\documentclass[11pt]{amsart}
\usepackage[margin=2.5cm]{geometry}                
\usepackage{amssymb,amsmath, mathabx, mathtools}  
\usepackage[colorlinks,
             linkcolor=blue,
             citecolor=green!70!black,
             pdfproducer={pdfLaTeX},
             pdfpagemode=None,
             bookmarksopen=true
             bookmarksnumbered=true]{hyperref}
\usepackage{cleveref}

\newcommand\arXiv[1]{\href{http://arxiv.org/abs/#1}{\nolinkurl{arXiv:#1}}}
\newcommand\MRnumber[1]{\href{http://www.ams.org/mathscinet-getitem?mr=#1}{\nolinkurl{MR#1}}}
\newcommand\ourDOI[1]{\href{http://dx.doi.org/#1}{\nolinkurl{DOI:#1}}}
\newcommand\MAILTO[1]{\href{mailto:#1}{\nolinkurl{#1}}}

\usepackage{tikz}
\usetikzlibrary{matrix,arrows,decorations.pathmorphing,decorations.pathreplacing,decorations.markings}
\tikzset{
    dot/.style={circle,draw,fill,inner sep=1pt},
    arrow/.style={->,thick,shorten <=2pt,shorten >=2pt},
    twoarrow/.style={double,double distance=1.5pt,shorten <=9pt,shorten >=10pt,decoration={markings,mark=at position -8pt with {\arrow[scale=1.75]{>}}},preaction={decorate}},
    twoarrowlonger/.style={double,double distance=1.5pt,shorten <=5pt,shorten >=6pt,decoration={markings,mark=at position -4pt with {\arrow[scale=1.75]{>}}},preaction={decorate}},
    twoarrowshorter/.style={double,double distance=1.5pt,shorten <=13pt,shorten >=14pt,decoration={markings,mark=at position -12pt with {\arrow[scale=1.75]{>}}},preaction={decorate}},
    twoarrowshorthead/.style={double,double distance=1.5pt,shorten <=9pt,shorten >=20pt,decoration={markings,mark=at position -18pt with {\arrow[scale=1.75]{>}}},preaction={decorate}},
    threearrowpart1/.style={ thick,double,double distance=3pt,shorten <=9pt,shorten >=11pt},
    threearrowpart2/.style={ thick,shorten <=9pt,shorten >=10pt},
    threearrowpart3/.style={ shorten <=9pt,shorten >=10pt,decoration={markings,mark=at position -8pt with {\arrow[scale=3]{>}}},preaction={decorate}},
fourarrowpart1/.style={thick, double,double distance=4pt,shorten <=1pt,shorten >=2.75pt},
fourarrowpart2/.style={thick,  double,double distance=1pt,shorten <=1pt,shorten >=1.25pt,decoration={markings,mark=at position -.05pt with {\arrow[scale=3,ultra thin]{>}}},preaction={decorate} }
}
\usepackage{tikz-cd}

\newcommand\threearrow{%
\begin{tikzpicture}
  \path coordinate (l) ++(3ex,0) coordinate (r);
  \draw[,double,double distance=3pt,shorten <=1pt,shorten >=2.5pt] (l) -- (r);
  \draw[shorten <=1pt,shorten >=.5pt,decoration={markings,mark=at position -.05pt with {\arrow[scale=3,ultra thin]{>}}},preaction={decorate}] (l) -- (r);
\end{tikzpicture}%
}

\newcommand\fourarrow{%
\begin{tikzpicture}
  \path coordinate (l) ++(3ex,0) coordinate (r);
  \draw[double,double distance=4pt,shorten <=1pt,shorten >=2.75pt] (l) -- (r);
  \draw[double,double distance=1pt,shorten <=1pt,shorten >=1.25pt,decoration={markings,mark=at position -.05pt with {\arrow[scale=3,ultra thin]{>}}},preaction={decorate}] (l) -- (r);
\end{tikzpicture}%
}

\newcommand\threearrowlong{%
\begin{tikzpicture}
  \path coordinate (l) ++(5ex,0) coordinate (r);
  \draw[,double,double distance=3pt,shorten <=1pt,shorten >=2.5pt] (l) -- (r);
  \draw[shorten <=1pt,shorten >=.5pt,decoration={markings,mark=at position -.05pt with {\arrow[scale=3,ultra thin]{>}}},preaction={decorate}] (l) -- (r);
\end{tikzpicture}%
}

\newcommand\fourarrowlong{%
\begin{tikzpicture}
  \path coordinate (l) ++(5ex,0) coordinate (r);
  \draw[double,double distance=4pt,shorten <=1pt,shorten >=2.75pt] (l) -- (r);
  \draw[double,double distance=1pt,shorten <=1pt,shorten >=1.25pt,decoration={markings,mark=at position -.05pt with {\arrow[scale=3,ultra thin]{>}}},preaction={decorate}] (l) -- (r);
\end{tikzpicture}%
}

\newcommand\ourunderbrace[2]{%
\begin{tikzpicture}[baseline=(mainnode.base)]
  \node[anchor=base] (mainnode) {$#1$};
  \path (mainnode.south east) ++(0,3pt) coordinate (se); \path (mainnode.south west) ++(0,3pt) coordinate (sw);
  \draw[decorate,decoration={brace,amplitude=6pt},thick] (se) -- (sw);
  \path (mainnode.south) ++(0,-3pt) node[anchor=north] {$\scriptstyle #2$};
\end{tikzpicture}%
}

\newtheorem{proposition}[subsection]{Proposition}
\newtheorem{theorem}[subsection]{Theorem}
\newtheorem{definition}[subsection]{Definition}
\newtheorem{theoremdefinition}[subsection]{Theorem/Definition}
\newtheorem{lemma}[subsection]{Lemma}
\newtheorem{corollary}[subsection]{Corollary}

\renewcommand\qedhere{\hfill\ensuremath\Box}

\theoremstyle{definition}
\newtheorem{remarknodiamond}[subsection]{Remark}
\newenvironment{remark}{\begin{remarknodiamond}}{\hfill\ensuremath\Diamond\end{remarknodiamond}}
\newtheorem{examplenodiamond}[subsection]{Example}
\newenvironment{example}{\begin{examplenodiamond}}{\hfill\ensuremath\Diamond\end{examplenodiamond}}

\crefname{section}{Section}{Section}
\crefformat{section}{#2Section~#1#3} 
\Crefformat{section}{#2Section~#1#3} 

\crefname{definition}{Definition}{Definitions}
\crefformat{definition}{#2Definition~#1#3} 
\Crefformat{definition}{#2Definition~#1#3} 

\crefname{lemma}{Lemma}{Lemmas}
\crefformat{lemma}{#2Lemma~#1#3} 
\Crefformat{lemma}{#2Lemma~#1#3} 

\crefname{proposition}{Proposition}{Propositions}
\crefformat{proposition}{#2Proposition~#1#3} 
\Crefformat{proposition}{#2Proposition~#1#3} 

\crefname{corollary}{Corollary}{Corollaries}
\crefformat{corollary}{#2Corollary~#1#3} 
\Crefformat{corollary}{#2Corollary~#1#3} 

\crefname{theorem}{Theorem}{Theorems}
\crefformat{theorem}{#2Theorem~#1#3} 
\Crefformat{theorem}{#2Theorem~#1#3}

\crefname{remark}{Remark}{Remarks}
\crefformat{remark}{#2Remark~#1#3} 
\Crefformat{remark}{#2Remark~#1#3} 

\crefname{remarknodiamond}{Remark}{Remarks}
\crefformat{remarknodiamond}{#2Remark~#1#3} 
\Crefformat{remarknodiamond}{#2Remark~#1#3}

\crefname{example}{Example}{Examples}
\crefformat{example}{#2Example~#1#3} 
\Crefformat{example}{#2Example~#1#3} 

\crefname{examplenodiamond}{Example}{Examples}
\crefformat{examplenodiamond}{#2Example~#1#3} 
\Crefformat{examplenodiamond}{#2Example~#1#3}

\usepackage[utf8]{inputenc}
\DeclareFontFamily{U}{min}{}
\DeclareFontShape{U}{min}{m}{n}{<-> udmj30}{}
\newcommand\yo{\!\text{\usefont{U}{min}{m}{n}\symbol{'207}}\!}

\usepackage{bbm}

\newcommand\Aa{\mathcal A} \newcommand\cA\Aa
\newcommand\Bb{\mathcal B} \newcommand\cB\Bb
\newcommand\Cc{\mathcal C} \newcommand\cC\Cc
\newcommand\Dd{\mathcal D} \newcommand\cD\Dd
\newcommand\Ee{\mathcal E} \newcommand\cE{\Ee}
\newcommand\cF{\mathcal F}
\newcommand\cG{\mathcal G}
\newcommand\cM{\mathcal M}

\newcommand\Ss{\mathcal S}  \newcommand\cS\Ss

\newcommand\cV{\mathcal V}
\newcommand\cX{\mathcal X}
\newcommand\cY{\mathcal Y}
\newcommand\cZ{\mathcal Z}

\newcommand\rB{\mathrm B}

\newcommand\bK{\mathbb K}
\newcommand\rL{\mathrm L}

\newcommand\NN{\mathbb N} \newcommand\bN\NN
\newcommand\bR{\mathbb R}
\newcommand\bZ{\mathbb Z}

\newcommand\h{\mathrm{h}}

\newcommand\inv{\mathrm{inv}}
\DeclareMathOperator\Gr{Gr}
\newcommand\id{\mathrm{id}}
\DeclareMathOperator\maps{maps}
\newcommand\inj{\mathrm{inj}}
\newcommand\proj{\mathrm{proj}}
\newcommand\op{\mathrm{op}}
\newcommand\Fin{\cat{Fin}}

\newcommand\fr{{\mathrm{fr}}}
\newcommand\fd{{\mathrm{fd}}}
\DeclareMathOperator\ev{ev}
\DeclareMathOperator\coev{coev}

\newcommand{\unit}{\mathbbm{1}}

\newcommand\mono\hookrightarrow
\newcommand\longto\longrightarrow
\newcommand\from\leftarrow

\newcommand\larunat{{\reflectbox{\ensuremath{\scriptstyle\natural}}}}

\newcommand\stratified{{\tikz[baseline=(bottom)]{\draw (-2.5pt,-2.5pt) coordinate (bottom) rectangle (2.5pt,2.5pt); \draw (0,-4pt) -- (0,4pt); \draw (0,0) node[fill=black, inner sep=1pt,circle] {};}}}

\DeclareMathOperator\Fact{Fact}

\newcommand\smallbox[1]{\tikz[baseline=(box.base)]\node[draw,rectangle,inner sep=1pt] (box) {\tiny\rm #1};}

\newcommand\define[1]{\emph{#1}}
\newcommand\cat[1]{\textsc{#1}}

\DeclareMathOperator\Alg{\cat{Alg}}
\DeclareMathOperator\Covers{\cat{Covers}}

\newcounter{notused}

\newcommand{\Subsection}[1]{

\mbox{}

\noindent\refstepcounter{notused}\textbf{#1.}\addcontentsline{toc}{subsection}{\mbox{}\quad\quad\quad#1}
\\
\indent}

\usepackage{lipsum}

\title[(Op)lax transfors, twisted QFTs, and even higher Morita categories]{(Op)lax natural transformations, twisted quantum field theories, and ``even higher'' Morita categories}
\author{Theo Johnson-Freyd and Claudia Scheimbauer}
\date{\today}
\email{\MAILTO{theojf@math.northwestern.edu} \\ \MAILTO{scheimbauer@mpim-bonn.mpg.de}}

\begin{document}
\begin{abstract}
Motivated by the challenge of defining twisted quantum field theories in the context of higher categories, we develop a general framework for lax and oplax transformations and their higher analogs between strong $(\infty, n)$-functors. We construct a double $(\infty,n)$-category built out of the target $(\infty, n)$-category governing the desired diagrammatics. We define (op)lax transformations as functors into parts thereof, and an (op)lax twisted field theory to be a symmetric monoidal (op)lax natural transformation between field theories. We verify that lax trivially-twisted relative field theories are the same as absolute field theories. As a second application, we extend the higher Morita category of $E_d$-algebras in a symmetric monoidal $(\infty, n)$-category $\cC$ to an $(\infty, n+d)$-category using the higher morphisms in $\cC$.
\end{abstract}
\maketitle

\tableofcontents

\section{Introduction}

This paper provides a universal answer to the following question.  Suppose that one is interested in objects equipped with some structure in a fixed higher category.  What are the \define{lax} and \define{oplax} (higher) transformations between such objects?  For example, what are the \emph{lax} and \emph{oplax} (higher) transformations between (strong) functors?

Although the answer to this question is interesting in the abstract, our motivation came from two particular applications in quantum field theory.    Our first motivation was to make precise the notion of ``twisted quantum field theory'' proposed in \cite{MR2742432} (and the closely related notion of ``relative quantum field theory'' from \cite{FreedTeleman2012}), which requires a lax version of ``symmetric monoidal natural transformation between symmetric monoidal $(\infty,n)$-functors.'' In the fully local case, the notion is closely related to that of a boundary field theory. Our second motivation was to extend the construction of ``higher'' Morita categories due to Calaque and the second author in \cite{ClaudiaDamien2} and to Haugseng \cite{HaugsengEn} to build ``even higher'' Morita categories in which higher-categorical intertwiners between bimodules are allowed --- in particular, the $(\infty,4)$-category of braided monoidal categories, monoidal categories, bimodule categories, bimodule functors, and bimodule intertwiners proposed by 
\cite[Section 9]{WalkerTQFT}, \cite{DSPS},  and \cite[Conjecture 6.5]{BZBJ}.
(Op)lax natural transformations and their higher cousins should have other applications we have not touched upon; for example, the paper \cite{HSS2015} has already applied our definition of ``symmetric monoidal oplax transfors'' to develop a notion of ``higher trace'' suitable for categorifying the Chern character.

In the remainder of the introduction we give a rather detailed account as an end-user guide avoiding the technicalities of higher categories.

\Subsection{Twisted quantum field theories and (op)lax natural transformations}

Atiyah \cite{MR1001453}, building on work of Segal \cite{MR2079383}, and extended by \cite{MR1256993,BaeDol95} and others, famously proposed that the mathematical definition of \define{quantum field theory} should be in terms of symmetric monoidal functors out of a suitable category of space-times.  More precisely, suppose that one wants to study $d$-dimensional quantum field theories defined in terms of some background geometry $\cG$ such as framings, orientations, but also conformal or Euclidean structures. To package these ``spacetimes'' together to obtain a symmetric monoidal $(\infty,n)$-category $\cat{Bord}^\cG_{d-n,\dots,d}$ for $n\leq d$, this geometry should make sense on manifolds of dimension $d-n+k$ which serve as $k$-morphisms for $k\leq n$; such field theories are called \define{$n$-extended}, and \define{fully extended, or fully local} when $n=d$. (See \cite{KolyaLong,MR2742432} for discussion of categories of cobordisms equipped with geometry and their role in quantum field theory.)  As target, one should consider a symmetric monoidal $(\infty,n)$-category $\cat{Vect}_n$ whose $n$-morphisms are numbers or matrices and whose $(n-1)$-morphisms are Hilbert spaces.  Then a \define{quantum field theory built on $\cat{Bord}^\cG_{d-n,\dots,d}$ with values in $\cat{Vect}_n$} is defined to be a symmetric monoidal functor $Z : \cat{Bord}^\cG_{d-n,\dots,d} \to \cat{Vect}_n$.

Functors $Z : \cat{Bord}^\cG_{d-n,\dots,d} \to \cat{Vect}_n$ capture many important examples of quantum field theories, but by no means all of them.  For example, they miss already conformal field theories in the sense of \cite{MR2079383} in which the space of conformal blocks (i.e.\ the modular functor) is nontrivial.  
 A generalization of the above definition was proposed in \cite{MR2742432} under the name \define{twisted quantum field theory}\footnote{The choice of terminology is unrelated to the ``topological twisting'' used for example in  \cite{MR958805}; the name is based instead on ``twisted K-theory''.}.  Suppose that one wants to model quantum field theories on a symmetric monoidal $(\infty,n)$-category $\cat{Bord}^\cG_{d-n,\dots,d}$.  A \define{twist} for such theories is a symmetric monoidal functor $T: \cat{Bord}^\cG_{d-n,\dots,d} \to \cat{Vect}_{n+1}$, meaning in particular that it takes top-dimensional ``spacetimes'' to vector or Hilbert spaces rather than numbers.  There is of course a trivial such functor $\unit : \cat{Bord}^\cG_{d-n,\dots,d} \to \cat{Vect}_{n+1}$, which takes every $k$-morphism in $\cat{Bord}^\cG_{d-n,\dots,d}$ to the identity $k$-morphism on the unit object $\unit \in \cat{Vect}_{n+1}$.  A \define{twisted quantum field theory} is supposed to be a symmetric monoidal natural transformation $Z : \unit \Rightarrow T$ of symmetric monoidal functors.

$$\begin{tikzpicture}
  \path node (B) {} node[anchor=east] {$\cat{Bord}^\cG_{d-n,\dots,d}$} +(2,0) node (C) {} node[anchor=west] {$\cC$};
  \draw[arrow] (B) .. controls +(1,-.5) and +(-1,-.5) .. coordinate (t)  (C) ;
  \draw[arrow] (B) .. controls +(1,.5) and +(-1,.5) .. coordinate (s)  (C);
  \draw[twoarrowlonger] (s) node[above] {$\scriptstyle \unit$} --   (t) node[below] {$\scriptstyle T$};
\end{tikzpicture}$$

\begin{remark} \label{remark.relative qft}
  A closely related notion, called \define{relative quantum field theory}, is proposed in \cite{FreedTeleman2012}.  In all their examples, the ``twist'' data is not just a functor $T: \cat{Bord}^\cG_{d-n,\dots,d} \to \cat{Vect}_{n+1}$, but in fact is the restriction of a functor $\tilde T : \cat{Bord}_{d-n,\dots,d+1} \to \cat{Vect}_{n+1}$, where $\cat{Bord}_{d-n,\dots,d+1}$ is the symmetric monoidal $(\infty,n+1)$-category of topological cobordisms up to dimension $d+1$.  Like twisted quantum field theories, \define{relative quantum field theories} with twist $T$ are thought of as transformations $Z : \unit \Rightarrow T$, but in \cite{FreedTeleman2012} it is suggested that the {definition} of such transformations should be in terms of $\cG$-geometric ``boundary conditions'' for the quantum field theory $\tilde T$. We will discuss a comparison in the fully local case in \cref{thm.comparison bdry and twisted}, see also \cite{FV2014}.
\end{remark}

Recall that given functors between bicategories $F,G: \cB \rightrightarrows \cC$, there are three natural notions of ``natural transformation'' $\eta : F \Rightarrow G$.  In all three cases, the data of $\eta$ includes, for each object $B \in \cB$, a morphism $\eta(B) : F(B) \to G(B)$ in $\cC$.  The difference is in how the three types of transformations handle ``naturality.''  If $\eta$ is \define{strong} (also called ``pseudo,'' to distinguish it from ``strict natural transformation between strict $2$-categories'') and $B_1 \overset b \to B_2$ is a 1-morphism in $B$, then the data of $\eta$ includes a 2-isomorphism $\eta(b) : G(b) \circ \eta(B_1) \cong \eta(B_2) \circ F(b)$, which is in turn compatible with compositions and natural for 2-morphisms in $\cB$.  But if $\eta$ is \define{lax}, then the data of $\eta$ only includes a 2-morphism $\eta(b) : G(b) \circ \eta(B_1) \Rightarrow \eta(B_2) \circ F(b)$ which may not be invertible, and if $\eta$ is \define{oplax} then we only have a 2-morphism $\eta(b) : \eta(B_2) \circ F(b) \Rightarrow G(b) \circ \eta(B_1)  $.

$$
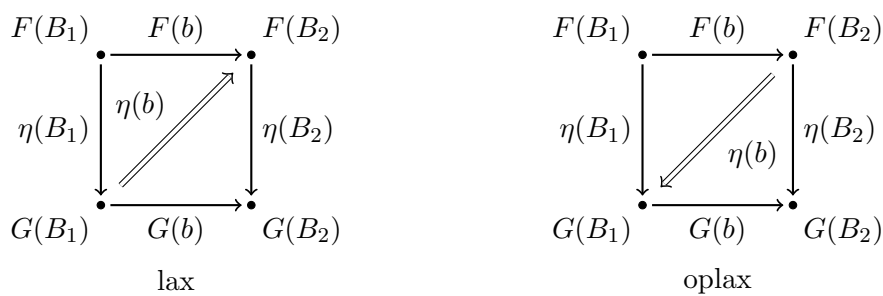
\begin{tikzpicture}[baseline=(r.base)]
  \path node[dot] (sl) {} +(2,0) node[dot] (sr) {} +(0,-2) node[dot] (tl) {} +(2,-2) node[dot] (tr) {} +(1,-3) node {lax};
  \draw[arrow] (sl) node[anchor = south east] {$F(B_{1})$} -- node[auto] {$F(b)$} (sr) node[anchor = south west] {$F(B_{2})$};
  \draw[arrow] (sl)  -- node[auto,swap] {$\eta(B_{1})$} (tl);
  \draw[arrow] (tl) node[anchor = north east] {$G(B_{1})$} -- node[auto,swap] {$G(b)$}  (tr) node[anchor = north west] {$G(B_{2})$};
  \draw[arrow] (sr) -- node[auto] (r) {$\eta(B_{2})$} (tr);
  \draw[twoarrow] (tl) -- node[auto] {$\eta(b)$} (sr);
\end{tikzpicture}
\hspace{1in}
\begin{tikzpicture}[baseline=(r.base)]
  \path node[dot] (sl) {} +(2,0) node[dot] (sr) {} +(0,-2) node[dot] (tl) {} +(2,-2) node[dot] (tr) {} +(1,-3) node {oplax};
  \draw[arrow] (sl) node[anchor = south east] {$F(B_{1})$} -- node[auto] {$F(b)$} (sr) node[anchor = south west] {$F(B_{2})$};
  \draw[arrow] (sl)  -- node[auto,swap] {$\eta(B_{1})$} (tl);
  \draw[arrow] (tl) node[anchor = north east] {$G(B_{1})$} -- node[auto,swap] {$G(b)$}  (tr) node[anchor = north west] {$G(B_{2})$};
  \draw[arrow] (sr) -- node[auto] (r) {$\eta(B_{2})$} (tr);
  \draw[twoarrow] (sr) -- node[auto] {$\eta(b)$} (tl);
\end{tikzpicture}
$$

\begin{remark}
  The literature provides a majority opinion, but by no means a consensus, on the question of which natural transformations should be called ``lax'' and which ``oplax.''  Our use agrees with the majority, including the original paper~\cite{MR0220789}; \cite{MR1953060,MR2063092} is a notable example 
  that reverses which natural transformations are called ``lax'' and which ``oplax.''   
  Discussion of the costs and benefits of the standard terminology can be found at~\cite{lax-versus-oplax}.
\end{remark}

Abstract nonsense of  $(\infty,n)$-categories provides a straightforward generalization of ``strong'' natural transformations.  Indeed, the $(\infty,1)$-category of  $(\infty,n)$-categories is Cartesian closed, so given  $(\infty,n)$-categories $\cB$ and $\cC$, there is naturally an  $(\infty,n)$-category $[\cB,\cC]$ whose objects can be identified with  functors $\cB \to \cC$ and whose 1-morphisms are the  strong natural transformations.   But this is not good enough for quantum field theory: strong natural transformations lead to very few examples of twisted quantum field theories $Z : \unit \to T$; both \cite{MR2742432,FV2014} call instead for lax natural transformations.
 The failure of strong transformations  can be seen most clearly in topological quantum field theory where all objects $B \in \cB=\cat{Bord}_{d-n,\dots,d}$ have duals.  Suppose that $B \in \cB$ has a dual object $B^*$ and that $\eta : F \Rightarrow G$ is a symmetric monoidal strong natural transformation.  Then the data of $\eta$ includes 1-morphisms $\eta(B) : F(B) \to G(B)$ and $\eta(B^*) : F(B)^* \simeq F(B^*) \to G(B^*) \simeq G(B)^*$ and also naturality squares coming from the pairing and copairing for $B$ which, since $\eta$ is strong, are filled in by equivalences.  A straightfoward diagram chase then shows that $\eta(B)$ and $\eta(B^*)^* : G(B) \simeq G(B)^{**} \to F(B)^{**} \simeq F(B)$ are each others' inverses, and so $\eta$ is invertible.  We are therefore faced with our first motivating question: How to define lax and oplax natural transformations between strong functors of symmetric monoidal $(\infty,n)$-categories?

To explain our proposed solution, we mention one more way of thinking about strong natural transformations.  Let $\Theta^{(1)} = \{\tikz{ \path node[dot] (l) {} +(.5,0) node[dot](r) {} +(0,-3pt) coordinate (base); \draw[arrow] (l) -- (r);}\}$ denote the ``walking $1$-morphism'' (compare \cref{defn.walkingimorphism}).  It has the property that for any $(\infty,n)$-category $\cC$, functors $\Theta^{(1)} \to \cC$ are the same as 1-morphisms in $\cC$.  Cartesian closedness of the $(\infty,1)$-category of $(\infty,n)$-categories provides an $(\infty,n)$-category $[\Theta^{(1)},\cC]$ of functors $\Theta^{(1)} \to \cC$ along with ``source'' and ``target'' functors $s,t: [\Theta^{(1)},\cC] \to \cC$ corresponding to the two inclusions $\Theta^{(0)} = \{\tikz \path node[dot] (l) {} +(0,-3pt) coordinate (base); \} \rightrightarrows \Theta^{(1)}$ (note that $\cC \simeq [\Theta^{(0)},\cC]$).  Given functors $F,G : \cB \rightrightarrows \cC$, a \define{strong natural transformation} $F \Rightarrow G$ is nothing but a functor $\eta : \cB \to [\Theta^{(1)},\cC]$ satisfying $s \circ \eta = F$ and $t \circ \eta = G$.

Similarly, in \cref{defn.(op)lax arrow category} we will construct, for any $(\infty,n)$-category $\cC$, two new $(\infty,n)$-categories which we will call $\cC^\downarrow$ and $\cC^\rightarrow$ with functors $s,t: \cC^\downarrow \rightrightarrows \cC$ and $s,t : \cC^\to \rightrightarrows \cC$.  All three $(\infty,n)$-categories $[\Theta^{(1)},\cC]$, $\cC^\downarrow$, and $\cC^\rightarrow$ have the same objects --- the one-morphisms in $\cC$ --- although we think of them slightly differently: in $\cC^\downarrow$ we think of the objects as ``vertical'' arrows 
$$ \begin{tikzpicture}[baseline=(c.base)]
  \path node[dot] (sl) {} +(0,-1) node[dot] (sr) {};
  \draw[arrow] (sl) node[anchor = east] (C) {$C_{1}$} -- node[auto] (c) {$c$} (sr) node[anchor = east] {$C_{2}$};
\end{tikzpicture} \quad \in \cC^\downarrow, $$
whereas we think of the objects of $\cC^\rightarrow$ as ``horizontal'' arrows
$$ \begin{tikzpicture}[baseline=(C.base)]
  \path node[dot] (sl) {} +(1,0) node[dot] (sr) {};
  \draw[arrow] (sl) node[anchor = east] (C) {$C_{1}$} -- node[auto] (c) {$c$} (sr) node[anchor = west] {$C_{2}$};
\end{tikzpicture} \quad \in \cC^{\rightarrow}.$$
The differences between $[\Theta^{(1)},\cC]$, $\cC^\downarrow$, and $\cC^\rightarrow$ are in the higher morphisms.  A 1-morphism in $[\Theta^{(1)},\cC]$ is by definition a commuting-up-to-isomorphism square in $\cC$, i.e.\ a diagram in $\cC$ of shape $\Theta^{(1)} \times \Theta^{(1)}$.  The 1-morphisms on $\cC^\downarrow$ and $\cC^\rightarrow$, for comparison, are each squares that commute only up to non-invertible 2-morphism.  A typical 1-morphism $c \to d$ in $\cC^{\downarrow}$ is a diagram in $\cC$ of shape
$$
\begin{tikzpicture}[baseline=(r.base)]
  \path node[dot] (sl) {} +(2,0) node[dot] (sr) {} +(0,-2) node[dot] (tl) {} +(2,-2) node[dot] (tr) {};
  \draw[arrow] (sl) node[anchor = south east] {$C_1$} -- node[auto] {} (sr) node[anchor = south west] {$D_{1}$};
  \draw[arrow] (sl)  -- node[auto,swap] {$c$} (tl);
  \draw[arrow] (tl) node[anchor = north east] {$C_{2}$} -- node[auto,swap] {}  (tr) node[anchor = north west] {$D_2$};
  \draw[arrow] (sr) -- node[auto] (r) {$d$} (tr);
  \draw[twoarrow] (tl) -- node[auto] {$f$} (sr);
\end{tikzpicture}
$$
and a typical 1-morphism $c \to d$ in $\cC^{\rightarrow}$ looks similar, but with reversed direction of the 2-morphism filling the square 
$$
\begin{tikzpicture}[baseline=(middle)]
  \path node[dot] (sl) {} +(2,0) node[dot] (sr) {} +(0,-2) node[dot] (tl) {} +(2,-2) node[dot] (tr) {};
  \draw[arrow] (sl) node[anchor = south east] {$C_1$} -- node[auto] {$c$} (sr) node[anchor = south west] {$C_2$};
  \draw[arrow] (sl)  -- node[auto,swap] {} (tl);
  \draw[arrow] (tl) node[anchor = north east] {$D_1$} -- node[auto,swap] {$d$}  (tr) node[anchor = north west] {$D_2$};
  \draw[arrow] (sr) -- node[auto] (r) {} (tr);
  \draw[twoarrow] (tl) -- node (middle) {} node[auto] {$f$} (sr);
\end{tikzpicture}
\quad \text{ or, flipping the diagram, } \quad
\begin{tikzpicture}[baseline=(middle)]
  \path node[dot] (sl) {} +(2,0) node[dot] (sr) {} +(0,-2) node[dot] (tl) {} +(2,-2) node[dot] (tr) {};
  \draw[arrow] (sl) node[anchor = south east] {$C_1$} -- node[auto] {} (sr) node[anchor = south west] {$D_{1}$};
  \draw[arrow] (sl)  -- node[auto,swap] {$c$} (tl);
  \draw[arrow] (tl) node[anchor = north east] {$C_{2}$} -- node[auto,swap] {}  (tr) node[anchor = north west] {$D_2$};
  \draw[arrow] (sr) -- node[auto] (r) {$d$} (tr);
  \draw[twoarrow] (sr) --node (middle) {} node[auto] {$f$} (tl);
\end{tikzpicture}
.
$$
In general, the $k$-morphisms of $\cC^\downarrow$ and $\cC^\rightarrow$ will be defined in terms of diagrams of certain shapes (which we will call $\Theta^{(k);(1)}$ and $\Theta^{(1);(k)}$ respectively) in $\cC$.  For example, the 2-morphisms $f \Rightarrow g$ in $\cC^\downarrow$ and $\cC^\rightarrow$ look, respectively, like
$$
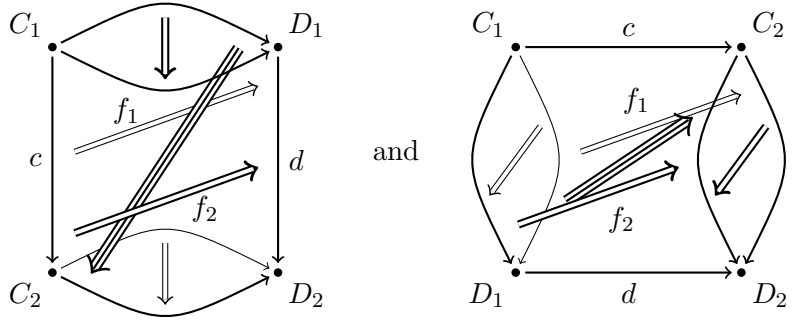
\begin{tikzpicture}[baseline=(middle)]
  \path node[dot] (a) {} node[anchor= south east] {$C_{1}$} +(3,0) node[dot] (b) {} node[anchor= south west] {$D_{1}$} +(0,-3) node[dot] (c) {} node[anchor= north east] {$C_{2}$} +(3,-3) node[dot] (d) {} node[anchor= north west] {$D_{2}$}+(0,-1.5) coordinate (middle);
  \path (b) +(-.35,.25) coordinate (alpha1);
  \path (c) +(.35,-.25) coordinate (alpha2);
  \draw[arrow] (a) -- coordinate (l1) node[anchor=east] {$c$} coordinate[very near end] (r1) (c);
  \draw[arrow] (b) -- coordinate[very near start] (l2) coordinate (r2) node[anchor=west] (rightlable) {$d$} (d);
  \draw[twoarrow,thin] (l1) -- node[auto,pos=.4,inner sep=1pt] {$f_{1}$} (l2);
  \draw[arrow] (a) .. controls +(1.5,.75) and +(-1.5,.75) .. coordinate (ss)  (b);
  \draw[arrow,thin] (c) .. controls +(1.5,.75) and +(-1.5,.75) .. coordinate (ts)  (d);
  \draw[arrow] (c) .. controls +(1.5,-.75) and +(-1.5,-.75) .. coordinate (tt) (d);
  \draw[twoarrowlonger,thin] (ts) --  (tt);
  \draw[threearrowpart1] (alpha1) --  (alpha2);
  \draw[threearrowpart2] (alpha1) --  (alpha2);
  \draw[threearrowpart3] (alpha1) --  (alpha2);
  \draw[arrow,thick] (a) .. controls +(1.5,-.75) and +(-1.5,-.75) .. coordinate (st) (b);
  \draw[twoarrowlonger,thick] (ss) --  (st);
  \draw[twoarrow,thick] (r1) -- node[auto,swap,pos=.6,inner sep=1pt] {$f_{2}$} (r2);
\end{tikzpicture} 
\quad \text{ and } \quad
\begin{tikzpicture}[baseline=(middle)]
  \path node[dot] (a) {} +(3,0) node[dot] (b) {} +(0,-3) node[dot] (c) {} +(3,-3) node[dot] (d) {} +(0,-1.5) coordinate (middle);
  \draw[arrow] (a) .. controls +(-.75,-1.5) and +(-.75,1.5) .. coordinate[near end] (lt) coordinate[very near end] (lt2) (c);
  \draw[arrow,thin] (a) .. controls +(.75,-1.5) and +(.75,1.5) .. coordinate[near start] (ls) coordinate (ls2) (c);
  \draw[twoarrowlonger] (ls) -- (lt);
  \draw[arrow] (b) .. controls +(.75,-1.5) and +(.75,1.5) .. coordinate[near start] (rs) coordinate[very near start] (rs0) (d);
  \draw[twoarrow] (ls2) -- coordinate[near end] (s) node[auto] {$f_1$} (rs0);
  \draw[arrow] (b) .. controls +(-.75,-1.5) and +(-.75,1.5) .. coordinate[near end] (rt) coordinate (rt0) (d);
  \draw[twoarrowlonger,thick] (rs) -- (rt);
  \draw[arrow] (a) node[anchor = south east] {$C_1$} -- node[auto] {$c$} (b) node[anchor = south west] {$C_2$};
  \draw[arrow] (c) node[anchor = north east] {$D_1$} -- node[below] {$d$} (d) node[anchor = north west] {$D_2$};
  \path (lt2) -- coordinate[near start] (t) (rt0);
  \draw[threearrowpart1] (t) --  (s); 
  \draw[threearrowpart2] (t) -- (s); 
  \draw[threearrowpart3] (t) -- (s); 
  \draw[twoarrow,thick] (lt2) -- node[auto,swap] {$f_2$} (rt0);
\end{tikzpicture}.$$

Our diagrammatics are an analog of the construction from \cite{MR1667313}.  The $(\infty,n)$-categories $\cC^\downarrow$ and $\cC^\rightarrow$ are built so that the following generalizes the usual notions from bicategories.

\begin{definition}\label{defn.oplax-transformation}
  Let $\cB$ and $\cC$ be $(\infty,n)$-categories and $F,G: \cB \rightrightarrows \cC$ functors.  
  A \define{lax natural transformation} $\eta : F \Rightarrow G$ is a  functor $\eta : \cB \to \cC^{\downarrow}$ such that $s\circ\eta = F$ and $t\circ \eta = G$.
  An \define{oplax natural transformation} $\eta: F \Rightarrow G$ is a  functor $\eta: \cB \to \cC^{\rightarrow}$ such that $s\circ\eta = F$ and $t\circ \eta = G$.
\end{definition}

When $\cC$ is symmetric monoidal, so too will be $\cC^\downarrow$ and $\cC^\rightarrow$, and the obvious symmetric monoidal variant of \cref{defn.oplax-transformation} is given in \cref{defn.sym mon lax transfor}.
Specializing to the case when $\cB = \cat{Bord}^\cG_{d-n,\dots,d}$ and $\cC = \cat{Vect}_{n+1}$, we have:
\begin{definition}\label{defn.T-twisted theory introduction}
  Let $T : \cat{Bord}^\cG_{d-n,\dots,d} \to \cat{Vect}_{n+1}$ be a symmetric monoidal functor.  A \define{lax $T$-twisted quantum field theory} is a symmetric monoidal lax natural transformation $Z : \unit \Rightarrow T$, i.e.\ a functor $Z : \cat{Bord}^\cG_{d-n,\dots,d} \to (\cat{Vect}_{n+1})^{\downarrow}$ such that $s\circ Z = \unit$ is the trivial field theory and $t \circ Z = T$ is the given twist $T$.  An \define{oplax $T$-twisted relative field theory} is a symmetric monoidal oplax natural transformation $Z : \unit \Rightarrow T$, i.e.\ a functor $Z : \cat{Bord} \to \cC^{\to}$ such that $s\circ Z = \unit$ and $t\circ Z = T$.
\end{definition}

Inspection of the literature reveals that $T$-twisted quantum field theories are sometimes taken to be lax and sometimes oplax.  \emph{Oplax} $T$-twisted  field theories enjoy the property that for any closed cobordism $b\in \cat{Bord}$, the value $Z(b)$ of the field theory on $b$ is an ``element'' of the twist $T(b)$ (see \cref{eg.lax versus oplax T-twisted qfts}); this is often taken as the defining property of ``$T$-twisted quantum field theories.''  On the other hand, we would be able to recover the usual notion of functorial QFT when twisting by the trivial twist (c.f.\ \cite[Lemma 5.7]{MR2742432}). This suggests that \emph{lax} $T$-twisted quantum field theories better deserve the name, because of the following fact which we will prove in \cref{trivially twisted versus untwisted}: 
\begin{theorem}\label{intro lax is better}
  Lax \define{trivially-twisted}  field theories in a symmetric monoidal $(\infty,n+1)$-category $\cC$ --- i.e.\ $T$-twisted theories for $T = \unit$ the trivial field theory --- are the same as ``absolute'' (also called ``untwisted'') field theories valued in the ``looping'' $\Omega\cC$ of $\cC$ --- i.e.\ the symmetric monoidal $(\infty,n)$-category of endomorphisms of the unit object in $\cC$.
\end{theorem}

A good definition of $\cC=\cat{Vect}_{n+1}$ as desired above should satisfy that $\Omega\cat{Vect}_{n+1} \simeq \cat{Vect}_n$, so \cref{intro lax is better} indeed recovers (untwisted) quantum field theories $Z:\cat{Bord}^\cG_{d-n,\dots,d} \to \cat{Vect}_n$ as defined earlier.

With our definition of twisted quantum field theories at hand, we turn to the fully local topological case.  Let $\cat{Bord}_n^\fr$ denote the fully extended framed topological bordism category from \cite{Lur09,ClaudiaDamien1}; the \define{Cobordism Hypothesis} of \cite{Lur09} identifies field theories based on $\cat{Bord} = \cat{Bord}_n^\fr$ with {$n$-dualizable objects} in the target category.   Using the cobordism hypothesis, we can classify twisted field theories based on $\cat{Bord}_n^\fr$.  The following theorem is a special case of \cref{cor.lax twisted framed tfts}, which also covers ``higher twists'' between twisted field theories.

\begin{theorem}
  Let $\cC$ be a symmetric monoidal $(\infty,n+1)$-category and $T : \cat{Bord}_n^\fr \to \cC$ a fully extended framed twist.  A lax $T$-twisted fully extended framed field theory is classified by a 1-morphism $f: \unit \to T(\mathrm{pt})$ in $\cC$ which is \define{$n$-times left-adjunctible} in that it has a left-adjoint $f^L$, the unit and counit 2-morphisms witnessing the adjunction between $f$ and $f^L$ both have left adjoints, the unit and counit 3-morphisms of those two adjunctions all have left adjoints, and so on $n$ times.  An oplax $T$-twisted fully extended framed field theory is classified by a 1-morphism $f: \unit \to T(\mathrm{pt})$ in $\cC$ which is \define{$n$-times right-adjunctible} in a similar sense.
\end{theorem}

Again, using the Cobordism Hypothesis, this time the version for manifolds with singularities, we obtain in \cref{thm.comparison bdry and twisted} a comparison result to fully extended boundary field theories, proposed in \cite{FreedTeleman2012} as the definition of ``relative field theory''.
\begin{theorem}\label{cor.intro comparison to boundary field theory}
Let $\cat{Bord} = \cat{Bord}_n^{\fr}$ denote the fully extended framed topological bordism category from \cite{Lur09,ClaudiaDamien1} and $\cat{Bord}_{n+1}^{\fr,\partial}$ the fully extended framed bordism category with ``free boundaries'' described in Section 4.3 in \cite{Lur09}.  Let $\cC$ be a symmetric monoidal $(\infty,n+1)$-category with duals.  The following are equivalent:
\begin{enumerate}
  \item 1-morphisms in $\cC$ with source $\unit$,
  \item fully extended $n$-dimensional {boundary} field theories $\cat{Bord}_{n+1}^\partial \to \cC$,
  \item lax twisted field theories with trivial source,
  \item oplax twisted field theories with trivial source.
\end{enumerate}
\end{theorem}

\Subsection{Higher (op)lax transfors and even higher Morita categories}

Our second motivation was to provide a higher-categorical generalization of Morita theory.  Classical Morita theory concerns the \define{Morita bicategory} $\cat{Alg}(\cat{Vect}_\bK)$, first introduced in \cite{MR0220789}, whose objects are algebras over $\bK$, 1-morphisms are bimodules, and 2-morphisms are homomorphisms of bimodules; equivalence in this bicategory is the \define{Morita equivalence} introduced in \cite{MR0096700}.
Generalizing this construction, given a symmetric monoidal category $\cC$ we can assign its Morita bicategory $\cat{Alg}(\cC) = \cat{Alg}_1(\cC)$, which, when instead considering symmetric monoidal $(\infty,1)$-categories $\cC$, is merely the first in an infinite sequence of ``higher Morita categories'' $\cat{Alg}_d(\cC)$. By definition, the objects of $\cat{Alg}_d(\cC)$ are $E_d$-algebras in $\cC$, i.e.\ algebras for the little $d$-disks operad, which objects equipped with $d$ compatible associative unital multiplications.  The 1-morphisms are bimodules between $E_d$-algebras which are themselves $E_{d-1}$-algebras.  The 2-morphisms are $E_{d-2}$-bimodules between the 1-morphisms, with compatibility conditions.  One continues in this way until one gets to the $d$-morphisms, where a choice must be made: to generalize Morita's original work, one should take merely bimodules between $E_1$-algebras; for many applications it is better to take \emph{pointed} bimodules (also called $E_0$-algebras) between the $E_1$-algebras.  In either case the construction produces from $\cC$ a symmetric monoidal $(\infty,d)$-category $\cat{Alg}_d(\cC)$.

Under certain mild technical conditions guaranteeing that bimodule tensor product is associative (see \cref{defn.sifted}), the papers \cite{HaugsengEn,ClaudiaDamien2}, using different techniques, construct versions of the higher Morita $(\infty,d)$-categories $\Alg_d(\cC)$, the former without, the latter with pointings.  Both constructions in fact turn the $(\infty,1)$-category $\cC$ into an $(\infty,d+1)$-categorical extension of $\Alg_d(\cC)$ whose $0$- through $d$-morphisms are algebras and bimodules and whose $(d+1)$-morphisms are homomorphisms of bimodules.  For applications in quantum field theory, however, this does not suffice: the higher morphisms in $\cC$ should be used as well. For example, \cite[Conjecture 6.5]{BZBJ} requires an $(\infty,4)$-categorical version of $\Alg_2(\cat{Cat})$ whose objects are braided monoidal categories, 1-morphisms are monoidal categories, 2-morphisms are bimodule categories, 3-morphisms are bimodule functors, and 4-morphisms are bimodule natural transformations. Similarly, the category $\cat{TC}$ from \cite{DSPS1,DSPS2} should be a subcategory of an $(\infty,3)$-categorical version of $\Alg_1(\cat{Cat})$.

There are, however, three different types of ``bimodule functors'' between bimodule categories \cite{MR1007911}: in a \define{strong bimodule functor}, compatibility with the actions can be imposed up to isomorphism; in \define{lax} and \define{oplax} bimodule functors, compatibility is imposed just by possibly-non-invertible natural transformation, the direction of which depends on the choice of ``lax'' or ``oplax.''  Our categories $\cC^\downarrow$ and $\cC^\to$ clarify and extend this trichotomy, cf.~\cref{ex.Morita category}:

\begin{definition}\label{defn.oplax-homomorphism}
  Let $\cC$ be a symmetric monoidal $(\infty,n)$-category. Given bimodule objects $A$ and $B$ in $\cC$, a \define{lax morphism} $f: A \to B$ is a bimodule object $f$ in $\cC^{\downarrow}$ such that   $s(f)=A$ and $t(f)=B$.  An \define{oplax morphism} $f: A \to B$ is a bimodule object $f$ in $\cC^{\rightarrow}$ such that   $s(f)=A$ and $t(f)=B$.  A \define{strong morphism} is a bimodule object in $[\Theta^{(1)},\cC]$.
\end{definition}

More generally, the framework allows to define lax/oplax/strong morphisms of $P$-algebras, for ``$P$-algebra'' a type of algebraic object such as a colored operad such as the bimodule operad above or the associative operad, or even an $(\infty,n)$-category such as $\cat{Bord}_n$.

\begin{definition}\label{defn.oplax-homomorphism}
  Let $\cC$ be a symmetric monoidal $(\infty,n)$-category.  Given $P$-algebras $A$ and $B$ in $\cC$, a \define{lax morphism} $f: A \to B$ is a $P$-algebra $f$ in $\cC^{\downarrow}$ such that   $s(f)=A$ and $t(f)=B$.  An \define{oplax morphism} $f: A \to B$ is a $P$-algebra $f$ in $\cC^{\rightarrow}$ such that   $s(f)=A$ and $t(f)=B$.  A \define{strong morphism} is a $P$-algebra in $[\Theta^{(1)},\cC]$.  
\end{definition}

\begin{example}
  When $P$ is the associative operad and $\cC$ is the bicategory of categories, $P$-algebras in $\cC$ are precisely monoidal categories.  A lax morphism $A \to B$ between monoidal categories is nothing but a lax monoidal functor --- a functor that takes algebra objects to algebra objects.  An oplax morphism is an oplax monoidal functor, taking coalgebras to coalgebras.
\end{example}

In the higher Morita category, we want not just (op)lax bimodule morphisms between bimodules in an $(\infty,n)$-category $\cC$, but (op)lax bimodule 2-morphisms between (op)lax morphisms using the 2-morphisms in $\cC$, and so on.  Similarly, when considering (op)lax natural transformations between functors, one should moreover consider higher transformations between natural transformations, called \define{$k$-transfors} in \cite{MR1667313} (numbered so that \define{$0$-transfors} are functors and \define{$1$-transfors} are natural transformations).  We will define \define{lax} and \define{oplax $k$-transfors} between functors 
 analagously to our definitions of lax and oplax natural transformations and morphisms.  
Namely, given an $(\infty,n)$-category $\cC$ and $k \leq n$, we construct  $(\infty,n)$-categories $\cC^{\smallbox{lax}}_{(k)}$ and $\cC^{\smallbox{oplax}}_{(k)}$ in terms of diagrams in $\cC$ of certain shapes generalizing the construction above, the objects of which are the $k$-morphisms of $\cC$: when $k=1$, $\cC^{\smallbox{lax}}_{(k)} = \cC^\downarrow$ and $\cC^{\smallbox{oplax}}_{(k)} = \cC^\rightarrow$.  Completing the picture, let $\Theta^{(k)}$ denote the ``walking $k$-morphism'' (see \cref{defn.walkingimorphism}; for example $\Theta^{(2)} = \bigl\{
\begin{tikzpicture}[baseline=(base),scale=.35]
  \path node[dot] (a) {} +(2,0) node[dot] (c) {} +(0,-8pt) coordinate (base);
  \draw[arrow] (a) .. controls +(1,-1) and +(-1,-1) .. coordinate (t)   (c);
  \draw[arrow] (a) .. controls +(1,1) and +(-1,1) .. coordinate (s)  (c);
  \draw[twoarrowlonger] (s) --   (t);
\end{tikzpicture} \bigr\}$), and set $\cC^{\smallbox{strong}}_{(k)} = [\Theta^{(k)},\cC]$. For all three, denoting $\ast = $ ``lax'', ``oplax'', or ``strong'', there are functors $s,t : \cC^{\boxasterisk}_{(k)} \rightrightarrows \cC^{\boxasterisk}_{(k-1)}$.

\begin{definition}\label{intro lax transfor lax morphism}
  Choose $\ast = $ ``lax'', ``oplax'', or ``strong''.
  
  Given $(\infty,n)$-categories $\cB$ and $\cC$, a \define{$\ast$-$k$-transfor}
is a functor $\eta : \cB \to \cC^{\boxasterisk}_{(k)}$.  The \define{source} and \define{target} of the $\ast$-$k$-transfor $\eta$ are the $\ast$-$(k-1)$-transfors $s\circ \eta,t\circ\eta : \cB \rightrightarrows \cC^{\boxasterisk}_{(k-1)}$.

  Given a symmetric monoidal $(\infty,n)$-category $\cC$ 
a \define{$\ast$-$k$-morphism} of $P$-algebras in $\cC$ is a $P$-algebra $\alpha$ in $\cC^{\boxasterisk}_{(k)}$.  The \define{source} and \define{target} of the $\ast$-$k$-morphism $\alpha$ $\ast$-$(k-1)$-morphisms $s(\alpha), t(\alpha) \in \cC^{\boxasterisk}_{(k-1)}$.
\end{definition}

With \cref{intro lax transfor lax morphism} in place, given a symmetric monoidal $(\infty,n)$-category $\cC$, we can extend the higher Morita $(\infty,d)$-categories $\cat{Alg}_d(\cC)$ from \cite{HaugsengEn,ClaudiaDamien2} to three different \define{even higher} Morita $(\infty,d+n)$-categories $\cat{Alg}_d^{\mathrm{strong}}(\cC)$, $\cat{Alg}_d^{\mathrm{lax}}(\cC)$, and $\cat{Alg}_d^{\mathrm{oplax}}(\cC)$.  In all three cases, the $0$- through $d$-morphisms of $\cat{Alg}_d^\ast(\cC)$ are those of $\cat{Alg}_d(\cC)$ --- the $E_d$-algebras, $E_{d-1}$-bimodules, and so on --- but for $k=0,\dots,n$, the $(d+k)$-morphisms of $\cat{Alg}_d^\ast(\cC)$ are defined to be $\ast$-$k$-morphisms of bimodules.

More precisely, because the $\cat{Alg}_d(-)$ construction requires certain technical conditions on $\cC$, so too we need $\cC$ to satisfy certain technical conditions, depending on the choice of ``lax'', ``oplax'', or ``strong'', in order to have any of $\cat{Alg}_d^{\mathrm{strong}}(\cC)$, $\cat{Alg}_d^{\mathrm{lax}}(\cC)$, or $\cat{Alg}_d^{\mathrm{oplax}}(\cC)$.  In \cref{thm.existence of even higher morita category} we prove that our technical conditions suffice. This allows, in \cref{defn.even higher Morita} to define the even higher Morita category:
\begin{theoremdefinition}
  Let $\ast = $ ``lax,'' ``oplax,'' or ``strong'', and let $\cC$ satisfy the required conditions specified in \cref{thm.existence of even higher morita category}.  
  Using either construction from \cite{HaugsengEn,ClaudiaDamien2}, then there an $(\infty,d)$-by-$(\infty,n)$ double category $ (\vec k;\vec l) \mapsto \Alg_d(\cC^\boxasterisk_{\vec l})_{\vec k} $.    
  The \define{even higher Morita category}  $\Alg_d^\ast(\cC)$ of $E_d$-algebras and $\ast$-morphisms in $\cC$ is its underlying $(\infty,d+n)$-category.
\end{theoremdefinition}

\begin{remark}
The technical conditions required are the natural conditions for defining the relative tensor product of bimodule objects, namely the existence and compatibility of certain colimits.
\end{remark}

Our main examples are the following:
\begin{examplenodiamond} \mbox{}
\begin{itemize}
\item  In \cref{eg.Alg2Rex} we observe that the bicategory $\cat{Rex}_\bK$ of finitely-cocomplete $\bK$-linear categories, finitely-cocontinuous $\bK$-linear functors, and natural transformations satisfies the required conditions for $\cat{Alg}_d^{\mathrm{strong}}(\cat{Rex}_\bK)$ to exist.
\begin{itemize}
\item For $d=2$, since $\cat{Rex}_\bK$ is a bicategory, the $(\infty,4)$-category $\cat{Alg}_2^{\mathrm{strong}}(\cat{Rex}_\bK)$ is the Morita category of braided monoidal categories predicted in \cite{WalkerTQFT,BZBJ}.
\item For $d=1$, the 3-category $\cat{TC}$ of \cite{DSPS1, DSPS2} is expected to be a subcategory of $\Alg_1^{\mathrm{strong}}(\cat{Rex})$; see \cref{eg.Alg1Rex}.
\end{itemize}
\item  \Cref{remark.strong cocomplete} suggests that many more $(\infty,n)$-categories $\cC$ of interest satisfy sufficient conditions to assure the existence of $\cat{Alg}_d^{\mathrm{strong}}(\cC)$.
\item For comparison, as a non-example, \cref{multiple examples} shows that the bicategory $\cat{Pres}_\bZ$ of locally presentable additive categories does not satisfy the conditions required for the oplax version of our construction. Since the conditions were the natural ones for the existence of relative tensor products of bimodules, we don't expect that a different construction would fix this problem. \hfill\ensuremath\Diamond
\end{itemize}
\end{examplenodiamond}

To conclude, we note that our framework for (op)lax transfors is of interest independently. In particular, we prove in \cref{cor.oplax-transformation} that (op)lax transfors assemble into an $(\infty,n)$-category:

\begin{theorem}\label{thm.intro lax transfor categories}
  Given $(\infty,n)$-categories $\cB$ and $\cC$ there are $(\infty,n)$-categories $\mathrm{Fun}^{\mathrm{lax}}(\cB,\cC)$ and $\mathrm{Fun}^{\mathrm{oplax}}(\cB,\cC)$ 
 whose objects are the functors from $\cB$ to $\cC$, and whose $k$-morphisms are lax or oplax $k$-transfors between functors $\cB \to \cC$.  They depend naturally on the inputs $\cB$ and $\cC$.  In particular, given equivalences of $(\infty,n)$-categories $\cB' \overset\sim\to \cB$ and $\cC \overset\sim\to \cC'$, the corresponding functors $\mathrm{Fun}^{\mathrm{lax}}(\cB,\cC) \to \mathrm{Fun}^{\mathrm{lax}}(\cB',\cC')$ and $\mathrm{Fun}^{\mathrm{oplax}}(\cB,\cC) \to \mathrm{Fun}^{\mathrm{oplax}}(\cB',\cC')$  
are equivalences.
\end{theorem}

  The ``strong'' version of \cref{thm.intro lax transfor categories} is automatic: $\mathrm{Fun}^{\mathrm{strong}}(\cB,\cC) = [\cB,\cC]$ is the functor category determined by Cartesian closedness. If $\cB$ and $\cC$ are symmetric monoidal, we can ask for symmetric monoidal $\ast$-$k$-transfors and $(\infty,n)$-categories thereof, described in \cref{defn.symm mon oplax-transfor,cor.symm mon oplax-transformation}; this variant can be used to show that for many notions of ``$P$-algebra,'' including any described in terms of algebras for an operad, the (op)lax $k$-morphisms of $P$-algebras in $\cC$ also package into an $(\infty,n)$-category.

\begin{remark}
  In spite of the names, it is worth emphasizing that our categories $\mathrm{Fun}^{\mathrm{lax}}(\cB,\cC)$ and $\mathrm{Fun}^{\mathrm{oplax}}(\cB,\cC)$ have as their objects the ``strong'' functors from $\cB$ to $\cC$, in the sense that functoriality (and symmetric monoidality in the case of $\mathrm{Fun}^{\mathrm{lax}}_{\otimes}(\cB,\cC)$ and $\mathrm{Fun}^{\mathrm{oplax}}_{\otimes}(\cB,\cC)$) are enforced up to equivalence: the words ``lax'' and ``oplax'' apply only to the transformations between functors and higher transfors thereof.  The case of bicategories illustrates the problem of allowing functoriality itself to be (op)lax.  Indeed, given bicategories $\cB$ and $\cC$, there is a bicategory whose objects are lax functors from $\cB$ to $\cC$, 1-morphisms are lax natural transformations, and 2-morphisms are modifications, but this category does not depend well on the choice of $\cB$: equivalent input bicategories can nevertheless produce \emph{inequivalent} bicategories of lax functors~\cite{MO7666}.  This particular bad behavior does not occur for the bicategory of strong (also called ``pseudo'') functors and lax transformations, and our results show that that pattern continues to the $(\infty,n)$-world.
\end{remark}

\Subsection{Organization of the paper}

In this note we model $(\infty,n)$-categories as complete $n$-fold Segal spaces; we provide a short review in \cref{section.CSS}.   
  We have tried to keep the paper reasonably self-contained and accessible, and so to supplement \cref{section.CSS} we include an Appendix on model categories as a crimp sheet for non-experts. 

 An important technical tool is the notion of ``computad,'' which we recall in \cref{section.walkingmorphisms}, where we also introduce the main players in the construction of the $(\infty,n)$-categories $\cC^{\smallbox{lax}}_{(k)}$ and $\cC^{\smallbox{oplax}}_{(k)}$, computads $\Theta^{(i);(j)}$ called the \define{walking  $i$-by-$j$-morphisms}.
Examples of walking $i$-by-$j$-morphisms for small $i$ and $j$ are computed in \cref{section.examples}.
  In \cref{section.thedefinition} we prove that for any $(\infty,n)$-category $\cC$, the spaces of ``$i$-by-$j$-morphisms in $\cC$,'' whose diagrammatics are controlled by $\Theta^{(i);(j)}$, package into a ``double $(\infty,n)$-category'' $\Cc^{\Box}$ (i.e.\ a complete $n$-fold Segal space internal to complete $n$-fold Segal spaces) which we call the \define{(op)lax square of $\cC$}.  We define $\cC^{\smallbox{lax}}_{(k)}$ and $\cC^{\smallbox{oplax}}_{(k)}$ as the $(\infty,n)$-categories of ``vertical $k$-morphisms'' and ``horizontal $k$-morphisms'' in $\cC^\Box$ respectively.
    
    Symmetric monoidal structures are discussed in \cref{section.symmetricmonoidal}, where we show that if $\cC$ is symmetric monoidal, so is $\cC^\Box$.  In \cref{section.QFT} we discuss twisted quantum field theories in some detail. In \cref{section.evenhigher} we construct various ``even higher'' $(\infty, d+n)$-Morita categories of $E_d$-algebras in a symmetric monoidal $(\infty,n)$-category.  
  Unfortunately, in order to keep \cref{section.evenhigher} from growing too long, 
    we will not be able to review all details of the constructions in \cite{HaugsengEn,ClaudiaDamien2}; we recall only the parts  necessary for our purposes, relying on the reader to consult the references for more background information.

\Subsection{Acknowledgments}

This article is based on the authors' conversations at the program  ``Modern Trends in Topological Quantum Field Theory,'' at the Erwin Schr\"odinger Institute in Vienna, Austria; 
at the program ``Interactions of Homotopy Theory and Algebraic Topology with Physics through Algebra and Geometry,'' at the Simons Center for Geometry and Physics in Stony Brook, New York;
 while CS was a visitor at the Institut des Hautes \'Etudes Scientifiques in Bures-sur-Yvette, France; 
 while TJF was a visitor at  the IBS Center for Geometry and Physics in Pohang, Korea; 
 and while TJF visited CS at the Max Planck Institut f\"ur Mathematik in Bonn, Germany. 
We thank all our host institutions for their generous hospitality.  

TJF is supported by the grant DMS-1304054 from the American National Science Foundation.  CS is supported by the grants 200021\textunderscore 137778 and P2EZP2\textunderscore 159113 from the Swiss National Science Foundation. 

We would also like to thank Rune Haugseng, Mauro Porta, Stephan Stolz, Peter Teichner, Constantin Teleman, and Alessandro Valentino for many clarifying conversations.  Our thanks go in particular to Chris Schommer-Pries, who provided guidance throughout this project.

\section{\texorpdfstring{\lowercase{$n$}}{n}-uple and complete \texorpdfstring{\lowercase{$n$}}{n}-fold Segal spaces}\label{section.CSS}

We will model $(\infty,1)$-categories as complete Segal spaces following \cite{Rezk} (with one modification: as is done in \cite{LurieGoodwillie,Horel2014}, we leave out the Reedy fibrancy condition included in the \cite{Rezk} definition).   
Their $n$-fold iteration models $n$-fold categories, the higher analogs of double categories, rather than $(\infty,n)$-categories. We follow~\cite{LurieGoodwillie,Haugseng2014} and use the term ``complete $n$-uple Segal space'' for them. To instead obtain a model for $(\infty,n)$-categories, following \cite{BSP2011} we have to add an extra condition, and obtain ``complete $n$-fold Segal spaces''.  We recall the ``user interface'' definitions in this section. In the Appendix we review some of the ``machine code'' model category theory that runs in the background.

Throughout we use the word \define{space} to mean Kan simplicial set.  We will write $\cat{Spaces}$ for the (strict) category of spaces; it is a full subcategory of the category $\cat{sSet}$ of simplicial sets, which will play a starring role in the Appendix.

\begin{definition}\label{defn.segalspace}
  A \define{(1-fold) Segal space} $\Cc$ is a simplicial object $\Cc_{\bullet} : \Delta^{\mathrm{op}} \to \cat{Spaces}$ satisfying the \define{Segal condition}, which says that for every $k\geq 1$ the natural map
  $$\Cc_{k} \to \ourunderbrace{\Cc_{1}\overset{h}{\underset{\Cc_{0}}\times} \Cc_{1} \overset{h}{\underset{\Cc_{0}}\times}  \dots \overset{h}{\underset{\Cc_{0}}\times}  \Cc_{1}}{k\text{ times}}, $$
induced by the maps $[0]\to[k]$ sending 0 to $i$ and $[1]\to[k]$ sending 0 to $i-1$ and $1$ to $i$, is an equivalence. 
\end{definition}

The space $\Cc_{0}$ is thought of as the space of objects of an $(\infty,1)$-category $\Cc$; the Segal condition says that $\Cc_{k}$ should be thought of as the space of $k$-tuples of composable $1$-morphisms. Moreover, 2-morphisms should be invertible and can be thought of as paths in the space $\Cc_1$ of 1-morphisms, 3-morphisms as homotopies between them, etc.

To any higher category one can  associate an ordinary category having the same objects, 
with morphisms being $2$-equivalence classes of $1$-morphisms:

\begin{definition}\label{defn.homotopy category}
The \define{homotopy category} $\h_1(\Cc)$ of a Segal space $\Cc=\Cc_\bullet$ is defined as follows.  Its set of objects is the underlying set (i.e.\ the set of zero-simplices) $\breve\Cc_{0}$ of the space $\Cc_{0}$.  For $x,y\in \breve\Cc_{0}$, the set of morphisms from $x$ to $y$ is
$$
\hom_{\h_1(\Cc)}(x,y):=\pi_0\left(\hom_\Cc(x,y)\right)
=\pi_0\left(\{x\}\underset{\Cc_0}{\overset{h}{\times}} \Cc_1 \underset{\Cc_0}{\overset{h}{\times}} \{y\}\right)\,.
$$
The composition of morphisms is defined by applying $\pi_{0}$ to the following zig-zag of spaces: 
\begin{eqnarray*}
\left(\{x\}\underset{\Cc_0}{\overset{h}{\times}} \Cc_1 \underset{\Cc_0}{\overset{h}{\times}} \{y\}\right)\times
\left(\{y\}\underset{\Cc_0}{\overset{h}{\times}} \Cc_1 \underset{\Cc_0}{\overset{h}{\times}} \{z\}\right)
& \longrightarrow & \{x\}\underset{\Cc_0}{\overset{h}{\times}} \Cc_1 \underset{\Cc_0}{\overset{h}{\times}}\Cc_1 \underset{\Cc_0}{\overset{h}{\times}} \{z\} \\
& \overset\sim\longleftarrow & \{x\}\underset{\Cc_0}{\overset{h}{\times}} \Cc_2 \underset{\Cc_0}{\overset{h}{\times}} \{z\} \\
& \longrightarrow & \{x\}\underset{\Cc_0}{\overset{h}{\times}} \Cc_1 \underset{\Cc_0}{\overset{h}{\times}} \{z\}\,.
\end{eqnarray*}
The second arrow, being a weak equivalence, induces a bijection on $\pi_0$. 
\end{definition}

\begin{definition} \label{defn.1fold complete}
Let $\Cc$ be a Segal space and $\Cc_{\inv}$  the space of elements in $\Cc_1$ which become invertible in the homotopy category of $\Cc$. Then $\Cc$ is said to be \define{complete} if the map $\Cc_0\to\Cc_{\inv}$ induced by the degeneracy map is an equivalence.
\end{definition}

The space $\cC_\inv$ is usually called $\cC_1^\inv$ in the literature.  The former notation will be convenient in the proofs of \cref{lemma.weaker completeness condition} and \cref{mainthm}.

\begin{remark} 
Completeness is a version of skeletality: it says that all isomorphisms in the category $\Cc$ are already paths in the space of objects $\Cc_{0}$.
\end{remark}

\begin{lemma}  \label{lemma.alternate description of completeness}
The space $\cC_{\inv}$ can be presented as the homotopy fiber product
$$ \cC_{\inv} = \cC_{3} \underset{\cC_{1}\times \cC_{1}}{\overset h \times} (\cC_{0} \times \cC_{0}) $$
where the two maps $\cC_{3} \rightrightarrows \cC_{1}$ correspond to the inclusions $[1] \to [3]$ as either the $(0,2)$-edge of the 3-simplex or as the $(1,3)$-edge.
\end{lemma}
This realizes the idea that a homotopy-invertible morphism is one that has both left and right up-to-homotopy inverses; the homotopy-invertible morphism itself corresponds to the $(1,2)$-edge.  We defer the proof to the end of this section.

We can iterate \cref{defn.segalspace} and inductively define the higher versions of Segal spaces, see e.g.~\cite{LurieGoodwillie,Haugseng2014}. A priori, we obtain higher versions of double categories and, by further induction steps, ``$n$-uple categories.'' Unfolding this iterative definition, we get the following.

\begin{definition}\label{defn.complete}
An \define{$n$-uple Segal space} is an $n$-fold simplicial space $\Cc=\Cc_{\bullet,\dots,\bullet}:(\Delta^{\mathrm{op}})^n\to \cat{Spaces}$ such that for every $1\leq i\leq n$, and every $k_1,\ldots, k_{i-1},k_{i+1},\ldots, k_n\geq0$,
$$\Cc_{k_1,\ldots, k_{i-1},\bullet, k_{i+1},\ldots, k_n}$$
is a Segal space.

An $n$-uple Segal space $\cC$ is \define{complete} if for every $1\leq i\leq n$ and every $k_1,\ldots, k_{i-1},k_{i+1},\ldots, k_n\geq0$
$$\Cc_{k_1,\ldots, k_{i-1},\bullet, k_{i+1},\ldots, k_n}$$
is a complete Segal space, i.e.~$\cC$ is complete in ``every direction''.
\end{definition}
Thus a (complete) $n$-uple Segal space $\Cc$ consists of a space $\Cc_{\vec k}$ for each $\vec k \in \NN^{n}$, thought of as the space of morphisms filling in a $n$-dimensional rectangle of size $k_{1}\times \dots \times k_{n}$, which satisfies the (complete) Segal condition in each direction.

To model $(\infty,n)$-categories, $n$-uple Segal spaces do not suffice, just as double categories are not equivalent to bicategories. One must instead impose one extra condition:

\begin{definition}\label{defn.nSS}
A \define{(complete) $n$-fold Segal space} is a (complete) $n$-uple Segal space $\Cc$ such that for every $1\leq i\leq n-1$, $\Cc_{k_{1},\dots,k_{i-1},0,\bullet,\dots,\bullet}$ is \define{essentially constant}, i.e.~the degeneracy maps
  $$ \Cc_{k_{1},\dots,k_{i-1},0,0,\dots,0} \to \Cc_{k_{1},\dots,k_{i-1},0,k_{i+1},\dots,k_n} $$
are equivalences.
\end{definition}

One may always treat a complete $n$-fold Segal space as a constant complete $(n+1)$-fold  Segal space by declaring $\Cc_{\vec k,k_{n+1}} = \Cc_{\vec k}$.  Conversely, to simplify notation, we will often drop trailing $0$s, and write $\Cc_{k_{1},\dots,k_{i}}$ for $\Cc_{k_{1},\dots,k_{i},0,\dots,0}$.  The essential constancy condition 
says that we may always assume (by slight abuse of notation) that $k_{1},k_{2},\dots,k_{i}$ is a sequence of strictly-positive integers, with the convention that the empty sequence ($i=0$) is allowed, so that $\Cc_{\emptyset} = \Cc_{0} = \Cc_{0,0,\dots,0}$ is the space of objects. 
 Finally, the Segal condition says that a complete $n$-fold Segal space is (essentially) determined by (face and degeneracy maps and) the spaces $\Cc_{1,\dots,1,0,\dots}$; when there are $i$ $1$s, we abbreviate $\Cc_{1,\dots,1}$ by $\Cc_{(i)}$, and think of it as ``the space of $i$-morphisms in $\Cc$.'' 
 More generally, we abbreviate any $i$ consecutive $1$s by $(i)$, e.g.~$\Cc_{2,(i),2}:=\Cc_{2,1,\ldots,1,2,0,\ldots}$

\begin{lemma} \label{lemma.weaker completeness condition}
  Suppose that $\cC_{\vec\bullet}$ is an $n$-fold Segal space which is not known to be complete, but which satisfies the a priori weaker condition that each Segal space $\Cc_{k_{1},\dots,k_{i},\bullet,0,0,\dots}$ is complete for all $0\leq i \leq n$ and $\vec k\in \bN^{i}$.  Then $\cC$ is in fact complete.
\end{lemma}

In~\cite{BSP2011}, this equivalent but seemingly weaker condition is taken as the definition of ``completeness'' of $n$-fold Segal spaces, whereas the seemingly  stronger version that we give matches the one used in \cite{LurieGoodwillie,Haugseng2014}.  \cref{lemma.weaker completeness condition} verifies that the two definitions of complete $n$-fold Segal spaces agree.  The result clearly depends on the essential constancy condition.  We defer the proof to the end of this section.

We turn now to the functors between complete $n$-fold Segal spaces.

\begin{definition}\label{defn.homomorphism of nCSS}
A \define{homomorphism} of (complete) $n$-fold Segal spaces $F: \cB_{\bullet,\ldots, \bullet} \to \cC_{\bullet,\ldots, \bullet}$ is a map of $n$-fold simplicial spaces, i.e.\ it consists of continuous maps
$$F_{k_1,\ldots, k_n}:\cB_{k_1,\ldots, k_n} \longrightarrow \cC_{k_1,\ldots, k_n}$$
for every $\vec k=(k_1,\ldots, k_n) \in \NN^n$ which commute with the face and degeneracy maps.

  An \define{equivalence} of (complete) $n$-fold Segal spaces $F: \cB_{\bullet,\ldots, \bullet} \to \cC_{\bullet,\ldots, \bullet}$ is a homomorphism such that each map $F_{k_1,\ldots, k_n}$ is a homotopy equivalence of spaces.  Two complete $n$-fold Segal spaces are \define{equivalent} if they are connected by a zig-zag of equivalences.
\end{definition}

This notion of homomorphism is, of course, a strict one; we would rather a notion which is invariant under equivalences.  Any such attempt will undoubtedly result in a space of maps $\cB \to \cC$, rather than just a set of them.  As a first attempt, note that the set of homomorphisms $\hom(\cB_{\vec\bullet},\cC_{\vec\bullet})$ between complete $n$-fold Segal spaces  is naturally the set of $0$-simplices in a simplicial set $\maps(\cB_{\vec\bullet},\cC_{\vec\bullet})$.  Indeed,
any complete $n$-fold Segal space is among other things a functor $(\Delta^{\times n})^{\mathrm{op}} \to \cat{sSet}$ (which happens to take values in Kan complexes).  Since $\cat{sSet}$ is locally presentable, the category of functors from any small category to $\cat{sSet}$  is naturally enriched in $\cat{sSet}$.

\begin{lemma}\label{lemma.mappingspace}
  Let $\cB_{\bullet}$ and $\cC_{\bullet}$ be (complete) $n$-fold Segal spaces.  Suppose that at least one of the following holds:
  \begin{enumerate}
    \item When considered as a functor $(\Delta^{\times n})^{\mathrm{op}} \to \cat{sSet}$, $\cB_{\bullet}$ is cofibrant for the projective model structure.
    \item When considered as a functor $(\Delta^{\times n})^{\mathrm{op}} \to \cat{sSet}$, $\cC_{\bullet}$ is fibrant for the injective model structure.
  \end{enumerate}
  Then the simplicial set $\maps(\cB_{\vec\bullet},\cC_{\vec\bullet})$ is a Kan complex.  Equivalences in either variable (preserving the chosen (co)fibrancy) induce homotopy equivalences of Kan complexes.
\end{lemma}

\begin{proof}
  The statement has nothing to do with (complete) $n$-fold Segal spaces, and relies only on: (i)~the existence of both injective and projective model structures (see \cref{defn.global model structure}) on the category of functors $(\Delta^{\times n})^{\mathrm{op}} \to \cat{sSet}$;  (ii)~the fact that both such model categories are simplicial; (iii)~the observation that every functor $(\Delta^{\times n})^{\mathrm{op}} \to \cat{sSet}$ is cofibrant for the injective model structure, and every  functor valued in the subcategory $\cat{Spaces}$, i.e.\ the full subcategory of Kan complexes, is fibrant for the projective model structure.
\end{proof}

Suppose now that $\cB_{\vec\bullet}$ and $\cC_{\vec\bullet}$ are arbitrary complete $n$-fold Segal space, and $\widetilde{\cB}_{\vec\bullet} \to \cB_{\vec\bullet}$ and $\Cc_{\vec\bullet} \to \widehat{\Cc}_{\vec\bullet}$ are equivalences with $\widetilde{\Bb}_{\vec\bullet}$ cofibrant for the projective model structure and $\widehat{\Cc}_{\vec\bullet}$ fibrant for the injective model structure.  Note that any $n$-uple simplicial space which is equivalent to a complete $n$-fold Segal space  is necessarily also complete $n$-fold Segal.    \cref{lemma.mappingspace} implies that  the induced maps $\maps(\widetilde{X}_{\vec\bullet},Y_{\vec\bullet}) \to \maps(\widetilde{X}_{\vec\bullet},\widehat{Y}_{\vec\bullet}) \leftarrow \maps(X_{\vec\bullet},\widehat{Y}_{\vec\bullet})$ are homotopy equivalences.

\begin{definition}
  Let $\cB_{\vec\bullet}$ and $\cC_{\vec\bullet}$ be (complete) $n$-fold Segal spaces.  With notation as above,  we refer to any of the homotopy equivalent spaces $\maps(\widetilde{\cB}_{\vec\bullet},\cC_{\vec\bullet})$, $\maps(\widetilde{\cB}_{\vec\bullet},\widehat{\cC}_{\vec\bullet})$, and $\maps(\cB_{\vec\bullet},\widehat{\cC}_{\vec\bullet})$ as the \define{derived mapping space} $\maps^{h}(\cB_{\vec\bullet},\cC_{\vec\bullet})$
\end{definition}

Note that it follows from \cref{lemma.mappingspace} that the derived mapping space is independent (up to homotopy) of the choice of cofibrant or fibrant resolution used.  These choices can be made functorially.  We henceforth choose fibrant and/or cofibrant resolution functors, thereby defining $\maps^{h}(-,-)$ so as to be strictly functorial in each variable.

The derived mapping spaces  $\maps^{h}(-,-)$ can be packaged together into an $(\infty,1)$-category known as the ``homotopy theory of $(\infty,n)$-categories.''  More details are in the Appendix.  In particular, we recall the model category theoretic definition of said homotopy theory in \cref{eg.Rezk}.  That that model category unpacks to the description presented herein follows from \cref{lemma.basic facts about localizations}.  In particular, complete $n$-fold Segal spaces are nothing but the fibrant objects of that model category, and weak equivalences in that model category between fibrant objects are precisely the equivalences in \cref{defn.homomorphism of nCSS}.

With \cref{defn.homomorphism of nCSS} in hand, we can prove \cref{lemma.alternate description of completeness}:
\begin{proof}[Proof of \cref{lemma.alternate description of completeness}]
  This is proved in \cite[Proposition 10.1]{MR2578310} under the assumption that the simplicial space $\cC_\bullet$ is fibrant for the injective model structure on simplicial spaces.  The space $\cC_{3} \times^{h}_{\cC_{1}\times \cC_{1}} (\cC_{0} \times \cC_{0})$ is obviously invariant up to homotopy under equivalences of Segal spaces.  It therefore suffices to see that an equivalence of Segal spaces $\cC_\bullet \overset\sim\to \cD_\bullet$ induces an equivalence $\cC_\inv \overset\sim\to \cD_\inv$, since we may then functorially replace $\cC$ by some injective-fibrant replacement.  (As in \cref{defn.1fold complete}, $\cC_\inv$ denotes the space usually called $\cC_1^\inv$.)
  
  Let $s,t: \cC_1 \rightrightarrows \cC_0$ denote the source and target maps.  Given elements $x,y\in \breve\cC_0$, by definition a $1$-morphism in $\hom_{\h_1\cC}(x,y)$ consists of a connected component $f\in \pi_0(\cC_1)$ along with homotopy classes of paths in $\cC_0$ connecting $s(f)$ to $x$ and connecting $t(f)$ to $y$.  Thus the sets $\hom_{\h_1\cC}(x,y)$ are locally constant in $x,y\in \cC_0$, and any data about a morphism in $\h_1\cC$ depends only on its connected component.  It follows that $\cC_\inv$ is a union of connected components of $\cC_1$: precisely those components that map to invertible morphisms in $\h_1\cC$.
  
  An equivalence of Segal spaces $\cC \overset\sim\to \cD$ induces a functor $\h_1(\Cc) \overset\sim\to \h_1(\cD)$ which is an isomorphism on $\hom$ sets and an essential surjection, hence an equivalence of categories.  It follows in particular that the connected components of $\cC_1$ appearing in $\cC_\inv$ correspond bijectively to the connected components of $\cD_1$ appearing in $\cD_\inv$.  Thus the 
   equivalence $\cC_1 \overset\sim\to\cD_1$ induces a 
    homotopy equivalence $\cC_\inv \overset\sim\to\cD_\inv$.
\end{proof}

To prove \cref{lemma.weaker completeness condition}
we will need the analog of the homotopy category of a Segal space for 2-fold Segal spaces.
\begin{definition}\label{defn.homotopy bicategory}
The \define{homotopy bicategory} $\h_2(\cC)$ of a $2$-fold Segal space $\cC=\cC_{\bullet,\bullet}$ is defined as follows.
Its set of objects is the underlying set (i.e.~the set of zero-simplices) $\breve\cC_{0,0}$ of the space $\cC_{0,0}$ and
$$
\mathrm{Hom}_{\h_2(\cC)}(x,y)=\h_1\big(\mathrm{Hom}_{\cC}(x,y)\big)
=\h_1\left(\{x\} \underset{\cC_{0,\bullet}}{\overset{h}{\times}} \cC_{1,\bullet} \underset{\cC_{0,\bullet}}{\overset{h}{\times}} \{y\}\right)
$$
as Hom categories. Horizontal composition is defined as follows:  
\begin{eqnarray*}
\left(\{x\}\underset{\cC_{0,\bullet}}{\overset{h}{\times}} \cC_{1,\bullet} \underset{\cC_{0,\bullet}}{\overset{h}{\times}} \{y\}\right)\times
\left(\{y\}\underset{\cC_{0,\bullet}}{\overset{h}{\times}} \cC_{1,\bullet} \underset{\cC_{0,\bullet}}{\overset{h}{\times}} \{z\}\right)
& \longrightarrow & \{x\}\underset{\cC_{0,\bullet}}{\overset{h}{\times}} \cC_{1,\bullet} \underset{\cC_{0,\bullet}}{\overset{h}{\times}}\cC_{1,\bullet} \underset{\cC_{0,\bullet}}{\overset{h}{\times}} \{z\} \\
& \tilde\longleftarrow & \{x\}\underset{\cC_{0,\bullet}}{\overset{h}{\times}} \cC_{2,\bullet} \underset{\cC_{0,\bullet}}{\overset{h}{\times}} \{z\} \\
& \longrightarrow & \{x\}\underset{\cC_{0,\bullet}}{\overset{h}{\times}} \cC_{1,\bullet} \underset{\cC_{0,\bullet}}{\overset{h}{\times}} \{z\}\,.
\end{eqnarray*}
The second arrow happens to go in the wrong way but it is a weak equivalence. Therefore after taking $\h_1$ it turns out to be an equivalence of categories, 
and thus to have an inverse (assuming the axiom of choice). 
\end{definition}

\begin{proof}[Proof of \cref{lemma.weaker completeness condition}]

By fixing the first few indices, it suffices to prove that if $\cC_{\vec \bullet}$ is an $n$-fold Segal space such that for each $i$ and $k_1,\dots,k_i$, the Segal space $\cC_{k_1,\dots,k_i,\bullet}$ is complete, then the map $\cC_{(0)} \simeq \cC_{0,\vec\bullet} \to \cC_{\inv,\vec\bullet}$ is an equivalence of $(n-1)$-fold Segal spaces, i.e.~for every $l_2,\ldots, l_n$, the Segal space $\Cc_{\bullet, l_2,\ldots, l_n}$ is complete.  (We continue the convention of dropping trailing $0$s.  As in \cref{defn.1fold complete}, for a Segal space $\cD_\bullet$, we denote by $\cD_\inv$ the space usually called $\cD_1^\inv$.)  Essential constancy and the Segal condition together imply that a homomorphism of $(n-1)$-fold Segal spaces is an equivalence as soon as it is an equivalence on all spaces of $(j)$-morphisms, i.e.~it is enough to show that for every $j$ the Segal space $\Cc_{\bullet, (j)}$ is complete.  We therefore suppose by induction that we know $\cC_{0,(i)} \to \cC_{\inv,(i)}$ is an equivalence for some $i$ and wish to show that $\cC_{0,(i+1)} \to \cC_{\inv,(i+1)}$ is an equivalence.  Note that the base case $j=0$ is part of the initial assumption on $\cC$ in the lemma.

Our goal is to move completeness from one index to another.
Thus, consider the $2$-fold Segal space $\cD_{\bullet,\bullet} = \cC_{\bullet,(i),\bullet}$.  Our inductive assumption asserts that $\cD_{\bullet,0}=\cC_{\bullet, (i), 0}=\cC_{\bullet, (i)}$ is complete, i.e.~that $\cD_{0,0} \to \cD_{\inv,0}$ is an equivalence. Furthermore, the assumption in the lemma asserts that for every $k$ the Segal space $\cD_{k,\bullet}=\cC_{k,(i),\bullet}$ is complete, i.e.~that $\cD_{\bullet,0} \to \cD_{\bullet,\inv}$ is an equivalence of Segal spaces.  It remains to show that for any $k$ the Segal space $\cD_{\bullet,k}=\cC_{\bullet,(i),k}$ is complete, which in turn by the Segal condition is reduced to showing the statement for $k=1$, namely that $\cD_{0,0} \simeq \cD_{0,1} \to \cD_{\inv,1}$ is an equivalence.

Consider the homotopy bicategory $\h_2\cD$ of $\cD$. Then $\cD_{\inv,1}$ consists of those elements in $\cD_{(2)}$ that represent 2-morphisms in $\h_2\cD$ which are invertible for the \emph{horizontal} (i.e.\ first) composition in $\h_2\cD$.  But an easy application of the Eckmann--Hilton argument shows that a 2-morphism in a bicategory is invertible for the horizontal composition if and only if its domain and codomain are invertible for the horizontal composition and the $2$-morphism is invertible for the vertical composition.  We find therefore
$$ \cD_{\inv,1} = \cD_{\inv,\inv} \simeq \cD_{\inv,0} \simeq \cD_{0,0}. $$
Here $\cD_{\inv,\inv}$ is the subspace of $\cD_{1,1}$ consisting of those $2$-morphisms that are invertible for both horizontal and vertical composition.  \cref{lemma.alternate description of completeness} implies that homotopy fiber products of Segal spaces restrict to homotopy fiber products of invertible pieces (since homotopy fiber products commute), and so the fact that $\cD_{\bullet,0} \to \cD_{\bullet,\inv}$ is an equivalence for each $\bullet\in \bN$ implies that $\cD_{\inv,0} \to \cD_{\inv,\inv}$ is an equivalence.
\end{proof}

\section{Walking  \texorpdfstring{\lowercase{$i$}}{i}-by-\texorpdfstring{\lowercase{$j$}}{j}-morphisms}\label{section.walkingmorphisms}

This section defines a collection of strict higher categories $\{\Theta^{(i);(j)}\}_{i,j\in \NN}$, called the \define{walking  $i$-by-$j$-morphisms}.  We will use these  in \cref{section.thedefinition} to define the (op)lax square $\Cc^{\Box}$ of an $(\infty,n)$-category~$\Cc$: we will set $\Cc^{\Box}_{(i);(j)} = \maps^{h}(\Theta^{(i);(j)},\Cc)$.  The motivation for the construction is the following fact about higher categories (which has been turned into an almost-definition of ``higher category'' in~\cite{BSP2011}): for each $i\in \NN$ there is a ``globe'' $\Theta^{(i)}$, called the \define{walking $i$-morphism}, such that for any $(\infty,n)$-category $\Cc$ there is a canonical homotopy equivalence $\Cc_{(i)} \cong \maps^{h}(\Theta^{(i)};\Cc)$; moreover, the source and target maps $\Cc_{(i)} \rightrightarrows \Cc_{(i-1)}$ are nothing but the pullbacks along canonical inclusions $s,t : \Theta^{(i-1)} \rightrightarrows \Theta^{(i)}$.   The use of ``walking'' in this sense comes from~\cite{Lauda:2005}.

The higher categories $\Theta^{(i);(j)}$ are in an appropriate sense ``free  $(i+j)$-categories.''  The term ``{$n$-computad}'' for ``presentation of free $n$-category'' was introduced in~\cite{MR0401868}, but has only worked its way into some corners of category theory land, and so we recall the definition.
Recall first that
  a \define{strict $n$-category} is a strict category strictly enriched in the strict category of strict $(n-1)$-categories. 
  The nerve of a strict category is a simplicial set that we may regard as a simplicial space which is discrete at every level;  by induction, the nerve of a strict $n$-category is an $n$-fold simplicial set, which again we will regard as an $n$-fold simplicial space which is discrete at every level.  Nerves of strict $n$-categories automatically satisfy the Segal and essential constancy conditions, and so are $n$-fold Segal spaces.  They are rarely complete: completeness of the nerve is equivalent to the category being \define{gaunt}, which requires every isomorphism to be an identity.

\begin{definition}
  An \define{$n$-computad} $\Theta$ is a presentation of a strict $n$-category freely generated by sets of $k$-morphisms for $0\leq k \leq n$.  It consists of: 
  \begin{itemize}
    \item an $(n-1)$-computad $\partial \Theta$, called the \define{$(n-1)$-skeleton of $\Theta$}; 
    \item a set of \define{generating $n$-morphisms}; 
    \item if $n\geq 1$, then for each generating $n$-morphism there is a pair of parallel $(n-1)$-morphisms, its \define{source} and \define{target}, in the strict $(n-1)$-category generated by $\partial\Theta$.  
    \end{itemize}
    For $k\leq n$, the $k$-skeleton of an $n$-computad $\Theta$ is the $k$-computad $\partial^{n-k}\Theta$.
  Any $n$-computad may be considered an $(n+1)$-computad with no generating $(n+1)$-morphisms; a \define{computad} is an $n$-computad for some $n$.
  By convention, the $(-1)$-skeleton of any computad is empty.
\end{definition}

For example, a $0$-computad is a set and a $1$-computad is a quiver.  Every $n$-computad $\Theta$ determines in particular an $n$-dimensional CW complex whose $k$-cells are the $k$-dimensional generators.  The extra data in a computad are the ``directions'' of the cells.  This CW complex is the ``groupoidification'' of $\Theta$, in which every generating morphism is forced to be invertible; we therefore call it~$\Gr(\Theta)$.

\begin{remarknodiamond}\label{remark.failure of cofibrancy}
  We will abuse notation and conflate an $n$-computad with the free strict $n$-category it generates and the nerve thereof regarded as an $n$-fold Segal space.  Note that computads are automatically gaunt and so the corresponding $n$-fold Segal spaces are complete.  Computads generally are not projectively cofibrant (compare \cref{lemma.mappingspace}).  This can be seen already for the quiver~$\{\tikz[baseline=(base)]{ \path node[dot] (l) {} +(0,-2.5pt) coordinate (base) +(.5,-4pt) node[dot](r) {} +(.5,4pt) node[dot](t) {}; \draw[arrow] (l) -- (r); \draw[arrow] (l) -- (t);}\}$ thought of as a $1$-computad.
  Indeed, suppose that $\cC_\bullet$ is a complete Segal space and $s,t: \cC_1 \rightrightarrows \cC_0$ are the source and target maps.  Then the underived mapping space from this computad to $\cC$ is a strict fibered product of spaces:
  $$ \maps\left( \tikz[baseline=(base)]{ \path node[dot] (l) {} +(0,-2.5pt) coordinate (base) +(.5,-4pt) node[dot](r) {} +(.5,4pt) node[dot](t) {}; \draw[arrow] (l) -- (r); \draw[arrow] (l) -- (t);} \,,\cC\right) = \cC_1 \underset{\cC_0}\times \cC_1 $$
  where both maps $\cC_1 \to \cC_0$ are the source map $s$.  This strict fibered product is not invariant under replacing $\cC$ by an equivalent complete Segal space, and so the (nerve of the) quiver $\{\tikz[baseline=(base)]{ \path node[dot] (l) {} +(0,-2.5pt) coordinate (base) +(.5,-4pt) node[dot](r) {} +(.5,4pt) node[dot](t) {}; \draw[arrow] (l) -- (r); \draw[arrow] (l) -- (t);}\}$ must not be cofibrant.
  
  One can eyeball a cofibrant replacement of $\{\tikz[baseline=(base)]{ \path node[dot] (l) {} +(0,-2.5pt) coordinate (base) +(.5,-4pt) node[dot](r) {} +(.5,4pt) node[dot](t) {}; \draw[arrow] (l) -- (r); \draw[arrow] (l) -- (t);}\}$: that left vertex should be replaced by an interval worth of vertices.  A more robust computation works as follows.  The nerve of $\{\tikz[baseline=(base)]{ \path node[dot] (l) {} +(0,-2.5pt) coordinate (base) +(.5,-4pt) node[dot](r) {} +(.5,4pt) node[dot](t) {}; \draw[arrow] (l) -- (r); \draw[arrow] (l) -- (t);}\}$ arrises as a pushout of simplicial sets:
  $$ 
  \{\tikz[baseline=(base)]{ \path node[dot] (l) {} +(0,-2.5pt) coordinate (base) +(.5,-4pt) node[dot](r) {} +(.5,4pt) node[dot](t) {}; \draw[arrow] (l) -- (r); \draw[arrow] (l) -- (t);}\} = 
  \{\tikz[baseline=(base)]{ \path node[dot] (l) {} +(0,-2.5pt) coordinate (base) +(.5,-4pt) node[dot](r) {} ; \draw[arrow] (l) -- (r);} \} 
  \underset{\{\tikz[baseline=(base)]{ \path node[dot] (l) {} +(0,-2.5pt) coordinate (base) ; }\} } \cup 
  \{\tikz[baseline=(base)]{ \path node[dot] (l) {} +(0,-2.5pt) coordinate (base)  +(.5,4pt) node[dot](t) {};  \draw[arrow] (l) -- (t);}\}.
  $$
  The simplicial sets in question are all projectively cofibrant as simplicial spaces: they are just the standard $0$- and $1$-dimensional simplices arising from the Yoneda embedding.  The maps $\{\tikz[baseline=(base)]{ \path node[dot] (l) {} +(0,-2.5pt) coordinate (base) ; }\} \hookrightarrow  \{\tikz[baseline=(base)]{ \path node[dot] (l) {} +(0,-2.5pt) coordinate (base) +(.5,0pt) node[dot](r) {} ; \draw[arrow] (l) -- (r);} \} $ are levelwise inclusions, and hence cofibrations for the \emph{injective} model structure on simplicial spaces.  Thus the above pushout is a homotopy pushout for the injective model structure.  On the other hand, homotopy pushouts in the injective and projective models are homotopy equivalent, and so a projective-cofibrant replacement of the quiver in question can be computed by using the homotopy colimit in the projective model:
   \\ \mbox{}\hfill$\displaystyle 
  \widetilde{\{\tikz[baseline=(base)]{ \path node[dot] (l) {} +(0,-2.5pt) coordinate (base) +(.5,-4pt) node[dot](r) {} +(.5,4pt) node[dot](t) {}; \draw[arrow] (l) -- (r); \draw[arrow] (l) -- (t);}\}} = 
  \{\tikz[baseline=(base)]{ \path node[dot] (l) {} +(0,-2.5pt) coordinate (base) +(.5,-4pt) node[dot](r) {} ; \draw[arrow] (l) -- (r);} \} 
  \underset{\{\tikz[baseline=(base)]{ \path node[dot] (l) {} +(0,-2.5pt) coordinate (base) ; }\} } {\overset h \cup}
  \{\tikz[baseline=(base)]{ \path node[dot] (l) {} +(0,-2.5pt) coordinate (base)  +(.5,4pt) node[dot](t) {};  \draw[arrow] (l) -- (t);}\}.
  $\hfill\ensuremath\Diamond
\end{remarknodiamond}

We now define the walking $i$-morphism as an $i$-computad:

\begin{definition} \label{defn.walkingimorphism}
  The \define{walking $i$-morphism} is the $i$-computad $\Theta^{(i)}$
   with $(i-1)$-skeleton $\partial\Theta^{(i)} = \Theta^{(i-1)} \cup_{\partial \Theta^{(i-1)}}\Theta^{(i-1)}$, and a unique generating $i$-morphism with source the unique generating $(i-1)$-morphism in the first copy of $\Theta^{(i-1)}$ and target the unique generating $(i-1)$-morphism in the second copy of $\Theta^{(i-1)}$.  We denote the inclusions $\Theta^{(i-1)}\mono \Theta^{(i)}$ on the first and second copies by $s^{(i)}$ and~$t^{(i)}$.
\end{definition}

For example, the walking $0$-morphism is the one-element set $\Theta^{(0)} = \{\tikz \path node[dot] (l) {}  +(0,-3pt) coordinate (base);\}$, with $(-1)$-skeleton $\partial \Theta^{(0)} = \emptyset$.  The walking $1$-morphism is the ``$A_{2}$ quiver'' $\Theta^{(1)} = \{\tikz{ \path node[dot] (l) {} +(.5,0) node[dot](r) {} +(0,-3pt) coordinate (base); \draw[arrow] (l) -- (r);}\}$.  The walking $2$-morphism is the standard bigon:
$$
\Theta^{(2)} = \left\{
\begin{tikzpicture}[baseline=(base)]
  \path node[dot] (a) {} +(2,0) node[dot] (c) {} +(0,-3pt) coordinate (base);
  \draw[arrow] (a) .. controls +(1,-.5) and +(-1,-.5) .. coordinate (t)   (c);
  \draw[arrow] (a) .. controls +(1,.5) and +(-1,.5) .. coordinate (s)  (c);
  \draw[twoarrowlonger] (s) --   (t);
\end{tikzpicture} \right\}
$$

\begin{remark} \label{remark.theta2}
  The walking bigon $\Theta^{(2)}$, or rather its nerve, is not projectively cofibrant as a $2$-uple simplicial space.  
  Consider the $2$-uple simplicial set
  $$
  
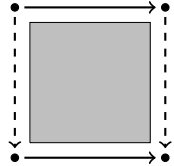
\begin{tikzpicture}[baseline=(base)]
    \path (0,0) node[dot] (sw) {} (0,2) node[dot] (nw) {} (2,0) node[dot] (se) {} (2,2) node[dot] (ne) {};
    \path (.2,.2) coordinate (sw1) (.2,1.8) coordinate (nw1) (1.8,.2) coordinate (se1) (1.8,1.8) coordinate (ne1);
    \draw[ultra thin,fill=black!25] (sw1) -- (nw1) -- (ne1) -- (se1) -- (sw1);
    \draw[arrow] (nw) -- (ne); \draw[arrow] (sw) -- (se);
    \draw[arrow,dashed] (nw) -- (sw); \draw[arrow,dashed] (ne) -- (se); 
    \path (0,1) coordinate (base);
  \end{tikzpicture}
  $$
  where dashed edges denote $1$-simplices in the second direction, and solid edges denote $1$-simplices in the first direction.  Thus it has four $(0\times 0)$-simplices, two nondegenerate $(1\times 0)$-simplices, two nondegenerate $(0\times 1)$-simplices, and one nondegenerate $(1\times 1)$-simplex.  All other uple-simplices are degenerate.  It is a complete $2$-uple simplicial space, but it is not essentially constant.  It corresponds via the Yoneda embedding to the object $([1],[1]) \in \Delta^{\times 2}$, and so is projective-cofibrant.
  
  The bigon $\Theta^{(2)}$ arrises as a pushout of simplicial sets:
  $$ \begin{tikzpicture}[baseline=(base)]
  \path node[dot] (a) {} +(2,0) node[dot] (c) {} +(0,-3pt) coordinate (base);
  \draw[arrow] (a) .. controls +(1,-.5) and +(-1,-.5) .. coordinate (t)   (c);
  \draw[arrow] (a) .. controls +(1,.5) and +(-1,.5) .. coordinate (s)  (c);
  \draw[twoarrowlonger] (s) --   (t);
\end{tikzpicture}
  \quad = \quad
  \tikz[baseline=(base)] \path node[dot] {} ++(0,-3pt) coordinate(base);
 \quad
 \underset{
   \begin{tikzpicture}[baseline=(base)]
    \path (0,0) node[dot] (sw) {} (0,.5) node[dot] (nw) {} (0,.6) coordinate(top);
    \draw[arrow,dashed] (nw) -- (sw);  
    \path (0,.3) coordinate (base);
  \end{tikzpicture}
 }
 \cup
 \quad
  \begin{tikzpicture}[baseline=(base)]
    \path (0,0) node[dot] (sw) {} (0,2) node[dot] (nw) {} (2,0) node[dot] (se) {} (2,2) node[dot] (ne) {};
    \path (.2,.2) coordinate (sw1) (.2,1.8) coordinate (nw1) (1.8,.2) coordinate (se1) (1.8,1.8) coordinate (ne1);
    \draw[ultra thin,fill=black!25] (sw1) -- (nw1) -- (ne1) -- (se1) -- (sw1);
    \draw[arrow] (nw) -- (ne); \draw[arrow] (sw) -- (se);
    \draw[arrow,dashed] (nw) -- (sw); \draw[arrow,dashed] (ne) -- (se); 
    \path (0,.8) coordinate (base);
  \end{tikzpicture}
 \quad
 \underset{
   \begin{tikzpicture}[baseline=(base)]
    \path (0,0) node[dot] (sw) {} (0,.5) node[dot] (nw) {} (0,.6) coordinate(top);
    \draw[arrow,dashed] (nw) -- (sw);  
    \path (0,.3) coordinate (base);
  \end{tikzpicture}
 }
 \cup
 \quad
  \tikz[baseline=(base)] \path node[dot] {} ++(0,-3pt) coordinate(base);
  $$
  Since the inclusion 
  $\begin{tikzpicture}[baseline=(base)]
    \path (0,0) node[dot] (sw) {} (0,.5) node[dot] (nw) {} ;
    \draw[arrow,dashed] (nw) -- (sw);  
    \path (0,.1) coordinate (base);
  \end{tikzpicture} \cup \begin{tikzpicture}[baseline=(base)]
    \path (0,0) node[dot] (sw) {} (0,.5) node[dot] (nw) {} ;
    \draw[arrow,dashed] (nw) -- (sw);  
    \path (0,.1) coordinate (base);
  \end{tikzpicture} \hookrightarrow 
  \begin{tikzpicture}[baseline=(base)]
    \path (0,0) node[dot] (sw) {} (0,.5) node[dot] (nw) {} (.5,0) node[dot] (se) {} (.5,.5) node[dot] (ne) {};
    \path (.1,.1) coordinate (sw1) (.1,.4) coordinate (nw1) (.4,.1) coordinate (se1) (.4,.4) coordinate (ne1);
    \path[fill=black!25] (sw1) -- (nw1) -- (ne1) -- (se1) -- (sw1);
    \draw[arrow] (nw) -- (ne); \draw[arrow] (sw) -- (se);
    \draw[arrow,dashed] (nw) -- (sw); \draw[arrow,dashed] (ne) -- (se); 
    \path (0,.1) coordinate (base);
  \end{tikzpicture}
  $
  is an injective-fibration, this pushout is also a homotopy pushout in the injective model on $2$-uple simplicial spaces.  Just as in \cref{remark.failure of cofibrancy}, it follows that a projective-cofibrant replacement of $\Theta^{(2)}$ can be computed by replacing the above pushouts instead as homotopy pushouts in the projective model structure on $2$-uple simplicial spaces:
  $$ \widetilde{\Theta^{(2)}} \quad = \quad
  \tikz[baseline=(base)] \path node[dot] {} ++(0,-3pt) coordinate(base);
 \quad
 \underset{
   \begin{tikzpicture}[baseline=(base)]
    \path (0,0) node[dot] (sw) {} (0,.5) node[dot] (nw) {} (0,.6) coordinate(top);
    \draw[arrow,dashed] (nw) -- (sw);  
    \path (0,.3) coordinate (base);
  \end{tikzpicture}
 }
 {\overset h\cup}
 \quad
  \begin{tikzpicture}[baseline=(base)]
    \path (0,0) node[dot] (sw) {} (0,2) node[dot] (nw) {} (2,0) node[dot] (se) {} (2,2) node[dot] (ne) {};
    \path (.2,.2) coordinate (sw1) (.2,1.8) coordinate (nw1) (1.8,.2) coordinate (se1) (1.8,1.8) coordinate (ne1);
    \draw[ultra thin,fill=black!25] (sw1) -- (nw1) -- (ne1) -- (se1) -- (sw1);
    \draw[arrow] (nw) -- (ne); \draw[arrow] (sw) -- (se);
    \draw[arrow,dashed] (nw) -- (sw); \draw[arrow,dashed] (ne) -- (se); 
    \path (0,.8) coordinate (base);
  \end{tikzpicture}
 \quad
 \underset{
   \begin{tikzpicture}[baseline=(base)]
    \path (0,0) node[dot] (sw) {} (0,.5) node[dot] (nw) {} (0,.6) coordinate(top);
    \draw[arrow,dashed] (nw) -- (sw);  
    \path (0,.3) coordinate (base);
  \end{tikzpicture}
 }
 {\overset h\cup}
 \quad
  \tikz[baseline=(base)] \path node[dot] {} ++(0,-3pt) coordinate(base);$$
  
   Given a complete 2-uple Segal space $\cC$, we conclude:
  $$ \maps^{h}(\Theta^{(2)},\cC) = \cC_{0,0} \times^{h}_{\cC_{0,1}} \cC_{1,1} \times^{h}_{\cC_{0,1}} \cC_{0,0}. $$
  This is equivalent to $\cC_{(2)} = \cC_{1,1}$ provided $\cC$ is essentially constant, as then the maps $\cC_{0,0} \to \cC_{0,1}$ are homotopy equivalences. 
\end{remark}

We turn now to defining $\Theta^{(i);(j)}$ as an $(i+j)$-computad.  
We construct $\Theta^{(i);(j)}$ by gluing two parallel copies of $\Theta^{(i-1);(j)}$ and two parallel copies of $\Theta^{(i);(j-1)}$ together along the inclusions of lower $\Theta$'s and filling it with an $(i+j)$-dimensional morphism in a suitable way.  
The end result is a higher-categorical version of the Crans--Gray tensor product of~$\Theta^{(i)}$ with~$\Theta^{(j)}$~\cite{MR0371990,MR1667313}. 
We phrase the definition as the following result:

\begin{proposition}\label{prop.i-by-j-morphism}
  There is a unique system of computads $\{\Theta^{(i);(j)}\}_{i,j\in \NN}$, called the \define{walking  $i$-by-$j$-morphisms}, satisfying the following:
  \begin{enumerate}
    \item When either $i$ or $j$ vanishes, we have $\Theta^{(i);(0)} \cong \Theta^{(i)}$ and $\Theta^{(0);(j)} \cong \Theta^{(j)}$.
    \item Suppose $i$ and $j$ and both positive.
    
    The finite CW complex $\Gr(\Theta^{(i);(j)})$ is isomorphic to $\Gr(\Theta^{(i)})\times \Gr(\Theta^{(j)})$.  In particular, $\Theta^{(i);(j)}$ is an $(i+j)$-computad with a unique $(i+j)$-dimensional generator, which we call~$\theta_{i;j}$.
    
    The inclusions 
\begin{align*}
\Gr(s^{(i)})\times \id,\Gr(t^{(i)})\times \id : \Gr(\Theta^{(i-1)}) \times \Gr(\Theta^{(j)}) &\mono \Gr(\Theta^{(i)}) \times \Gr(\Theta^{(j)})\\
\id \times \Gr(s^{(j)}),\id \times \Gr(t^{(j)}) : \Gr(\Theta^{(i)})\times\Gr(\Theta^{(j-1)}) &\mono \Gr(\Theta^{(i)})\times\Gr(\Theta^{(j)})
\end{align*}
lift to inclusions of computads:
    \begin{gather*}
    s_{h}^{(i)}, t_{h}^{(i)} : \Theta^{(i-1);(j)} \mono \Theta^{(i);(j)} \\
    s_{v}^{(j)}, t_{v}^{(j)} : \Theta^{(i);(j-1)} \mono \Theta^{(i);(j)} 
    \end{gather*}
    In particular,  the combined map $s_{h}^{(i)} \sqcup t_{h}^{(i)} \sqcup s_{v}^{(j)}\sqcup t_{v}^{(j)} : \Theta^{(i-1);(j)} \sqcup \Theta^{(i-1);(j)} \sqcup \Theta^{(i);(j-1)} \sqcup \Theta^{(i);(j-1)} \to \Theta^{(i);(j)}$ is a surjection on $(i+j-1)$-skeleta and a bijection on the set of $(i+j-1)$-dimensional generators, of which there are four: $s_{h}^{(i)}\theta_{i-1;j}$, $t_{h}^{(i)}\theta_{i-1;j}$, $s_{v}^{(j)}\theta_{i,j-1}$, and $t_{v}^{(j)}\theta_{i,j-1}$.
    
    The subscripts ``$h$'' and ``$v$'' stand for ``horizontal'' and ``vertical.''
    \item Suppose $i$ and $j$ are both positive.  The $(i+j)$-dimensional generator $\theta_{i;j}$ of $\Theta^{(i);(j)}$ is glued onto the $(i+j-1)$-skeleton as follows:
    \begin{description}
      \item[$i$ is odd]  If $i$ is odd, each of the pairs $\{t_{v}^{(j)}\theta_{i,j-1},s_{h}^{(i)}\theta_{i-1;j}\}$ and $\{t_{h}^{(i)}\theta_{i-1;j},s_{v}^{(j)}\theta_{i,j-1}\}$ is composable up to whiskering by lower-dimensional generators, and their compositions are parallel.  The generator $\theta_{i;j}$ of $\Theta^{(i);(j)}$ has source $t_{v}^{(j)}\theta_{i;j-1} \circ s_{h}^{(i)}\theta_{i-1;j}$ and target $t_{h}^{(i)}\theta_{i-1;j} \circ s_{v}^{(j)}\theta_{i;j-1}$.
      \item[$i$ is even]  If $i$ is even, each of the pairs $\{s_{v}^{(j)}\theta_{i,j-1},s_{h}^{(i)}\theta_{i-1;j}\}$ and $\{t_{h}^{(i)}\theta_{i-1;j},t_{v}^{(j)}\theta_{i,j-1}\}$ is composable up to whiskering by lower-dimensional generators, and their compositions are parallel. The generator $\theta_{i;j}$ has source $s_{v}^{(j)}\theta_{i;j-1} \circ s_{h}^{(i)}\theta_{i-1;j}$ and target $t_{h}^{(i)}\theta_{i-1;j} \circ t_{v}^{(j)}\theta_{i;j-1}$.
    \end{description}
  \end{enumerate}
\end{proposition}

The rule unifying the two cases of condition~(3) is ``$\theta_{i;j}$ always goes from $s_{h}(\theta_{i-1;j})$ to $t_{h}(\theta_{i-1;j})$, with whichever compositions are necessary'' --- the ``horizontal'' direction is always given preference.  By employing also the rule ``omit things with negative sub- and superscripts,'' comparison with \cref{defn.walkingimorphism} shows that condition~(1) may be folded into condition~(3).

\begin{proof}
  Uniqueness is obvious, as condition~(2) determines the $(i+j-1)$-skeleton of $\Theta^{(i);(j)}$ and condition~(3) determines the $(i+j)$-dimensional data.  
  For existence, we must check three things: first, that the inclusions asserted in condition~(2) are compatible on double overlaps, so that they do indeed define the $(i+j-1)$-skeleton of $\Theta^{(i);(j)}$; second, that the pairs asserted to be composable in condition~(3) are indeed composable; and, third, that their compositions are parallel, so that the $(i+j)$-dimensional generator $\theta_{i;j}$ may be attached.  
  We work by induction in $i$ and $j$.
  In the remainder of the proof, we adopt the following convention: for sufficiently low $i$ and $j$, some sub- and superscripts might be negative, in which case the corresponding item is simply omitted.
  
  \textbf{Double overlaps.}
From the geometry of $\Gr(\Theta^{(i)})\times \Gr(\Theta^{(j)}) \cong \Gr(\Theta^{(i);(j)})$, we see that there are eight double overlaps, each with the geometry of $\Gr(\Theta^{(i')})\times \Gr(\Theta^{(j')}) \cong \Gr(\Theta^{(i');(j')})$ for $i'+j' = i+j-2$.  Indeed, in each case the double overlap asserts an equality of two {a priori} different inclusions of $\Theta^{(i');(j')}$, namely:
    \begin{align*}
     s^{(i)}_{h}s^{(i-1)}_{h}\Theta^{(i-2);(j)} & \equiv t^{(i)}_{h}s^{(i-1)}_{h}\Theta^{(i-2);(j)} ,  &
      s^{(i)}_{h}t^{(i-1)}_{h}\Theta^{(i-2);(j)} & \equiv t^{(i)}_{h}t^{(i-1)}_{h}\Theta^{(i-2);(j)} , \\
      s^{(j)}_{v}s^{(j-1)}_{v}\Theta^{(i);(j-2)} & \equiv t^{(j)}_{v}s^{(j-1)}_{v}\Theta^{(i);(j-2)},  &
       t^{(j)}_{v}s^{(j-1)}_{v}\Theta^{(i);(j-2)} & \equiv t^{(j)}_{v}t^{(j-1)}_{v}\Theta^{(i);(j-2)} , \\
       s^{(i)}_{h}s^{(j)}_{v}\Theta^{(i-1);(j-1)} & \equiv s^{(j)}_{v}s^{(i)}_{h}\Theta^{(i-1);(j-1)} ,  &
       s^{(i)}_{h}t^{(j)}_{v}\Theta^{(i-1);(j-1)} & \equiv t^{(j)}_{v}s^{(i)}_{h}\Theta^{(i-1);(j-1)} , \\
       t^{(i)}_{h}s^{(j)}_{v}\Theta^{(i-1);(j-1)} & \equiv s^{(j)}_{v}t^{(i)}_{h}\Theta^{(i-1);(j-1)} ,   &
       t^{(i)}_{h}t^{(j)}_{v}\Theta^{(i-1);(j-1)} & \equiv t^{(j)}_{v}t^{(i)}_{h}\Theta^{(i-1);(j-1)} .
  \end{align*}
  Inspection guarantees that these inclusions do match.

  \textbf{Composability.}
   We will describe the case when $i$ is odd; we leave the completely analogous even case to the reader.  Then $i-1$ is even, and by induction $s_{h}^{(i)}\theta_{i-1;j}$ has as source and target, respectively,
\begin{align*}
s_{h}^{(i)}\left( s_{v}^{(j)}(\theta_{i-1,j-1}) \circ s_{h}^{(i-1)}(\theta_{i-2;j}) \right) &= s_{h}^{(i)}s_{v}^{(j)}\theta_{i-1,j-1} \circ s_{h}^{(i)}s_{h}^{(i-1)}\theta_{i-2;j},\\
s_{h}^{(i)}\left( t_h^{(i-1)}(\theta_{i-2;j})\circ t_v^{(j)}(\theta_{i-1;j-1}) \right) &= s_{h}^{(i)}t_{h}^{(i-1)}\theta_{i-2;j} \circ s_{h}^{(i)}t_{v}^{(j)}\theta_{i-1;j-1}.
\end{align*}
  Similarly,
  \begin{alignat*}{2}
    t_{h}^{(i)}\theta_{i-1;j} :&\quad& t_{h}^{(i)}s_{v}^{(j)}\theta_{i-1;j-1} \circ t_{h}^{(i)}s_{h}^{(i-1)}\theta_{i-2;j} &\Rightarrow t_{h}^{(i)}t_{h}^{(i-1)}\theta_{i-2;j} \circ t_{h}^{(i)}t_{v}^{(j)}\theta_{i-1;j-1} \\
    s_{v}^{(j)}\theta_{i;j-1} :&& s_{v}^{(j)}t_{v}^{(j-1)}\theta_{i;j-2} \circ s_{v}^{(j)}s_{h}^{(i)}\theta_{i-1;j-1} & \Rightarrow s_{v}^{(j)}t_{h}^{(i)}\theta_{i-1;j-1} \circ s_{v}^{(j)}s_{v}^{(j-1)}\theta_{i;j-2} \\
    t_{v}^{(j)}\theta_{i;j-1} :&& t_{v}^{(j)}t_{v}^{(j-1)}\theta_{i;j-2} \circ t_{v}^{(j)}s_{h}^{(i)}\theta_{i-1;j-1} &\Rightarrow t_{v}^{(j)}t_{h}^{(i)}\theta_{i-1;j-1} \circ t_{v}^{(j)}s_{v}^{(j-1)}\theta_{i;j-2} 
  \end{alignat*}
Compare the target of $s_h^{(i)}\theta_{i-1;j}$ with the source of $t_v^{(j)}\theta_{i;j-1}$.  The double overlaps imply that $s_{h}^{(i)}t_{v}^{(j)}\theta_{i-1;j-1} = t_{v}^{(j)}s_{h}^{(i)}\theta_{i-1;j-1}$, and the cells $s_{h}^{(i)}t_{h}^{(i-1)}\theta_{i-2;j}$ and $t_{v}^{(j)}t_{v}^{(j-1)}\theta_{i;j-2}$ meet in dimension $i+j-4$, as follows:
  $$ 
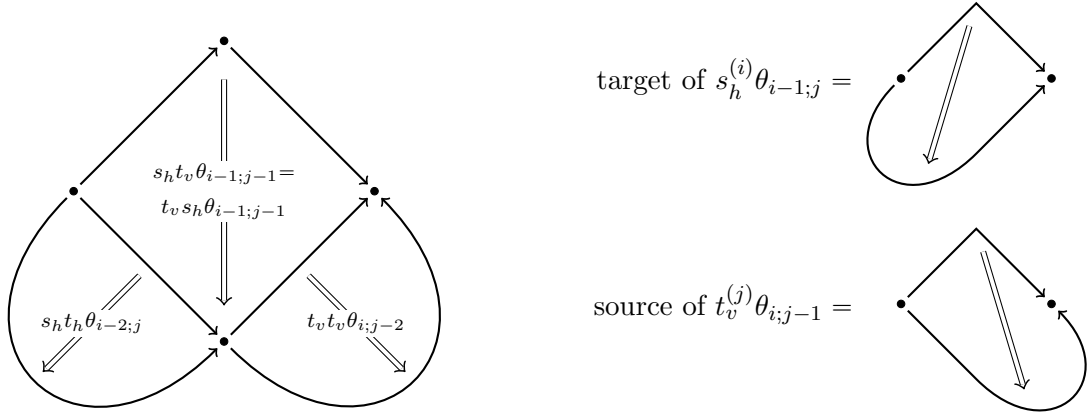
\begin{tikzpicture}
    \path (0,0) node[dot] (n) {} (-2,-2) node[dot] (w) {} (2,-2) node[dot] (e) {} (0,-4) node[dot] (s) {};
    \draw[arrow] (w) -- (n); \draw[arrow] (n) -- (e); 
    \draw[arrow] (w) -- coordinate (sw1) (s);
    \draw[arrow] (w) .. controls +(-2,-2) and +(-2,-2) .. coordinate (sw2) (s);
    \draw[arrow] (s) -- coordinate (se1) (e);
    \draw[arrow] (s) .. controls +(2,-2) and +(2,-2) .. coordinate (se2) (e);
    \draw[twoarrowlonger] (sw1) -- node[fill=white,inner sep=1pt] {$\scriptstyle s_{h}t_{h}\theta_{i-2;j}$} (sw2);
    \draw[twoarrowlonger] (se1) -- node[fill=white,inner sep=1pt] {$\scriptstyle t_{v}t_{v}\theta_{i;j-2}$} (se2);
    \draw[twoarrowshorter] (n) -- node[fill=white,inner sep=1pt,text width=2cm,text centered] {$\scriptstyle s_{h}t_{v}\theta_{i-1;j-1}  =  t_{v}s_{h}\theta_{i-1;j-1}$}(s);
    \path (9,-.5) node[dot] (upper L) {} +(2,0) node[dot] (upper R) {} +(-.5,0) node[anchor=east] {target of $s_h^{(i)}\theta_{i-1;j} =$};
    \path (9,-3.5) node[dot] (lower L) {} +(2,0) node[dot] (lower R) {} ++(-.5,0) node[anchor=east] {source of $t_v^{(j)}\theta_{i;j-1} =$};
    \draw[arrow] (upper L) -- +(1,1) coordinate (upper T) -- (upper R);
    \draw[arrow] (upper L) .. controls +(-1,-1) and +(-1,-1) .. coordinate[near end] (upper B) +(1,-1) -- (upper R);
    \draw[twoarrow] (upper T) -- (upper B);
    \draw[arrow] (lower L) -- +(1,1) coordinate (lower T) -- (lower R);
    \draw[arrow] (lower L) -- ++(1,-1) .. controls +(1,-1) and +(1,-1) .. coordinate[near start] (lower B) (lower R);
    \draw[twoarrow] (lower T) -- (lower B);
  \end{tikzpicture}$$
  It follows that $s_{h}^{(i)}\theta_{i-1;j}$ may be whiskered by $t_{v}^{(j)}t_{v}^{(j-1)}\theta_{i;j-2}$ to produce an $(i+j-2)$-cell with target the above triple composition.  Similarly, the whiskering of $t_{v}^{(j)}\theta_{i;j-1}$ by $s_{h}^{(i)}t_{h}^{(i-1)}\theta_{i-2;j}$ has the above composition as its source.  Thus, up to these whiskerings, $s_{h}^{(i)}\theta_{i-1;j}$ and $t_{v}^{(j)}\theta_{i;j-1}$ are a composable:
  $$  \hspace*{-1cm}
  \tikz \path (0,0) node[inner sep=0pt] (a) {$\scriptstyle
       t_{v}^{(j)}t_{v}^{(j-1)}\theta_{i;j-2} \circ s_{h}^{(i)}s_{v}^{(j)}\theta_{i-1;j-1} \circ s_{h}^{(i)}s_{h}^{(i-1)}\theta_{i-2;j}
  $} (10,0) node (b) {$\scriptstyle 
       s_{h}^{(i)}t_{h}^{(i-1)}\theta_{i-2;j} \circ t_{v}^{(j)}t_{h}^{(i)}\theta_{i-1;j-1} \circ t_{v}^{(j)}s_{v}^{(j-1)}\theta_{i;j-2}
  $} [draw,twoarrowlonger] (a) -- node[auto] {$\scriptstyle 
       t_{v}^{(j)}\theta_{i,j-1} \circ s_{h}^{(i)}\theta_{i-1;j}
  $} (b); 
  \hspace*{-1cm} $$

  A similar analysis verifies that one may whisker $s_{v}^{(j)}\theta_{i;j-1}$ by $t_{h}^{(i)}s_{h}^{(i-1)}\theta_{i-2;j}$
  to get an $(i+j-1)$-morphism which may be composed with the whiskering of $t_{h}^{(i)}\theta_{i-1;j}$ by $s_{v}^{(j)}s_{v}^{j-1)}\theta_{i;j-2}$:
  $$  \hspace*{-1cm}
  \tikz \path (0,0) node[inner sep=0pt] (a) {$\scriptstyle
       s_{v}^{(j)}t_{v}^{(j-1)}\theta_{i;j-2} \circ s_{v}^{(j)}s_{h}^{(i)}\theta_{i-1;j-1} \circ t_{h}^{(i)}s_{h}^{(i-1)}\theta_{i-2;j}
  $} (10,0) node (b) {$\scriptstyle 
       t_{h}^{(i)}t_{h}^{(i-1)}\theta_{i-2;j} \circ t_{h}^{(i)}t_{v}^{(j)}\theta_{i-1;j-1} \circ s_{v}^{(j)}s_{v}^{j-1)}\theta_{i;j-2}
  $} [draw,twoarrowlonger] (a) -- node[auto] {$\scriptstyle 
       t_{h}^{(i)}\theta_{i-1;j} \circ s_{v}^{(j)}\theta_{i;j-1}
  $} (b); 
  \hspace*{-1cm} $$
  
  \textbf{Parallelism.}
  Finally, we must check that these two compositions 
  $t_{v}^{(j)}\theta_{i,j-1} \circ s_{h}^{(i)}\theta_{i-1;j}$ and $t_{h}^{(i)}\theta_{i-1;j} \circ s_{v}^{(j)}\theta_{i;j-1}$  
   are parallel.  This will follow from the double overlap identitifications.  For the sources of the two morphisms, we have
\begin{align*}
t_{v}^{(j)}t_{v}^{(j-1)}\theta_{i;j-2} &= s_{v}^{(j)}t_{v}^{(j-1)}\theta_{i;j-2},\\
s_{h}^{(i)}s_{v}^{(j)}\theta_{i-1;j-1} &= s_{v}^{(j)}s_{h}^{(i)}\theta_{i-1;j-1},\\
s_{h}^{(i)}s_{h}^{(i-1)}\theta_{i-2;j} &= t_{h}^{(i)}s_{h}^{(i-1)}\theta_{i-2;j};
\end{align*}
hence the sources of the two compositions agree.  For the targets we have \begin{align*}
s_{h}^{(i)}t_{h}^{(i-1)}\theta_{i-2;j} &= t_{h}^{(i)}t_{h}^{(i-1)}\theta_{i-2;j},\\
t_{v}^{(j)}t_{h}^{(i)}\theta_{i-1;j-1} &= t_{h}^{(i)}t_{v}^{(j)}\theta_{i-1;j-1},\\
t_{v}^{(j)}s_{v}^{(j-1)}\theta_{i;j-2} &= s_{v}^{(j)}s_{v}^{j-1)}\theta_{i;j-2}.
\end{align*}
  Thus $t_{v}^{(j)}\theta_{i,j-1} \circ s_{h}^{(i)}\theta_{i-1;j}$ and $t_{v}^{(j)}\theta_{i,j-1} \circ s_{h}^{(i)}\theta_{i-1;j}$ are parallel $(i+j-1)$-morphisms in $\partial\Theta^{(i);(j)}$, and $\theta_{i;j} : t_{v}^{(j)}\theta_{i,j-1} \circ s_{h}^{(i)}\theta_{i-1;j} \Rightarrow t_{v}^{(j)}\theta_{i,j-1} \circ s_{h}^{(i)}\theta_{i-1;j}$ may be defined, and the induction may be continued.
\end{proof}

\begin{remark}
  \cref{prop.i-by-j-morphism} can be rephrased as saying that (for $i,j > 1$) the skeleton $\partial \Theta^{(i);(j)}$ arises as a strict colimit:
  $$ \partial \Theta^{(i);(j)} = \operatorname{colim} \left\{
    \begin{tikzpicture}[baseline=(base)]
      \path (0,3.35) coordinate (base)
        (0,0) node (i0j2A) {} node[anchor=east] {$\Theta^{(i);(j-2)}$}
        ++(0,1) node (i0j2B) {} node[anchor=east] {$\Theta^{(i);(j-2)}$}
        ++(0,1) node (i1j1A) {} node[anchor=east] {$\Theta^{(i-1);(j-1)}$}
        ++(0,1) node (i1j1B) {} node[anchor=east] {$\Theta^{(i-1);(j-1)}$}
        ++(0,1) node (i1j1C) {} node[anchor=east] {$\Theta^{(i-1);(j-1)}$}
        ++(0,1) node (i1j1D) {} node[anchor=east] {$\Theta^{(i-1);(j-1)}$}
        ++(0,1) node (i2j0A) {} node[anchor=east] {$\Theta^{(i-2);(j)}$}
        ++(0,1) node (i2j0B) {} node[anchor=east] {$\Theta^{(i-2);(j)}$};
      \path
        (5,.5) node (i0j1A) {} node[anchor=west] {$\Theta^{(i);(j-1)}$}
        ++(0,2) node (i0j1B) {} node[anchor=west] {$\Theta^{(i);(j-1)}$}
        ++(0,2) node (i1j0A) {} node[anchor=west] {$\Theta^{(i-1);(j)}$}
        ++(0,2) node (i1j0B) {} node[anchor=west] {$\Theta^{(i-1);(j)}$};
      \draw[right hook->] (i0j2A) -- (i0j1A);
      \draw[right hook->] (i0j2A) -- (i0j1B);
      \draw[right hook->] (i0j2B) -- (i0j1A);
      \draw[right hook->] (i0j2B) -- (i0j1B);
      \draw[right hook->] (i2j0A) -- (i1j0A);
      \draw[right hook->] (i2j0B) -- (i1j0A);
      \draw[right hook->] (i2j0A) -- (i1j0B);
      \draw[right hook->] (i2j0B) -- (i1j0B);
      \draw[right hook->] (i1j1A) -- (i0j1A);
      \draw[right hook->] (i1j1A) -- (i1j0A);
      \draw[right hook->] (i1j1B) -- (i0j1A);
      \draw[right hook->] (i1j1B) -- (i1j0B);
      \draw[right hook->] (i1j1C) -- (i0j1B);
      \draw[right hook->] (i1j1C) -- (i1j0A);
      \draw[right hook->] (i1j1D) -- (i0j1B);
      \draw[right hook->] (i1j1D) -- (i1j0B);
    \end{tikzpicture}
  \right\} $$
  Although the individual maps are inclusions, this is not a pushout along inclusions, because the $(i+j-2)$-computads in question overlap too much in dimension $(i+j-3)$.  So the above colimit cannot be upgraded to a homotopy colimit without including skeleta at all dimensions.
\end{remark}
  
  \begin{definition}\label{defn.horizontal skeleton}
  In the proof  of \cref{mainthm}, we will use the following two subcomputads of $\partial\Theta^{(i);(j)}$.  The \define{horizontal skeleton} $\partial_{h}\Theta^{(i);(j)}$ of $\Theta^{(i);(j)}$ consists of two copies of $\Theta^{(i-1);(j)}$ glued along two copies of $\Theta^{(i-2);(j)}$; its underlying CW complex is $\Gr(\partial_{h}\Theta^{(i);(j)}) = \Gr(\partial \Theta^{(i)})\times \Gr(\Theta^{(j)})$. The \define{vertical skeleton} $\partial_{v}\Theta^{(i);(j)}$ consists of two copies of $\Theta^{(i);(j-1)}$ glued along two copies of $\Theta^{(i);(j-2)}$.  \label{remark.horizontal skeleton}
\end{definition}

\section{Examples for small \texorpdfstring{\lowercase{$i$}}{i} and \texorpdfstring{\lowercase{$j$}}{j}} \label{section.examples}

We will henceforth drop the superscripts $(i)$ and $(j)$ from the maps $s_{h},\dots,t_{v}$.  In diagrams we will sometimes drop $\theta_{i;j}$s; they can be filled in from the geometry of the diagram.  It is worth working out explicitly the walking  $i$-by-$j$ morphisms for low $i$ and $j$. 

\begin{example} When $j$ vanishes we get the walking $i$-morphism and when $i$ vanishes we get the walking $j$-morphism:
\\[6pt]
 \mbox{} \hfill
$ \displaystyle \Theta^{(i);(0)} = \Theta^{(i)},\quad \Theta^{(0);(j)} = \Theta^{(j)}. $ \end{example}

\begin{example}
Perhaps the most important example is $\Theta^{(1);(1)}$:
\\
\mbox{}\\ \mbox{} \hfill
$ \displaystyle \Theta^{(1);(1)} = \left\{ \begin{tikzpicture}[baseline=(r.base)]
  \path node[dot] (sl) {} +(2,0) node[dot] (sr) {} +(0,-2) node[dot] (tl) {} +(2,-2) node[dot] (tr) {};
  \draw[arrow] (sl) -- node[auto] {$\scriptstyle s_{v}\theta_{1;0}$} (sr);
  \draw[arrow] (sl) -- node[auto,swap] {$\scriptstyle s_{h}\theta_{0;1}$} (tl);
  \draw[arrow] (tl) -- node[auto,swap] {$\scriptstyle t_{v}\theta_{1;0}$}  (tr);
  \draw[arrow] (sr) -- node[auto] (r) {$\scriptstyle t_{h}\theta_{0;1}$} (tr);
  \draw[twoarrow] (tl) -- node[fill=white,inner sep=1pt] {$ \theta_{1;1}$} (sr);
\end{tikzpicture} \right\}$ \end{example}

\begin{example}
Raising $j$ to $2$ gives:
$$\Theta^{(1);(2)} = \left\{ \begin{tikzpicture}[baseline=(middle)]
  \path node[dot] (a) {} +(3,0) node[dot] (b) {} +(0,-3) node[dot] (c) {} +(3,-3) node[dot] (d) {} +(0,-1.5) coordinate (middle);
  \draw[arrow] (a) .. controls +(-.75,-1.5) and +(-.75,1.5) .. coordinate[near end] (lt) coordinate[very near end] (lt2) (c);
  \draw[arrow,thin] (a) .. controls +(.75,-1.5) and +(.75,1.5) .. coordinate[near start] (ls) coordinate (ls2) (c);
  \draw[twoarrowlonger] (ls) -- node[fill=white,inner sep=1pt] {$\scriptstyle s_{h}\theta_{0;2}$} (lt);
  \draw[arrow] (b) .. controls +(.75,-1.5) and +(.75,1.5) .. coordinate[near start] (rs) coordinate[very near start] (rs0) (d);
  \draw[twoarrow] (ls2) -- coordinate[near end] (s) node[near end,auto] {$\scriptstyle s_{v}\theta_{1;1}$} (rs0);
  \draw[arrow] (b) .. controls +(-.75,-1.5) and +(-.75,1.5) .. coordinate[near end] (rt) coordinate (rt0) (d);
  \draw[twoarrowlonger,thick] (rs) -- node[fill=white,inner sep=1pt] {$\scriptstyle t_{h}\theta_{0;2}$} (rt);
  \draw[arrow] (a) -- (b);
  \draw[arrow] (c) -- (d);
  \path (lt2) -- coordinate[near start] (t) (rt0);
  \draw[threearrowpart1] (t) --  (s); 
  \draw[threearrowpart2] (t) -- (s); 
  \path[threearrowpart3] (t) -- node[fill=white,inner sep=1pt] {$\theta_{1;2}$} (s); 
  \draw[twoarrow,thick] (lt2) -- node[auto,swap] {$\scriptstyle t_{v}\theta_{1;1}$} (rt0);
\end{tikzpicture}
\right\}$$
The source of the 3-morphism $\theta_{1;2}$ is the left and front sides, and the target is the back and right. Unfolding the left and right sides,
$$ 
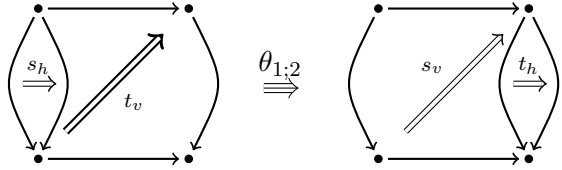
\begin{tikzpicture}[baseline=(ls1.base)]
  \path node[dot] (a) {} +(2,0) node[dot] (b) {} +(0,-2) node[dot] (c) {} +(2,-2) node[dot] (d) {};
  \draw[arrow] (a) .. controls +(-.5,-1) and +(-.5,1) .. coordinate (ls)  node (ls1) {\phantom{1}} (c);
  \draw[arrow] (a) .. controls +(.5,-1) and +(.5,1) .. coordinate (lt)  (c);
  \draw[twoarrowlonger] (ls) --  node[auto] {$\scriptstyle s_{h}$} (lt);
  \draw[arrow] (a) -- (b);
  \draw[arrow] (c) -- (d);
  \draw[arrow] (b) .. controls +(.5,-1) and +(.5,1) .. coordinate (rt)  (d);
  \draw[twoarrowshorter, thick] (c) --  node[auto,swap] {$\scriptstyle t_{v}$} (b);
\end{tikzpicture} 
\quad\overset{\textstyle\theta_{1;2}}\threearrow\quad
\begin{tikzpicture}[baseline=(ls1.base)]
  \path node[dot] (a) {} +(2,0) node[dot] (b) {} +(0,-2) node[dot] (c) {} +(2,-2) node[dot] (d) {};
  \draw[arrow] (a) .. controls +(-.5,-1) and +(-.5,1) .. coordinate (ls)  node (ls1) {\phantom{1}} (c);
  \draw[arrow] (a) -- (b);
  \draw[arrow] (c) -- (d);
  \draw[arrow] (b) .. controls +(-.5,-1) and +(-.5,1) .. coordinate (rs)  (d);
  \draw[arrow] (b) .. controls +(.5,-1) and +(.5,1) .. coordinate (rt)  (d);
  \draw[twoarrowshorter, thin] (c) --  node[auto] {$\scriptstyle s_{v}$} (b);
  \draw[twoarrowlonger] (rs) --  node[auto] {$\scriptstyle t_{h}$} (rt);
\end{tikzpicture} 
$$
From above, the diagram looks like:
$$
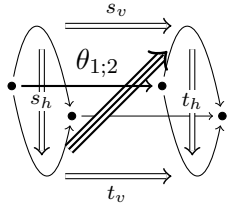
\begin{tikzpicture}
  \path node (sl) {} +(2,0) node (sr) {} +(0,-2) node (tl) {} +(2,-2) node (tr) {};
  \draw[twoarrowlonger] (sl) -- node[auto] {$\scriptstyle s_{v}$} (sr);
  \draw[twoarrowlonger] (sl) -- coordinate (l) node[fill=white,inner sep=1pt] {$\scriptstyle s_{h}$} (tl);
  \draw[twoarrowlonger] (tl) -- node[auto,swap] {$\scriptstyle t_{v}$}  (tr);
  \path (sr) -- coordinate (r)  (tr);
  \path (l) +(-.4,.2) node[dot] (a) {};
  \path (l) +(.4,-.2) node[dot] (b) {};
  \path (r) +(-.4,.2) node[dot] (c) {};
  \path (r) +(.4,-.2) node[dot] (d) {};
  \draw[arrow,thin] (a) .. controls +(.25,1.25) and +(-.25,1.25) .. (b);
  \draw[arrow,thin] (a) .. controls +(.25,-1.25) and +(-.25,-1.25) .. (b);
  \draw[arrow,thin] (c) .. controls +(.25,1.25) and +(-.25,1.25) .. (d);
  \draw[arrow,thin] (c) .. controls +(.25,-1.25) and +(-.25,-1.25) .. (d);
  \draw[arrow,very thin] (b) -- (d);
  \draw[threearrowpart1] (tl) -- node[auto,pos=.6] {$ \theta_{1;2}$} (sr);
  \draw[threearrowpart2] (tl) --  (sr);
  \draw[threearrowpart3] (tl) --  (sr);
  \draw[arrow, thick] (a) -- (c);
  \draw[twoarrowlonger] (sr) -- node[fill=white,inner sep=1pt] {$\scriptstyle t_{h}$} (tr);
\end{tikzpicture}$$
Looking at the 2-morphisms and 3-morphism, this perspective makes clear that $\Theta^{(1);(2)}$ has combinatorics similar to $\Theta^{(1);(1)}$.  This is because $i=1$ is odd. \end{example} \begin{example} Fixing $i=1$ but letting $j$ increase gives a sequence of analogous diagrams.  For example, $\Theta^{(1);(3)}$ has the following four faces:
$$
  
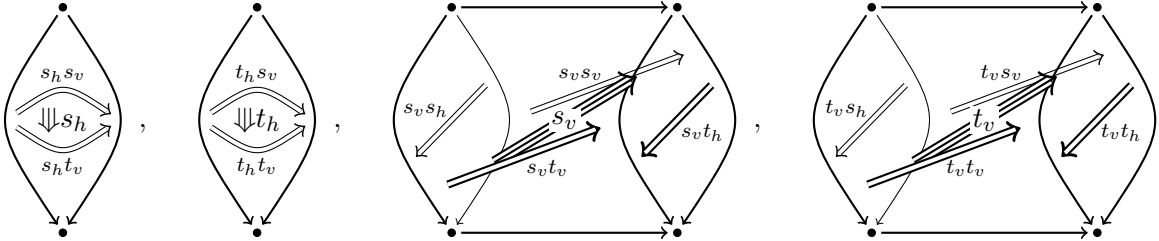
\begin{tikzpicture}[baseline=(l.base)]
    \path node[dot] (a) {} +(0,-3) node[dot] (b) {};
    \draw[arrow] (a) .. controls +(-1,-1.5) and +(-1,1.5) .. coordinate (ab1) (b);
    \draw[arrow] (a) .. controls +(1,-1.5) and +(1,1.5) .. coordinate (ab2) (b);
    \draw[twoarrowlonger] (ab1) .. controls +(.75,.5) and +(-.75,.5) .. node[auto] (sl) {$\scriptstyle s_{h}s_{v}$} (ab2);
    \draw[twoarrowlonger] (ab1) .. controls +(.75,-.5) and +(-.75,-.5) .. node[auto,swap] (tl) {$\scriptstyle s_{h}t_{v}$} (ab2);
    \path (sl) -- node (l) {$\raisebox{-3pt}{\rotatebox{90}{\ensuremath\Lleftarrow}} s_{h}$} (tl);
  \end{tikzpicture}
  , \quad
  \begin{tikzpicture}[baseline=(l.base)]
    \path node[dot] (a) {} +(0,-3) node[dot] (b) {};
    \draw[arrow] (a) .. controls +(-1,-1.5) and +(-1,1.5) .. coordinate (ab1) (b);
    \draw[arrow] (a) .. controls +(1,-1.5) and +(1,1.5) .. coordinate (ab2) (b);
    \draw[twoarrowlonger] (ab1) .. controls +(.75,.5) and +(-.75,.5) .. node[auto] (sl) {$\scriptstyle t_{h}s_{v}$} (ab2);
    \draw[twoarrowlonger] (ab1) .. controls +(.75,-.5) and +(-.75,-.5) .. node[auto,swap] (tl) {$\scriptstyle t_{h}t_{v}$} (ab2);
    \path (sl) -- node (l) {$\raisebox{-3pt}{\rotatebox{90}{\ensuremath\Lleftarrow}} t_{h}$} (tl);
  \end{tikzpicture}
  ,\quad
\begin{tikzpicture}[baseline=(alpha.base),inner sep=1pt]
  \path node[dot] (a) {} +(3,0) node[dot] (b) {} +(0,-3) node[dot] (c) {} +(3,-3) node[dot] (d) {};
  \path (a) -- node (alpha) {\phantom{\raisebox{-3pt}{\rotatebox{90}{\ensuremath\Lleftarrow}}}} (c);
  \draw[arrow] (a) .. controls +(-1,-1.5) and +(-1,1.5) .. coordinate[near end] (lt) coordinate[very near end] (lt2) (c);
  \draw[arrow,thin] (a) .. controls +(1,-1.5) and +(1,1.5) .. coordinate[near start] (ls) coordinate (ls2) (c);
  \draw[twoarrowlonger] (ls) -- node[auto,swap] {$\scriptstyle s_{v}s_{h}$} (lt);
  \draw[arrow] (b) .. controls +(1,-1.5) and +(1,1.5) .. coordinate[near start] (rs) coordinate[very near start] (rs0) (d);
  \draw[twoarrow] (ls2) -- coordinate[near end] (s) node[auto] {$\scriptstyle s_{v}s_{v}$} (rs0);
  \draw[arrow] (b) .. controls +(-1,-1.5) and +(-1,1.5) .. coordinate[near end] (rt) coordinate (rt0) (d);
  \draw[twoarrowlonger,thick] (rs) -- node[auto] {$\scriptstyle s_{v}t_{h}$} (rt);
  \draw[arrow] (a) -- (b);
  \draw[arrow] (c) -- (d);
  \path (lt2) -- coordinate[near start] (t) (rt0);
  \draw[threearrowpart1] (t) --   (s); 
  \draw[threearrowpart2] (t) -- (s); 
  \path[threearrowpart3] (t) -- node  [fill=white,inner sep=1pt] {$s_{v}$} (s); 
  \draw[twoarrow,thick] (lt2) -- node[auto,swap] {$\scriptstyle s_{v}t_{v}$} (rt0);
\end{tikzpicture}
  ,\quad
\begin{tikzpicture}[baseline=(alpha.base),inner sep=1pt]
  \path node[dot] (a) {} +(3,0) node[dot] (b) {} +(0,-3) node[dot] (c) {} +(3,-3) node[dot] (d) {};
  \path (a) -- node (alpha) {\phantom{\raisebox{-3pt}{\rotatebox{90}{\ensuremath\Lleftarrow}}}} (c);
  \draw[arrow] (a) .. controls +(-1,-1.5) and +(-1,1.5) .. coordinate[near end] (lt) coordinate[very near end] (lt2) (c);
  \draw[arrow,thin] (a) .. controls +(1,-1.5) and +(1,1.5) .. coordinate[near start] (ls) coordinate (ls2) (c);
  \draw[twoarrowlonger] (ls) -- node[auto,swap] {$\scriptstyle t_{v}s_{h}$} (lt);
  \draw[arrow] (b) .. controls +(1,-1.5) and +(1,1.5) .. coordinate[near start] (rs) coordinate[very near start] (rs0) (d);
  \draw[twoarrow] (ls2) -- coordinate[near end] (s) node[auto] {$\scriptstyle t_{v}s_{v}$} (rs0);
  \draw[arrow] (b) .. controls +(-1,-1.5) and +(-1,1.5) .. coordinate[near end] (rt) coordinate (rt0) (d);
  \draw[twoarrowlonger,thick] (rs) -- node[auto] {$\scriptstyle t_{v}t_{h}$} (rt);
  \draw[arrow] (a) -- (b);
  \draw[arrow] (c) -- (d);
  \path (lt2) -- coordinate[near start] (t) (rt0);
  \draw[threearrowpart1] (t) --   (s); 
  \draw[threearrowpart2] (t) -- (s); 
  \path[threearrowpart3] (t) -- node  [fill=white,inner sep=1pt] {$t_{v}$} (s); 
  \draw[twoarrow,thick] (lt2) -- node[auto,swap] {$\scriptstyle t_{v}t_{v}$} (rt0);
\end{tikzpicture}
$$
These are glued together along $t_{v}t_{v} = s_{v}t_{v}$, $t_{v}s_{v} = s_{v}s_{v}$, $s_{v}s_{h} = s_{h}s_{v}$, $t_{v}s_{h} = s_{h}t_{v}$, $s_{v}t_{h} = t_{h}s_{v}$, and $t_{v}t_{h} = t_{h}t_{v}$.  With these gluings, the compositions $t_{v}\circ s_{h}$ and $t_{h}\circ s_{v}$ make sense, and to complete $\Theta^{(1);(3)}$ we attach a generating $4$-morphism $\theta_{1;3}:t_{v}\circ s_{h} \fourarrow t_{h}\circ s_{v}$. Note that this shows that again, $\Theta^{(1);(3)}$ has combinatorics similar to $\Theta^{(1);(1)}$:
$$
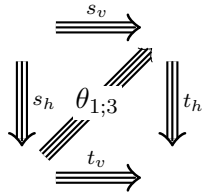
\begin{tikzpicture}
  \path node (sl) {} +(2,0) node (sr) {} +(0,-2) node (tl) {} +(2,-2) node (tr) {};
  \draw[threearrowpart1] (sl) -- node[auto] {$\scriptstyle s_{v}$} (sr);
  \draw[threearrowpart2] (sl) -- (sr);
  \draw[threearrowpart3] (sl) -- (sr);
  \draw[threearrowpart1] (sl) -- node[auto] {$\scriptstyle s_{h}$} (tl);
  \draw[threearrowpart2] (sl) -- (tl);
  \draw[threearrowpart3] (sl) -- (tl);
  \draw[threearrowpart1] (tl) -- node[auto] {$\scriptstyle t_{v}$} (tr);
  \draw[threearrowpart2] (tl) -- (tr);
  \draw[threearrowpart3] (tl) -- (tr);
  \draw[threearrowpart1] (sr) -- node[auto] {$\scriptstyle t_{h}$} (tr);
  \draw[threearrowpart2] (sr) -- (tr);
  \draw[threearrowpart3] (sr) -- (tr);
\draw (tl) node[anchor=south west] (TL) {};
\draw (sr) node[anchor= north east] (SR) {};
  \draw[fourarrowpart1] (TL) --  (SR);
  \draw[fourarrowpart2] (TL) -- node[fill=white,inner sep=1pt] {$ \theta_{1;3}$} (SR);
\end{tikzpicture}$$
\end{example}

\begin{example}
When $i=2$ the pattern changes, since $2$ is even.  The picture for $\Theta^{(2);(1)}$ is:
$$
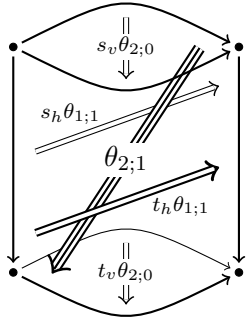
\begin{tikzpicture}[baseline=(middle)]
  \path node[dot] (a) {} +(3,0) node[dot] (b) {} +(0,-3) node[dot] (c) {} +(3,-3) node[dot] (d) {} +(0,-1.5) coordinate (middle);
  \path (b) +(-.35,.25) coordinate (alpha1);
  \path (c) +(.35,-.25) coordinate (alpha2);
  \draw[arrow] (a) -- coordinate (l1) coordinate[very near end] (r1) (c);
  \draw[arrow] (b) -- coordinate[very near start] (l2) coordinate (r2) (d);
  \draw[twoarrow,thin] (l1) -- node[auto,pos=.4,inner sep=1pt] {$\scriptstyle s_{h}\theta_{1;1}$} (l2);
  \draw[arrow] (a) .. controls +(1.5,.75) and +(-1.5,.75) .. coordinate (ss)  (b);
  \draw[arrow,thin] (c) .. controls +(1.5,.75) and +(-1.5,.75) .. coordinate (ts)  (d);
  \draw[arrow] (c) .. controls +(1.5,-.75) and +(-1.5,-.75) .. coordinate (tt) (d);
  \draw[twoarrowlonger,thin] (ts) -- node[fill=white,inner sep=1pt] {$\scriptstyle t_{v}\theta_{2;0}$} (tt);
  \draw[threearrowpart1] (alpha1) --  (alpha2);
  \draw[threearrowpart2] (alpha1) --  (alpha2);
  \draw[threearrowpart3] (alpha1) -- node[fill=white,inner sep=1pt] {$ \theta_{2;1}$} (alpha2);
  \draw[arrow,thick] (a) .. controls +(1.5,-.75) and +(-1.5,-.75) .. coordinate (st) (b);
  \draw[twoarrowlonger] (ss) -- node[fill=white,inner sep=1pt] {$\scriptstyle s_{v}\theta_{2;0}$} (st);
  \draw[twoarrow,thick] (r1) -- node[auto,swap,pos=.6,inner sep=1pt] {$\scriptstyle t_{h}\theta_{1;1}$} (r2);
\end{tikzpicture}$$
The 3-morphism $\theta_{2;1}$ goes from the composition of back and top to the composition of front and bottom (the picture above has ``folds'' along the back/top edge and along the front/bottom edge, obscuring slightly the composition; we unfold them below):
$$
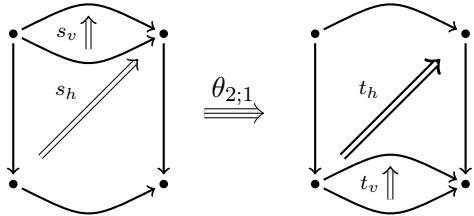
\begin{tikzpicture}[baseline=(ls1.base)]
  \path node[dot] (a) {} +(2,0) node[dot] (c) {} +(0,-2) node[dot] (b) {} +(2,-2) node[dot] (d) {};
  \draw[arrow] (a) .. controls +(1,-.5) and +(-1,-.5) .. coordinate (sl)   (c);
  \draw[arrow] (a) .. controls +(1,.5) and +(-1,.5) .. coordinate (sr)  (c);
  \draw[twoarrowlonger] (sl) --  node[auto] {$\scriptstyle s_{v}$} (sr);
  \draw[arrow] (a) -- node (ls1) {\phantom{1}} (b);
  \draw[arrow] (c) -- (d);
  \draw[arrow] (b) .. controls +(1,-.5) and +(-1,-.5) .. coordinate (tl)  (d);
  \draw[twoarrowshorter, thin] (b) --  node[auto] {$\scriptstyle s_{h}$} (c);
\end{tikzpicture} 
\quad\overset{\textstyle\theta_{2;1}}\threearrowlong\quad
\begin{tikzpicture}[baseline=(ls1.base)]
  \path node[dot] (a) {} +(2,0) node[dot] (c) {} +(0,-2) node[dot] (b) {} +(2,-2) node[dot] (d) {};
  \draw[arrow] (a) .. controls +(1,.5) and +(-1,.5) .. coordinate (sr)  (c);
  \draw[arrow] (a) -- node (ls1) {\phantom{1}} (b);
  \draw[arrow] (c) -- (d);
  \draw[arrow] (b) .. controls +(1,-.5) and +(-1,-.5) .. coordinate (tl)  (d);
  \draw[arrow] (b) .. controls +(1,.5) and +(-1,.5) .. coordinate (tr)  (d);
  \draw[twoarrowshorter, thick] (b) --  node[auto] {$\scriptstyle t_{h}$} (c);
  \draw[twoarrowlonger] (tl) --  node[auto] {$\scriptstyle t_{v}$} (tr);
\end{tikzpicture} 
$$
Here is a side view that puts $s_{h}$ on the left and $t_{h}$ on the right: note that this shows the combinatorics now have changed.

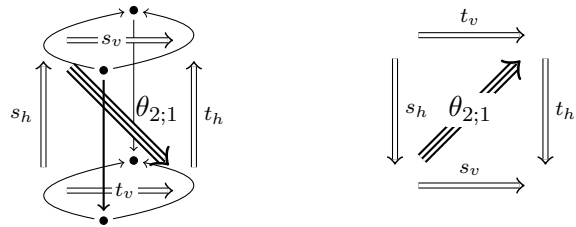
\mbox{} \\ \mbox{} \hfill $\displaystyle
\begin{tikzpicture}[baseline=(theta.base)]
  \path node (sl) {} +(2,0) node (sr) {} +(0,-2) node (tl) {} +(2,-2) node (tr) {};
  \path (sl) -- coordinate (s) (sr);
  \draw[twoarrowlonger] (tl) -- coordinate (l) node[auto] {$\scriptstyle s_{h}$} (sl);
  \draw[twoarrowlonger] (tl) -- coordinate (t) node[fill=white,inner sep=1pt,pos=.55] {$\scriptstyle t_{v}$}  (tr);
  \path (sr) -- coordinate (r)  (tr);
  \path (s) +(-.2,-.4) node[dot] (a) {};
  \path (s) +(.2,.4) node[dot] (b) {};
  \path (t) +(-.2,-.4) node[dot] (c) {};
  \path (t) +(.2,.4) node[dot] (d) {};
  \draw[arrow,thin] (a) .. controls +(1.25,.25) and +(1.25,-.25) .. (b);
  \draw[arrow,thin] (a) .. controls +(-1.25,.25) and +(-1.25,-.25) .. (b);
  \draw[arrow,thin] (c) .. controls +(1.25,.25) and +(1.25,-.25) .. (d);
  \draw[arrow,thin] (c) .. controls +(-1.25,.25) and +(-1.25,-.25) .. (d);
  \draw[arrow,very thin] (b) -- (d);
  \draw[twoarrowlonger] (sl) -- node[fill=white,inner sep=1pt,pos=.45]  {$\scriptstyle s_{v}$} (sr);
  \draw[threearrowpart1] (sl) -- node[auto,pos=.6,inner sep=1pt] (theta) {$ \theta_{2;1}$} (tr);
  \draw[threearrowpart2] (sl) --  (tr);
  \draw[threearrowpart3] (sl) --  (tr);
  \draw[arrow, thick] (a) -- (c);
  \draw[twoarrowlonger] (tr) -- node[auto,swap] {$\scriptstyle t_{h}$} (sr);
\end{tikzpicture}
\hspace{2cm}
\begin{tikzpicture}[baseline=(theta.base)]
  \path node (sl) {} +(2,0) node (sr) {} +(0,-2) node (tl) {} +(2,-2) node (tr) {};
  \path (sl) -- coordinate (s) (sr);
  \draw[twoarrowlonger] (sl) -- coordinate (l) node[auto] {$\scriptstyle s_{h}$} (tl);
  \draw[twoarrowlonger] (tl) -- coordinate (t) node[auto] {$\scriptstyle s_{v}$}  (tr);
  \path (sr) -- coordinate (r)  (tr);
  \draw[twoarrowlonger] (sl) -- node[auto]  {$\scriptstyle t_{v}$} (sr);
  \draw[threearrowpart1] (tl) --  (sr);
  \draw[threearrowpart2] (tl) --  (sr);
  \draw[threearrowpart3] (tl) -- node[fill=white,inner sep=1pt] (theta) {$ \theta_{2;1}$} (sr);
  \draw[twoarrowlonger] (sr) -- node[auto] {$\scriptstyle t_{h}$} (tr);
\end{tikzpicture}
$
\end{example}

\begin{example}
As a final example, the walking  $2$-by-$2$-morphism $\Theta^{(2);(2)}$ illustrates all the required gluings of the general case.  The $3$-skeleton $\partial\Theta^{(2);(2)}$ is covered by two copies of $\Theta^{(1);(2)}$,
$$
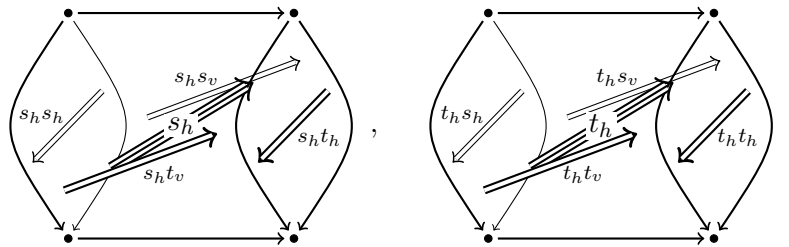
\begin{tikzpicture}[baseline=(alpha.base),inner sep=1pt]
  \path node[dot] (a) {} +(3,0) node[dot] (b) {} +(0,-3) node[dot] (c) {} +(3,-3) node[dot] (d) {};
  \path (a) -- node (alpha) {\phantom{\raisebox{-3pt}{\rotatebox{90}{\ensuremath\Lleftarrow}}}} (c);
  \draw[arrow] (a) .. controls +(-1,-1.5) and +(-1,1.5) .. coordinate[near end] (lt) coordinate[very near end] (lt2) (c);
  \draw[arrow,thin] (a) .. controls +(1,-1.5) and +(1,1.5) .. coordinate[near start] (ls) coordinate (ls2) (c);
  \draw[twoarrowlonger] (ls) -- node[auto,swap] {$\scriptstyle s_{h}s_{h}$} (lt);
  \draw[arrow] (b) .. controls +(1,-1.5) and +(1,1.5) .. coordinate[near start] (rs) coordinate[very near start] (rs0) (d);
  \draw[twoarrow] (ls2) -- coordinate[near end] (s) node[auto] {$\scriptstyle s_{h}s_{v}$} (rs0);
  \draw[arrow] (b) .. controls +(-1,-1.5) and +(-1,1.5) .. coordinate[near end] (rt) coordinate (rt0) (d);
  \draw[twoarrowlonger,thick] (rs) -- node[auto] {$\scriptstyle s_{h}t_{h}$} (rt);
  \draw[arrow] (a) -- (b);
  \draw[arrow] (c) -- (d);
  \path (lt2) -- coordinate[near start] (t) (rt0);
  \draw[threearrowpart1] (t) --   (s); 
  \draw[threearrowpart2] (t) -- (s); 
  \path[threearrowpart3] (t) -- node  [fill=white,inner sep=1pt] {$s_{h}$} (s); 
  \draw[twoarrow,thick] (lt2) -- node[auto,swap] {$\scriptstyle s_{h}t_{v}$} (rt0);
\end{tikzpicture}
  ,\quad
\begin{tikzpicture}[baseline=(alpha.base),inner sep=1pt]
  \path node[dot] (a) {} +(3,0) node[dot] (b) {} +(0,-3) node[dot] (c) {} +(3,-3) node[dot] (d) {};
  \path (a) -- node (alpha) {\phantom{\raisebox{-3pt}{\rotatebox{90}{\ensuremath\Lleftarrow}}}} (c);
  \draw[arrow] (a) .. controls +(-1,-1.5) and +(-1,1.5) .. coordinate[near end] (lt) coordinate[very near end] (lt2) (c);
  \draw[arrow,thin] (a) .. controls +(1,-1.5) and +(1,1.5) .. coordinate[near start] (ls) coordinate (ls2) (c);
  \draw[twoarrowlonger] (ls) -- node[auto,swap] {$\scriptstyle t_{h}s_{h}$} (lt);
  \draw[arrow] (b) .. controls +(1,-1.5) and +(1,1.5) .. coordinate[near start] (rs) coordinate[very near start] (rs0) (d);
  \draw[twoarrow] (ls2) -- coordinate[near end] (s) node[auto] {$\scriptstyle t_{h}s_{v}$} (rs0);
  \draw[arrow] (b) .. controls +(-1,-1.5) and +(-1,1.5) .. coordinate[near end] (rt) coordinate (rt0) (d);
  \draw[twoarrowlonger,thick] (rs) -- node[auto] {$\scriptstyle t_{h}t_{h}$} (rt);
  \draw[arrow] (a) -- (b);
  \draw[arrow] (c) -- (d);
  \path (lt2) -- coordinate[near start] (t) (rt0);
  \draw[threearrowpart1] (t) --   (s); 
  \draw[threearrowpart2] (t) -- (s); 
  \path[threearrowpart3] (t) -- node  [fill=white,inner sep=1pt] {$t_{h}$} (s); 
  \draw[twoarrow,thick] (lt2) -- node[auto,swap] {$\scriptstyle t_{h}t_{v}$} (rt0);
\end{tikzpicture},
$$
which are glued along $s_{h}s_{h} = t_{h}s_{h}$ and $s_{h}t_{h} = t_{h}t_{h}$ to make a solid torus, 
 and two copies of~$\Theta^{(2);(1)}$,
$$
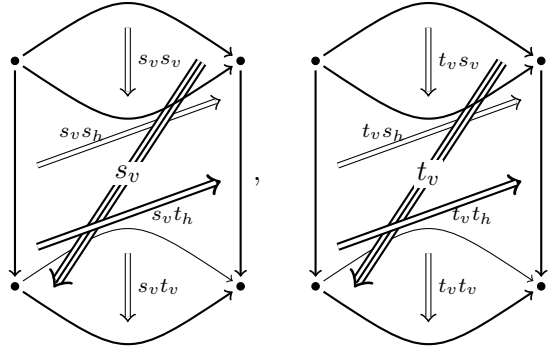
\begin{tikzpicture}[baseline=(beta.base)]
  \path node[dot] (a) {} +(3,0) node[dot] (b) {} +(0,-3) node[dot] (c) {} +(3,-3) node[dot] (d) {};
  \path (a) -- node (beta) {\phantom{\raisebox{-3pt}{\rotatebox{90}{\ensuremath\Lleftarrow}}}} (c);
  \path (b) +(-.35,.25) coordinate (alpha1);
  \path (c) +(.35,-.25) coordinate (alpha2);
  \draw[arrow] (a) -- coordinate (l1) coordinate[very near end] (r1) (c);
  \draw[arrow] (b) -- coordinate[very near start] (l2) coordinate (r2) (d);
  \draw[twoarrow,thin] (l1) -- node[auto,pos=.4,inner sep=0pt] {$\scriptstyle s_{v}s_{h}$} (l2);
  \draw[arrow] (a) .. controls +(1.5,1) and +(-1.5,1) .. coordinate (ss)  (b);
  \draw[arrow,thin] (c) .. controls +(1.5,1) and +(-1.5,1) .. coordinate (ts)  (d);
  \draw[arrow] (c) .. controls +(1.5,-1) and +(-1.5,-1) .. coordinate (tt) (d);
  \draw[twoarrow,thin] (ts) -- node[auto] {$\scriptstyle s_{v}t_{v}$} (tt);
  \draw[threearrowpart1] (alpha1) --  (alpha2);
  \draw[threearrowpart2] (alpha1) --  (alpha2);
  \draw[threearrowpart3] (alpha1) -- node  [fill=white,inner sep=1pt] {$s_{v}$} (alpha2);
  \draw[arrow,thick] (a) .. controls +(1.5,-1) and +(-1.5,-1) .. coordinate (st) (b);
  \draw[twoarrow] (ss) -- node[auto] {$\scriptstyle s_{v}s_{v}$} (st);
  \draw[twoarrow,thick] (r1) -- node[auto,swap,pos=.6,inner sep=0pt] {$\scriptstyle s_{v}t_{h}$} (r2);
\end{tikzpicture}
\;,\quad
\begin{tikzpicture}[baseline=(beta.base)]
  \path node[dot] (a) {} +(3,0) node[dot] (b) {} +(0,-3) node[dot] (c) {} +(3,-3) node[dot] (d) {};
  \path (a) -- node (beta) {\phantom{\raisebox{-3pt}{\rotatebox{90}{\ensuremath\Lleftarrow}}}} (c);
  \path (b) +(-.35,.25) coordinate (alpha1);
  \path (c) +(.35,-.25) coordinate (alpha2);
  \draw[arrow] (a) -- coordinate (l1) coordinate[very near end] (r1) (c);
  \draw[arrow] (b) -- coordinate[very near start] (l2) coordinate (r2) (d);
  \draw[twoarrow,thin] (l1) -- node[auto,pos=.4,inner sep=0pt] {$\scriptstyle t_{v}s_{h}$} (l2);
  \draw[arrow] (a) .. controls +(1.5,1) and +(-1.5,1) .. coordinate (ss)  (b);
  \draw[arrow,thin] (c) .. controls +(1.5,1) and +(-1.5,1) .. coordinate (ts)  (d);
  \draw[arrow] (c) .. controls +(1.5,-1) and +(-1.5,-1) .. coordinate (tt) (d);
  \draw[twoarrow,thin] (ts) -- node[auto] {$\scriptstyle t_{v}t_{v}$} (tt);
  \draw[threearrowpart1] (alpha1) --  (alpha2);
  \draw[threearrowpart2] (alpha1) --  (alpha2);
  \draw[threearrowpart3] (alpha1) -- node  [fill=white,inner sep=1pt] {$t_{v}$} (alpha2);
  \draw[arrow,thick] (a) .. controls +(1.5,-1) and +(-1.5,-1) .. coordinate (st) (b);
  \draw[twoarrow] (ss) -- node[auto] {$\scriptstyle t_{v}s_{v}$} (st);
  \draw[twoarrow,thick] (r1) -- node[auto,swap,pos=.6,inner sep=0pt] {$\scriptstyle t_{v}t_{h}$} (r2);
\end{tikzpicture}
\;,
$$
which are likewise glued into a solid torus along $s_{v}s_{v} = t_{v}s_{v}$ and $s_{v}t_{v} = t_{v}t_{v}$.  Finally, the two tori are glued together to form a 3-sphere, by identifying the following copies of $\theta_{1;1}$: $s_{h}s_{v} = s_{v}s_{h}$, $s_{h}t_{v}=t_{v}s_{h}$, $t_{h}s_{v} = s_{v}t_{h}$, and $t_{h}t_{v} = t_{v}t_{h}$.  This 3-sphere decomposes into two composable 3-balls, namely $s_{v}\circ s_{h}$ and $t_{h}\circ t_{v}$.  Completing $\Theta^{(2);(2)}$ is a 4-morphism $\theta_{2;2} : s_{v}\circ s_{h} \fourarrow t_{h}\circ t_{v}$ filling in the 4-ball with boundary this glued-up 3-sphere:
$$
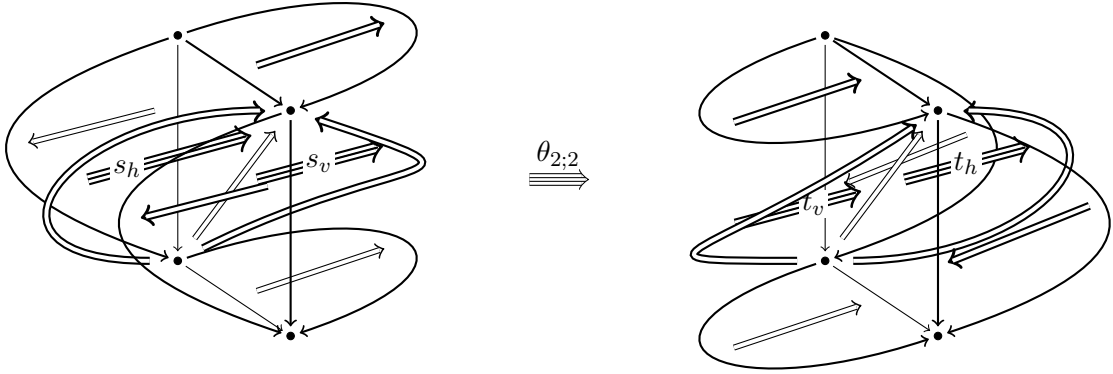
\begin{tikzpicture}[baseline=(middle)]
  \path node[dot] (a) {} +(0,-3) node[dot] (c) {} +(1.5,-1) node[dot] (b) {} +(1.5,-4) node[dot] (d) {};
  \path (a) -- coordinate(middle) (d); \path (middle) ++(2.25,.6) coordinate (sv1); \path (middle) +(.5,.75) coordinate (sh2) ++(-2.25,0) coordinate (sh1);
  \draw[arrow,thin] (a) -- coordinate [pos=.3] (ac) (c);
  \draw[arrow,thin] (c) -- coordinate (cd) (d);
  \draw[threearrowpart1] (middle) --  (sv1);
  \draw[threearrowpart2] (middle) --  (sv1);
  \draw[threearrowpart3] (middle) -- node  [fill=white,inner sep=1pt] {$s_{v}$} (sv1);
  \draw[threearrowpart1] (sh1) --  (sh2);
  \draw[threearrowpart2] (sh1) --  (sh2);
  \draw[threearrowpart3] (sh1) -- node  [fill=white,inner sep=1pt,pos=.3] {$s_{h}$} (sh2);
  \draw[arrow] (a) -- coordinate (ab) (b);
  \draw[arrow] (a) .. controls +(-3,-1) and +(-3,1) .. coordinate (ac1) (c);
  \draw[twoarrow] (ac) -- (ac1);
  \draw[arrow] (c) .. controls +(3,1) and +(3,1) .. coordinate (cd1) (d);
  \draw[twoarrow] (cd) -- (cd1);
  \draw[twoarrow] (c) -- (b);
  \draw[twoarrow,thick] (c) .. controls +(-3,0) and +(-3,0) .. (b);
  \draw[twoarrow,thick] (c) .. controls +(3,1.5) and +(3,-1) .. (b);
  \draw[arrow] (b) -- coordinate [pos=.3] (bd) (d);
  \draw[arrow] (b) .. controls +(-3,-1) and +(-3,1) .. coordinate (bd1) (d);
  \draw[twoarrow,thick] (bd) -- (bd1);
  \draw[arrow] (a) .. controls +(3,1) and +(3,1) .. coordinate (ab1) (b);
  \draw[twoarrow,thick] (ab) -- (ab1);
\end{tikzpicture}
\overset{\textstyle \theta_{2;2}} \fourarrowlong
\begin{tikzpicture}[baseline=(middle)]
  \path node[dot] (a) {} +(0,-3) node[dot] (c) {} +(1.5,-1) node[dot] (b) {} +(1.5,-4) node[dot] (d) {};
  \path (a) -- coordinate(middle) (d); \path (middle) ++(2.25,.6) coordinate (sv1); \path (middle) +(0,0) coordinate (sh2) ++(-2.25,-.6) coordinate (sh1);
  \draw[arrow,thin] (a) -- coordinate [pos=.7] (ac) (c);
  \draw[arrow,thin] (c) -- coordinate (cd) (d);
  \draw[threearrowpart1] (sh1) --  (sh2);
  \draw[threearrowpart2] (sh1) --  (sh2);
  \draw[threearrowpart3] (sh1) -- node  [fill=white,inner sep=1pt,pos=.6] {$t_{v}$} (sh2);
  \draw[arrow] (a) -- coordinate (ab) (b);
  \draw[arrow] (a) .. controls +(3,-1) and +(3,1) .. coordinate [pos=.4] (ac1) (c);
  \draw[twoarrow] (ac1) -- (ac);
  \draw[arrow] (c) .. controls +(-3,-1) and +(-3,-1) .. coordinate (cd1) (d);
  \draw[twoarrow] (cd1) -- (cd);
  \draw[threearrowpart1] (middle) --  (sv1);
  \draw[threearrowpart2] (middle) --  (sv1);
  \draw[threearrowpart3] (middle) -- node  [fill=white,inner sep=1pt] {$t_{h}$} (sv1);
  \draw[twoarrow] (c) -- (b);
  \draw[twoarrow,thick] (c) .. controls +(-3,0) and +(-3,-2) .. (b);
  \draw[twoarrow,thick] (c) .. controls +(3,0) and +(3,0) .. (b);
  \draw[arrow] (b) -- coordinate [pos=.7] (bd) (d);
  \draw[arrow] (b) .. controls +(3,-1) and +(3,1) .. coordinate [pos=.4] (bd1) (d);
  \draw[twoarrowlonger,thick] (bd1) -- (bd);
  \draw[arrow] (a) .. controls +(-3,-1) and +(-3,-1) .. coordinate (ab1) (b);
  \draw[twoarrow,thick] (ab1) -- (ab);
\end{tikzpicture}
$$
In spite of the distortions above, both $s_{v}\circ s_{h}$ and $t_{h}\circ t_{v}$ are 3-morphisms with the same source and target 2-morphisms:

$$
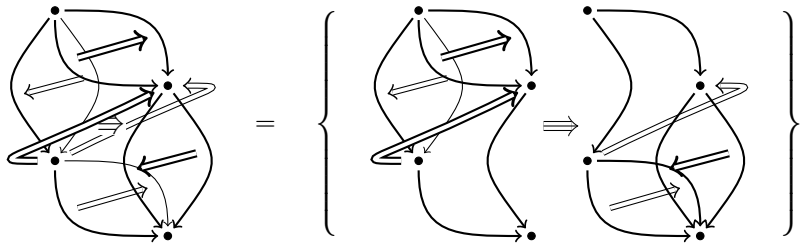
\begin{tikzpicture}[baseline=(middle)]
  \path node[dot] (a) {} +(0,-2) node[dot] (c) {} +(1.5,-1) node[dot] (b) {} +(1.5,-3) node[dot] (d) {} +(.75,-1.6) coordinate (middle);
  \draw[arrow,thin] (a) .. controls +(.75,-1) and +(.75,1) .. coordinate [pos=.4] (ac) (c);
  \draw[arrow] (a) .. controls +(-.75,-1) and +(-.75,1) .. coordinate [pos=.6] (ac1) (c);
  \draw[twoarrowlonger] (ac) -- (ac1);
  \draw[arrow,thin] (c) .. controls +(1,0) and +(0,1) .. coordinate[pos=.7] (cd1) (d);
  \draw[arrow] (c) .. controls +(0,-1) and +(-1,0) .. coordinate[pos=.3] (cd) (d);
  \draw[twoarrowlonger] (cd) -- (cd1);
  \draw[twoarrowlonger] (c) .. controls +(2,1) and +(1,0) .. (b);
  \path (middle) node[inner sep=0pt,fill=white,anchor=base] {$\Rrightarrow$};
  \draw[arrow] (b) .. controls +(.75,-1) and +(.75,1) .. coordinate [pos=.4] (bd) (d);
  \draw[arrow] (b) .. controls +(-.75,-1) and +(-.75,1) .. coordinate [pos=.6] (bd1) (d);
  \draw[twoarrowlonger,thick] (bd) -- (bd1);
  \draw[arrow] (a) .. controls +(1,0) and +(0,1) .. coordinate[pos=.7] (ab1) (b);
  \draw[arrow] (a) .. controls +(0,-1) and +(-1,0) .. coordinate[pos=.3] (ab) (b);
  \draw[twoarrowlonger,thick] (ab) -- (ab1);
  \draw[twoarrowlonger,thick] (c) .. controls +(-1,0) and +(-2,-1) .. (b);
\end{tikzpicture}
= \quad\left\lbrace
\begin{tikzpicture}[baseline=(middle)]
  \path node[dot] (a) {} +(0,-2) node[dot] (c) {} +(1.5,-1) node[dot] (b) {} +(1.5,-3) node[dot] (d) {} +(.75,-1.6) coordinate (middle);
  \draw[arrow,thin] (a) .. controls +(.75,-1) and +(.75,1) .. coordinate [pos=.4] (ac) (c);
  \draw[arrow] (a) .. controls +(-.75,-1) and +(-.75,1) .. coordinate [pos=.6] (ac1) (c);
  \draw[twoarrowlonger] (ac) -- (ac1);
  \draw[arrow] (c) .. controls +(0,-1) and +(-1,0) .. coordinate[pos=.3] (cd) (d);
  \draw[arrow] (b) .. controls +(-.75,-1) and +(-.75,1) .. coordinate [pos=.6] (bd1) (d);
  \draw[arrow] (a) .. controls +(1,0) and +(0,1) .. coordinate[pos=.7] (ab1) (b);
  \draw[arrow] (a) .. controls +(0,-1) and +(-1,0) .. coordinate[pos=.3] (ab) (b);
  \draw[twoarrowlonger,thick] (ab) -- (ab1);
  \draw[twoarrowlonger,thick] (c) .. controls +(-1,0) and +(-2,-1) .. (b);
\end{tikzpicture}
\threearrow
\begin{tikzpicture}[baseline=(middle)]
  \path node[dot] (a) {} +(0,-2) node[dot] (c) {} +(1.5,-1) node[dot] (b) {} +(1.5,-3) node[dot] (d) {} +(.75,-1.6) coordinate (middle);
  \draw[arrow] (a) .. controls +(.75,-1) and +(.75,1) .. coordinate [pos=.4] (ac) (c);
  \draw[arrow] (c) .. controls +(1,0) and +(0,1) .. coordinate[pos=.7] (cd1) (d);
  \draw[arrow] (c) .. controls +(0,-1) and +(-1,0) .. coordinate[pos=.3] (cd) (d);
  \draw[twoarrowlonger] (cd) -- (cd1);
  \draw[twoarrowlonger] (c) .. controls +(2,1) and +(1,0) .. (b);
  \draw[arrow] (b) .. controls +(.75,-1) and +(.75,1) .. coordinate [pos=.4] (bd) (d);
  \draw[arrow] (b) .. controls +(-.75,-1) and +(-.75,1) .. coordinate [pos=.6] (bd1) (d);
  \draw[twoarrowlonger,thick] (bd) -- (bd1);
  \draw[arrow] (a) .. controls +(1,0) and +(0,1) .. coordinate[pos=.7] (ab1) (b);
\end{tikzpicture}
\right\rbrace
$$
Note that the combinatorics again are similar to the previous example:
 \mbox{} \\ \mbox{} \hfill
$\displaystyle
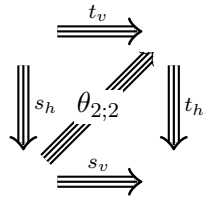
\begin{tikzpicture}
  \path node (sl) {} +(2,0) node (sr) {} +(0,-2) node (tl) {} +(2,-2) node (tr) {};
  \draw[threearrowpart1] (sl) -- node[auto] {$\scriptstyle t_{v}$} (sr);
  \draw[threearrowpart2] (sl) -- (sr);
  \draw[threearrowpart3] (sl) -- (sr);
  \draw[threearrowpart1] (sl) -- node[auto] {$\scriptstyle s_{h}$} (tl);
  \draw[threearrowpart2] (sl) -- (tl);
  \draw[threearrowpart3] (sl) -- (tl);
  \draw[threearrowpart1] (tl) -- node[auto] {$\scriptstyle s_{v}$} (tr);
  \draw[threearrowpart2] (tl) -- (tr);
  \draw[threearrowpart3] (tl) -- (tr);
  \draw[threearrowpart1] (sr) -- node[auto] {$\scriptstyle t_{h}$} (tr);
  \draw[threearrowpart2] (sr) -- (tr);
  \draw[threearrowpart3] (sr) -- (tr);
\draw (tl) node[anchor=south west] (TL) {};
\draw (sr) node[anchor= north east] (SR) {};
  \draw[fourarrowpart1] (TL) --  (SR);
  \draw[fourarrowpart2] (TL) -- node[fill=white,inner sep=1pt] {$ \theta_{2;2}$} (SR);
\end{tikzpicture}$$
$ 
\end{example}

We end this section by recording a useful fact relating the computads $\Theta^{1,\vec \bullet}$ and $\Theta^{\vec\bullet;(1)}$ defined in \cref{defn.veck} and \cref{defn.veck-by-vecl}.  We will use this fact in the proof of \cref{thm.qft}.

\begin{lemma}\label{lemma.untwisted}
  For each $i$, there is a canonical equivalence
  $$ \Theta^{(0);(0)} \underset{\Theta^{(i);(0)}}{\overset h \cup} \Theta^{(i);(1)} \underset{\Theta^{(i);(0)}}{\overset h \cup} \Theta^{(0);(0)} \simeq \Theta^{(i+1)} = \Theta^{1,(i)}$$
  compatible with (horizontal) source and target maps, where the two maps $\Theta^{(i);(0)} \rightrightarrows \Theta^{(i);(1)}$ are the inclusions $s_v,t_v$ as the vertical source and target.
\end{lemma}

\begin{proof}
  The combined map $s_v \cup t_v : \Theta^{(i);(0)} \cup \Theta^{(i);(0)} \to \Theta^{(i);(1)}$ is an inclusion, hence an injective-fibration, so the homotopy pushout is equivalent to the corresponding strict pushout.  We claim, therefore, that there is an isomorphism of computads
  $$ \Theta^{(0);(0)} \underset{\Theta^{(i);(0)}}{\cup} \Theta^{(i);(1)} \underset{\Theta^{(i);(0)}}{\cup} \Theta^{(0);(0)} \cong \Theta^{(i+1)}$$
  intertwining $s_h$ with $s : \Theta^{(i)} \to \Theta^{(i+1)}$ and $t_h$ with $t : \Theta^{(i)} \to \Theta^{(i+1)}$.  This follows from induction in $i$ and the description of $\Theta^{(i);(1)}$ in \cref{prop.i-by-j-morphism}:
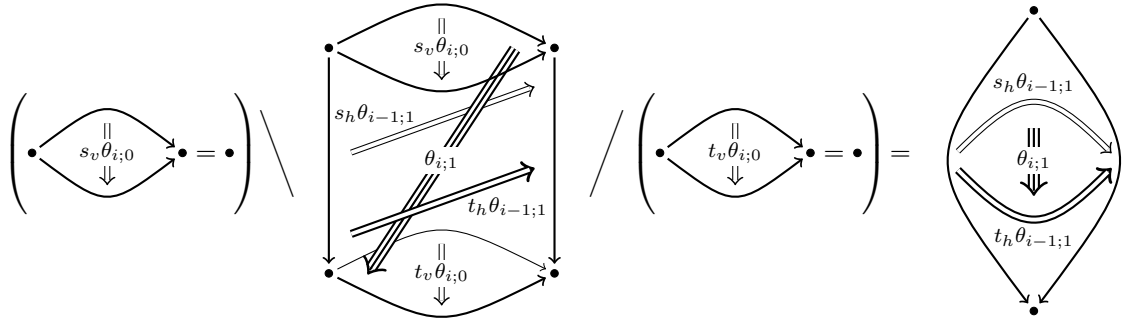
  $$
 \left( \begin{tikzpicture}[baseline=(base)]
  \path node[dot] (a) {} +(2,0) node[dot] (c) {} +(0,-3pt) coordinate (base);
  \draw[arrow] (a) .. controls +(1,-.75) and +(-1,-.75) .. coordinate (t)   (c);
  \draw[arrow] (a) .. controls +(1,.75) and +(-1,.75) .. coordinate (s)  (c);
  \draw[twoarrowlonger] (s) --  node[fill=white,inner sep=1pt] {$\scriptstyle s_{v}\theta_{i;0}$} (t);
\end{tikzpicture} = \tikz[baseline=(base)]\path node[dot]{}  +(0,-3pt) coordinate (base);\right)  {\Bigg \backslash} \quad
\begin{tikzpicture}[baseline=(middle)]
  \path node[dot] (a) {} +(3,0) node[dot] (b) {} +(0,-3) node[dot] (c) {} +(3,-3) node[dot] (d) {} +(0,-1.5) coordinate (middle);
  \path (b) +(-.35,.25) coordinate (alpha1);
  \path (c) +(.35,-.25) coordinate (alpha2);
  \draw[arrow] (a) -- coordinate (l1) coordinate[very near end] (r1) (c);
  \draw[arrow] (b) -- coordinate[very near start] (l2) coordinate (r2) (d);
  \draw[twoarrow,thin] (l1) -- node[auto,pos=.4,inner sep=1pt] {$\scriptstyle s_{h}\theta_{i-1;1}$} (l2);
  \draw[arrow] (a) .. controls +(1.5,.75) and +(-1.5,.75) .. coordinate (ss)  (b);
  \draw[arrow,thin] (c) .. controls +(1.5,.75) and +(-1.5,.75) .. coordinate (ts)  (d);
  \draw[arrow] (c) .. controls +(1.5,-.75) and +(-1.5,-.75) .. coordinate (tt) (d);
  \draw[twoarrowlonger,thin] (ts) -- node[fill=white,inner sep=1pt] {$\scriptstyle t_{v}\theta_{i;0}$} (tt);
  \draw[threearrowpart1] (alpha1) --  (alpha2);
  \draw[threearrowpart2] (alpha1) --  (alpha2);
  \draw[threearrowpart3] (alpha1) -- node[fill=white,inner sep=1pt] {$\scriptstyle \theta_{i;1}$} (alpha2);
  \draw[arrow,thick] (a) .. controls +(1.5,-.75) and +(-1.5,-.75) .. coordinate (st) (b);
  \draw[twoarrowlonger] (ss) -- node[fill=white,inner sep=1pt] {$\scriptstyle s_{v}\theta_{i;0}$} (st);
  \draw[twoarrow,thick] (r1) -- node[auto,swap,pos=.6,inner sep=1pt] {$\scriptstyle t_{h}\theta_{i-1;1}$} (r2);
\end{tikzpicture}
\quad {\Bigg /} 
\left( \begin{tikzpicture}[baseline=(base)]
  \path node[dot] (a) {} +(2,0) node[dot] (c) {} +(0,-3pt) coordinate (base);
  \draw[arrow] (a) .. controls +(1,-.75) and +(-1,-.75) .. coordinate (t)   (c);
  \draw[arrow] (a) .. controls +(1,.75) and +(-1,.75) .. coordinate (s)  (c);
  \draw[twoarrowlonger] (s) --  node[fill=white,inner sep=1pt] {$\scriptstyle t_{v}\theta_{i;0}$} (t);
\end{tikzpicture} = \tikz[baseline=(base)]\path node[dot]{}  +(0,-3pt) coordinate (base);\right) 
=
\begin{tikzpicture}[baseline=(middle)]
  \path node[dot] (a) {}  +(0,-4) node[dot] (c) {} +(0,-2) coordinate (middle);
  \draw[arrow] (a) .. controls +(-1.5,-2) and +(-1.5,2) .. coordinate (l) (c);
  \draw[arrow] (a) .. controls +(1.5,-2) and +(1.5,2) .. coordinate (r) (c);
  \draw[twoarrowlonger,thin] (l) .. controls +(1,1) and +(-1,1) .. node[auto] (s) {$\scriptstyle s_{h}\theta_{i-1;1}$} (r);
  \draw[twoarrowlonger,thick] (l) .. controls +(1,-1) and +(-1,-1) .. node[auto,swap] (t) {$\scriptstyle t_{h}\theta_{i-1;1}$} (r);
  \draw[threearrowpart1] (s) --  (t);
  \draw[threearrowpart2] (s) --  (t);
  \draw[threearrowpart3] (s) -- node[fill=white,inner sep=1pt] {$\scriptstyle \theta_{i;1}$} (t);
\end{tikzpicture}
$$
    Note in particular that the $(i+1)$-dimensional cell $\theta_{i;1} \in \Theta^{(i);(1)}$ goes from $s_h\theta_{i-1;1}$ to $t_h\theta_{i-1;1}$.
\end{proof}

\begin{remark}\label{remark.computad for oplax trivially twisted case}
  There is also an equivalence $\Theta^{(0);(0)} \cup^h_{\Theta^{(0);(j)}} \Theta^{(1);(j)} \cup^h_{\Theta^{(0);(j)}} \Theta^{(0);(0)} \simeq \Theta^{(j+1)}$, but it is not compatible with source and target maps.  More precisely, it intertwines $s_v$ with $s$ and $t_v$ with $t$ when $j$ is even, but exchanges them when $j$ is odd.
\end{remark}

\section{The (op)lax square construction}\label{section.thedefinition}

The walking  $i$-by-$j$-morphisms $\Theta^{(i);(j)}$ have an immediate generalization to walking $\vec k$-by-$\vec l$-tuples $\Theta^{\vec k;\vec l}$ thereof, which we will describe in \cref{defn.veck-by-vecl}.  These computads will package into a $2n$-fold cosimplicial category $\Theta^{\vec \bullet;\vec \bullet}$.  Given a complete $n$-fold Segal space $\Cc$, we will define the $2n$-fold simplicial space $\Cc^{\Box}_{\vec \bullet;\vec \bullet} = \maps^{h}(\Theta^{\vec \bullet;\vec \bullet},\Cc)$.  Thus defined, we will prove in \cref{mainthm} that $\Cc^{\Box}$ satisfies the axioms of a complete $n$-fold Segal objects internal to complete $n$-fold Segal spaces.  

We first generalize the walking $i$-morphisms to walking tuples:

\begin{definition}\label{defn.veck}
Let $\vec k  \in \bN^{i}$ be an $i$-tuple of strictly positive integers.  
    The \define{walking $\vec k$-tuple of composable $i$-morphisms} $\Theta^{\vec k}$ is the $i$-computad defined  as follows.  
\begin{itemize}
\item First, when $\vec k = (i) $ consists entirely of $1$s, we set $\Theta^{\vec k} = \Theta^{(i)}$ as defined in \cref{defn.walkingimorphism}. 
\item Suppose now that only the last entry $k_i$ in $\vec k$ is not $1$. Recall that there are two inclusions $s^{(i)},t^{(i)} : \Theta^{(i-1)} \rightrightarrows \Theta^{(i)}$.  We set
    $$ \Theta^{\vec k}=\Theta^{(i-1),k_{i}} = \ourunderbrace{\Theta^{(i)} \underset{\Theta^{(i-1)}}{ \cup} \Theta^{(i)} \underset{\Theta^{(i-1)}}{ \cup} \dots \underset{\Theta^{(i-1)}}{ \cup} \Theta^{(i)}}{k_{i}\text{ times}} $$
    where all leftward inclusions are along $t^{(i)}$ and all rightward inclusions are along $s^{(i)}$, and the pushouts are of computads, or equivalently of strict higher categories.
  
\item  Let $\vec k  \in \bN_{>0}^{i}$. For each $i' \in \{1,\dots,i\}$, there are two inclusions $s^{(i')},t^{(i')} : \Theta^{(i'-1)} \rightrightarrows \Theta^{(i)} = \Theta^{(i'),1,\dots,1}$.     
  By induction, assume that we have defined $\Theta^{(i'),k_{i'+1},\dots,k_{i}}$ and two inclusions, which by abuse of notation we also denote by $s^{(i')},t^{(i')} : \Theta^{(i'-1)} \rightrightarrows \Theta^{(i'),k_{i'+1},\dots,k_{i}}$.  Then we set
    $$ \Theta^{(i'-1),k_{i'},k_{i'+1},\dots,k_{i}} = \ourunderbrace{\Theta^{(i'),k_{i'+1},\dots,k_{i}} \underset{\Theta^{(i'-1)}}{ \cup} \Theta^{(i'),k_{i'+1},\dots,k_{i}} \underset{\Theta^{(i'-1)}}{ \cup} \dots \underset{\Theta^{(i'-1)}}{ \cup} \Theta^{(i'),k_{i'+1},\dots,k_{i}}}{k_{i'}\text{ times}} $$
    where all leftward inclusions are along $t^{(i')}$ and all rightward inclusions are along $s^{(i')}$.
    
\item    For $\vec k \in \bN^{i}$ with some vanishing entries, we set $\Theta^{k_{1},\dots,k_{i'-1},0,k_{i'+1},\dots} = \Theta^{k_{1},\dots,k_{i'-1}}$. 
\end{itemize}
Note that for arbitrary $\vec k \in \bN^i$, the $i$-cells in $\Theta^{\vec k}$ naturally form a $k_1\times \dots \times k_i$ grid.
\end{definition}

\begin{lemma}\label{walking tuples are strict}
  The  pushouts in \cref{defn.veck} are homotopy pushouts.
\end{lemma}
\begin{proof}
  We have $\Theta^{(i'-1),k_{i'}+1,k_{i'+1},\dots,k_{i}} \simeq \Theta^{(i'-1),k_{i'},k_{i'+1},\dots,k_{i}} \cup_{\Theta^{(i'-1)}} \Theta^{(i'),k_{i'+1},\dots,k_{i}}$.  But by the induction step above, the map $s^{(i')} :\Theta^{(i'-1)} \to \Theta^{(i'),k_{i'+1},\dots,k_{i}}$ is an inclusion, and so this last pushout is equivalent to the corresponding homotopy pushout.  By induction, $\Theta^{(i'-1),k_{i'}+1,k_{i'+1},\dots,k_{i}} \simeq \Theta^{(i'),k_{i'+1},\dots,k_{i}} \cup^h_{\Theta^{(i'-1)}} \dots \cup^h_{\Theta^{(i')}} \Theta^{(i'),k_{i'+1},\dots,k_{i}}$.
\end{proof}

\begin{remark}
  The nerve construction does not take pushouts of strict higher categories to pushouts of simplicial sets.  Rather, the nerve of the pushout is built from the pushout of the nerve by attaching higher-dimensional uple-simplices as necessary to ensure the Segal condition.  In particular, the difference between the nerve of the pushout and the pushout of the nerve is invisible to the functor $\maps^h(-,\cC)$ for any complete $n$-fold Segal space $\cC$.  In the language of model categories, the nerve of the pushout and the pushout of the nerve are equivalent, but only the former is fibrant.
\end{remark}

 Some early examples of walking tuples of $2$-morphisms are:
$$
\Theta^{2,1} = \left\{
\begin{tikzpicture}[baseline=(base)]
  \path node[dot] (a) {}  +(0,-3pt) coordinate (base) ++(2,0) node[dot] (b) {} ++(2,0) node[dot] (c) {};
  \draw[arrow] (a) .. controls +(1,-.5) and +(-1,-.5) .. coordinate (t1)   (b);
  \draw[arrow] (a) .. controls +(1,.5) and +(-1,.5) .. coordinate (s1)  (b);
  \draw[twoarrowlonger] (s1) --   (t1);
  \draw[arrow] (b) .. controls +(1,-.5) and +(-1,-.5) .. coordinate (t2)   (c);
  \draw[arrow] (b) .. controls +(1,.5) and +(-1,.5) .. coordinate (s2)  (c);
  \draw[twoarrowlonger] (s2) --   (t2);
\end{tikzpicture} 
\right\}, \quad \Theta^{1,2} = 
\left\{
\begin{tikzpicture}[baseline=(base)]
  \path node[dot] (a) {} +(2,0) node[dot] (c) {} +(0,-3pt) coordinate (base);
  \draw[arrow] (a) .. controls +(1,1) and +(-1,1) .. coordinate (s)  (c);
  \draw[arrow] (a) -- coordinate (t)  (c);
  \draw[arrow] (a) .. controls +(1,-1) and +(-1,-1) .. coordinate (u)   (c);
  \draw[twoarrowlonger] (s) --   (t); \draw[twoarrowlonger] (t) --   (u);
\end{tikzpicture} \right\}
$$

\begin{remark}\label{underlying n-fold segal space}
Set $i=n$ and let $\Cc$ be a complete $n$-fold Segal space.  The walking tuples of composable $n$-morphisms are designed to know everything about $\Cc$: there are canonical homotopy equivalences
$$ \Cc_{\vec k} \cong \maps^{h}(\Theta^{\vec k},\Cc) $$
and the face and degeneracy maps are pullbacks along the canonical functors between strict $n$-categories $\Theta^{\vec k}$.

Furthermore, if $\Cc$ is a (complete) $n$-uple Segal space, then
$$\vec k \mapsto \maps^{h}(\Theta^{\vec k},\Cc)$$
is a (complete) $n$-fold Segal space, the \define{underlying $n$-fold Segal space of $\cC$}.
\end{remark}

\begin{remark}
  It is worth emphasizing that (the nerves of) the computads $\Theta^{\vec k}$ in general are not projectively cofibrant as $n$-fold simplicial spaces, and so to define the derived mapping space $\maps^{h}(\Theta^{\vec k},\Cc)$ one must either apply a projective-cofibrant resolution functor to $\Theta$ or an injective-fibrant resolution functor to $\Cc$.  See  \cref{lemma.mappingspace}, \cref{remark.theta2}, and \cref{eg.theta11}.
\end{remark}

\begin{remark}
  The computads $\Theta^{\vec k}$ are among the objects of Joyal's category $\Theta_n$ \cite{JoyalTheta,MR2578310}, which contains more generally the computads that can be built by gluing walking morphisms $\Theta^{(i)}$ together end-to-end.  For example, the following $3$-computad is in Joyal's $\Theta_3$, but is not one of our~$\Theta^{\vec k}$s:
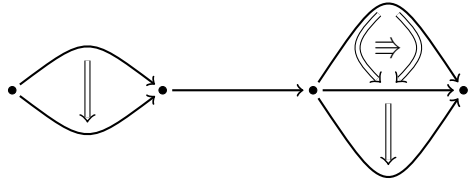
  $$
\begin{tikzpicture}[baseline=(base)]
  \path node[dot] (a) {}  +(0,-3pt) coordinate (base) ++(2,0) node[dot] (b) {} ++(2,0) node[dot] (c) {} ++(2,0) node[dot] (d) {};
  \draw[arrow] (a) .. controls +(1,-.75) and +(-1,-.75) .. coordinate (t1)   (b);
  \draw[arrow] (a) .. controls +(1,.75) and +(-1,.75) .. coordinate (s1)  (b);
  \draw[twoarrowlonger] (s1) --   (t1);
  \draw[arrow] (b) --  (c);
  \draw[arrow] (c) .. controls +(1,1.5) and +(-1,1.5) .. coordinate (s)  (d);
  \draw[arrow] (c) -- coordinate (t)  (d);
  \draw[arrow] (c) .. controls +(1,-1.5) and +(-1,-1.5) .. coordinate (u)   (d);
    \draw[twoarrowlonger] (s) .. controls +(-.5,-.5) and +(-.5,.5) .. node[auto] (x) {} (t);
    \draw[twoarrowlonger] (s) .. controls +(.5,-.5) and +(.5,.5) .. node[auto,swap] (y) {} (t);
    \path (x) -- node {$\Rrightarrow$} (y);
  \draw[twoarrowlonger] (t) --   (u);
\end{tikzpicture} 
$$
Our story applies just as well to these more complicated computads, but the notation would be more involved.  We would use them if we were modeling $(\infty,n)$-categories by a version of Rezk's complete Segal $\Theta_n$-spaces  rather than a version of Barwick's complete $n$-fold Segal spaces.
\end{remark}

A generalization of \cref{defn.veck} defines tuples of walking $i$-by-$j$ morphisms by similarly gluing the walking $i$-by-$j$ morphisms together in a grid.

Recall that for $i,j > 0$ and $i'\leq i$, $j'\leq j$, there are inclusions $s_{h}^{(i')},t_{h}^{(i')} : \Theta^{(i'-1);(j)} \rightrightarrows \Theta^{(i);(j)}$ and $s_{v}^{(j')},t_{v}^{(j')} : \Theta^{(i);(j'-1)} \rightrightarrows \Theta^{(i);(j)}$.
\begin{definition} \label{defn.veck-by-vecl} Let $\vec k, \vec l \in \bN^{i}$.  The \define{walking $\vec k$-by-$\vec l$-tuple of composable oplax $i$-by-$j$-morphisms} is the computad $\Theta^{\vec k;\vec l}$ defined as follows:
\begin{itemize}
\item If $\vec k=0$, we define $\Theta^{\vec k;\vec l}=\Theta^{\vec l}$. If $\vec l=0$, we define $\Theta^{\vec k;\vec l}=\Theta^{\vec k}$.
\item If $\vec k =(i)$ and $\vec l =(j)$, set $\Theta^{\vec k;\vec l}=\Theta^{(i);(j)}$.
\item  For $\vec k, \vec l \in \bN^{i}_{>0}$, define by induction over $i'$, respectively $j'$, from top to bottom with base case $i'=i$, respectively $j=j'$,
  \begin{align*}
     \Theta^{(i'-1),k_{i'},\dots,k_{i}; (j)} & = \ourunderbrace{\Theta^{(i'),k_{i'+1},\dots,k_{i}; (j)} \underset{\Theta^{(i'-1); (j)}}{\cup} \dots \underset{\Theta^{(i'-1); (j)}}{\cup} \Theta^{(i'),k_{i'+1},\dots,k_{i}; (j)}}{k_{i'}\text{ times}} ,\\
     \Theta^{(i);(j'-1),l_{j'},\dots,l_{j}} & = \ourunderbrace{\Theta^{(i);(j'),l_{j'+1},\dots,l_{j}} \underset{\Theta^{(i); (j'-1)}}{\cup} \dots \underset{\Theta^{(i); (j'-1)}}{\cup} \Theta^{(i);(j'),l_{j'+1},\dots,l_{j}}}{l_{j'}\text{ times}}.
  \end{align*}
\item By combining the two gluing steps above, again by induction over $i'$ and $j'$, define
  \begin{multline*}
    \Theta^{(i'-1),k_{i'},\dots;(j'-1),l_{j'},\dots}  
    \\ 
    \begin{aligned}
     & = \mathrm{colim} \left(
      \begin{tikzpicture}[baseline=(center)]
        \path 
          (0,-1.5) coordinate (center)
          (0,0) node (nw) {$\Theta^{(i'),k_{i'+1},\dots;(j'),l_{j'+1},\dots}$}
          (0,-1) node (w) {$\Theta^{(i'),k_{i'+1},\dots;(j'-1)}$}
          (0,-2) node (wdots) {$\cdots$}
          (0,-3) node (sw) {$\Theta^{(i'),k_{i'+1},\dots;(j'),l_{j'+1},\dots}$}
          (4,0) node (n) {$\Theta^{(i'-1);(j'),l_{j'+1},\dots}$}
          (4,-1) node (m) {$\Theta^{(i'-1);(j'-1)}$}
          (4,-2) node (wmdots) {$\cdots$}
          (4,-3) node (s) {$\Theta^{(i'-1);(j'),l_{j'+1},\dots}$}
          (6.5,0) node (ndots) {$\cdots$}
          (6.5,-1) node (nmdots) {$\cdots$}
          (6.5,-3) node (sdots) {$\cdots$}
          (9.25,0) node (ne) {$\Theta^{(i'),k_{i'+1},\dots;(j'),l_{j'+1},\dots}$}
          (9.25,-1) node (e) {$\Theta^{(i'),k_{i'+1},\dots;(j'-1)}$}
          (9.25,-2) node (edots) {$\cdots$}
          (9.25,-3) node (se) {$\Theta^{(i'),k_{i'+1},\dots;(j'),l_{j'+1},\dots}$}
        ;
        \draw[left hook->] (n) -- (nw);   \draw[right hook->] (n) -- (ndots);   \draw[right hook->] (ndots) -- (ne);
        \draw[left hook->] (m) -- (w);   \draw[right hook->] (m) -- (nmdots);   \draw[right hook->] (nmdots) -- (e);
        \draw[left hook->] (s) -- (sw);   \draw[right hook->] (s) -- (sdots);   \draw[right hook->] (sdots) -- (se);
        \draw[right hook->] (w) -- (nw); \draw[left hook->] (w) -- (wdots); \draw[left hook->] (wdots) -- (sw);
        \draw[right hook->] (m) -- (n); \draw[left hook->] (m) -- (wmdots); \draw[left hook->] (wmdots) -- (s);
        \draw[right hook->] (e) -- (ne); \draw[left hook->] (e) -- (edots); \draw[left hook->] (edots) -- (se);
      \end{tikzpicture}
    \right)
    \\ & \simeq \ourunderbrace{\Theta^{(i'),k_{i'+1},\dots; (j'-1),l_{j'},\dots} \underset{\Theta^{(i'-1); (j'-1),l_{j'},\dots}}{\cup} \dots \underset{\Theta^{(i'-1); (j'-1),l_{j'},\dots}}{\cup} \Theta^{(i'),k_{i'+1},\dots; (j'-1),l_{j'},\dots}}{k_{i'}\text{ times}}
    \\ & \simeq \ourunderbrace{\Theta^{(i'-1),k_{i'},\dots;(j'),l_{j'+1},\dots} \underset{\Theta^{(i'-1),k_{i'},\dots; (j'-1)}}{\cup} \dots \underset{\Theta^{(i'-1),k_{i'},\dots; (j'-1)}}{\cup} \Theta^{(i'-1),k_{i'},\dots;(j'),l_{j'+1},\dots}}{l_{j'}\text{ times}}
    \end{aligned}
  \end{multline*}
\item  Finally, if any $k_{i'}$ or $l_{j'}$ vanishes, we truncate: $\Theta^{k_1,\dots,k_{i'-1},0,k_{i'+1},\dots; \vec l} = \Theta^{k_1,\dots,k_{i'-1}; \vec l}$ and $\Theta^{\vec k;l_1,\dots,l_{j'-1},0,l_{j'+1},\dots} = \Theta^{\vec k;l_1,\dots,l_{j'-1}}$.
\end{itemize}
 \end{definition}

Just as in \cref{walking tuples are strict}, enough of the maps involved in the colimits are inclusions, and so an induction argument shows:
\begin{lemma} \label{walking tuples are strict 2}
  The pushouts presenting $\Theta^{\vec k;\vec l}$ are homotopy pushouts. \qedhere
\end{lemma}

\begin{remark}
  At the level of underlying CW complexes, we have $\Gr(\Theta^{\vec k;\vec l}) \cong \Gr(\Theta^{\vec k}) \times \Gr(\Theta^{\vec l})$.
\end{remark}

The computads $\Theta^{\vec \bullet;\vec \bullet}$ are together ``$2n$-fold cosimplicial,'' i.e.\ cosimplicial in each index, by mapping the cells of one computad to compositions of cells of another computad.

\begin{definition}\label{defn.Cbox}
  Let $\Cc$ be a complete $n$-uple Segal space.  Its \define{(op)lax square} $\Cc^{\Box}$ is the $2n$-fold simplicial space 
  $$ \Cc^{\Box}_{\vec \bullet;\vec \bullet} = \maps^{h}(\Theta^{\vec \bullet;\vec \bullet},\Cc). $$
\end{definition}

\begin{theorem}\label{mainthm}
  For any complete $n$-fold Segal space $\Cc$, the $2n$-fold simplicial space $\Cc^{\Box}$ from \cref{defn.Cbox} is a complete $n$-fold Segal object internal to complete $n$-fold Segal spaces.
\end{theorem}
\begin{proof}
  There are three conditions to check: essential constancy, the Segal condition, and completeness.  Only the third condition is not an immediate consequence of the construction of $\Theta^{\vec \bullet;\vec \bullet}$.
  
  \textbf{Essential constancy.}  We must verify that the degeneracy maps 
$$\Cc^{\Box}_{k_{1},\dots,k_{i-1},0,0,\dots;\vec l} \to \Cc^{\Box}_{k_{1},\dots,k_{i-1},0,k_{i+1},\dots;\vec l} \quad \mbox{and} \quad \Cc^{\Box}_{\vec k; l_{1},\dots,l_{j-1},0,0,\dots} \to \Cc^{\Box}_{\vec k; l_{1},\dots,l_{j-1},0,l_{j+1},\dots}$$
are homotopy equivalences.  This follows from the equalities $\Theta^{k_{1},\dots,k_{i-1},0,k_{i+1},\dots;\vec l} = \Theta^{k_{1},\dots,k_{i-1},0,0,\dots;\vec l}$ and $\Theta^{\vec k; l_{1},\dots,l_{j-1},0,l_{j+1},\dots} = \Theta^{\vec k; l_{1},\dots,l_{j-1},0,0,\dots}$.
  
  \textbf{Segal condition.} 
  Since colimits commute, we have isomorphisms
  \begin{align*}
    \Theta^{k_1,\dots,k_{i-1},k_i,k_{i+1},\dots;\vec l} &  \cong \Theta^{k_1,\dots,k_{i-1},1,k_{i+1},\dots;\vec l} \underset{\Theta^{k_1,\dots,k_{i-1},0;\vec l}}\cup  \dots \underset{\Theta^{k_1,\dots,k_{i-1},0;\vec l}}\cup  \Theta^{k_1,\dots,k_{i-1},1,k_{i+1},\dots;\vec l},\\
    \Theta^{\vec k;l_1,\dots,l_{j-1},l_j,l_{j+1},\dots} & \cong \Theta^{\vec k;l_1,\dots,l_{j-1},1,l_{j+1},\dots} \underset{\Theta^{\vec k;l_1,\dots,l_{j-1},0} }\cup \dots  \underset{\Theta^{\vec k;l_1,\dots,l_{j-1},0} }\cup \Theta^{\vec k;l_1,\dots,l_{j-1},1,l_{j+1},\dots} .
  \end{align*}
  Just as in \cref{walking tuples are strict,walking tuples are strict 2}, these pushouts are also homotopy pushouts.  Indeed, we have
  \begin{align*}
    \Theta^{k_1,\dots,k_{i-1},k_i,k_{i+1},\dots;\vec l} &\cong \Theta^{k_1,\dots,k_{i-1},k_i-1,k_{i+1},\dots;\vec l}\underset{\Theta^{k_1,\dots,k_{i-1},0;\vec l}}\cup  \Theta^{k_1,\dots,k_{i-1},1,k_{i+1},\dots;\vec l}, \\
    \Theta^{\vec k;l_1,\dots,l_{j-1},l_j,l_{j+1},\dots} &\cong \Theta^{\vec k;l_1,\dots,l_{j-1},l_j-1,l_{j+1},\dots} \underset{\Theta^{\vec k;l_1,\dots,l_{j-1},0} }\cup \Theta^{\vec k;l_1,\dots,l_{j-1},1,l_{j+1},\dots},
  \end{align*}
  but the maps $\Theta^{k_1,\dots,k_{i-1},0;\vec l} \to \Theta^{k_1,\dots,k_{i-1},1,k_{i+1},\dots;\vec l}$ and $\Theta^{\vec k;l_1,\dots,l_{j-1},0} \to \Theta^{\vec k;l_1,\dots,l_{j-1},1,l_{j+1},\dots}$
  are inclusions, and so
  \begin{align*}
    \Theta^{k_1,\dots,k_{i-1},k_i,k_{i+1},\dots;\vec l} &\simeq \Theta^{k_1,\dots,k_{i-1},k_i-1,k_{i+1},\dots;\vec l}\underset{\Theta^{k_1,\dots,k_{i-1},0;\vec l}}{\overset h\cup}  \Theta^{k_1,\dots,k_{i-1},1,k_{i+1},\dots;\vec l}, \\
    \Theta^{\vec k;l_1,\dots,l_{j-1},l_j,l_{j+1},\dots} &\simeq \Theta^{\vec k;l_1,\dots,l_{j-1},l_j-1,l_{j+1},\dots} \underset{\Theta^{\vec k;l_1,\dots,l_{j-1},0} }{\overset h \cup} \Theta^{\vec k;l_1,\dots,l_{j-1},1,l_{j+1},\dots}.
  \end{align*}
  The Segal condition follows, since $\maps^{h}(-,\Cc)$ turns homotopy colimits into homotopy limits.
   
  \textbf{Completeness.}  
  By \cref{lemma.weaker completeness condition}, essential constancy, and the Segal condition, it suffices to show that $\cC^\Box_{(i);(j),\bullet,0,\dots}$ and $\cC^\Box_{(i),\bullet,0,\dots;(j)}$ are complete.  We will check the former; the latter is completely analogous.
  We wish to show that the map $\cC^\Box_{(i);(j)} \to \cC^\Box_{(i);(j),\inv}$ is a homotopy equivalence. 
  As in \cref{defn.1fold complete}, for a Segal space $\cD_\bullet$, we  write $\cD_\inv$ for what is usually called $\cD_1^\inv$.  For example,  $\cC^\Box_{(i);(j),\inv}$ denotes the subspace of $\cC^\Box_{(i);(j),1} = \cC^\Box_{(i);(j+1)}$ of elements that represent  invertible morphisms in $\h_1(\cC^\Box_{(i);(j),\bullet})$. As in the proof of \cref{lemma.alternate description of completeness}, it is a union of connected components of $\cC^\Box_{(i);(j+1)}$.
 
We proceed by induction on $i$. For $i=0$,  by definition $\cC^\Box_{(i);(j),\bullet,0,\dots} = \cC^\Box_{(j),\bullet,0,\dots} = \cC_{(j),\bullet,0,\dots}$, which is complete since $\cC$ is complete. For $i>0$,  as usual let $s_h,t_h : \cC^\Box_{(i);(j+1)} \rightrightarrows \cC^\Box_{(i-1);(j+1)}$ denote the source and target maps in the $i$th direction.  They restrict to maps $\cC^\Box_{(i);(j),\inv} \rightrightarrows \cC^\Box_{(i-1);(j),\inv}$.  By induction, $\cC^\Box_{(i-1);(j),\inv} \simeq \cC^\Box_{(i-1);(j)}$ consists just of identities.  

Note that, since source and target of a map are themselves parallel, we can consider the pair $(s_{h},t_{h})$ as mapping
  $$ \cC^\Box_{(i);(j+1)} \to \cC^{\Box}_{\partial_{h}(i);(j+1)} := \maps^{h}\bigl( \partial_{h}\Theta^{(i);(j+1)},\cC\bigr),$$
  where $\partial_{h}\Theta^{(i);(j+1)}$ is the ``horizontal skeleton'' of $\Theta^{(i);(j+1)}$ from \cref{remark.horizontal skeleton}.  By working more generally with the Segal space $\cC^{\Box}_{\partial_{h}(i);(j),\bullet} = \maps^{h}\bigl( \partial_{h}\Theta^{(i);(j),\bullet,0,\dots},\cC\bigr)$, one sees that $\cC^{\Box}_{\partial_{h}(i);(j+1)}$ is the space of one-morphisms of this Segal space, and the image of $\cC^{\Box}_{(i);(j),\inv} \subset \cC^\Box_{(i);(j+1)}$ therein lands in the invertible elements $\cC^{\Box}_{\partial_{h}(i);(j),\inv}$.  

Since $\partial_h\Theta^{(i);(j+1)}$ consists of two copies of $\Theta^{(i-1);(j+1)}$ glued along two copies of $\Theta^{(i-2);(j+1)}$, the induction hypothesis implies that $\cC^{\Box}_{\partial_{h}(i);(j),\inv} \simeq \cC^{\Box}_{\partial_{h}(i);(j),0}$ consists just of identities.
  
  We find therefore
  \begin{align*}
    \cC^\Box_{(i);(j),\inv} & \simeq \cC^\Box_{(i);(j),\inv} \underset{\cC^{\Box}_{\partial_{h}(i);(j),\inv}}{\overset h \times} \cC^{\Box}_{\partial_{h}(i);(j),\inv} \\
     &  \simeq \cC^\Box_{(i);(j),\inv} \underset{\cC^{\Box}_{\partial_{h}(i);(j),\inv}}{\overset h \times} \cC^{\Box}_{\partial_{h}(i);(j)}.
\intertext{The right-hand side is the subspace of}
  & \subseteq \maps^{h}\left( \Theta^{(i);(j+1)} \underset{\partial_{h}\Theta^{(i);(j+1)}}{\overset h \cup} \partial_{h}\Theta^{(i);(j)}, \cC\right)
  \end{align*}
  consisting of those things that are invertible for composition in the last direction.
  
  The map $\partial_{h}\Theta^{(i);(j+1)} \to \Theta^{(i);(j+1)}$ is an inclusion, hence an injective-fibration, and so the homotopy pushout $\Theta^{(i);(j+1)} \cup_{\partial_{h}\Theta^{(i);(j+1)}}^{h} \partial_{h}\Theta^{(i);(j)}$ is a cofibrant replacement of the strict pushout $\Theta^{(i);(j+1)} \cup_{\partial_{h}\Theta^{(i);(j+1)}} \partial_{h}\Theta^{(i);(j)}$.
  But this strict pushout consists of two copies of $\Theta^{(i);(j)}$ glued along $\partial \Theta^{(i);(j)}$, and filled in by an $(i+j+1)$-morphism.  It is that $(i+j+1)$-morphism which is invertible.  Completeness of $\cC$ in the $(i+j+1)$th direction implies that it can be contracted out, from which we conclude  $\cC^\Box_{(i);(j),\inv} \simeq \cC^\Box_{(i);(j)}$. 
 \end{proof}

\begin{remark}  \label{remark.cartoon of completeness}
  Here is the cartoon of our argument for completeness, when $i=j=1$. The first (i.e.\ $i$th) direction is drawn from left to right and the last direction is drawn vertically.  Suppose that
\;
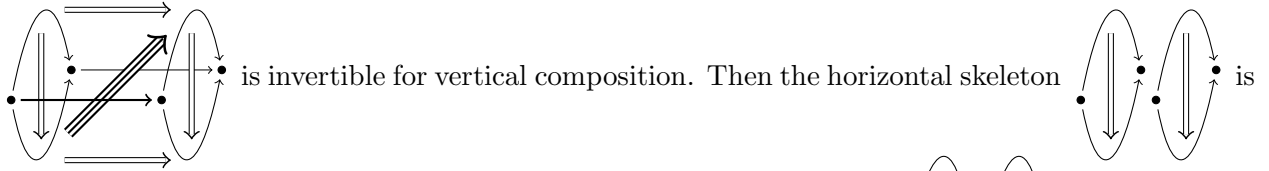
\begin{tikzpicture}[baseline=(center.base)]
  \path node (sl) {} +(2,0) node (sr) {} +(0,-2) node (tl) {} +(2,-2) node (tr) {};
  \draw[twoarrowlonger] (sl) --  (sr);
  \draw[twoarrowlonger] (sl) -- coordinate (l)  (tl);
  \draw[twoarrowlonger] (tl) --  (tr);
  \path (sr) -- coordinate (r)  (tr);
  \path (l) +(-.4,-.2) node[dot] (a) {};
  \path (l) +(.4,.2) node[dot] (b) {};
  \path (r) +(-.4,-.2) node[dot] (c) {};
  \path (r) +(.4,.2) node[dot] (d) {};
  \draw[arrow,thin] (a) .. controls +(.25,1.25) and +(-.25,1.25) .. (b);
  \draw[arrow,thin] (a) .. controls +(.25,-1.25) and +(-.25,-1.25) .. (b);
  \draw[arrow,thin] (c) .. controls +(.25,1.25) and +(-.25,1.25) .. (d);
  \draw[arrow,thin] (c) .. controls +(.25,-1.25) and +(-.25,-1.25) .. (d);
  \draw[arrow,very thin] (b) -- (d);
  \draw[threearrowpart1] (tl) -- node (center) {} (sr);
  \draw[threearrowpart2] (tl) --  (sr);
  \draw[threearrowpart3] (tl) --  (sr);
  \draw[arrow, thick] (a) -- (c);
  \draw[twoarrowlonger] (sr) --  (tr);
\end{tikzpicture}\;\vspace{-1cm}  is invertible for vertical composition.  Then the horizontal skeleton \;\begin{tikzpicture}[baseline=(center.base)]
  \path node (sl) {} +(1,0) node (sr) {} +(0,-2) node (tl) {} +(1,-2) node (tr) {};
  \draw[twoarrowlonger] (sl) -- coordinate (l)  (tl);
  \path (sr) -- coordinate (r)  (tr);
  \path (l) +(-.4,-.2) node[dot] (a) {};
  \path (l) +(.4,.2) node[dot] (b) {};
  \path (r) +(-.4,-.2) node[dot] (c) {};
  \path (r) +(.4,.2) node[dot] (d) {};
  \draw[arrow,thin] (a) .. controls +(.25,1.25) and +(-.25,1.25) .. (b);
  \draw[arrow,thin] (a) .. controls +(.25,-1.25) and +(-.25,-1.25) .. (b);
  \draw[arrow,thin] (c) .. controls +(.25,1.25) and +(-.25,1.25) .. (d);
  \draw[arrow,thin] (c) .. controls +(.25,-1.25) and +(-.25,-1.25) .. (d);
  \draw[twoarrowlonger] (sr) --  (tr);
\end{tikzpicture}\; is invertible for vertical composition, and hence an identity by induction:  \vspace{-1cm}$\;
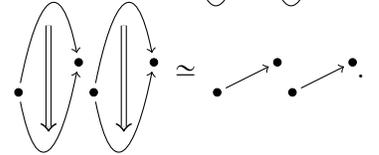
\begin{tikzpicture}[baseline=(center.base)]
  \path node (sl) {} +(1,0) node (sr) {} +(0,-2) node (tl) {} +(1,-2) node (tr) {};
  \draw[twoarrowlonger] (sl) -- coordinate (l)  (tl);
  \path (sr) -- coordinate (r)  (tr);
  \path (l) +(-.4,-.2) node[dot] (a) {};
  \path (l) +(.4,.2) node[dot] (b) {};
  \path (r) +(-.4,-.2) node[dot] (c) {};
  \path (r) +(.4,.2) node[dot] (d) {};
  \draw[arrow,thin] (a) .. controls +(.25,1.25) and +(-.25,1.25) .. (b);
  \draw[arrow,thin] (a) .. controls +(.25,-1.25) and +(-.25,-1.25) .. (b);
  \draw[arrow,thin] (c) .. controls +(.25,1.25) and +(-.25,1.25) .. (d);
  \draw[arrow,thin] (c) .. controls +(.25,-1.25) and +(-.25,-1.25) .. (d);
  \draw[twoarrowlonger] (sr) --  (tr);
\end{tikzpicture} \;\simeq\;
\begin{tikzpicture}[baseline=(center.base)]
  \path node (sl) {} +(1,0) node (sr) {} +(0,-2) node (tl) {} +(1,-2) node (tr) {};
  \path (sr) -- coordinate (r)  (tr);
  \path (l) +(-.4,-.2) node[dot] (a) {};
  \path (l) +(.4,.2) node[dot] (b) {};
  \path (r) +(-.4,-.2) node[dot] (c) {};
  \path (r) +(.4,.2) node[dot] (d) {};
  \draw[arrow,thin] (a) -- (b);
  \draw[arrow,thin] (c) -- (d);
\end{tikzpicture}$.  But then $\;
\begin{tikzpicture}[baseline=(center.base)]
  \path node (sl) {} +(2,0) node (sr) {} +(0,-2) node (tl) {} +(2,-2) node (tr) {};
  \draw[twoarrowlonger] (sl) --  (sr);
  \draw[twoarrowlonger] (sl) -- coordinate (l)  (tl);
  \draw[twoarrowlonger] (tl) --  (tr);
  \path (sr) -- coordinate (r)  (tr);
  \path (l) +(-.4,-.2) node[dot] (a) {};
  \path (l) +(.4,.2) node[dot] (b) {};
  \path (r) +(-.4,-.2) node[dot] (c) {};
  \path (r) +(.4,.2) node[dot] (d) {};
  \draw[arrow,thin] (a) .. controls +(.25,1.25) and +(-.25,1.25) .. (b);
  \draw[arrow,thin] (a) .. controls +(.25,-1.25) and +(-.25,-1.25) .. (b);
  \draw[arrow,thin] (c) .. controls +(.25,1.25) and +(-.25,1.25) .. (d);
  \draw[arrow,thin] (c) .. controls +(.25,-1.25) and +(-.25,-1.25) .. (d);
  \draw[arrow,very thin] (b) -- (d);
  \draw[threearrowpart1] (tl) -- node (center) {} (sr);
  \draw[threearrowpart2] (tl) --  (sr);
  \draw[threearrowpart3] (tl) --  (sr);
  \draw[arrow, thick] (a) -- (c);
  \draw[twoarrowlonger] (sr) --  (tr);
\end{tikzpicture}
\;\simeq\;
\begin{tikzpicture}[baseline=(center.base)]
  \path node (sl) {} +(2,0) node (sr) {} +(0,-2) node (tl) {} +(2,-2) node (tr) {};
  \path (sl)  -- coordinate (sr1) (sr);
  \path (sl) -- coordinate (l)  (tl);
 \path (tl) -- coordinate (tl1) (tr);
  \path (sr) -- coordinate (r)  (tr);
  \path (l) +(-.4,-.2) node[dot] (a) {};
  \path (l) +(.4,.2) node[dot] (b) {};
  \path (r) +(-.4,-.2) node[dot] (c) {};
  \path (r) +(.4,.2) node[dot] (d) {};
  \draw[arrow,thin] (a) -- (b);
  \draw[arrow,thin] (c) -- (d);
  \draw[arrow,very thin] (b) -- (d);
  \draw[twoarrowlonger] (b) .. controls +(.5,-1) and +(-.5,-1) .. (c);
  \draw[threearrowpart1] (tl1) -- node (center) {} (sr1);
  \draw[threearrowpart2] (tl1) --  (sr1);
  \draw[threearrowpart3] (tl1) --  (sr1);
  \draw[arrow, thick] (a) -- (c);
  \draw[twoarrowlonger] (b) .. controls +(.5,1) and +(-.5,1) .. (c);
\end{tikzpicture}\;$ is just an $(i+j+1)$-morphism  connecting some $(i+j)$-morphisms whose sources and targets have been factored in some way, and is invertible in the $(i+j)$th direction.  It is therefore an identity.
\end{remark}

\begin{example}\label{eg.theta11}
  Perhaps the most important example is the space $\cC^{\Box}_{(1);(1)} = \maps^{h}(\Theta^{(1);(1)},\cC)$.  
  The computad $\Theta^{(1);(1)}$ is a strict $2$-category, and so its nerve is a $2$-fold simplicial set, with the following nondegenerate simplices:
  $$
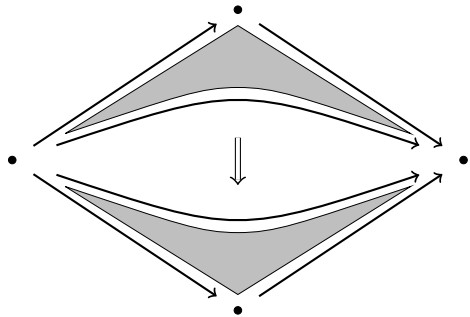
\begin{tikzpicture}
    \path (0,0) node[dot] (w) {} +(20pt,10pt) coordinate (w1)  +(20pt,-10pt) coordinate (w2) (3,2) node[dot] (n) {}  +(0pt,-6pt) coordinate (n1)  (6,0) node[dot] (e) {}  +(-20pt,10pt) coordinate (e1)  +(-20pt,-10pt) coordinate (e2)  (3,-2) node[dot] (s) {} +(0pt,6pt) coordinate (s2);
    \draw[arrow,shorten <=8pt,shorten >=8pt] (w) -- (n); \draw[arrow,shorten <=8pt,shorten >=8pt] (w) -- (s); \draw[arrow,shorten <=8pt,shorten >=8pt] (n) -- (e); \draw[arrow,shorten <=8pt,shorten >=8pt] (s) -- (e);
    \draw[arrow,shorten <=16pt,shorten >=16pt] (w) .. controls +(3,1) and +(-3,1) .. coordinate (upper) (e);
    \draw[arrow,shorten <=16pt,shorten >=16pt] (w) .. controls +(3,-1) and +(-3,-1) .. coordinate (lower) (e);
    \draw[twoarrowshorter] (upper) -- (lower);
    \path (upper) +(0,6pt) coordinate (upper1) (lower) +(0,-6pt) coordinate (lower2);
    \draw[ultra thin,fill=black!25] (w1) .. controls +(1.5,.5) and +(-.5,0) .. (upper1) .. controls +(.5,0) and +(-1.5,.5) .. (e1) -- (n1) -- (w1);
    \draw[ultra thin,fill=black!25] (w2) .. controls +(1.5,-.5) and +(-.5,0) .. (lower2) .. controls +(.5,0) and +(-1.5,-.5) .. (e2) -- (s2) -- (w2);
  \end{tikzpicture}$$
  In particular, there are four $(0,0)$-simplices, six nondegenerate $(1,0)$-simplices, two nondegenerate $(2,0)$-simplices, no nondegenerate $(0,1)$-simplices, and one non-degenerate $(1,1)$-simplex.
  
  This simplicial set is not projective-cofibrant as a $2$-fold simplicial space (compare \cref{remark.failure of cofibrancy}), and so the derived mapping space $\maps^{h}(\Theta^{(1);(1)},\cC)$ cannot be computed as the (simplicially-enriched) strict mapping space, unless $\cC$ happens to be injective-fibrant as an $n$-fold simplicial space.  But the above picture can be interpreted as a pushout of multi-simplices, each of which is projective-cofibrant, along injective-cofibrations; therefore a cofibrant replacement of $\Theta^{(1);(1)}$ can be constructed by interpreting the picture instead as a homotopy colimit, just as in \cref{remark.failure of cofibrancy}, with the bigon handled as in \cref{remark.theta2}.  One finds:
\begin{align*}
  \Cc^{\Box}_{(1);(1)} 
  & \;=\;
  \operatorname{holim}\left(
\begin{tikzpicture}[baseline=(middle.base),anchor=mid]
    \path (0,0) node (w) {$\Cc_{0}$} 
          (4,3) node (n) {$\Cc_{0}$}    
          (8,0) node (e) {$\Cc_{0}$}    
          (4,-3) node (s) {$\Cc_{0}$} ;
    \path (w) -- node (nw) {$\cC_{1}$} (n); \path (w) -- node (sw) {$\cC_{1}$} (s); \path (n) -- node (ne) {$\cC_{1}$}  (e); \path (s) -- node (se) {$\cC_{1}$} (e);
    \path (4,1.25) node (upper) {$\cC_{1}$} (e);
    \path (4,-1.25) node (lower) {$\cC_{1}$} (e);
    \path (lower) -- node (middle) {$\cC_{1,1}$} (upper);
    \path (2.5,0) node (wm) {$\cC_{0,1}$} (5.5,0) node (em) {$\cC_{0,1}$};
    \path (4,2.25) node (nn) {$\cC_{2}$} (4,-2.25) node (ss) {$\cC_{2}$};
    \draw[->] (nw) -- (w); \draw[->] (nw) -- (n); \draw[->] (ne) -- (e); \draw[->] (ne)--(n);
    \draw[->] (sw) -- (w); \draw[->] (sw) -- (s); \draw[->] (se) -- (e); \draw[->] (se)--(s);
    \draw[->] (upper) -- (w); \draw[->] (upper) -- (e); \draw[->] (lower) -- (w); \draw[->] (lower) -- (e);
    \draw[->] (nn) -- (upper); \draw[->] (nn) -- (nw); \draw[->] (nn) -- (ne);
    \draw[->] (ss) -- (lower); \draw[->] (ss) -- (sw); \draw[->] (ss) -- (se);
    \draw[->] (middle) -- (upper); \draw[->] (middle) -- (lower); \draw[->] (middle) -- (wm); \draw[->] (middle) -- (em);
    \draw[->] (w) -- node[auto,pos=.7] {$\sim$} (wm);
    \draw[->] (e) -- node[auto,pos=.7,swap] {$\sim$} (em);
  \end{tikzpicture}
  \right)
  \\ & \;\simeq\; \cC_{2} \underset{\cC_{1}}{\overset h \times} \cC_{1,1} \underset{\cC_{1}}{\overset h \times} \cC_{2}.
\end{align*}
Recall that our convention is to drop trailing $0$s; for example $\cC_{2} = \cC_{2,0,0,\dots}$.
The maps $ \cC_{1,\vec\bullet} \to \cC_{0,\vec\bullet}$ that point leftwards are source maps in the first direction; the rightward pointing ones are target maps.  The maps $\cC_{2} \to \cC_{1}$ are, from left to right, the first, middle, and last face maps.  The upward map $\cC_{1,1} \to \cC_{1}$ is the source in the second direction; the downward pointing map is the target.  Finally, the maps $\cC_{0} \to \cC_{0,1}$ are homotopy equivalences by essential constancy.  
\end{example}

Since we will use the following notation repeatedly from now on, it is worth formulating it as a definition.

\begin{definition}\label{defn.(op)lax arrow category}
Given an $(\infty,n)$-category $\cC$, i.e.~a complete $n$-fold Segal space, we will use the notations
$$\cC^{\Box}_{\vec \bullet;(1)} = \cC^{\downarrow}_{\vec \bullet} \quad  \mbox{and} \quad  \cC^{\Box}_{(1);\vec \bullet} = \cC^{\to}_{\vec \bullet}$$
for the $(\infty,n)$-categories of lax and oplax arrows, respectively.
As in the introduction, we will write $[\Theta^{(1)},\cC]$ for the $(\infty,n)$-category of ``strong arrows'' defined by the universal property
$$ \maps^h(\cB,[\Theta^{(1)},\cC]) \simeq \maps^h(\cB \times^h \Theta^{(1)},\cC). $$

More generally, for $\ast = $ ``lax'', ``oplax'', or ``strong'', and for $\vec k \in \bN^n$, we have an $(\infty,n)$-category $\cC^\boxasterisk_{\vec k}$ defined by
$$\bigl(\cC^\boxasterisk_{\vec k}\bigr)_{\vec\bullet}
 = \begin{cases}
\cC^\Box_{\vec \bullet;\vec k}, &\ast =\mbox{lax},\\
\cC^\Box_{\vec k;\vec \bullet}, &\ast =\mbox{oplax},\\
[\Theta^{\vec k},\cC]_{\vec\bullet}, &\ast =\mbox{strong}.
\end{cases}$$
For fixed $\vec k$, all three $(\infty,n)$-categories $\cC^{\smallbox{lax}}_{\vec k}$, $\cC^{\smallbox{oplax}}_{\vec k}$, and $\cC^{\smallbox{strong}}_{\vec k}$ have the same objects, but they have different morphisms.
\end{definition}

\begin{remark}
  \Cref{mainthm} asserts that for fixed $\vec k$, the $n$-fold simplicial spaces $\cC^{\smallbox{lax}}_{\vec k}$ and $\cC^{\smallbox{oplax}}_{\vec k}$ are $(\infty,n)$-categories, and moreover that $\vec k \mapsto \cC^{\smallbox{lax}}_{\vec k}$ and $\vec k \mapsto \cC^{\smallbox{oplax}}_{\vec k}$ are complete $n$-fold Segal objects among $(\infty,n)$-categories.  That $\vec k \mapsto \cC^{\smallbox{strong}}_{\vec k}$ is a complete $n$-fold Segal object among $(\infty,n)$-categories follows from a similar, somewhat easier argument.
\end{remark}

Now we have all ingredients to recall the definition of (op)lax natural transformations between strong functors given in the introduction in \cref{defn.oplax-transformation}.
\begin{definition}\label{defn.oplax-transformation2}
  Let $\cB$ and $\cC$ be $(\infty,n)$-categories and $F,G: \cB \rightrightarrows \cC$ strong functors.  
  A \define{lax transformation} $\eta : F \Rightarrow G$ is a strong functor $\eta : \cB \to \cC^{\downarrow}$ such that $s_{v}\circ\eta = F$ and $t_{v}\circ \eta = G$.
  An \define{oplax transformation} $\eta: F \Rightarrow G$ is a strong functor $\eta: \cB \to \cC^{\rightarrow}$ such that $s_{h}\circ\eta = F$ and $t_{h}\circ \eta = G$.
\end{definition}

\begin{example}\label{eg.recollection table}
An (op)lax transformation $\eta:F\Rightarrow G$ in particular assigns the following:
$$
\begin{array}{rc|c|c}
 && \text{lax} & \text{oplax}\\[3pt] \hline \hline &&& \\[-6pt]
\text{object} & B \in \cB_{(0)} & \eta(B) : F(B) \to G(B) & \eta(B) : F(B) \to G(B)\\[6pt] \hline &&& \\[-6pt]
1\text{-morphism} & \{B_{1}\overset b \to B_{2}\} \in \cB_{(1)} &
\begin{tikzpicture}[baseline=(r.base)]
  \path node[dot] (sl) {} +(2,0) node[dot] (sr) {} +(0,-2) node[dot] (tl) {} +(2,-2) node[dot] (tr) {};
  \draw[arrow] (sl) node[anchor = south east] {$F(B_{1})$} -- node[auto] {$F(b)$} (sr) node[anchor = south west] {$F(B_{2})$};
  \draw[arrow] (sl)  -- node[auto,swap] {$\eta(B_{1})$} (tl);
  \draw[arrow] (tl) node[anchor = north east] {$G(B_{1})$} -- node[auto,swap] {$G(b)$}  (tr) node[anchor = north west] {$G(B_{2})$};
  \draw[arrow] (sr) -- node[auto] (r) {$\eta(B_{2})$} (tr);
  \draw[twoarrow] (tl) -- node[auto] {$\eta(b)$} (sr);
\end{tikzpicture}
&
\begin{tikzpicture}[baseline=(r.base)]
  \path node[dot] (sl) {} +(2,0) node[dot] (sr) {} +(0,-2) node[dot] (tl) {} +(2,-2) node[dot] (tr) {};
  \draw[arrow] (sl) node[anchor = south east] {$F(B_{1})$} -- node[auto] {$\eta(B_{1})$} (sr) node[anchor = south west] {$G(B_{1})$};
  \draw[arrow] (sl)  -- node[auto,swap] {$F(b)$} (tl);
  \draw[arrow] (tl) node[anchor = north east] {$F(B_{2})$} -- node[auto,swap] {$\eta(B_{2})$}  (tr) node[anchor = north west] (GB2) {$G(B_{2})$};
  \draw[arrow] (sr) -- node[auto] (r) {$G(b)$} (tr);
  \draw[twoarrow] (tl) -- node[auto] {$\eta(b)$} (sr);
  \path (GB2.south) ++(0,-5pt) coordinate (bottom);
\end{tikzpicture}
\\[6pt] \hline &&& \\[-6pt]
2\text{-morphism} &
\begin{tikzpicture}[baseline=(base), scale=0.8]
  \path node[dot] (a) {} node[anchor=east] {$B_1$} +(2,0) node[dot] (c) {} node[anchor=west] {$B_2$} +(0,-3pt) coordinate (base);
  \draw[arrow] (a) .. controls +(1,-.5) and +(-1,-.5) .. coordinate (t)   (c);
  \draw[arrow] (a) .. controls +(1,.5) and +(-1,.5) .. coordinate (s)  (c);
\draw (s) node[anchor=south] {$\scriptstyle b_1$};
\draw (t) node[anchor=north] (b2) {$\scriptstyle b_2$};
  \draw[twoarrowlonger] (s) -- node[anchor=east] {$\scriptstyle \beta$}   (t);
\path (b2) +(0,-0.7) node {$\in \cB_{(2)}$};
\end{tikzpicture}
&
\begin{tikzpicture}[baseline=(middle), scale=0.8]
  \path node[dot] (a) {} node[anchor= south east] {$F(B_{1})$} +(3,0) node[dot] (b) {} node[anchor= south west] {$F(B_{2})$} +(0,-3) node[dot] (c) {} node[anchor= north east] {$G(B_{1})$} +(3,-3) node[dot] (d) {} node[anchor= north west] {$G(B_{2})$}+(0,-1.5) coordinate (middle);
  \path (b) +(-.35,.25) coordinate (alpha1);
  \path (c) +(.35,-.25) coordinate (alpha2);
  \draw[arrow] (a) -- coordinate (l1) node[anchor=east] {$\eta(B_1)$} coordinate[very near end] (r1) (c);
  \draw[arrow] (b) -- coordinate[very near start] (l2) coordinate (r2) node[anchor=west] (rightlable) {$\eta(B_2)$} (d);
  \draw[twoarrow,thin] (l1) -- node[auto,pos=.4,inner sep=1pt] {$\scriptstyle \eta(b_{1})$} (l2);
  \draw[arrow] (a) .. controls +(1.5,.75) and +(-1.5,.75) .. coordinate (ss)  (b);
  \draw[arrow,thin] (c) .. controls +(1.5,.75) and +(-1.5,.75) .. coordinate (ts)  (d);
  \draw[arrow] (c) .. controls +(1.5,-.75) and +(-1.5,-.75) .. coordinate (tt) (d);
  \draw[twoarrowlonger,thin] (ts) -- node[anchor=west] {$\scriptstyle G(\beta)$}  (tt);
  \draw[threearrowpart1] (alpha1) --  (alpha2);
  \draw[threearrowpart2] (alpha1) --  (alpha2);
  \draw[threearrowpart3] (alpha1) -- node[inner sep= 1pt, fill=white] {$\scriptstyle \eta(\beta)$}  (alpha2);
  \draw[arrow,thick] (a) .. controls +(1.5,-.75) and +(-1.5,-.75) .. coordinate (st) (b);
  \draw[twoarrowlonger,thick] (ss) -- node[anchor=east] {$\scriptstyle F(\beta)$}  (st);
  \draw[twoarrow,thick] (r1) -- node[auto,swap,pos=.6,inner sep=1pt] {$\scriptstyle \eta(b_{2})$} (r2);
\end{tikzpicture}
&
\begin{tikzpicture}[baseline=(middle), scale=0.8]
  \path node[dot] (a) {} +(3,0) node[dot] (b) {} +(0,-3) node[dot] (c) {} +(3,-3) node[dot] (d) {} +(0,-1.5) coordinate (middle);
  \draw[arrow] (a) .. controls +(-.75,-1.5) and +(-.75,1.5) .. coordinate[near end] (lt) coordinate[very near end] (lt2) (c);
  \draw[arrow,thin] (a) .. controls +(.75,-1.5) and +(.75,1.5) .. coordinate[near start] (ls) coordinate (ls2) (c);
  \draw[twoarrowlonger] (ls) -- node[inner sep=1pt,fill=white] {$\scriptstyle F(\beta)$} (lt);
  \draw[arrow] (b) .. controls +(.75,-1.5) and +(.75,1.5) .. coordinate[near start] (rs) coordinate[very near start] (rs0) (d);
  \draw[twoarrow] (ls2) -- coordinate[near end] (s) node[auto] {$\scriptstyle \eta(b_1)$} (rs0);
  \draw[arrow] (b) .. controls +(-.75,-1.5) and +(-.75,1.5) .. coordinate[near end] (rt) coordinate (rt0) (d);
  \draw[twoarrowlonger,thick] (rs) -- node[inner sep=1pt,fill=white] {$\scriptstyle G(\beta)$} (rt);
  \draw[arrow] (a) node[anchor = south east] {$F(B_1)$} -- node[auto] {$\eta(B_1)$} (b) node[anchor = south west] {$G(B_1)$};
  \draw[arrow] (c) node[anchor = north east] {$F(B_2)$} -- node[below] {$\eta(B_2)$} (d) node[anchor = north west] {$G(B_2)$};
  \path (lt2) -- coordinate[near start] (t) (rt0);
  \draw[threearrowpart1] (t) --  (s); 
  \draw[threearrowpart2] (t) -- (s); 
  \draw[threearrowpart3] (t) -- node[inner sep= 1pt, fill=white] {$\scriptstyle \eta(\beta)$} (s); 
  \draw[twoarrow,thick] (lt2) -- node[auto,swap] {$\scriptstyle \eta(b_2)$} (rt0);
\end{tikzpicture} 
\end{array}
$$
For comparison, in a \define{strong} transformation, i.e.\ a strong functor $\cB \to [\Theta^{(1)},\cC]$, the  ``bulk'' morphisms $\eta(b)$, $\eta(b_1)$, $\eta(b_2)$, and $\eta(\beta)$ are required to be invertible.
\end{example}

Using the full oplax square provides definitions of (op)lax higher transfors:
\begin{definition}\label{defn.oplax-transfor}
  Let $\cB$ and $\cC$ be $(\infty,n)$-categories and let $k>1$.
\begin{enumerate}
\item An \define{(op)lax 1-transfor} is an (op)lax natural transformation between strong functors.
\item A \define{lax $k$-transfor} is a strong functor $\eta : \cB \to \cC^{\smallbox{lax}}_{(k)} = \cC^{\Box}_{\vec\bullet;(k)}$. Its source and target are the lax $(k-1)$-transfors $s_{v}\circ\eta$ and $t_{v}\circ \eta$.
\item An \define{oplax $k$-transfor} is a strong functor $\eta : \cB \to \cC^{\smallbox{oplax}}_{(k)} = \cC^{\Box}_{(k);\vec\bullet}$. Its source and target are the oplax $(k-1)$-transfors $s_{h}\circ\eta$ and $t_{h}\circ \eta$.
\end{enumerate}
\end{definition}

\begin{corollary}\label{cor.oplax-transformation}
  Let $\cB$ and $\cC$ be complete $n$-fold Segal spaces.  The $n$-fold simplicial spaces
\begin{align*}
\vec k &\mapsto \mathrm{Fun}^{\mathrm{lax}}(\cB,\cC)_{\vec k} :=\maps^{h}(\cB,\cC^{\smallbox{lax}}_{\vec k}), \quad\mbox{and}\\
\vec k &\mapsto \mathrm{Fun}^{\mathrm{oplax}}(\cB,\cC)_{\vec k} :=\maps^{h}(\cB,\cC^{\smallbox{oplax}}_{\vec k})
\end{align*}
are complete $n$-fold Segal spaces, i.e.~$(\infty,n)$-categories, which depend naturally on the choices of $\cB$ and $\cC$.
\end{corollary}

\begin{proof}
  It suffices to observe that $\maps^h(\cB,-)$ preserves weak equivalences and homotopy fiber products.  Naturality in $\cB$ is manifest, and naturality in $\cC$ is clear once it is observed that $\cC^\Box_{\vec \bullet;\vec\bullet} = \maps^h(\Theta^{\vec\bullet;\vec\bullet},\cC)$ depends naturally in $\cC$.
\end{proof}

By definition unpacking, $\vec k \mapsto \mathrm{Fun}^{\mathrm{strong}}(\cB,\cC)_{\vec k} :=\maps^{h}(\cB,\cC^{\smallbox{strong}}_{\vec k})$ is nothing but the complete $n$-fold Segal space $\vec k \mapsto [\cB,\cC]_{\vec k}$ satisfying the universal property $\maps^h(\cX,[\cB,\cC]) \simeq \maps^h(\cX \times^h \cB,\cC)$.

\begin{remark}
Note that the space of objects of each of the $(\infty,n)$-categories $\mathrm{Fun}^{\ast}(\cB,\cC)$, for $\ast = $ ``lax'', ``oplax'', or ``strong'', is the space $\maps^{h}(\cB,\cC) \cong \maps^{h}(\cB,\cC^\boxasterisk_{(0)})$, i.e.\ the space of (strong) functors $\cB \to \cC$.
Their spaces of 1-morphisms are the spaces of lax, oplax, or strong natural transformations between functors.  The remainder of the $(\infty,n)$-categories $\vec k \mapsto \maps^{h}(\cB,\cC^{\boxasterisk}_{\vec k})$, for $\asterisk = $ ``lax'', ``oplax'', or ``strong'', implements composition of lax/oplax/strong natural transformations as well as the higher transfors discussed in \cref{defn.oplax-transfor}.
\end{remark}

\section{Symmetric monoidal structures}\label{section.symmetricmonoidal}

In this section we prove in \cref{prop.c box is symmetric monoidal} that if $\cC$ is symmetric monoidal than so are $\cC^{\smallbox{lax}}_{\vec\bullet}$ and $\cC^{\smallbox{oplax}}_{\vec\bullet}$.  This allows us to present in \cref{defn.symm mon oplax-transfor} the notion of ``symmetric monoidal (op)lax natural transformation between symmetric monoidal functors.''  
We then recall the ``based loop $(\infty,n-1)$-category'' $\Omega \cC = \mathrm{End}_\cC(\unit)$ of a symmetric monoidal $(\infty,n)$, where $\unit$ denotes the monoidal unit in $\cC$, and identify it in \cref{thm.qft} with the category of lax arrows from $\unit$ to $\unit$.  We will use this result in \cref{section.QFT} to identify lax trivially-twisted and untwisted quantum field theories.

We begin by recalling the definition of ``symmetric monoidal complete $n$-fold Segal space'':

\begin{definition}\label{defn.Gamma}
  The category $\Fin_{*}$ 
is the category whose objects are the pointed (at $0$) sets
$$[k]=\{0,\ldots, k\},$$
for every $k\geq0$ and morphisms are pointed functions, i.e.~functions sending $0\mapsto 0$.  In particular, for each $k$, there are $k$ canonical morphisms  
$$\gamma_i:[k]\longrightarrow [1], \quad 0<i\leq k,$$
 such that $\gamma_i(j)=\delta_{ij}$, called the \define{Segal morphisms}.
\end{definition}

\begin{remark}\label{remark.Gamma}
  In~\cite{MR0353298}, Segal defined the category $\Gamma$ to be the opposite of $\Fin_{*}$, although many modern authors use the name $\Gamma$ for $\Fin_*$ instead of $\Fin_*^{\mathrm{op}}$.  Segal set $\Gamma = \Fin_*^{\mathrm{op}}$ in analogy with the category $\Delta$ of simplices, which allows to interpret his ``$\Gamma$-spaces'' as simplicial spaces by composition with the obvious covariant functor $\Delta\to \Fin_*^{\mathrm{op}}$.
In this analogy, the Segal morphisms correspond to the $k$ standard inclusions $[1] \to [k]$ in $\Delta$.
\end{remark}

\begin{definition}\label{defn.Gammaobject}
  A \define{symmetric monoidal complete $n$-fold Segal space} is a strict functor $\cC^{\otimes}$ from $\Fin_{*}$ to the strict category of complete $n$-fold Segal spaces
  satisfying the following \define{Segal condition}:
   for each $[m]$, the map
 \begin{equation*} \prod_{i=1}^{m} \gamma_{i} : \cC^{\otimes}[m]_{\vec \bullet} \longrightarrow \prod_{i=1}^{m}\cC^{\otimes}[1]_{\vec \bullet} 
\end{equation*}
  is a weak equivalence of complete $n$-fold Segal spaces.  (In both the projective and injective model structures this boils down to levelwise weak equivalences of spaces --- note that the product on the right hand side is computed levelwise).
\end{definition}

\begin{remark}
We should define a symmetric monoidal complete $n$-fold Segal space to be an $(\infty,1)$-functor of $(\infty,1)$-categories from $\Fin_{*}$ to the $(\infty,1)$-category of complete $n$-fold Segal spaces. However, by the strictification theorem of To\"en-Vezzosi (Proposition 4.2.1 in \cite{ToenVezzosi}) which in the setting of complete Segal spaces can be found in \cite[Section 8.9]{Rezk}, any such functor can be represented by a strict functor.
\end{remark}

  \define{Homomorphisms} of symmetric monoidal complete $n$-fold Segal spaces are natural transformations of functors (in the strict sense).
Symmetric monoidal complete $n$-fold Segal spaces are the fibrant objects of a model structure  
on the category of functors
$\Fin_{*} \times (\Delta^{\op})^{\times n} \to \cat{sSet}$
defined in \cref{eg.model category of symmetric monoidal $n$-fold complete Segal spaces}.
 Its \define{equivalences} are  homomorphisms that are levelwise equivalences.  A version of \cref{lemma.mappingspace} holds for symmetric monoidal complete $n$-fold Segal spaces considered as functors $\Fin_{*} \times (\Delta^{\op})^{\times n} \to \cat{sSet}$.  In particular, one can define derived mapping spaces $\maps^{h}_\otimes(\cB^{\otimes},\cC^{\otimes})$ between symmetric monoidal complete $n$-fold Segal spaces $\cB^\otimes$ and $\cC^{\otimes}$ by either cofibrantly resolving $\cB^{\otimes}$ in the projective model structure on functors $\Fin_{*} \times (\Delta^{\op})^{\times n} \to \cat{sSet}$ or by fibrantly resolving $\cC^\otimes$ in the injective model structure.

Given a complete $n$-fold Segal space $\cC$, we will say that $\cC$ \define{has a symmetric monoidal structure} if a symmetric monoidal complete $n$-fold Segal space $\cC^{\otimes}$ is given with $\cC^{\otimes}[1] = \cC$.  We will henceforth abuse notation slightly, writing ``$\cC$'' both for a symmetric monoidal complete $n$-fold Segal space and for its underlying complete $n$-fold Segal space $\cC^\otimes[1]$; we will also shorten notation to $\cC[m]:=\cC^\otimes[m]$.  Derived mapping spaces of underlying $n$-fold complete Segal spaces will be denoted $\maps^h(-,-)$; for the derived mapping spaces of symmetric monoidal $n$-fold complete Segal spaces, we will write $\maps^h_\otimes$.

We can now make sense of the term ``symmetric monoidal (op)lax natural transformation.''

\begin{proposition}\label{prop.c box is symmetric monoidal}
  The (op)lax square of a symmetric monoidal complete $n$-fold Segal space $\cC^\otimes$ is symmetric monoidal in each component, i.e.~for any $\vec k$, the assignments
$$\cC^{\smallbox{lax}}_{\vec k}[m]:=  (\cC[m])^{\smallbox{lax}}_{\vec k} \quad \text{and}\quad \cC^{\smallbox{oplax}}_{ \vec k}[m]:= (\cC[m])^{\smallbox{oplax}}_{ \vec k}$$
endow $\cC^{\smallbox{lax}}_{\vec k}$ and $\cC^{\smallbox{oplax}}_{\vec k}$ with symmetric monoidal structures.
\end{proposition}

\begin{proof}
  Define $\cC^{\Box}[m]_{\vec\bullet;\vec\bullet} := (\cC[m])^{\Box}_{\vec\bullet;\vec\bullet} = \maps^{h}(\Theta^{\vec\bullet;\vec\bullet},\cC[m])$; since derived mapping spaces are functorial, this depends functorially on $[m] \in \Fin_{*}$.  We have only to check that $\prod \gamma_{i} : \maps^{h}(\Theta^{\vec k;\vec l},\cC[m]) \to \maps^{h}(\Theta^{\vec k;\vec l},\cC[1])^{\times m}$ is a weak equivalence for each $\vec k,\vec l$, which indeed holds with $\Theta^{\vec k;\vec l}$ replaced by any complete $n$-fold Segal space.
\end{proof}

\begin{remark}
Note that we also have a strong variant: if $\cC$ is symmetric monoidal, then $\cC^{\smallbox{strong}}_{\vec k}[m]:=(\cC[m])^{\smallbox{strong}}_{\vec k}= [\Theta^{\vec{k}}, \cC[m]]$ defines a symmetric monoidal structure on $\cC^{\smallbox{strong}}_{\vec k}$.
\end{remark}

In particular $\cC^{\downarrow} = \cC^{\smallbox{lax}}_{(1)}$ and $\cC^\to = \cC^{\smallbox{oplax}}_{(1)}$ are symmetric monoidal complete $n$-fold Segal spaces if $\cC$ is, and the structure maps $s_{v},t_{h} : \cC^{\downarrow} \rightrightarrows \cC$ and $s_h,t_h : \cC^\to \rightrightarrows \cC$ are homomorphisms of symmetric monoidal complete $n$-fold Segal spaces.

\begin{definition}\label{defn.sym mon lax transfor}
Let $\cB$ and $ \cC$ be symmetric monoidal complete $n$-fold Segal spaces.  Given symmetric monoidal functors $F, G \in \maps_{\otimes}^{h}(\cB,\cC)$, the space of \define{symmetric monoidal lax transformations} from $F$ to $G$ is the fibered product
$$\mathrm{Lax}_\otimes(F,G)= \{F \} \underset{\maps_{\otimes}^{h}(\cB,\cC)}{\overset h \times} \maps_{\otimes}^{h}(\cB,\cC^{\downarrow}) \underset{\maps_{\otimes}^{h}(\cB,\cC)}{\overset h \times} \{G\}, $$
where the two maps $\maps_{\otimes}^{h}(\cB,\cC^{\downarrow}) \rightrightarrows \maps_{\otimes}^{h}(\cB,\cC)$ are given by composition with $s_{v},t_{v} : \cC^{\downarrow} \rightrightarrows \cC$.  The space of \define{symmetric monoidal oplax transformations} from $F$ to $G$ is
$$\mathrm{Oplax}_\otimes(F,G) = \{F \} \underset{\maps_{\otimes}^{h}(\cB,\cC)}{\overset h \times} \maps_{\otimes}^{h}(\cB,\cC^{\to}) \underset{\maps_{\otimes}^{h}(\cB,\cC)}{\overset h \times} \{G\}, $$
where the two maps $\maps_{\otimes}^{h}(\cB,\cC^{\to}) \rightrightarrows \maps_{\otimes}^{h}(\cB,\cC)$ are given by composition with $s_{h},t_{h} : \cC^{\to} \rightrightarrows \cC$.
\end{definition}

Using the full (op)lax square we obtain the symmetric monoidal versions of  higher transfors.
\begin{definition}\label{defn.symm mon oplax-transfor}
  Let $\cB$ and $\cC$ be symmetric monoidal $(\infty,n)$-categories and let $k>1$.
\begin{enumerate}
\item A \define{symmetric monoidal (op)lax 1-transfor} is a symmetric monoidal (op)lax natural transformation between symmetric monoidal strong functors.
\item A \define{symmetric monoidal lax $k$-transfor} is a symmetric monoidal strong functor $\eta : \cB \to \cC^{\smallbox{lax}}_{(k)}$. Its source and target are the symmetric monoidal lax $(k-1)$-transfors $s_{v}\circ\eta$ and $t_{v}\circ \eta$.
\item A \define{symmetric monoidal oplax $k$-transfor} is a symmetric monoidal strong functor $\eta : \cB \to \cC^{\smallbox{oplax}}_{(k)}$. Its source and target are the symmetric monoidal oplax $(k-1)$-transfors $s_{h}\circ\eta$ and $t_{h}\circ \eta$.
\end{enumerate}
\end{definition}

This allows for the symmetric monoidal version of \cref{cor.oplax-transformation}:
\begin{corollary}\label{cor.symm mon oplax-transformation}
  Let $\cB$ and $\cC$ be symmetric monoidal complete $n$-fold Segal spaces.  The $n$-fold simplicial spaces
\begin{align*}
\vec k &\mapsto \mathrm{Fun}_\otimes^{\mathrm{lax}}(\cB,\cC)_{\vec k} :=\maps^{h}_\otimes(\cB,\cC^{\smallbox{lax}}_{\vec k}), \quad\mbox{and}\\
\vec k &\mapsto \mathrm{Fun}_\otimes^{\mathrm{oplax}}(\cB,\cC)_{\vec k} :=\maps^{h}_\otimes(\cB_{\vec\bullet},\cC^{\smallbox{oplax}}_{\vec k})
\end{align*}
are  complete $n$-fold Segal spaces which depend naturally on the choices of $\cB$ and $\cC$.
\end{corollary}

We end this section by studying the special case of symmetric monoidal (op)lax natural natural endomorphisms of the ``trivial'' symmetric monoidal functor $\cB \to \cC$ whose constant value is the monoidal unit $\unit \in \cC$ (or a higher identity  morphism thereon).

\begin{definition}\label{defn.trivial theory}
  The Segal condition for $m=0$ in \cref{defn.Gammaobject} gives an equivalence of complete $n$-fold Segal spaces
$$\Cc[0] \simeq *.$$
  The map $\unit: *\simeq\Cc[0] \to \Cc[1]$ picks out the \define{unit}.

  We make the contractible space $\cC[0]$ into a symmetric monoidal complete $n$-fold Segal space by declaring $\cC[0][m] = \cC[0]$; the Segal condition follows from contractibility.  Then the unit map $\unit : \cC[0] \to \cC[1]$ is in fact symmetric monoidal.  

Let $\cB$ and $\cC$ be symmetric monoidal complete $n$-fold Segal spaces.  Since $\cC[0]$ is contractible, so too is $\maps_{\otimes}^{h}(\cB,\cC[0])$.  The \define{trivial symmetric monoidal functor} from $\cB$ to $\cC$ is picked out by the map $\unit : \ast \simeq \maps_{\otimes}^{h}(\cB,\cC[0]) \to \maps_{\otimes}^{h}(\cB,\cC)$.
\end{definition}

  \Cref{lemma.basic facts about localizations} implies that homotopy fibered products of symmetric monoidal complete $n$-fold Segal spaces can be computed levelwise: $(\cX \times^h_\cY \cZ)[m]_{\vec k} \simeq \cX[m]_{\vec k} \times^h_{\cY[m]_{\vec k}} \cZ[m]_{\vec k}$. Abbreviating $\cC^{\unit\downarrow\unit} = \cC[0] \times^{h}_{\cC} \cC^{\downarrow} \times^{h}_{\cC} \cC[0]$ and $\cC^{\unit\to\unit} = \cC[0] \times^{h}_{\cC} \cC^{\to} \times^{h}_{\cC} \cC[0]$, it follows that:
\begin{align*}
\mathrm{Lax}_\otimes(\unit,\unit) &
   \simeq \maps_\otimes^h\left(\cB, \cC[0] \underset{\cC }{\overset h\times} \cC^\downarrow  \underset{\cC }{\overset h\times} \cC[0] \right) \,\,\,= \maps_\otimes^h\left(\cB,  \cC^{\unit\downarrow\unit} \right).
\\
\mathrm{Oplax}_\otimes(\unit,\unit) &
 \simeq \maps_\otimes^h\left(\cB, \cC[0] \underset{\cC }{\overset h\times} \cC^\to  \underset{\cC }{\overset h\times} \cC[0] \right)  = \maps_\otimes^h\left(\cB,  \cC^{\unit\to\unit} \right).
   \end{align*}  
The objects of both $\cC^{\unit\downarrow\unit}$ and $\cC^{\unit\to\unit}$ are the 1-morphisms in $\cC$ with source and target $\unit \in \cC$.  There is also an $(\infty,n-1)$-category with this property: the ``based loop category'' $\Omega\cC$ of  endomorphisms of $\unit\in \cC$.  The based loop category $\Omega\cC$ is an example of the general phenomenon that complete $n$-fold Segal spaces are ``enriched'' in complete $(n-1)$-fold Segal spaces:

\begin{definition}\label{defn.loop space}
  Let $\cC = \cC_{\vec\bullet}$ be a complete $n$-fold Segal space.  For any objects $X,Y\in \cC_{(0)}$, the \define{hom category} $\cC(X,Y)$ is the complete $(n-1)$-fold Segal space
  $$ \cC(X,Y)_{\vec\bullet} = \{X\} \underset{\cC_{0,\vec \bullet}}{\overset h \times} \cC_{1,\vec\bullet} \underset{\cC_{0,\vec \bullet}}{\overset h \times} \{Y\} $$
  where the two maps $\cC_{1,\vec\bullet} \rightrightarrows \cC_{0,\vec\bullet}$ are the source and target maps in the first direction, and $\{X\}$ (resp.\ $\{Y\}$) denotes the constant complete $n$-fold Segal space $\{X\}_{\vec\bullet} = \{X\}$ (resp.\ $\{Y\}_{\vec\bullet} = \{Y\}$).
  
  If $\cC$ is additionally symmetric monoidal, the \define{based loop category} $\Omega\cC$ is the symmetric monoidal complete $(n-1)$-fold Segal space
  $$ (\Omega\cC)_{\vec\bullet} = \cC(\unit,\unit)_{\vec\bullet} = \cC[0]_{0,\vec\bullet} \underset{\cC_{0,\vec \bullet}}{\overset h \times} \cC_{1,\vec\bullet} \underset{\cC_{0,\vec \bullet}}{\overset h \times} \cC[0]_{0,\vec\bullet}. $$
Being an $(\infty,n-1)$-category, we can treat $\Omega\cC$ as an $(\infty,n)$-category with only identity $n$-morphisms.
\end{definition}

\begin{proposition}\label{thm.qft}
  Let $\cC$ be a symmetric monoidal complete $n$-fold Segal space.  Then $\cC^{\unit\downarrow\unit}$ and $\Omega\cC$ are canonically equivalent as symmetric monoidal complete $n$-fold Segal spaces.
\end{proposition}

\begin{proof}
 \cref{lemma.untwisted} implies that there is an equivalence of uple-cosimplicial computads
$$ \Theta^{(0);(0)} \underset{\Theta^{\vec\bullet;(0)}}{\overset h \cup} \Theta^{\vec\bullet;(1)} \underset{\Theta^{\vec\bullet;(0)}}{\overset h \cup} \Theta^{(0);(0)} \simeq \Theta^{1,\vec\bullet}.$$

  It follows that for any complete $n$-fold Segal space $\cD$ and any objects $X,Y\in \cD_{(0)}$, there is a functorial-in-$\cD$ equivalence of complete $n$-fold Segal spaces
  $$ \{X\} \underset{\cD}{\overset h \times}  
  \cD^\downarrow
  \underset{\cD}{\overset h \times} \{Y\}
  \simeq \cD(X,Y), $$
  where the two maps $\cD^\downarrow \rightrightarrows \cD$ are $s_v$ and $t_v$, and the inclusions $\{X\},\{Y\} \to \cD$ extend the maps $\{X\},\{Y\} \to \cD_{(0)}$ along degeneracies.  Indeed, the left-hand side is just the derived mapping space into $\cD$ from $\Theta^{(0);(0)} \underset{\Theta^{\vec\bullet;(0)}}{\overset h \cup} \Theta^{\vec\bullet;(1)} \underset{\Theta^{\vec\bullet;(0)}}{\overset h \cup} \Theta^{(0);(0)}$ such that the two copies of $\Theta^{(0);(0)}$ go to $X$ and $Y$ respectively; the right-hand side is the derived mapping space into $\cD$ from $\Theta^{1,\vec\bullet}$ such that the two $0$-cells in $\Theta^{1,\vec\bullet}$ (corresponding to the two inclusions $\Theta^{(0)} = \Theta^{0,\vec\bullet} \rightrightarrows \Theta^{1,\vec\bullet}$) go to $X$ and $Y$.
  
  Let $\cC$ now be a symmetric monoidal complete $n$-fold Segal space and let $\cD$ range over the complete $n$-fold Segal spaces $\cC[m]$, with $X = Y = \unit : \cC[0] \to \cC[m]$ for every $m$.  The \emph{symmetric monoidal} equivalence $\{\unit\} \times^h_\cC \cC^\downarrow \times^h_\cC \{\unit\} \simeq \cC(\unit,\unit)$ then follows from functoriality-in-$\cD$ of the above equivalence.
\end{proof}

\begin{remark}\label{remark.oplax wrong way}
In particular, $\cC^{\unit\downarrow\unit}$, which is {a priori} an $n$-fold Segal space, is actually (equivalent to) an $(n-1)$-fold Segal space. 
 It turns out that $\cC^{\unit \to \unit}$, which controls oplax natural transformations, as well as the fibered product $\cC[0] \times_{\cC}^{h} [\Theta^{(1)},\cC] \times_{\cC}^{h} \cC[0]$ controlling strong natural transformations, are also only $(n-1)$-fold Segal, even though they are {a priori} $n$-fold Segal.  However, already in the case of bicategories, one may check that neither $\cC^{\unit \to \unit}$ nor $\cC[0] \times_{\cC}^{h} [\Theta^{(1)},\cC] \times_{\cC}^{h} \cC[0]$ is equivalent to $\Omega\cC$ (although the former always is equivalent to what you get from $\Omega\cC$ by reversing all odd-dimensional morphisms).
\end{remark}

\section{Twisted quantum field theories}\label{section.QFT}

In this section we apply \cref{defn.sym mon lax transfor} to clarify the notion of ``twisted quantum field theory'' from \cref{defn.T-twisted theory introduction}.  We begin by letting $\cat{Bord} = \cat{Bord}_{d-n,\dots,d}^\cG$ be an arbitrary symmetric monoidal complete $n$-fold Segal space, which we think of as the $(\infty,n)$-category of ``geometric bordisms'' upon which quantum field theories are to be based.  We will later specialize to the case $\cat{Bord} = \cat{Bord}_n^{\fr}$ of framed cobordisms.  Recall that a \define{quantum field theory} based on $\cat{Bord}$ and valued in a symmetric monoidal $(\infty,n)$-category $\cC$ is a symmetric monoidal functor $\cat{Bord} \to \cC$.

\begin{definition}\label{defn.relative}
Let $\cC$ be a symmetric monoidal complete $(n+1)$-fold Segal space.  The space
$$\maps_\otimes^h(\cat{Bord},\cC^\downarrow), \quad\mbox{respectively}\quad \maps_\otimes^h(\cat{Bord},\cC^\to),$$
is the space of \define{lax}, respectively \define{oplax}, \define{twisted field theories} valued in $\cC$.  More generally, the space of \define{lax}, resp.\ \define{oplax}, \define{$k$-times-twisted field theories} in $\cC$ is $\maps_\otimes^h(\cat{Bord},\cC^{\smallbox{lax}}_{(k)})$, resp.\ $\maps_\otimes^h(\cat{Bord},\cC^{\smallbox{oplax}}_{(k)})$.  By \cref{cor.symm mon oplax-transformation}, lax and oplax $k$-times-twisted field theories in $\cC$ are the $k$-morphisms in $(\infty,n)$-category $\operatorname{Fun}_\otimes^{\mathrm{lax}}(\cat{Bord},\cC)$ and $\operatorname{Fun}_\otimes^{\mathrm{oplax}}(\cat{Bord},\cC)$ respectively.

$$
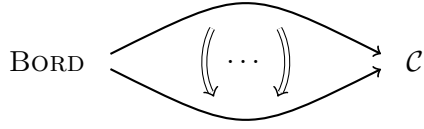
\begin{tikzpicture}[scale=2]
  \path node (B) {} node[anchor=east] {$\cat{Bord}$} +(2,0) node (C) {} node[anchor=west] {$\cC$};
  \draw[arrow] (B) .. controls +(1,-.5) and +(-1,-.5) .. (C) node[pos=0.4] (t1) {} node[pos=0.6] (t2) {};
  \draw[arrow] (B) .. controls +(1,.5) and +(-1,.5) .. (C) node[pos=0.4] (s1) {} node[pos=0.6] (s2) {};
  \draw[twoarrowlonger] (s1) .. controls +(-.15, -.275) and +(-0.15, 0.275) .. (t1);
  \draw[twoarrowlonger] (s2) .. controls +(.15, -.275) and +(0.15, 0.275) .. (t2);
\draw (1,0) node {$\dots$};
\end{tikzpicture}$$

For an (op)lax twisted field theory, we call its compositions with the source and target maps $\cC^\downarrow\rightrightarrows \cC$, respectively $\cC^\to\rightrightarrows \cC$, its \define{twists $S$ and $T$} and say the (op)lax twisted field theory is \define{$(S,T)$-twisted}.

In the special case when $S=\unit$ is the trivial twist from \cref{defn.trivial theory}, we say the twisted field theory in question is an \define{(op)lax $T$-twisted  field theory}. 
As in the discussion after \cref{defn.trivial theory}, we see that the spaces of lax and oplax $T$-twisted field theories can be computed as:
\begin{align*}
\mathrm{Lax}_\otimes(\unit, T) &\simeq 
  \maps_\otimes^h\left(\cat{Bord}, \cC[0] \underset{\cC }{\overset h\times} \cC^\downarrow\right) \underset{\maps_\otimes^h(\cat{Bord},\cC)}{\overset h\times} \{T\},\\
\mathrm{Oplax}_\otimes(\unit, T) &\simeq 
   \maps_\otimes^h\left(\cat{Bord}, \cC[0] \underset{\cC }{\overset h\times} \cC^\to\right) \underset{\maps_\otimes^h(\cat{Bord},\cC)}{\overset h\times} \{T\} . 
\end{align*}
We will henceforth denote $\cC^{\unit\downarrow}:=\cC[0] \underset{\cC }{\overset h\times} \cC^\downarrow$ and $\cC^{\unit\to}:=\cC[0] \underset{\cC }{\overset h\times} \cC^\to$.
\end{definition}

\begin{example}
Twisted field theories as defined in \cite[Definition 5.2]{MR2742432} unravel to lax $T$-twisted field theories based on their category $(G,\mathbb M)\cat{-Bord}$ of super-Euclidean bordisms and valued in their category $\cat{TA}$ of topological algebras. 
\end{example}

 \begin{example} \label{eg.lax versus oplax T-twisted qfts}
 To illustrate \cref{defn.relative}, consider the values of a lax or oplax $T$-twisted  field theory on a \emph{closed} bordism $b\in \cat{Bord}$, i.e.\ a $k$-morphism whose source and target are ($(k-1)$-dimensional identity morphisms on) the monoidal unit $\emptyset \in \cat{Bord}$.  By symmetric monoidality, $T(\emptyset) = \unit \in \cC$, and so $T(b)$ is a $k$-morphism from $\unit$ to $\unit$ (or rather, $(k-1)$-dimensional identity morphisms thereon).
 
 First, consider the case when $Z : \unit \Rightarrow T$ is an oplax transformation, i.e.\ a functor $Z : \cat{Bord} \to \cC^{\to}$ with $s_{h}\circ Z = \unit$ and $t_{h}\circ Z = T$.  Since $Z$ is also symmetric monoidal, $Z(\emptyset) = \unit$.  Thus $Z(b)$ is an  $1$-by-$k$ morphism in $\cC$ with:
 $$ s_{h}(Z(b)) = \unit, \quad t_{h}(Z(b)) = T(b), \quad s_{v}(Z(b)) = \unit, \quad t_{v}(Z(b)) = \unit. $$
 Finally, the ``square'' with these boundaries is to be filled in with the $(k+1)$-morphism $Z(b)$.  According to \cref{prop.i-by-j-morphism}, such filling-in cell always goes ``from $s_{h}$ to $t_{h}$,''  with compositions as necessary.  In our case such compositions are trivial, and the end result is that $Z(b)$ is a $(k+1)$-morphism in $\cC$ with source $\unit$ and target $T(b)$.  Such a $(k+1)$-morphism can be thought of as an ``element'' of the $k$-morphism $T(b)$.
 
 For comparison, when $Z : \unit \Rightarrow T$ is a lax transformation, the  $k$-by-$1$ morphism $Z(b)$ enjoys
 $$ s_{v}(Z(b)) = \unit, \quad t_{v}(Z(b)) = T(b), \quad s_{h}(Z(b)) = \unit, \quad t_{h}(Z(b)) = \unit. $$
 According to \cref{prop.i-by-j-morphism}, whether the filling $(k+1)$-morphism 
 goes from $s_v$ to $t_v$ or from $t_v$ to $s_v$ depends on the parity of $k$: when $k$ is even, $Z(b) : \unit \to T(b)$ is an element of $T(b)$, whereas when $k$ is odd, $Z(b) : T(b) \to \unit$ is a ``coelement.''
 
 The cases when $k=0$, $1$, and $2$ can be read from \cref{eg.recollection table}:
$$
\begin{array}{rc|c|c}
 && \text{lax} & \text{oplax}\\[3pt] \hline \hline &&& \\[-6pt]
\text{object} & b \in \cB_{(0)} & Z(b) : \unit \to T(b) & Z(b) : \unit \to T(b)\\[6pt] \hline &&& \\[-6pt]
1\text{-morphism} & \unit\overset b \to \unit  \in \cB_{(1)} &
\begin{tikzpicture}[baseline=(r.base)]
  \path node[dot] (sl) {} +(2,0) node[dot] (sr) {} +(0,-2) node[dot] (tl) {} +(2,-2) node[dot] (tr) {};
  \draw[arrow] (sl) node[anchor = south east] {$\unit$} -- node[auto] {$\unit$} (sr) node[anchor = south west] {$\unit$};
  \draw[arrow] (sl)  -- node[auto,swap] {$\unit$} (tl);
  \draw[arrow] (tl) node[anchor = north east] {$\unit$} -- node[auto,swap] (Tb) {$T(b)$}  (tr) node[anchor = north west] {$\unit$};
  \draw[arrow] (sr) -- node[auto] (r) {$\unit$} (tr);
  \draw[twoarrow] (tl) -- node[auto] {$Z(b)$} (sr);
\path (Tb) +(0,-1) node {$Z(b):  T(b) \Rightarrow \unit$};
\end{tikzpicture}
&
\begin{tikzpicture}[baseline=(r.base)]
  \path node[dot] (sl) {} +(2,0) node[dot] (sr) {} +(0,-2) node[dot] (tl) {} +(2,-2) node[dot] (tr) {};
  \draw[arrow] (sl) node[anchor = south east] {$\unit$} -- node[auto] {$\unit$} (sr) node[anchor = south west] {$\unit$};
  \draw[arrow] (sl)  -- node[auto,swap] {$\unit$} (tl);
  \draw[arrow] (tl) node[anchor = north east] {$\unit$} -- node[auto,swap] (Tb) {$\unit$}  (tr) node[anchor = north west] (GB2) {$\unit$};
  \draw[arrow] (sr) -- node[auto] (r) {$T(b)$} (tr);
  \draw[twoarrow] (tl) -- node[auto] {$Z(b)$} (sr);
  \path (GB2.south) ++(0,-5pt) coordinate (bottom);
\path (Tb) +(0,-1) node {$Z(b):  \unit \Rightarrow T(b)$};
\end{tikzpicture}
\\[6pt] \hline &&& \\[-6pt]
2\text{-morphism} &
\begin{tikzpicture}[baseline=(base), scale=0.8]
  \path node[dot] (a) {} node[anchor=east] {$\unit$} +(2,0) node[dot] (c) {} node[anchor=west] {$\unit$} +(0,-3pt) coordinate (base);
  \draw[arrow] (a) .. controls +(1,-.5) and +(-1,-.5) .. coordinate (t)   (c);
  \draw[arrow] (a) .. controls +(1,.5) and +(-1,.5) .. coordinate (s)  (c);
\draw (s) node[anchor=south] {$\scriptstyle \unit$};
\draw (t) node[anchor=north] {$\scriptstyle \unit$};
  \draw[twoarrowlonger] (s) -- node[anchor=east] {$\scriptstyle b$}   (t);
\end{tikzpicture} \in \cB_{(2)} 
&
\begin{tikzpicture}[baseline=(middle), scale=0.8]
  \path node[dot] (a) {} node[anchor= south east] {} +(3,0) node[dot] (b) {} node[anchor= south west] {} +(0,-3) node[dot] (c) {} node[anchor= north east] {} +(3,-3) node[dot] (d) {} node[anchor= north west] {}+(0,-1.5) coordinate (middle);
  \path (b) +(-.35,.25) coordinate (alpha1);
  \path (c) +(.35,-.25) coordinate (alpha2);
  \draw[arrow] (a) -- coordinate (l1) node[anchor=east] {} coordinate[very near end] (r1) (c);
  \draw[arrow] (b) -- coordinate[very near start] (l2) coordinate (r2) node[anchor=west] (rightlable) {} (d);
  \draw[twoarrow,thin] (l1) -- node[auto,pos=.4,inner sep=1pt] {$\scriptstyle \unit$} (l2);
  \draw[arrow] (a) .. controls +(1.5,.75) and +(-1.5,.75) .. coordinate (ss)  (b);
  \draw[arrow,thin] (c) .. controls +(1.5,.75) and +(-1.5,.75) .. coordinate (ts)  (d);
  \draw[arrow] (c) .. controls +(1.5,-.75) and +(-1.5,-.75) .. coordinate (tt) (d);
  \draw[twoarrowlonger,thin] (ts) -- node[anchor=west] {$\scriptstyle T(b)$}  (tt);
  \draw[threearrowpart1] (alpha1) --  (alpha2);
  \draw[threearrowpart2] (alpha1) --  (alpha2);
  \draw[threearrowpart3] (alpha1) -- node[inner sep= 1pt, fill=white] {$\scriptstyle Z(b)$}  (alpha2);
  \draw[arrow,thick] (a) .. controls +(1.5,-.75) and +(-1.5,-.75) .. coordinate (st) (b);
  \draw[twoarrowlonger,thick] (ss) -- node[anchor=east] {$\scriptstyle \unit$}  (st);
  \draw[twoarrow,thick] (r1) -- node[auto,swap,pos=.6,inner sep=1pt] {$\scriptstyle \unit$} (r2);
\path (tt) +(0,-1) node {$Z(b):  \unit \Rrightarrow T(b)$};
\end{tikzpicture}
&
\begin{tikzpicture}[baseline=(middle), scale=0.8]
  \path node[dot] (a) {} +(3,0) node[dot] (b) {} +(0,-3) node[dot] (c) {} +(3,-3) node[dot] (d) {} +(0,-1.5) coordinate (middle);
  \draw[arrow] (a) .. controls +(-.75,-1.5) and +(-.75,1.5) .. coordinate[near end] (lt) coordinate[very near end] (lt2) (c);
  \draw[arrow,thin] (a) .. controls +(.75,-1.5) and +(.75,1.5) .. coordinate[near start] (ls) coordinate (ls2) (c);
  \draw[twoarrowlonger] (ls) -- node[auto,swap] {$\scriptstyle \unit$} (lt);
  \draw[arrow] (b) .. controls +(.75,-1.5) and +(.75,1.5) .. coordinate[near start] (rs) coordinate[very near start] (rs0) (d);
  \draw[twoarrow] (ls2) -- coordinate[near end] (s) node[auto] {$\scriptstyle \unit$} (rs0);
  \draw[arrow] (b) .. controls +(-.75,-1.5) and +(-.75,1.5) .. coordinate[near end] (rt) coordinate (rt0) (d);
  \draw[twoarrowlonger,thick] (rs) -- node[auto,fill=white,inner sep=1pt] {$\scriptstyle T(b)$} (rt);
  \draw[arrow] (a) node[anchor = south east] {} -- node[auto] {} (b) node[anchor = south west] {};
  \draw[arrow] (c) node[anchor = north east] {} -- node[below] {} (d) node[anchor = north west] {};
  \path (lt2) -- coordinate[near start] (t) (rt0);
  \draw[threearrowpart1] (t) --  (s); 
  \draw[threearrowpart2] (t) -- (s); 
  \draw[threearrowpart3] (t) -- node[inner sep= 1pt, fill=white] {$\scriptstyle Z(b)$} (s); 
  \draw[twoarrow,thick] (lt2) -- node[auto,swap] {$\scriptstyle  \unit$} (rt0);
\path (tt) +(0,-1) node {$Z(b):  \unit \Rrightarrow T(b)$};
\end{tikzpicture} 
\end{array}
$$
 \end{example}
 
 Based on the examples described in the literature,
it is tempting to take  the slogan ``$Z(b)$ is an element of $T(b)$'' as a defining property of ``$T$-twisted quantum field theory'' --- for example, \cite{FreedTeleman2012} take this slogan as part of their closely related notion of ``relative quantum field theory.'' Based on \cref{eg.lax versus oplax T-twisted qfts}, this suggests that a ``{$T$-twisted field theory}'' is best defined to be what we have called an \emph{oplax} $T$-twisted field theory.  However, another basic desired property suggested is that we should be able to recover the usual notion of functorial QFT when twisting by the trivial twist (c.f.\ \cite[Lemma 5.7]{MR2742432}).  For this criterion we find that \emph{lax} $T$-twisted field theories behave better:

\begin{theorem}\label{trivially twisted versus untwisted}
  For any  symmetric monoidal $(\infty,n)$-category $\cat{Bord}$ and any target symmetric monoidal $(\infty,n+1)$-category $\cC$, the space $\mathrm{Lax}_{\otimes}(\unit,\unit) \simeq \maps^{h}_{\otimes}(\cat{Bord},\cC^{\unit\downarrow\unit})$ of \define{lax trivially-twisted field theories} (based on $\cat{Bord}$) is canonically equivalent to the space $\maps^{h}_{\otimes}(\cat{Bord},\Omega\cC)$ of ``absolute'' field theories valued in the loop category $\Omega\cC$. 
\end{theorem}
\begin{proof}
  This follows immediately from \cref{thm.qft}.
\end{proof}

\begin{remark}
\Cref{remark.oplax wrong way} implies that the space $\mathrm{Oplax}_\otimes(\unit,\unit)$ of \define{oplax trivially-twisted field theories} is also equivalent to a space of absolute field theories.  However, they are not  valued in $\Omega \cC$, but rather in the ``odd opposite'' of $\Omega \cC$ in which the directions of the odd morphisms are reversed.
\end{remark}

\Subsection{Fully extended twisted field theories}

We now turn to the fully local topological case.  This will allow us, in the next subsection, to compare our notion of (op)lax twisted field theory to the boundary field theories proposed in \cite{FreedTeleman2012} as the definition of ``relative field theory''; another very nice account of the comparison of boundary conditions to (op)lax twisted field theories is available in~\cite{FV2014}.

Let $\cat{Bord} = \cat{Bord}_n^{\mathrm{fr}}$ denote the fully extended framed topological bordism category from \cite{Lur09,ClaudiaDamien1}.  
The main theorem of~\cite{Lur09}, the so-called \define{Cobordism Hypothesis}, 
asserts that for any $(\infty,N)$-category $\cC$, possibly with $N \geq n$, the $(\infty,N)$-category $\operatorname{Fun}_\otimes^{\mathrm{strong}}(\cat{Bord}_n^{\fr},\cC)$ (of symmetric monoidal functors $\cat{Bord}_n^{\fr} \to\cC$ and \emph{strong} transformations as morphisms) is equivalent to the maximal $\infty$-groupoid of $\cC$ whose objects are ``$n$-dualizable objects in $\cC$.''

Although this story is well-studied, we include a brief review to establish notation.  A 1-morphism $f : X \to Y$ in a bicategory $\cC$ is \define{left-adjunctible} if there is a 1-morphism $f^L : Y \to X$, called the \define{left adjoint} of $f$, and 2-morphisms $\ev_f : f^L \circ f \Rightarrow \id_X$ and $\coev_f : \id_Y \Rightarrow f \circ f^L$, called \define{evaluation} (or \define{counit}) and \define{coevaluation} (or \define{unit}), such that the two natural compositions ``$\ev \circ \coev$'' are identities on $f$ and $f^L$.  Similarly, a 1-morphism $f$ is \define{right-adjunctible} if there is a \define{right adjoint} $f^R$ and 2-morphisms $f \circ f^R \Rightarrow \id_Y$ and $\id_X \Rightarrow f^R \circ f$ whose two possible compositions are identities.  Left and right adjoints are unique up to unique isomorphism if they exist, and clearly $f \cong (f^L)^R \cong (f^R)^L$ if $f$ is both left- and right-adjunctible.  We will call $f$ just \define{adjunctible} if it is both left and right adjunctible and moreover $f^L$ is left adjunctible, as is $(f^L)^L$, as is $((f^L)^L)^L)$, etc., and $f^R$ is right adjunctible, as is $(f^R)^R$, etc.; this notion does not seem to have a standard name in the literature.

For $k\geq 1$, a $k$-morphism in an $(\infty,N)$-category $\cC$ is called \define{left-adjunctible}, \define{right-adjunctible}, or \define{adjunctible} if it is so is the appropriate homotopy bicategory (defined analogously to \cref{defn.homotopy bicategory}).  Thus left-adjunctibility (or right-adjunctibility, or adjunctibility) of a $k$-morphism asserts the existence of certain $(k+1)$-morphisms; we will say that those $(k+1)$-morphisms \define{witness} left- (or right-) adjunctibility of the given $k$-morphism.  By induction, we will say that a $k$-morphism in $\cC$ is \define{$n$-times left-adjunctible} for $n\geq 2$ if it is $(n-1)$-times left-adjunctible and the $(k+n-1)$-morphisms witnessing $(n-1)$-times left-adjunctibility are themselves left-adjunctible.

If $\cC$ is a symmetric monoidal category, an object $X \in \cC$ is \define{1-dualizable} if there is a \define{dual} object $X^*$ and 1-morphisms $\ev_X : X^* \otimes X \to \unit$ and $\coev_X : \unit \to X \otimes X^*$ such that the two compositions $(\id_X \otimes \ev_X) \circ (\coev_X \otimes \id_X) : X \to X$ and $(\ev_X \otimes \id_{X^*}) \circ (\id_{X^*} \otimes \coev_X) : X^* \to X^*$ are identities; 
these are called the \define{snake relations}.
If $\cC$ is a symmetric monoidal $(\infty,N)$-category, an object is 1-dualizable if it is so in the homotopy category of $\cC$.  Now proceeding by induction, an object in a symmetric monoidal $(\infty,N)$-category is \define{$n$-dualizable} for $n\geq 2$ if it is $(n-1)$-dualizable and all the evaluation and coevaluation $(n-1)$-morphisms witnessing $(n-1)$-dualizability are themselves adjunctible.
Lest $n$-dualizability look like too much information to check in practice, we will later recall a few standard lemmas simplifying the problem.

In the remainder of this section, we will use the Cobordism Hypothesis to analyze fully extended framed topological twisted field theories and their higher analogs by unravelling $n$-dualizability in $\cC^\downarrow$, $\cC^\to$, and their higher analogs. Our main theorem in this section, \cref{thm.topological higher twisted field theories}, gives a full classification thereof. From now on, let $n\geq 0$ and let $\cC$ be a symmetric monoidal $(\infty,N)$-category, possibly with $N \geq n$.

\begin{theorem}\label{thm.topological higher twisted field theories}
For $\ast = $ ``lax'', respectively ``oplax,'' an object in $\cC^\boxasterisk_{(j)}$ is $n$-dualizable if and only if its vertical, respectively horizontal, source and target under $\cC^\boxasterisk_{(j)} \rightrightarrows\cC^\boxasterisk_{(j-1)}$ are $n$-dualizable and moreover the corresponding $j$-morphism in $\cC$ is $n$-times left-adjunctible, respectively right-adjunctible.
\end{theorem}
\begin{proof}
  By \cref{corollary.left is enough}, to check $n$-dualizability it is enough to check $n$-times left-adjunctibility.  The claim then follows by induction, using \cref{prop.1dualizability in C-down} as the base case and \cref{lax left adjunctibility} as the induction step.
\end{proof}
Assuming the Cobordism Hypothesis, \cref{thm.topological higher twisted field theories} has an immediate corollary:
\begin{corollary} \label{cor.lax twisted framed tfts}
There is an equivalence of $(\infty,N)$-categories between:
  \begin{enumerate}
  \item  
  the $(\infty,N)$-category $\operatorname{Fun}_\otimes^{\mathrm{lax}}(\cat{Bord}_n^{\fr},\cC)$ whose objects are fully-extended framed topological field theories, 1-morphisms are lax twisted field theories between them, and in general whose $k$-morphisms are lax $k$-times-twisted field theories;
  \item the sub-$(\infty,N)$-category of $\cC$ whose objects are $n$-dualizable objects in $\cC$, and 1-morphisms are $n$-times left-adjunctible 1-morphisms between $n$-dualizable objects in $\cC$, and in general whose $k$-morphisms are the $n$-times left-adjunctible $k$-morphisms in $\cC$ between the allowed $(k-1)$-morphisms. \label{cor.lax twisted framed tfts.item2}
  \end{enumerate}
The same statement holds for ``lax'' replaced by ``oplax'' and ``left'' replaced by ``right''. \qedhere
\end{corollary}

\begin{remark}\label{comment about Cfd}
The $(\infty,N)$-category in (\ref{cor.lax twisted framed tfts.item2}) is not the same as the $(\infty,N)$-category $\cC^\fd$ appearing in the statement of the Cobordism Hypothesis in \cite{Lur09}.  The objects of $\cC^\fd$ are required to be $N$-dualizable; when $N>n$, this is strictly stronger than the $n$-dualizability required in (\ref{cor.lax twisted framed tfts.item2}).  On the other hand, $k$-morphisms in $\cC^\fd$ are required to be $(N-k)$-times adjunctible, whereas $k$-morphisms in the category from (\ref{cor.lax twisted framed tfts.item2}) are required to be $n$-times left-adjunctible; the latter is strictly stronger when $k > N-n$, as then $n$-times left adjunctibility forces the $k$-morphism to be invertible.  
\end{remark}

In the proofs of the results leading to \cref{thm.topological higher twisted field theories}, we will adopt a certain cavalier attitude towards the difference between equivalence and equality in higher categories, generally assuming that any category in question is strict.  This is allowed by the \cref{lemma about gauntification}, which implies that   when studying questions about dualizability and adjunctibility we can assume that the categories in question are gaunt.

For a bicategory $\cC$, let its \define{gauntification} $\operatorname{gau}(\cC)$ be the strict 2-category given by the left adjoint to the inclusion of gaunt 2-categories into bicategories
$$\cat{Gaunt}_2\hookrightarrow 2\cat{Cat} \hookrightarrow \cat{Bicat},$$
i.e.~it is formed by first forcing all invertible 2-morphisms in $\cC$ to be identities (thereby identifying isomorphic 1-morphisms) and then forcing all invertible 1-morphisms to be identities (thereby identifying isomorphism objects).
\begin{lemma}\label{lemma about gauntification}
  Let $\cC$ be a bicategory. A 1-morphism $f : X \to Y$ in $\cC$ is left- (resp.\ right-) adjunctible iff its class $[f] : [X] \to [Y]$ in $\operatorname{gau}(\cC)$ is left- (resp.\ right-) adjunctible.
\end{lemma}
\begin{proof}
  Adjunctions are preserved by functors, including the quotient functor $\cC \to \operatorname{gau}(\cC)$; this establishes the ``only if'' direction.  For the ``if'' direction, assume $[f]$ is left-adjunctible (the right-adjunctible case being analogous), and let $[f]^L : [Y] \to [X]$ be its left adjoint.  Choose any lift of $[f]^L$ to $\cC$; it has domain some object equivalent to $Y$ and codomain some object equivalent to $X$, and by composing with these equivalences we can assume that our choice of lift is a 1-morphism $Y \to X$.  We will suggestively write $f^L : Y \to X$ for this choice of lift.  Now choose lifts of $\ev_{[f]}$ and $\coev_{[f]}$.  Whiskering by equivalences as necessary, we can assume that the lifts are 2-morphisms $f^L \circ f \Rightarrow \id_X$ and $\id_Y \Rightarrow f \circ f^L$ respectively.  We will suggestively write $\ev_f: f^L \circ f \Rightarrow \id_X$ for this choice of a lift of $\ev_{[f]}$; we will write $\widetilde{\coev_f}:\id_Y \Rightarrow f \circ f^L$ for the choice of lift of $\coev_{[f]}$.  Now consider the two compositions between $\ev_f$ and $\widetilde{\coev_f}$.  Both lift identities in $\operatorname{gau}(\cC)$, and therefore are isomorphisms.  By multiplying $\widetilde{\coev_f}$ by the appropriate isomorphism, we can guarantee that the composition $f \Rightarrow f$ is the identity.  Let $\phi : f^* \Rightarrow f^*$ denote the other composition.  On the one hand it is an automorphism; on the other hand, the fact that the other composition is the identity implies that $\phi^2 = \phi$.  It follows that $\phi$ is an identity.
\end{proof}

We now proof the base case for \cref{thm.topological higher twisted field theories}. The oplax case is also proven in \cite{HSS2015}.
\begin{proposition} \label{prop.1dualizability in C-down}
  Let $\cC$ be a symmetric monoidal $(\infty,n)$-category and $f : X\to Y$ a 1-morphism in $\cC$.  Then $f$ is 1-dualizable as an object in $\cC^\downarrow$ if and only if $X$ and $Y$ are 1-dualizable in $\cC$ and $f : X \to Y$ is (1-times) left-adjunctible. Similarly, $f$ is 1-dualizable as an object in $\cC^\to$ if and only if $X$ and $Y$ are 1-dualizable in $\cC$ and $f : X \to Y$ is (1-times) right-adjunctible.
\end{proposition}

\begin{proof}
  Since the functors $s_v,t_v : \cC^\downarrow \rightrightarrows \cC$ are symmetric monoidal and symmetric monoidal functors preserve dualizability, dualizability of both $X$ and $Y$ is certainly a necessary condition for dualizability of $f$.
  
  Suppose that $f$ has dual $f^* : X^* \to Y^*$ in $\cC^\downarrow$.  Then we have 2-morphisms in $\cC$
$$
\begin{tikzpicture}[baseline=(r.base)]
  \path node[dot] (sl) {} +(2,0) node[dot] (sr) {} +(0,-2) node[dot] (tl) {} +(2,-2) node[dot] (tr) {} +(1,1.5) node (e) {$\ev_f: \ev_Y \circ (f\otimes f^*) \Rightarrow \ev_X$};
  \draw[arrow] (sl) node[anchor = south east] {$X\otimes X^*$} -- node[auto] {$\ev_X$} (sr) node[anchor = south west] {$\unit$};
  \draw[arrow] (sl)  -- node[auto,swap] {$f\otimes f^*$} (tl);
  \draw[arrow] (tl) node[anchor = north east] {$Y\otimes Y^*$} -- node[auto,swap] {$\ev_Y$}  (tr) node[anchor = north west] {$\unit$};
  \draw[arrow] (sr) -- node[auto] (r) {$\id_\unit$} (tr);
  \draw[twoarrow] (tl) -- node[auto] {$\ev_f$} (sr);
\end{tikzpicture}
\hspace{2cm}
\begin{tikzpicture}[baseline=(r.base)]
  \path node[dot] (sl) {} +(2,0) node[dot] (sr) {} +(0,-2) node[dot] (tl) {} +(2,-2) node[dot] (tr) {} +(1,1.5) node {$\coev_f: \coev_Y \Rightarrow (f^*\otimes f)\circ \coev_X$};
  \draw[arrow] (sl) node[anchor = south east] {$\unit$} -- node[auto] {$\coev_X$} (sr) node[anchor = south west] {$X^*\otimes X$};
  \draw[arrow] (sl)  -- node[auto,swap] {$\id_\unit$} (tl);
  \draw[arrow] (tl) node[anchor = north east] {$\unit$} -- node[auto,swap] {$\coev_Y$}  (tr) node[anchor = north west] {$Y^*\otimes Y$};
  \draw[arrow] (sr) -- node[auto] (r) {$f^*\otimes f$} (tr);
  \draw[twoarrow] (tl) -- node[auto] {$\coev_f\!\!$} (sr);
\end{tikzpicture}
$$
such that the compositions
$$
\begin{tikzpicture}[baseline=(r.base), scale=1.5]
  \path node[dot] (sl) {} +(2,0) node[dot] (sr) {} +(-2,0) node[dot] (sll) {} +(0,-2) node[dot] (tl) {} +(2,-2) node[dot] (tr) {} +(-2,-2) node[dot] (tll) {};
  \draw[arrow] (sll) node[anchor=south east] {$X$} -- node[auto, swap] {$\scriptstyle \id_X\otimes \coev_X$} (sl) node[anchor = south] {$X\otimes X^*\otimes X$} ;
 \draw[arrow] (sll) -- node[auto, swap] (ll) {$f$} (tll) node [anchor= north east] {$Y$};
\draw[arrow] (tll) -- node[auto] {$\scriptstyle \id_Y\otimes \coev_Y$} (tl);
 \draw[arrow] (sl) -- node[auto, swap] {$\scriptstyle \ev_X \otimes \id_X$} (sr) node[anchor = south west] {$X$};
  \draw[arrow] (sl)  -- node[inner sep= 1pt, fill=white] {$f\otimes f^* \otimes f$} (tl);
  \draw[arrow] (tl) node[anchor = north] {$Y\otimes Y^* \otimes Y$} -- node[auto] {$\scriptstyle \ev_Y \otimes \id_Y$}  (tr) node[anchor = north west] {$Y$};
  \draw[arrow] (sr) -- node[auto] (r) {$f$} (tr);
  \draw[twoarrow] (tl) -- node[auto, pos=0.6] {$\scriptstyle \ev_f \otimes \id_f$} (sr);
 \draw[twoarrow] (tll) -- node[auto, pos=0.6] {$\scriptstyle \id_f\otimes \coev_f$} (sl);
\end{tikzpicture}
\mbox{and}
\begin{tikzpicture}[baseline=(r.base), scale=1.5]
  \path node[dot] (sl) {} +(2,0) node[dot] (sr) {} +(-2,0) node[dot] (sll) {} +(0,-2) node[dot] (tl) {} +(2,-2) node[dot] (tr) {} +(-2,-2) node[dot] (tll) {};
  \draw[arrow] (sll) node[anchor=south east] {$X^*$} -- node[auto, swap] {$\scriptstyle \coev_X \otimes \id_{X^*}$} (sl) node[anchor = south] {$X^*\otimes X\otimes X^*$} ;
 \draw[arrow] (sll) -- node[auto, swap] (ll) {$f^*$} (tll) node [anchor= north east] {$Y^*$};
\draw[arrow] (tll) -- node[auto] {$\scriptstyle \coev_Y \otimes \id_{Y^*}$} (tl);
 \draw[arrow] (sl) -- node[auto, swap] {$\scriptstyle \id_{X^*}\otimes \ev_X$} (sr) node[anchor = south west] {$X^*$};
  \draw[arrow] (sl)  -- node[inner sep= 1pt, fill=white] {$f^*\otimes f\otimes f^*$} (tl);
  \draw[arrow] (tl) node[anchor = north] {$Y^* \otimes Y\otimes Y^*$} -- node[auto] {$\scriptstyle \id_{Y^*} \otimes \ev_Y$}  (tr) node[anchor = north west] {$Y^*$};
  \draw[arrow] (sr) -- node[auto] (r) {$f^*$} (tr);
  \draw[twoarrow] (tl) -- node[auto, pos=0.6] {$\scriptstyle \id_{f^*} \otimes \ev_f$} (sr);
 \draw[twoarrow] (tll) -- node[auto, pos=0.6] {$\scriptstyle \coev_f \otimes \id_{f^*}$} (sl);
\end{tikzpicture}
$$
are identities on $f$ respectively $f^*$.

We claim that the \define{mate} $f^\ddagger$ of $f^*$, defined by
$$f^\ddagger: Y \xrightarrow{\id_Y\otimes \coev_X} Y\otimes X^* \otimes X \xrightarrow{\id_Y\otimes f^* \otimes \id_X} Y\otimes Y^*\otimes X \xrightarrow{\ev_Y\otimes \id_X} X,$$
is a left adjoint to $f$.
 To see this, we compute
\begin{align*}
f\circ f^\ddagger &= f\circ (\ev_Y\otimes \id_X)\circ (\id_Y\otimes f^*\otimes \id_X) \circ (\id_Y \otimes \coev_X)\\
&= (\ev_Y\otimes \id_X)\circ (\id_Y\otimes f^*\otimes f) \circ (\id_Y \otimes \coev_X)\\
&= (\ev_Y\otimes \id_X)\circ (\id_Y\otimes ((f^*\otimes f)\circ  \coev_X))
\end{align*}
Note that $\coev_f: \coev_Y \Rightarrow (f^*\otimes f)\circ \coev_X$, so we define the unit of the adjunction to be
$$u=(\ev_Y\otimes \id_Y)\circ(\id_Y\otimes \coev_f) : \id_Y=(\ev_Y\otimes \id_Y)\circ (\id_Y\otimes \coev_Y) \Longrightarrow f\circ f^\ddagger.$$
Similarly,
\begin{align*}
f^\ddagger \circ f &= (\ev_Y\otimes \id_X)\circ (\id_Y\otimes f^*\otimes \id_X) \circ (\id_Y \otimes \coev_X) \circ f\\
&= (\ev_Y\otimes \id_X)\circ (f \otimes f^*\otimes \id_X) \circ (\id_X \otimes \coev_X)\\
&= ((\ev_Y \circ (f\otimes f^*) )\otimes \id_X)\circ  (\id_X \otimes \coev_X))
\end{align*}
and using $\ev_f$ we can define the counit of the adjunction
$$v=(\ev_f\otimes \id_X)\circ (\id_X\otimes \coev_X):  f^\ddagger \circ f  \Longrightarrow (\ev_X\otimes \id_X)\circ (\id_X\otimes \coev_X) =id_X.$$
To see that they satisfy the conditions to be an adjunction, note that the bold red composition
$$
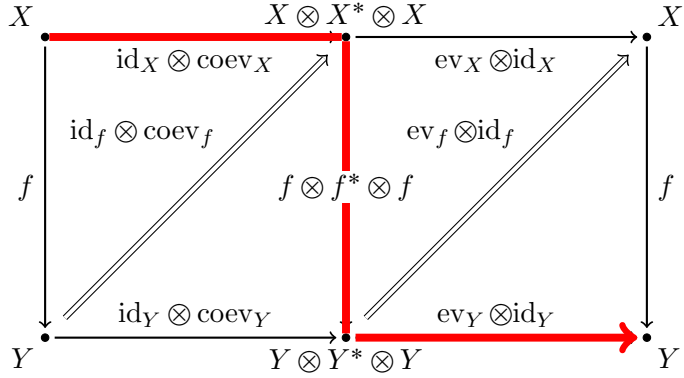
\begin{tikzpicture}[baseline=(r.base), scale=2]
  \path node[dot] (sl) {} +(2,0) node[dot] (sr) {} +(-2,0) node[dot] (sll) {} +(0,-2) node[dot] (tl) {} +(2,-2) node[dot] (tr) {} +(-2,-2) node[dot] (tll) {};
  \draw[arrow] (sll) node[anchor=south east] {$X$} -- node[auto, swap] {$\id_X\otimes \coev_X$} (sl) node[anchor = south] {$X\otimes X^*\otimes X$} ;
 \draw[arrow] (sll) -- node[auto, swap] (ll) {$f$} (tll) node [anchor= north east] {$Y$};
\draw[arrow] (tll) -- node[auto] {$\id_Y\otimes \coev_Y$} (tl);
 \draw[arrow] (sl) -- node[auto, swap] {$\ev_X \otimes \id_X$} (sr) node[anchor = south west] {$X$};
  \draw[arrow] (sl)  --  (tl);
  \draw[arrow] (tl) node[anchor = north] {$Y\otimes Y^* \otimes Y$} -- node[auto] {$\ev_Y \otimes \id_Y$}  (tr) node[anchor = north west] {$Y$};
  \draw[arrow] (sr) -- node[auto] (r) {$f$} (tr);
  \draw[twoarrow] (tl) -- node[auto, pos=0.6] {$\ev_f \otimes \id_f$} (sr);
 \draw[twoarrow] (tll) -- node[auto, pos=0.6] {$\id_f\otimes \coev_f$} (sl);
\draw [arrow, color=red, line width=3pt, join=round] (sll) -- (sl) -- node[inner sep= 1pt, fill=white] {$\color{black} f\otimes f^* \otimes f$} (tl) --(tr);
\end{tikzpicture}
$$
is $f\circ f^\ddagger \circ f$, and the composition of the 2-cells is the one we need. Similarly for $f^\ddagger \circ f \circ f^\ddagger$.

Conversely, suppose that $X$ and $Y$ are dualizable and that $f : X \to Y$ has a left adjoint $f^L : Y \to X$.  In keeping with the above notation, we will write $u:\id_Y \Rightarrow f\circ f^L$ and $v: f^L\circ f \Rightarrow \id_X$ for the unit and counit of the adjunction, i.e.\ the compositions
\begin{gather*}
f= \id_Y \circ f \overset{u\times \id_f}{\Rightarrow} f\circ f^L \circ f \overset{\id_f\times v}{\Rightarrow} f \circ id_X = f,\\
f^L  = f^L\circ \id_Y \overset{\id_{f^L} \times u}{\Rightarrow} f^L \circ f\circ f^L \overset{v\times \id_{f^L}}{\Rightarrow} \id_X \circ f^L = f^L
\end{gather*}
are identities.  Define $f^*$ to be the mate of $f^L$:
$$f^*: X^* \xrightarrow{\coev_Y\otimes \id_{X^*}} Y^*\otimes Y\otimes X^* \xrightarrow{\id_{Y^*}\otimes g \otimes \id_{X^*}} Y^*\otimes X\otimes X^* \xrightarrow{\id_{Y^*} \otimes \ev_X} Y^*.$$
We claim that $f^*$ is a dual to $f$ in $\cC^\downarrow$.  To see this, note first that
$$(f^*\otimes f)\circ \coev_X = (\id_{Y^*} \otimes (f\circ f^L))\circ \coev_Y,$$
as follows from the commutativity if the following diagram:
$$
\begin{tikzpicture}
\matrix(m)[matrix of math nodes,
row sep=2.6em, column sep=2.4em,
text height=1.5ex, text depth=0.25ex]
{\unit & X^*\otimes X & Y^*\otimes Y \otimes X^*\otimes X &Y^* \otimes X \otimes X^*\otimes X & Y^*\otimes X & Y^*\otimes Y\\
&&Y^*\otimes Y&Y^*\otimes X\\};
\path[->,font=\scriptsize]
(m-1-1) edge node[auto] {$\coev_X$} (m-1-2)
edge node[auto, swap] {$\coev_Y$} (m-2-3)
(m-1-2) edge node[auto] {$\coev_Y$} (m-1-3)
 edge[bend left=20] node[auto] {$f^*$} (m-1-5)
(m-1-3) edge node[auto] {$f^L$} (m-1-4)
(m-1-4) edge node[auto] {$\ev_X$} (m-1-5)
(m-1-5) edge node[auto] {$f$} (m-1-6)
(m-2-3) edge node[auto] {$\coev_X$} (m-1-3)
 edge node[auto] {$f^L$} (m-2-4)
(m-2-4) edge node[auto] {$\coev_X$} (m-1-4)
 edge node[auto, swap] {$\id$} (m-1-5);
\end{tikzpicture}
$$
This allows to define
$$\coev_f= (\id_{Y^*} \otimes u)\circ \coev_Y: \coev_Y \Longrightarrow (\id_{Y^*} \otimes (f\circ f^L))\circ \coev_Y = (f^*\otimes f)\circ \coev_X.$$

Similarly,
$$\ev_Y\circ(f\otimes f^*) = \ev_X \circ ((f^L\circ f)\otimes \id_{X^*}),$$
because of the commutativity of
$$
\begin{tikzpicture}
\matrix(m)[matrix of math nodes,
row sep=2.6em, column sep=2.4em,
text height=1.5ex, text depth=0.25ex]
{X\otimes X^* & Y\otimes X^* & Y\otimes Y^* \otimes Y\otimes X^* &Y\otimes Y^* \otimes X\otimes X^* & Y\otimes Y^* & \unit\\
&&Y\otimes X^* &X\otimes X^*\\};
\path[->,font=\scriptsize]
(m-1-1) edge node[auto] {$f$} (m-1-2)
(m-1-2) edge node[auto, swap] {$\id$} (m-2-3)
(m-1-2) edge node[auto] {$\coev_Y$} (m-1-3)
 edge[bend left=20] node[auto] {$f^*$} (m-1-5)
(m-1-3) edge node[auto] {$f^L$} (m-1-4)
edge node[auto] {$\ev_Y$} (m-2-3)
(m-1-4) edge node[auto] {$\ev_X$} (m-1-5)
(m-1-5) edge node[auto] {$\ev_Y$} (m-1-6)
(m-2-3) edge node[auto] {$f^L$} (m-2-4)
(m-1-4) edge node[auto] {$\ev_Y$} (m-2-4)
 (m-2-4) edge node[auto, swap] {$\ev_X$} (m-1-5);
\end{tikzpicture}
$$
This allows to define
$$\ev_f = \ev_X\circ(v\otimes \id_{X^*}): \ev_Y\circ(f\otimes f^*) = \ev_X \circ ((f^L\circ f)\otimes \id_{X^*}) \Longrightarrow \ev_X.$$
A diagram chase, which we leave to the reader, shows that $\ev_f$ and $\coev_f$ satisfy the snake relations using the fact that $u$ and $v$ form the unit and counit of an adjunction.
\end{proof}

In order to establish the induction step in \cref{thm.topological higher twisted field theories}, we first simplify the conditions we need to check:

\begin{lemma}\label{lemma simplifying adjunctibility}
  Let $\cC$ be an $(\infty,n)$-category and $f : X \to Y$ a $k$-morphism in $\cC$.  The following are equivalent:
  \begin{enumerate}
    \item \label{item1.lemma simplifying adjunctibility} $f$ is adjunctible and all the evaluation and coevaluation $(k+1)$-morphisms witnessing adjunctibility are themselves adjunctible.
    \item \label{item2.lemma simplifying adjunctibility} $f$ is left-adjunctible and both $(k+1)$-morphisms $\ev_f : f^L \circ f \to \id_X$ and $\coev_f : \id_Y \to f \circ f^L$ are both left- and right-adjunctible.
    \item \label{item3.lemma simplifying adjunctibility} $f$ is both left- and right-adjunctible and all four $(k+1)$-morphisms $\ev_f : f^L \circ f \to \id_X$, $\coev_f : \id_X \to f \circ f^L$, $\ev_{f^R} : f \circ f^R \to \id_Y$, and $\coev_{f^R} : \id_X \to f^R \circ f$ are left-adjunctible.
  \end{enumerate}
The same statement hold with the words ``left'' and ``right'' reversed.
\end{lemma}

\begin{proof}
  That (\ref{item1.lemma simplifying adjunctibility}) implies both (\ref{item2.lemma simplifying adjunctibility}) and (\ref{item3.lemma simplifying adjunctibility}) is clear.  
  
  Suppose that $f$ satisfies (\ref{item2.lemma simplifying adjunctibility}).  Then the left adjoints $\ev_f^L : \id_X \to f^L \circ f$ and $\coev_f^L : f \circ f^L \to \id_Y$ together witness $f^L$ as the \emph{right} adjoint to $f$, see \cite[Remark 3.4.22]{Lur09} and \cite[Lemma 2.4.4]{DSPS1}. Similarly, the right adjoints $\ev_f^R : \id_X \to f^L \circ f$ and $\coev_f^R : f \circ f^L \to \id_Y$ also witness $f^L$ as a right adjoint to $f$. 

It follows that there is an automorphism $S$ of $f^L$ formed by composing $\ev_f^R\times \id_{f^L}$ with $\id_{f^L}\times\coev_f^L$; its inverse is formed by composing $\ev_f^L\times \id_{f^L}$ with $\id_{f^L}\times\coev_f^R$. It intertwines $\ev_f^L\times \id_{f^L\circ f}$ with $\ev_f^R\times\id_{f^L\circ f}$ and $\id_{f^L}\times \coev_f^L\times \id_{f^L}$ and $\id_{f^L}\times \coev_f^R\times \id_f$:
$$
\begin{tikzpicture}[anchor=base]
  \path
   node (a) {$S\times \id_f: \quad f^L\circ f$}
   ++(6.5cm,0) node (b) {$(f^L\circ f)\circ (f^L\circ f)$}
   ++(6.5cm,0) node (c) {$f^L\circ f \quad :S^{-1} \times \id_f.$};
  \path 
   (a.east) +(0,2pt) coordinate (ane)
   (a.east) +(0,-2pt) coordinate (ase)
   (b.west |- ane) coordinate (bnw)
   (b.west |- ase) coordinate (bsw)
   (b.east |- ane) coordinate (bne)
   (b.east |- ase) coordinate (bse)
   (c.west |- ane) coordinate (cnw)
   (c.west |- ase) coordinate (csw)
  ;
  \draw[->] (ane) -- node[auto] {$\scriptstyle \ev_f^L\times \id_{f^L\circ f}$} (bnw);
  \draw[->] (bsw) -- node[auto] {$\scriptstyle \id_{f^L}\times \coev_f^L\times \id_{f^L}$} (ase);
  \draw[->] (bne) -- node[auto] {$\scriptstyle \id_{f^L}\times \coev_f^R\times \id_f$} (cnw);
  \draw[->] (csw) -- node[auto] {$\scriptstyle \ev_f^R\times\id_{f^L\circ f}$} (bse);
\end{tikzpicture}
$$
Invertible morphisms are adjunctible, so $S$ and $S^{-1}$ are adjunctible. Since compositions of adjunctible morphisms are adjunctible and $\ev_f^L$ is right-adjunctible by construction, $\ev_f^R$ is right-adjunctible. Similarly, $\ev_f^L$ is left-adjunctible. Iterating gives adjunctibility of $\ev_f$, and a similar argument gives adjunctibility of $\coev_f$ and of the other evaluation and coevaluation morphisms.  This finishes the proof that (\ref{item2.lemma simplifying adjunctibility}) implies (\ref{item1.lemma simplifying adjunctibility}).
  
  Finally, assume that $f$ satisfies (\ref{item3.lemma simplifying adjunctibility}).  Similarly to the previous case, the left adjoints $\ev_{f^R}^L : \id_Y \to f \circ f^R$ and $\coev_{f^R}^L : f^R \circ f \to \id_X$ establish $f^R$ as a left adjoint to $f$.  It follows that there is an isomorphism $f^R \cong f^L$ intertwining $\ev_{f^R}^L$ with $\coev_f$ and $\coev_{f^R}^L$ with $\ev_f$.  But $\ev_{f^R}^L$ and $\coev_{f^R}^L$ are right-adjunctible.  This establishes that (\ref{item3.lemma simplifying adjunctibility}) implies (\ref{item2.lemma simplifying adjunctibility}) and completes the proof.
\end{proof}

\begin{corollary} \label{corollary.left is enough}
  Let $X$ be an object in a symmetric monoidal $(\infty,N)$-category. The following are equivalent:
\begin{enumerate}
\item $X$ is $n$-dualizable.
\item $X$ is 1-dualizable and the evaluation and coevaluation morphisms $\ev_X : X^* \otimes X \to \unit$ and $\coev_X : \unit \to X \otimes X^*$ are $(n-1)$-times left-adjunctible.
\item $X$ is 1-dualizable and the evaluation and coevaluation morphisms $\ev_X : X^* \otimes X \to \unit$ and $\coev_X : \unit \to X \otimes X^*$ are $(n-1)$-times right-adjunctible.\qedhere
\end{enumerate} 
\end{corollary}

We turn now to $n$-dualizability in the categories $\cC^{\boxasterisk}_{(j)}$ of lax, respectively oplax, $j$-morphisms in an $(\infty,n)$-category $\cC$. For simplicity, we focus on the \emph{lax} case, and thus left-adjunctibility; the oplax case works similarly.

Recall that an $i$-morphism in $\cC^{\smallbox{lax}}_{(j)}$ is by definition an element of $\cC^\Box_{(i);(j)}$.  Given $f \in \cC^\Box_{(i);(j)}$, its horizontal source and target $s_hf,t_hf \in \cC^\Box_{(i-1);(j)}$ are its source and target $(i-1)$-morphisms in $\cC^{\smallbox{lax}}_{(j)}$ and its vertical source and target $s_vf,t_vf \in \cC^\Box_{(i);(j-1)}$ are $i$-morphisms in $\cC^{\smallbox{lax}}_{(j-1)}$.  
We will write $f^\#$ for the ``bulk'' $(i+j)$-morphism in $\cC$ that fills in the ``box'' with sides $s_hf$, $t_hf$, $s_vf$, and $t_vf$.  Since its source and target depend on the parity of $j$, we unify notation in both cases: let $a = s_vf$ when $j$ is even and $a = t_vf$ when $j$ is odd, and let $b = t_vf$ when $j$ is even and $b = s_vf$ with $j$ is odd, and let $s = s_hf$ and $t = t_hf$; then \cref{prop.i-by-j-morphism} says that there are whiskerings $a^\natural$, $b^\natural$, $s^\natural$, and $t^\natural$ of the ``bulk'' $(i+j-1)$-morphisms $a^\#,b^\#,s^\#,t^\#$ such that 
$$f^\# : a^\natural \circ s^\natural \Rightarrow t^\natural \circ b^\natural.$$

\begin{proposition} \label{lax left adjunctibility}
  Let $f \in \cC^\Box_{(i);(j)}$ with notation as above.  Then $f$ is left-adjunctible as an $i$-morphism in $\cC^{\smallbox{lax}}_{(j)}$ if and only if 
  $a$ and $b$ are left-adjunctible as $i$-morphisms in $\cC^{\smallbox{lax}}_{(j-1)}$
  and $f^\#$ is left-adjunctible as an $(i+j)$-morphism in $\cC$.  Moreover, letting $f^L$ denote the left adjoint of $f$ in $\cC^{\smallbox{lax}}_{(j)}$, the $(i+j)$-morphisms $(f^L)^\#$ and $(f^\#)^L$ in $\cC$ are \define{mates} in the sense that they are related by units and counits of the adjunctions for $a$ and $b$.
\end{proposition}

\begin{proof}
  The proof is almost the same as the proof of \cref{prop.1dualizability in C-down}, and so we give only the outline.  The functors $s_v,t_v : \cC^{\smallbox{lax}}_{(j)} \rightrightarrows \cC^{\smallbox{lax}}_{(j-1)}$ preserve left-adjunctibility, and so a necessary condition for left-adjunctibility of $f$ is left-adjunctibility of $a$ and $b$.  Any alleged adjoint $f^L$ of $f$ in $\cC^{\smallbox{lax}}_{(j)}$ necessarily has a ``bulk'' $(i+j)$-morphism
  $$ (f^L)^\# : (a^L)^\natural \circ t^{\larunat} \Rightarrow s^{\larunat} \circ (b^L)^\natural,$$
  where $s^{\larunat}$ and $t^{\larunat}$ are some whiskerings of $s^\#$ and $t^\#$, possibly different from $s^\natural$ and $t^\natural$.  For comparison, any alleged adjoint $(f^\#)^L$ of $f^\#$ in $\cC$ is a map
  $$ (f^\#)^L : t^\natural \circ b^\natural \Rightarrow a^\natural \circ s^\natural.$$  
  
  Whiskering is functorial, and so preserves left-adjunctibility.  Given any $g : t^\natural \circ b^\natural \Rightarrow a^\natural \circ s^\natural$, define its \define{mate}
  $$ g^\dagger : (a^\natural)^L \circ t^\natural \Rightarrow s^\natural \circ (b^\natural)^L $$
  as the composition
  $$ (a^\natural)^L \circ t^\natural \xRightarrow{\coev_{b^\natural}} (a^\natural)^L \circ t^\natural \circ b^\natural \circ (b^\natural)^L \xRightarrow{g}  (a^\natural)^L \circ a^\natural \circ s^\natural \circ (b^\natural)^L \xRightarrow{\ev_{a^\natural}} s^\natural \circ (b^\natural)^L. $$
  Being functorial, whiskering also preserves mates.  By induction, $(a^L)^\natural$, a whiskering of $(a^L)^\#$, is a mate of a whiskering of $(a^\#)^L$.  It follows that, up to whiskering, $(a^L)^\natural$ is a mate of $(a^\natural)^L$, and $(b^L)^\natural$ is similarly a mate of $(b^\natural)^L$.  Therefore we can whisker $g^\dagger$ into a map
  $$ g^\ddagger : (a^L)^\natural \circ t^{\larunat} \Rightarrow s^{\larunat} \circ (b^L)^\natural.$$
  Conversely, given such a $g^\ddagger$, we can whisker it to a map $(a^\natural)^L \circ t^\natural \Rightarrow s^\natural \circ (b^\natural)^L$ and take its mate to restore $g$.
  
  We claim that $(f^L)^\#$ and $(f^\#)^L$ are {mates} in the sense that
  $$ ((f^\#)^L)^\ddagger = (f^L)^\#,$$
  by which we mean in particular that given one of $(f^\#)^L$ and $(f^L)^\#$, we can define the other as the above composition of mates and whiskerings.
  The evaluation and coevaluation morphisms are defined by dragging those from one side through the mates-and-whiskerings to the other side.  Checking the appropriate relations is then a diagram chase no harder than those in the proof of \cref{prop.1dualizability in C-down} and is left to the reader.
\end{proof}

\begin{remark}\label{oplax handedness}
  For comparison, right-adjunctibility of an $i$-morphism in $\cC^{\smallbox{lax}}_{(j)}$ is not the same as right adjunctibility of the bulk $(i+j)$-morphism in $\cC$, but rather as right-adjunctibility of a certain mate of the bulk.  In the oplax case, however, an $i$-morphism in $\cC^{\smallbox{oplax}}_{(j)}$ is right-adjunctible if and only if its horizontal source and target are right-adjunctible in $\cC^{\smallbox{oplax}}_{(j-1)}$ and its bulk $(i+j)$-morphism is right-adjunctible in $\cC$.
\end{remark}

\Subsection{Comparison with boundary field theories}

Lurie's paper \cite{Lur09} on the Cobordism Hypothesis focuses mostly on categories ``with duals.''  By definition, an $(\infty,n)$-category \define{has duals} if every object is $n$-dualizable and every $k$-morphism for $k\geq 1$ is $(n-k)$-times adjunctible.  Given an arbitrary symmetric monoidal $(\infty,n)$-category $\cC$, the $(\infty,n)$-category $\cC^\fd$ denotes the maximal subcategory of $\cC$ with duals.

 In Section 4.3 of \cite{Lur09}, Lurie describes a category $\cat{Bord}_{n+1}^{\fr,\partial}$ of fully extended framed bordisms with ``free boundaries'' along with an embedding $\cat{Bord}_{n+1}^{\fr} \hookrightarrow \cat{Bord}_{n+1}^{\fr,\partial}$.
  Given an $(n+1)$-dimensional framed topological field theory $Z : \cat{Bord}_{n+1}^{\fr} \to \cC$, a \define{boundary field theory for $Z$} is an extension of $Z$ to a symmetric monoidal functor $\bar Z : \cat{Bord}_{n+1}^{\fr,\partial} \to \cC$.  The field theory $Z$ plays the same role for boundary field theories as a twist plays for twisted field theories.
  \Cref{cor.lax twisted framed tfts} has the following corollary, providing a comparison between lax twisted field theories and boundary theories. 

\begin{theorem}\label{thm.comparison bdry and twisted}
Let $\cat{Bord} = \cat{Bord}_n^{\fr}$ denote the fully extended framed topological bordism category from \cite{Lur09,ClaudiaDamien1} and $\cat{Bord}_{n+1}^{\fr,\partial}$ the fully extended framed bordism category with ``free boundaries'' described in Section 4.3 in \cite{Lur09}.  Let $\cC$ be a symmetric monoidal $(\infty,n+1)$-category with duals.  The following are equivalent:
\begin{enumerate}
  \item \label{thm.comparison bdry and twisted.item1} 1-morphisms in $\cC$ with source $\unit$,
  \item \label{thm.comparison bdry and twisted.item2} fully extended $n$-dimensional {boundary} field theories $\cat{Bord}_{n+1}^\partial \to \cC$,
  \item \label{thm.comparison bdry and twisted.item3} lax twisted field theories with trivial source $\cat{Bord}_n\to \cC^{\unit\downarrow}$,
  \item \label{thm.comparison bdry and twisted.item4} oplax twisted field theories with trivial source $\cat{Bord}_n\to \cC^{\unit\rightarrow}$.
\end{enumerate}
\end{theorem}

\begin{proof}
  The equivalence between (\ref{thm.comparison bdry and twisted.item1}) and (\ref{thm.comparison bdry and twisted.item2}) is the first example of Lurie's \define{Cobordism Hypothesis for Manifolds with Singularities}: assuming $\cC$ has duals, given a fully extended framed topological field theory $Z: \cat{Bord}_{n+1}^{\fr} \to \cC$, the data of an extension to a boundary theory $\bar{Z}: \cat{Bord}_{n+1}^{\fr,\partial} \to \cC$ is equivalent to the data of a morphism $\unit \to Z(\mathrm{pt})$.
  The equivalence between (\ref{thm.comparison bdry and twisted.item1}) and (\ref{thm.comparison bdry and twisted.item3}) follows from \cref{cor.lax twisted framed tfts} and the observation that if $\cC$ is an $(\infty,n+1)$-category with duals, then in particular every object is $n$-dualizable and every 1-morphism is $n$-times left-adjunctible.  For the equivalence between (\ref{thm.comparison bdry and twisted.item1}) and (\ref{thm.comparison bdry and twisted.item4}), one follows  the same arguments as above except consistently substituting ``right'' for ``left'' and ``oplax'' for ``lax.''
\end{proof}

\begin{remark}
  The Cobordism Hypothesis for Manifolds with Singularities furthermore describes functors from fully-extended cobordism categories with certain types of higher-codimension defects.  \Cref{cor.lax twisted framed tfts} can also be used to compare those theories with {higher} twisted theories in the case when the target category has enough duals.  
  
  The most interesting case is when $\cC$ is not required to have all duals.  Then to have a boundary field theory $\bar Z : \cat{Bord}_{n+1}^{\fr,\partial} \to \cC$ still requires that the ``bulk'' theory $Z : \cat{Bord}_{n+1}^{\fr} \to \cC$ assigns to a point an $(n+1)$-dualizable object $Z(\mathrm{pt}) \in \cC$.  On the other hand, in a lax twisted field theory, the twist $T:\cat{Bord}_n\to\cC$ only requires that $T(\mathrm{pt})$ be $n$-dualizable.  Compare \cref{comment about Cfd}.
\end{remark}

\section{The even higher Morita category of \texorpdfstring{$E_d$}{E_d}-algebras}\label{section.evenhigher}

In this section, as an application of (op)lax homomorphisms, we construct an ``even higher'' version of the 
``higher (pointed) Morita'' $(\infty,d)$-category $\Alg_d(\cC)$ of $E_d$-algebras in a given symmetric monoidal $(\infty,n)$-category~$\cC$. Its objects are $E_d$-algebras in $\cC$, its 1-morphisms are (pointed) bimodules in $ E_{d-1}$-algebras in $\cC$, 2-morphisms are (pointed) bimodules of (pointed) bimodules, etc.,\ up to its $d$-morphisms, which are (pointed) bimodules of (pointed) bimodules \dots of (pointed) bimodules. 

We will extend the $(\infty,d)$-category $\Alg_d(\cC)$ to an $(\infty,d+n)$-category whose $k$-morphisms for $k>d$ are $(k-d)$-morphisms in $\cC$ between such (pointed) bimodules.   
 As in \cref{defn.oplax-homomorphism}, in addition to strong morphisms between algebraic structures in higher categories, there are also lax and oplax morphisms.
 As such, we will build three separate ``even higher'' Morita $(\infty,d+n)$-categories, which we will call $\Alg_d^{\mathrm{lax}}(\cC)$, $\Alg_d^{\mathrm{oplax}}(\cC)$, and $\Alg_d^{\mathrm{strong}}(\cC)$.
 The main example, described in \cref{eg.Alg2Rex}, is the $(\infty,4)$-category $\Alg_2^{\mathrm{strong}}(\cat{Rex}_\bK)$ controlling the Morita theory of braided monoidal finitely-cocomplete $\bK$-linear categories.

An unpointed version of $\Alg_d(\cC)$ was constructed in \cite{HaugsengEn} using $\infty$-operadic methods. A pointed version was constructed in \cite{ClaudiaDamien2} using  more geometric tools, namely factorization algebras.  
Our generalization applies to both versions.
The two constructions seem to be closely related and a comparison statement is the subject of an ongoing project. 
We will briefly recall outlines of both versions later in this section.
  We will need only certain formal properties from the constructions, and so will not recall all details. 
  Both constructions need certain technical conditions which we recall here in order to state the results.

\begin{definition}\label{defn.sifted}
\begin{enumerate}

    \item A \define{geometric realization} is a colimit of a simplicial object, i.e.\ a colimit of a diagram indexed by the category $\Delta^\op$.  
An $(\infty,1)$-category is \define{$\otimes$-GR-cocomplete} if it is symmetric monoidal, has geometric realizations, and if the symmetric monoidal structure distributes over geometric realizations in each variable.  A functor between $\otimes$-sifted-cocomplete $(\infty, 1)$-categories is \define{$\otimes$-GR-cocontinuous} if it is symmetric monoidal and preserves geometric realizations.
    
    We will call a symmetric monoidal $(\infty,n)$-category $\cC_{\vec\bullet}$ ``$\otimes$-GR-cocomplete'' if its underlying $(\infty,1)$-category $\cC_{\bullet,0,\dots,0}$ is.
    
    \item   A \define{sifted colimit} is
  a colimit over a nonempty diagram $D$ such that the diagonal functor $D \to D \times D$ is final (see \cite[Definition 5.5.8.1]{LurieHTT} for details).
    An $(\infty,1)$-category is \define{$\otimes$-sifted-cocomplete} if it is symmetric monoidal, has sifted colimits, and if the symmetric monoidal structure distributes over sifted colimits in each variable.  A functor between $\otimes$-sifted-cocomplete $(\infty, 1)$-categories is \define{$\otimes$-sifted-cocontinuous} if it is symmetric monoidal and preserves sifted colimits.
    
    We will call a symmetric monoidal $(\infty,n)$-category $\cC_{\vec\bullet}$ ``$\otimes$-sifted-cocomplete'' if its underlying $(\infty,1)$-category $\cC_{\bullet,0,\dots,0}$ is.

\end{enumerate}
\end{definition}

Since $\Delta^\op$ is sifted,    geometric realizations are special cases of sifted colimits. Some readers might know geometric realizations from their strict-categorical counterparts, which are called \define{reflexive coequalizers}.  
A well-developed theory of colimits within $(\infty,1)$-categories exists \cite{LurieHTT,RV2014} but consistently uses quasicategories to model $(\infty,1)$-categories, rather than complete Segal spaces.  These models are known to be equivalent \cite{MR2342834}, and in this section we will use them interchangeably.

The following unifies the main results of \cite{HaugsengEn, ClaudiaDamien2}:

\begin{theorem} \label{CS and Haugseng theorem} 
Let $\cC$ be a symmetric monoidal $(\infty,1)$-category.
\begin{enumerate}
  \item (\cite{HaugsengEn})   If $\cC$ is $\otimes$-GR-cocomplete,
  then there is a symmetric monoidal $d$-fold Segal space $\Alg_d(\cC)_{\vec \bullet}$ (which is not complete) whose objects are $E_d$-algebras in $\cC$, 1-morphisms are unpointed bimodules in $\mathrm E_{d-1}$-algebras in $\cC$, \dots, and $d$-morphisms are  unpointed bimodules of pointed bimodules \dots of pointed bimodules. The assignment $\cC \mapsto \Alg_d(\cC)_{\vec\bullet}$ is functorial for $\otimes$-GR-cocontinuous functors. 
  \item (\cite{ClaudiaDamien2}) 
  If $\cC$ is $\otimes$-sifted-cocomplete,
  then there is a symmetric monoidal complete $d$-fold Segal space i.e.~a symmetric monoidal $(\infty,d)$-category $\Alg_d(\cC)_{\vec \bullet}$ whose objects are $E_d$-algebras in $\cC$, 1-morphisms are pointed bimodules in $\mathrm E_{d-1}$-algebras in $\cC$, \dots, and $d$-morphisms are  pointed bimodules of pointed bimodules \dots of pointed bimodules.  The assignment $\cC \mapsto \Alg_d(\cC)_{\vec\bullet}$ is functorial for $\otimes$-sifted-cocontinuous functors.
  \qedhere
\end{enumerate}
\end{theorem}

Both papers \cite{HaugsengEn, ClaudiaDamien2} use the name ``$\Alg_d$'' for their constructions, even though the two constructions are different.  In this article we will not try to give them different names, since our generalization applies to both.  The construction given in \cite{ClaudiaDamien2} naturally produces a complete $d$-fold Segal space.  The construction in  \cite{HaugsengEn} produces a $d$-fold Segal space which is not complete; the ``higher Morita category'' from \cite{HaugsengEn} is its completion to a complete $d$-fold Segal space.  We will let $\Alg_d$ denote the the uncompleted version, as it's the one we will need for our construction. 

In both \cite{HaugsengEn, ClaudiaDamien2}, a stronger statement holds: the (complete) $d$-fold Segal space $\Alg_d(\cC)_{\vec \bullet}$ is naturally the ``object'' part of a (complete) $d$-fold Segal object internal to complete Segal spaces $\Alg_d(\cC)_{\vec \bullet;\bullet}$, whose morphisms in the latter direction are homomorphisms of bimodules.  This gives (after completion, in the case of \cite{HaugsengEn}) an $(\infty,d)$-by-$(\infty,1)$ ``double'' category whose underlying $(\infty,d+1)$-category from \cref{underlying n-fold segal space} is the ``higher Morita category of $E_d$-algebras in $\cC$.''

Our strategy for generalizing to the case when $\cC$ is an $(\infty,n)$-category is to apply the $\Alg_d(-)$ construction to the double $(\infty,n)$-category $\cC^\boxasterisk$ from \cref{defn.(op)lax arrow category}, for $\ast = $ ``lax,'' ``oplax,'' or ``strong.''  More precisely, since $\Alg_d(-)$ inputs an $(\infty,1)$-category, we use for each $\vec l$ only the underlying $(\infty,1)$-category of the $(\infty,n)$-category $\cC^\boxasterisk_{\vec l}$.  The motivation for this choice comes from \cref{defn.oplax-homomorphism}: a {lax morphism} of some type of structure in $\cC$ is that type of structure in $\cC^\downarrow = \cC^{\smallbox{lax}}_{(1)}$, an {oplax morphism} is that type of structure in $\cC^\rightarrow = \cC^{\smallbox{oplax}}_{(1)}$, and a {strong morphism} of some type is that type of structure in $[\Theta^{(1)}, \cC] = \cC^{\smallbox{strong}}_{(1)}$.

However, since $\Alg_d(-)$ is not functorial for arbitrary functors, the spaces $\Alg_d(\cC^\boxasterisk_{\vec \bullet})_{\vec k}$ may fail to be an $n$-uple simplicial space, let alone an $(\infty,n)$-category.  Instead, we must ask that $\cC^\boxasterisk_{\vec \bullet}$ enjoy one further property:

\begin{definition}
  Let $\cC$ be a symmetric monoidal $(\infty,n)$-category and $\ast = $ ``lax,'' ``oplax,'' or ``strong.'' 
  \begin{enumerate}
  \item
   We will say that $\cC^\boxasterisk$ is \define{\ensuremath\otimes-GR-cocomplete} if it is an $n$-fold simplicial diagram not just of symmetric monoidal categories and symmetric monoidal functors, but of $\otimes$-GR-cocomplete categories and $\otimes$-GR-cocontinuous functors.  I.e.\ $\cC^\boxasterisk$ is {\ensuremath\otimes-GR-cocomplete} if {and only if} every $\cC^\boxasterisk_{\vec l}$ is $\otimes$-GR-cocomplete and  all face and degeneracy functors preserve geometric realizations.
  \item
   We will say that $\cC^\boxasterisk$ is \define{\ensuremath\otimes-sifted-cocomplete} if it is an $n$-fold simplicial diagram not just of symmetric monoidal categories and symmetric monoidal functors, but of $\otimes$-sifted-cocomplete categories and $\otimes$-sifted-cocontinuous functors.  I.e.\ $\cC^\boxasterisk$ is {\ensuremath\otimes-sifted-cocomplete} if {and only if} every $\cC^\boxasterisk_{\vec l}$ is $\otimes$-sifted-cocomplete and  all face and degeneracy functors preserve sifted colimits.
   \end{enumerate}
\end{definition}

\begin{remark}\label{complete n-fold segality among sifted-cocomplete functors}
By \cite[Lemma 5.4.5.5]{LurieHTT}, for any collection $X$ of shapes of colimits, fiber products of categories which each have colimits of shape $X$ along functors that preserve colimits of shape $X$ again have colimits of shape $X$.  It follows that if $\cC^\boxasterisk$ is \ensuremath\otimes-sifted-cocomplete (resp.\ \ensuremath\otimes-GR-cocomplete), then \cref{mainthm,prop.c box is symmetric monoidal} imply in fact that $\cC^\boxasterisk_{\vec\bullet}$ is a complete $n$-fold Segal object in the $(\infty,1)$-category of $\otimes$-sifted-cocomplete (resp.\ \ensuremath\otimes-GR-cocomplete) categories and $\otimes$-sifted-cocontinuous  (resp.\ \ensuremath\otimes-GR-cocontinuous) functors.
\end{remark}

The functoriality results from \cref{CS and Haugseng theorem} then say the following: if $\cC^\boxasterisk$ is $\otimes$-sifted-cocomplete and ``$\Alg_d$'' is interpreted in the sense of \cite{ClaudiaDamien2}, then $(\vec k,\vec l)\mapsto\Alg_d(\cC^\boxasterisk_{\vec l})_{\vec k}$ is a $(d+n)$-fold simplicial space; if $\cC^\boxasterisk$ is $\otimes$-GR-cocomplete and ``$\Alg_d$'' is interpreted in the sense of \cite{HaugsengEn}, then $(\vec k,\vec l)\mapsto\Alg_d(\cC^\boxasterisk_{\vec l})_{\vec k}$ is a $(d+n)$-fold simplicial space.
Our main result of this section is that this $(d+n)$-uple simplicial space is an $(\infty,d)$-by-$(\infty,n)$ double category:

\begin{theorem}\label{thm.existence of even higher morita category}
  Let $\cC$ be a symmetric monoidal $(\infty,n)$-category and $\ast = $ ``lax,'' ``oplax,'' or ``strong.'' 
  \begin{enumerate}
    \item Suppose that $\cC^\boxasterisk$ is \ensuremath\otimes-sifted-cocomplete and that ``$\Alg_d$'' is interpreted in the sense of \cite{ClaudiaDamien2}.  Then the $(d+n)$-uple simplicial space $\Alg_d(\cC^\boxasterisk_{\vec\bullet})_{\vec\bullet}$ is a complete $d$-fold Segal object internal to complete $n$-fold Segal spaces.
    \item Suppose that $\cC^\boxasterisk$ is \ensuremath\otimes-GR-cocomplete and that ``$\Alg_d$'' is interpreted in the sense of \cite{HaugsengEn}.  Then the $(d+n)$-uple simplicial space $\Alg_d(\cC^\boxasterisk_{\vec\bullet})_{\vec\bullet}$ is a  $d$-fold Segal object internal to complete $n$-fold Segal spaces.
  \end{enumerate}
\end{theorem}

The complete proof of \cref{thm.existence of even higher morita category} will require reviewing a bit about the constructions from \cite{HaugsengEn,ClaudiaDamien2}, which we defer to the end of this section.  But most of the proof follows formally from the theorems already established:

\begin{proof}
  For fixed $\vec l$, \cref{CS and Haugseng theorem} implies that $\Alg_d(\cC^\boxasterisk_{\vec l })_{\vec\bullet}$ is a (complete, in the case of \cite{ClaudiaDamien2}) $d$-fold Segal space.  It thus suffices to prove that for fixed $\vec k$, $\Alg_d(\cC^\boxasterisk_{\vec\bullet})_{\vec k}$ is a complete $n$-fold Segal space.  Since $\cC^\boxasterisk_{\vec\bullet}$ is a complete $n$-fold Segal object by \cref{complete n-fold segality among sifted-cocomplete functors}, it suffices to prove that the functor $\Alg_d(-)_{\vec k}$ preserves fibered products.  For the version of $\Alg_d$ from \cite{ClaudiaDamien2}, we will prove this in \cref{end of the proof CS version}.  For the version of $\Alg_d$ from \cite{HaugsengEn}, we will prove this in \cref{end of the proof Haugseng version}.
\end{proof}

The theorem allows us to give our main definition of this section.
\begin{definition}\label{defn.even higher Morita}
  Let $\ast = $ ``lax,'' ``oplax,'' or ``strong'', and let $\cC$ be as in \cref{thm.existence of even higher morita category}.  The \define{even higher Morita category}  $\Alg_d^\ast(\cC)$ of $E_d$-algebras and $\ast$-morphisms in $\cC$ is the underlying $(\infty,d+n)$-category of the $(\infty,d)$-by-$(\infty,n)$ double category $ (\vec k;\vec l) \mapsto \Alg_d(\cC^\boxasterisk_{\vec l})_{\vec k} $.  
\end{definition}

\begin{remark}
  \Cref{defn.even higher Morita} is a slight abuse of notation when ``$\Alg_d$'' is interpreted in the sense of \cite{HaugsengEn}, as then for fixed $\vec l$ the $d$-fold Segal space $\Alg_d(\cC^\boxasterisk_{\vec l})_{\vec\bullet}$ is generally not complete.  To define the {even higher Morita category}  $\Alg_d^\ast(\cC)$ for the \cite{HaugsengEn}-version of $\Alg_d$, one must first complete $\Alg_d(\cC^\boxasterisk_{\vec\bullet})_{\vec\bullet}$ before taking its underlying $(\infty,d+n)$-category.
\end{remark}

\begin{example}\label{ex.Morita category}
To justify \cref{defn.even higher Morita}, assume that $\cC$ is an $(\infty,2)$-category such that $\cC^\boxasterisk$ is $\otimes$-sifted-cocomplete, where $\ast$ is one of ``lax,'' ``oplax,'' or ``strong.''  We will work out the $(\infty,1)$-by-$(\infty,1)$-category underlying the $(\infty,1)$-by-$(\infty,2)$-category $\Alg^\ast_1(\cC)$, where ``$\Alg_1$'' is understood in the sense of \cite{ClaudiaDamien2}. Note that if $\cC^\boxasterisk$ is $\otimes$-GR-cocomplete and ``$\Alg_1$'' is understood in the sense of \cite{HaugsengEn}, the same interpretations without the pointings hold.

Since $\cC_{\vec\bullet} = \cC^\Box_{(0); \vec\bullet} = \cC^\Box_{\vec\bullet; (0)} = [\Theta^0, \cC]_{\vec\bullet}$,
the objects of $\Alg_{1}^{\ast}(\cC)$  are elements in $\left(\Alg^\ast_1(\cC)\right)_{(0)} = \Alg_1(\cC)_0$, which are $E_1$-algebras in $\Cc$. For the same reason, the 1-morphisms in $\Alg_{1}^{\ast}(\cC)$ are elements in $\left(\Alg^\ast_1(\cC)\right)_{1,(0)} = \Alg_1(\cC)_{1}$, which are (pointed, in the case of \cite{ClaudiaDamien2}) bimodules in $\cC$.

Setting $\vec l = (1)$, the $(0,1,0)$-morphisms in the 3-uple Segal space are
$$\left(\Alg^\ast_1(\cC)\right)_{0,(1)} = \Alg_1\left(\cC^\boxasterisk_{(1)}\right)_0 = \begin{cases} \Alg_1(\cC^\downarrow)_0, &\ast =\mbox{lax}\\
\Alg_1(\cC^\to)_0, &\ast =\mbox{oplax}\\
\Alg_1([\Theta^{(1)},\cC])_0, &\ast =\mbox{strong}.
\end{cases}$$
An element of this space is an $E_1$-algebra object in $\cC^{\boxasterisk}_{(1)}$, i.e.\ an arrow $A \to B$ in $\cC$, thought of as an object in $\cC^{\boxasterisk}_{(1)}$, along with a coherently homotopy-associative homotopy-unital map $(A\to B) \otimes(A\to B) \to (A\to B)$ in $\cC^{\boxasterisk}_{(1)}$.  In the lax, oplax, and strong cases, such a map unpacks to a diagram of the form
$$
\begin{tikzpicture}[baseline=(label.base)]
  \path (0,2) node (AA) {$A \otimes A$} (2,2) node (BB) {$B \otimes B$} (0,0) node (A) {$A$} (2,0) node (B) {$B$};
  \draw[arrow] (AA) -- (A);
  \draw[arrow] (BB) -- (B);
  \draw[arrow] (AA) -- (BB);
  \draw[arrow] (A) -- (B);
  \path (1,-1) node[anchor=base] (label) {lax};
  \draw[twoarrow] (A) -- (BB);
\end{tikzpicture}
\quad\quad
\begin{tikzpicture}[baseline=(label.base)]
  \path (0,2) node (AA) {$A \otimes A$} (2,2) node (BB) {$B \otimes B$} (0,0) node (A) {$A$} (2,0) node (B) {$B$};
  \draw[arrow] (AA) -- (A);
  \draw[arrow] (BB) -- (B);
  \draw[arrow] (AA) -- (BB);
  \draw[arrow] (A) -- (B);
  \path (1,-1) node[anchor=base] (label) {oplax};
  \draw[twoarrow] (BB) -- (A);
\end{tikzpicture}
\quad\quad
\begin{tikzpicture}[baseline=(label.base)]
  \path (0,2) node (AA) {$A \otimes A$} (2,2) node (BB) {$B \otimes B$} (0,0) node (A) {$A$} (2,0) node (B) {$B$};
  \draw[arrow] (AA) -- (A);
  \draw[arrow] (BB) -- (B);
  \draw[arrow] (AA) -- (BB);
  \draw[arrow] (A) -- (B);
  \path (1,-1) node[anchor=base] (label) {strong};
  \draw[twoarrow] (A) -- coordinate (midpoint) (BB); \path (midpoint) ++(-5pt,5pt) node[rotate=45,anchor=mid]{$\sim$};
  \path[decoration={markings,mark=at position -8pt with {\arrow[scale=1.75]{>}}},preaction={decorate}] (BB) -- (A);
\end{tikzpicture}
$$
in $\cC$.  Said another way, $A$ and $B$ are $E_1$-algebra objects in $\cC$ and the arrow $A \to B$ is a lax/oplax/strong algebra morphism.

The 2-morphisms in $\Alg_{1}^{\ast}(\cC)$ are elements of
$$\left(\Alg^\ast_1(\cC)\right)_{1,(1)} = \Alg_1\left(\cC^\boxasterisk_{(1)}\right)_1 = \begin{cases} \Alg_1(\cC^\downarrow)_1, &\ast =\mbox{lax}\\
\Alg_1(\cC^\to)_1, &\ast =\mbox{oplax}\\
\Alg_1([\Theta^{(1)},\cC])_1, &\ast =\mbox{strong}
\end{cases}$$
We will spell this out explicitly for $\ast=\mbox{lax}$. An element in $\Alg_1(\cC^\downarrow)_1$ is a pointed bimodule in $\cC^\downarrow$. This data consists of:
\begin{itemize}
  \item an object in $\cC^{\downarrow}$, i.e.\ an arrow $M \to N$ in $\cC$;
  \item a pointing in $\cC^{\downarrow}$, i.e.\ a morphism $(\unit \to \unit) \to (M\to N)$, i.e.\ a diagram in $\cC$ of shape
$$
\begin{tikzpicture}[baseline=(id.base)]
  \path (0,2) node (unit) {$\unit$} (2,2) node (M) {$M$} (0,0) node (unit2) {$\unit$} (2,0) node (N) {$N$} (1,1) coordinate (center);
  \draw[arrow] (unit) -- node[auto,swap] (id) {$\scriptstyle \id$} (unit2);
  \draw[arrow] (M) -- (N);
  \draw[arrow] (unit) -- (M);
  \draw[arrow] (unit2) -- (N);
  \draw[twoarrow] (unit2) -- (M);
\end{tikzpicture}
\quad \text{or equivalently} \quad
\begin{tikzpicture}[baseline=(unit.base)]
  \path (0,1) node (unit) {$\unit$} (2,2) node (M) {$M$}  (2,0) node (N) {$N$} (1,1) coordinate (center);
  \draw[arrow] (M) -- (N);
  \draw[arrow] (unit) -- (M);
  \draw[arrow] (unit) -- coordinate (unit2) (N);
  \draw[twoarrow] (unit2) -- (M);
\end{tikzpicture}
;$$
  \item a bimodule structure on $(M \to N)$ between algebra objects $(A \to B)$ and $(C \to D)$.  The actions in such a bimodule are diagrams in $\cC$ of the form
  $$
\begin{tikzpicture}[baseline=(center)]
  \path (0,2) node (AA) {$A \otimes M$} (2,2) node (BB) {$B \otimes N$} (0,0) node (A) {$M$} (2,0) node (B) {$N$} (0,1) coordinate (center);
  \draw[arrow] (AA) -- (A);
  \draw[arrow] (BB) -- (B);
  \draw[arrow] (AA) -- (BB);
  \draw[arrow] (A) -- (B);
  \draw[twoarrow] (A) -- (BB);
\end{tikzpicture}
\quad\text{and}\quad
\begin{tikzpicture}[baseline=(center)]
  \path (0,2) node (AA) {$M \otimes C$} (2,2) node (BB) {$N \otimes D$} (0,0) node (A) {$M$} (2,0) node (B) {$N$} (0,1) coordinate (center);
  \draw[arrow] (AA) -- (A);
  \draw[arrow] (BB) -- (B);
  \draw[arrow] (AA) -- (BB);
  \draw[arrow] (A) -- (B);
  \draw[twoarrow] (A) -- (BB);
\end{tikzpicture}
  .$$
\end{itemize}
Repackaging this data, we see that $M$ is a pointed $(A,C)$-bimodules, $N$ is a pointed $(B,D)$-bimodule, and the the map $M\to N$ is a lax morphism of the pointed bimodules, intertwining with the lax morphisms $A\to B$ and $C\to D$.

Similarly, if we choose $\ast=\mbox{oplax}$ or $\ast=\mbox{strong}$, we have two pointed bimodules $M$ and $N$ together with an oplax or strong morphism $M \to N$ of bimodules.
\end{example}

\begin{example} \label{eg.Alg2Rex}
  Fix a ground commutative commutative ring $\bK$.  Let $\cat{Rex}_\bK$ denote the bicategory, thought of as an $(\infty,2)$-category, of small finitely-cocomplete $\bK$-linear  categories, finitely-cocontinuous $\bK$-linear functors, and natural transformations.  As observed in \cite[Proposition 3.10]{BZBJ}, $\cat{Rex}_\bK$ is $\otimes$-sifted-cocomplete.
  Using the well-developed theory of 2-colimits in a bicategory (e.g.\ \cite{MR0401868,Street1980,Kelly1989}), one can show moreover that $\cat{Rex}_\bK^{\smallbox{strong}}$ is $\otimes$-sifted-cocomplete.  (We will not spell out the argument; it uses the same ideas as in \cref{remark.strong cocomplete}, which in the case of 2-limits in bicategories are well established.)
  Using the unpointed version of ``$\Alg_2$'' from \cite{HaugsengEn}, one can therefore build the $(\infty,4)$-category $\Alg_2^{\mathrm{strong}}(\cat{Rex}_\bK)$ predicted in \cite[Conjecture 6.5]{BZBJ}.
\end{example}

\begin{example}\label{eg.Alg1Rex}
By the same argument as in the example above, we can use the unpointed version of ``$\Alg_1$'' from \cite{HaugsengEn} to build the $(\infty,3)$-category $\Alg_1^{\mathrm{strong}}(\cat{Rex}_\bK)$. The $(\infty,3)$-category $\cat{TC}$ from \cite{DSPS1,DSPS2} is expected to be a subcategory of this.
\end{example}

\begin{example}\label{multiple examples}
  Fix a ground commutative ring $\bK$.  Let $\cat{Pres}_\bK$ denote the bicategory, thought of as an $(\infty,2)$-category, of  locally presentable $\bK$-linear categories categories,   cocontinuous $\bK$-linear functors, and natural transformations.
  It contains all small limits and colimits~\cite{BirdThesis} and 
   is closed symmetric monoidal with the Kelly tensor product ``$\boxtimes_{\bK}$'' of linear cocomplete categories~ \cite{MR2177301}, and hence is $\otimes$-sifted-cocomplete.
  Again using the well-developed theory of 2-colimits, one can show directly that $\cat{Pres}_\bK^{\smallbox{strong}}$ is $\otimes$-sifted-cocomplete.
  
We will prove in this example, however, that $\cat{Pres}_{\bZ}^{\smallbox{oplax}}$ fails to be $\otimes$-GR-cocomplete.  Suppose that it were.  Then in particular $\cat{Pres}_{\bZ}^{\to}$ would be $\otimes$-GR-cocomplete, and the source and target functors $s,t: \cat{Pres}_{\bZ}^{\to} \to \cat{Pres}_{\bZ}$ would preserve geometric realizations.  Consider then the following geometric realization in $\cat{Pres}_{\bZ}^{\to}$ (we draw only the 2-truncation of the corresponding $\Delta^\op$-indexed diagram, since $\cat{Pres}_{\bZ}^{\to}$ is a bicategory):
$$ \operatorname{hocolim} \left(
\begin{tikzpicture}[baseline=(ML.base)]
  \path (0,0) node (LL) {$\cat{Mod}_\bZ$} +(6,0) node (LR) {$\cat{Mod}_\bZ$}
        (0,2) node (ML) {$\cat{Mod}_\bZ$} +(6,0) node (MR) {$\cat{Mod}_\bZ$}
        (0,4) node (UL) {$\cat{Mod}_\bZ$} +(6,0) node (UR) {$\cat{Mod}_\bZ$};
  \draw[arrow] (LL.north) -- node[fill=white,inner sep=1pt] {$\scriptstyle \bZ$} (ML.south); 
  \draw[arrow] (ML.south -| .25,0) -- node[fill=white,inner sep=1pt] {$\scriptstyle \bZ$} (LL.north -| .25,0); 
  \draw[arrow] (ML.south -| -.25,0) -- node[fill=white,inner sep=1pt] {$\scriptstyle \bZ$} (LL.north -| -.25,0); 
  \draw[arrow] (LR.north) -- node[fill=white,inner sep=1pt] {$\scriptstyle \bZ$} (MR.south); 
  \draw[arrow] (MR.south -| 6.25,0) -- node[fill=white,inner sep=1pt] {$\scriptstyle \bZ$} (LR.north -| 6.25,0); 
  \draw[arrow] (MR.south -| 5.75,0) -- node[fill=white,inner sep=1pt] {$\scriptstyle \bZ$} (LR.north -| 5.75,0); 
  \draw[arrow] (ML.north -| .25,0) -- node[fill=white,inner sep=1pt] {$\scriptstyle \bZ$} (UL.south -| .25,0);
  \draw[arrow] (ML.north -| -.25,0) -- node[fill=white,inner sep=1pt] {$\scriptstyle \bZ$} (UL.south -| -.25,0);
  \draw[arrow] (UL.south) -- node[fill=white,inner sep=1pt] {$\scriptstyle \bZ$} (ML.north);
  \draw[arrow] (UL.south -| -.5,0) -- node[fill=white,inner sep=1pt] {$\scriptstyle \bZ$} (ML.north -| -.5,0);
  \draw[arrow] (UL.south -| .5,0) -- node[fill=white,inner sep=1pt] {$\scriptstyle \bZ$} (ML.north -| .5,0);
  \draw[arrow] (MR.north -| 6.25,0) -- node[fill=white,inner sep=1pt] {$\scriptstyle \bZ$} (UR.south -| 6.25,0);
  \draw[arrow] (MR.north -| 5.75,0) -- node[fill=white,inner sep=1pt] {$\scriptstyle \bZ$} (UR.south -| 5.75,0);
  \draw[arrow] (UR.south) -- node[fill=white,inner sep=1pt] {$\scriptstyle \bZ$} (MR.north);
  \draw[arrow] (UR.south -| 5.5,0) -- node[fill=white,inner sep=1pt] {$\scriptstyle \bZ$} (MR.north -| 5.5,0);
  \draw[arrow] (UR.south -| 6.5,0) -- node[fill=white,inner sep=1pt] {$\scriptstyle \bZ$} (MR.north -| 6.5,0);
  \draw[arrow] (LL) -- node[fill=white,inner sep=1pt] {$\scriptstyle \bZ/4$} (LR);
  \draw[arrow] (ML) -- node[fill=white,inner sep=1pt] {$\scriptstyle \bZ/4 \oplus \bZ/2$} (MR);
  \draw[arrow] (UL) -- node[fill=white,inner sep=1pt] {$\scriptstyle (\bZ/4 \oplus \bZ/2) \oplus (\bZ/4 \oplus \bZ/2)$} (UR);
  \draw[twoarrowlonger] (LL.north -| 1,0) -- node[fill=white,inner sep=1pt] {$\scriptstyle (1,0)$} (MR.south -| 3,0);
  \draw[twoarrowlonger] (MR.south -| 4,0) -- node[fill=white,inner sep=1pt] {$\scriptstyle p_1$} (LL.north -| 2,0);
  \draw[twoarrowlonger] (LL.north -| 3,0) -- node[fill=white,inner sep=1pt] {$\scriptstyle (1,1)$} (MR.south -| 5,0);
  \draw[twoarrowlonger] (ML.north -| 1,0) -- node[fill=white,inner sep=1pt] {$\scriptstyle (\id,0)$} (UR.south -| 1.8,0);
  \draw[twoarrowlonger] (UR.south -| 2.6,0) -- node[fill=white,inner sep=1pt] {$\scriptstyle p_1$} (ML.north -| 1.8,0);
  \draw[twoarrowlonger] (ML.north -| 2.6,0) -- node[fill=white,inner sep=1pt] {$\scriptstyle (\id,\id)$} (UR.south -| 3.4,0);
  \draw[twoarrowlonger] (UR.south -| 4.2,0) -- node[fill=white,inner sep=1pt] {$\scriptstyle p_2$} (ML.north -| 3.4,0);
  \draw[twoarrowlonger] (ML.north -| 4.2,0) -- node[fill=white,inner sep=1pt] {$\scriptstyle (0,\id)$} (UR.south -| 5,0);
\end{tikzpicture}
  \right)
  $$
Here $\cat{Mod}_{\bZ}$ denotes the category of all abelian groups (which is indeed additive and locally presentable), and we have used the Eilenberg--Watts theorem to identitify the category of additive cocontinuous endofunctors of $\cat{Mod}_{\bZ}$ with $\cat{Mod}_{\bZ}$ itself (via $M \leftrightarrow (\otimes M)$).  By $p_i$ we mean projection onto the $i$th factor.  A map like $(1,1) : \bZ/4 \to \bZ/4 \oplus \bZ/2$ is the identity in the first factor and the quotient map $\bZ/4 \to \bZ/2$ in the second.  Of course, the arrows indexed by degeneracy maps in $\Delta$ do not effect the colimit, but their existence witnesses that in fact this is a geometric realization.

By assumption, the horizontal source and target maps $\cat{Pres}_\bZ^\to \rightrightarrows \cat{Pres}_\bZ$ preserve geometric realizations (in fact, one can prove that they preserve all colimits), and 
$$ \operatorname{hocolim} \left(
\begin{tikzpicture}[baseline=(ML.base)]
  \path (0,0) node (LL) {$\cat{Mod}_\bZ$} 
        (0,2) node (ML) {$\cat{Mod}_\bZ$} 
        (0,4) node (UL) {$\cat{Mod}_\bZ$}; 
  \draw[arrow] (LL.north) -- node[fill=white,inner sep=1pt] {$\scriptstyle \bZ$} (ML.south); 
  \draw[arrow] (ML.south -| .25,0) -- node[fill=white,inner sep=1pt] {$\scriptstyle \bZ$} (LL.north -| .25,0); 
  \draw[arrow] (ML.south -| -.25,0) -- node[fill=white,inner sep=1pt] {$\scriptstyle \bZ$} (LL.north -| -.25,0); 
  \draw[arrow] (ML.north -| .25,0) -- node[fill=white,inner sep=1pt] {$\scriptstyle \bZ$} (UL.south -| .25,0);
  \draw[arrow] (ML.north -| -.25,0) -- node[fill=white,inner sep=1pt] {$\scriptstyle \bZ$} (UL.south -| -.25,0);
  \draw[arrow] (UL.south) -- node[fill=white,inner sep=1pt] {$\scriptstyle \bZ$} (ML.north);
  \draw[arrow] (UL.south -| -.5,0) -- node[fill=white,inner sep=1pt] {$\scriptstyle \bZ$} (ML.north -| -.5,0);
  \draw[arrow] (UL.south -| .5,0) -- node[fill=white,inner sep=1pt] {$\scriptstyle \bZ$} (ML.north -| .5,0);
\end{tikzpicture}
  \right)
  \quad\simeq\quad \cat{Mod}_\bZ. $$
Thus the colimit in $\cat{Pres}_{\bZ}^{\to}$ is an object of the form $(\cat{Mod}_{\bZ} \overset M \to \cat{Mod}_{\bZ}) \in \cat{Pres}_{\bZ}^{\to}$, and it suffices to compute~$M$.  Unpacking the universal property verifies that
$$ M = \lim \left( \bZ/4 \overset{(1,0)}{\underset{(1,1)}\rightrightarrows} \bZ/4 \oplus \bZ/2 \right) \cong \bZ/2. $$
The colimit turned into a limit because the 2-cell involved in a 1-morphism $\cC^\to$ points ``the wrong way.''

Now, continuing to suppose that $\cat{Pres}_{\bZ}^{\to}$ is $\otimes$-GR-cocomplete, this colimit should commute with tensoring by arbitrary objects of  $\cat{Pres}_{\bZ}^{\to}$.  Consider tensoring with the object $(\cat{Mod}_{\bZ/2} \overset \id \to \cat{Mod}_{\bZ/2})$.  The same  calculations as above then show that
\begin{multline*} \operatorname{hocolim} \left(
\begin{tikzpicture}[baseline=(ML.base)]
  \path (0,0) node (LL) {$\cat{Mod}_{\bZ/2}$} +(6,0) node (LR) {$\cat{Mod}_{\bZ/2}$}
        (0,2) node (ML) {$\cat{Mod}_{\bZ/2}$} +(6,0) node (MR) {$\cat{Mod}_{\bZ/2}$}
        (0,4) node (UL) {$\cat{Mod}_{\bZ/2}$} +(6,0) node (UR) {$\cat{Mod}_{\bZ/2}$};
  \draw[arrow] (LL.north) --  (ML.south); 
  \draw[arrow] (ML.south -| .25,0) --  (LL.north -| .25,0); 
  \draw[arrow] (ML.south -| -.25,0) --  (LL.north -| -.25,0); 
  \draw[arrow] (LR.north) --  (MR.south); 
  \draw[arrow] (MR.south -| 6.25,0) --  (LR.north -| 6.25,0); 
  \draw[arrow] (MR.south -| 5.75,0) --  (LR.north -| 5.75,0); 
  \draw[arrow] (ML.north -| .25,0) --  (UL.south -| .25,0);
  \draw[arrow] (ML.north -| -.25,0) --  (UL.south -| -.25,0);
  \draw[arrow] (UL.south) --  (ML.north);
  \draw[arrow] (UL.south -| -.5,0) --  (ML.north -| -.5,0);
  \draw[arrow] (UL.south -| .5,0) --  (ML.north -| .5,0);
  \draw[arrow] (MR.north -| 6.25,0) --  (UR.south -| 6.25,0);
  \draw[arrow] (MR.north -| 5.75,0) --  (UR.south -| 5.75,0);
  \draw[arrow] (UR.south) --  (MR.north);
  \draw[arrow] (UR.south -| 5.5,0) --  (MR.north -| 5.5,0);
  \draw[arrow] (UR.south -| 6.5,0) --  (MR.north -| 6.5,0);
  \draw[arrow] (LL) -- node[fill=white,inner sep=1pt] {$\scriptstyle \bZ/2$} (LR);
  \draw[arrow] (ML) -- node[fill=white,inner sep=1pt] {$\scriptstyle \bZ/2 \oplus \bZ/2$} (MR);
  \draw[arrow] (UL) -- node[fill=white,inner sep=1pt] {$\scriptstyle (\bZ/2 \oplus \bZ/2) \oplus (\bZ/2 \oplus \bZ/2)$} (UR);
  \draw[twoarrowlonger] (LL.north -| 1,0) -- node[fill=white,inner sep=1pt] {$\scriptstyle (1,0)$} (MR.south -| 3,0);
  \draw[twoarrowlonger] (MR.south -| 4,0) -- node[fill=white,inner sep=1pt] {$\scriptstyle \pi_1$} (LL.north -| 2,0);
  \draw[twoarrowlonger] (LL.north -| 3,0) -- node[fill=white,inner sep=1pt] {$\scriptstyle (1,1)$} (MR.south -| 5,0);
  \draw[twoarrowlonger] (ML.north -| 1,0) -- node[fill=white,inner sep=1pt] {$\scriptstyle (\id,0)$} (UR.south -| 1.8,0);
  \draw[twoarrowlonger] (UR.south -| 2.6,0) -- node[fill=white,inner sep=1pt] {$\scriptstyle \pi_1$} (ML.north -| 1.8,0);
  \draw[twoarrowlonger] (ML.north -| 2.6,0) -- node[fill=white,inner sep=1pt] {$\scriptstyle (\id,\id)$} (UR.south -| 3.4,0);
  \draw[twoarrowlonger] (UR.south -| 4.2,0) -- node[fill=white,inner sep=1pt] {$\scriptstyle \pi_2$} (ML.north -| 3.4,0);
  \draw[twoarrowlonger] (ML.north -| 4.2,0) -- node[fill=white,inner sep=1pt] {$\scriptstyle (0,\id)$} (UR.south -| 5,0);
\end{tikzpicture}
  \right) \\
  \simeq \quad
\bigg(
  \begin{tikzpicture}[baseline=(L.base)]
  \path (0,0) node (L) {$\cat{Mod}_{\bZ/2}$} +(6,0) node (R) {$\cat{Mod}_{\bZ/2}$} ;
  \draw[arrow] (L) -- node[auto]  {$\lim \left( \bZ/2 \overset{(1,0)}{\underset{(1,1)}\rightrightarrows} \bZ/2 \oplus \bZ/2 \right)$} (R);
  \end{tikzpicture}
\bigg)
  \quad\simeq \quad
\bigg(
  \begin{tikzpicture}[baseline=(L.base)]
  \path (0,0) node (L) {$\cat{Mod}_{\bZ/2}$} +(3,0) node (R) {$\cat{Mod}_{\bZ/2}$} ;
  \draw[arrow] (L) -- node[auto]  {$0$} (R);
  \end{tikzpicture}\bigg).
  \end{multline*}
  But $(\cat{Mod}_\bZ \overset {\bZ/2} \to \cat{Mod}_\bZ) \boxtimes (\cat{Mod}_{\bZ/2} \overset {\id} \to \cat{Mod}_{\bZ/2}) \simeq (\cat{Mod}_{\bZ/2} \overset {\id} \to \cat{Mod}_{\bZ/2})$, not $(\cat{Mod}_{\bZ/2} \overset {0} \to \cat{Mod}_{\bZ/2})$.  Thus the symmetric monoidal structure in $\cat{Pres}_\bZ^\rightarrow$ does not distribute over geometric realizations, and so $\cat{Pres}_\bZ^\rightarrow$ is not $\otimes$-GR-cocomplete.
\end{example}

\begin{remark}\label{remark.strong cocomplete}
  Recall from \cref{defn.loop space} that $(\infty,n)$-categories are ``$(\infty,1)$-categories enriched in $(\infty,n-1)$-categories.''  Although much work on enriched $(\infty,1)$-categories exists (see for example \cite{MR3109865,BergnerRezk2014}, where it is shown that ``category enriched in $(\infty,n-1)$-categories'' can be taken as a definition for ``$(\infty,n)$-category'', \cite{MR3345192,MR3402334} for the general theory of ``enriched $(\infty,1)$-categories'', and \cite{GepnerHaugsengNikolaus} for weighted (co)limits), it seems that a complete accounting of limits and colimits in enriched $(\infty,1)$-categories does not yet exist.
  
  What should be true is the following.  Suppose that $\cC$ is a $\cV$-enriched $(\infty,1)$-category with underlying $(\infty,1)$-category $\underline \cC$.  Given a diagram in $\cC$, one can study its colimit in $\underline \cC$: the universal property of the colimit asserts an equivalence of spaces.  For example, \cref{defn.sifted} only uses such ``unenriched colimits.''  One can also study its ``enriched colimit'' in which the universal property holds not just at the level of spaces but also at the level of $\cV$-objects.  If an enriched colimit exists, then it is also the unenriched colimit, but it is possible that unenriched colimits  exist while their enriched versions  fail to exist.  For example, consider the bicategory $\rB^2\bN$ with one object $x$, one 1-morphism $\id_x$, and $\bN$-many 2-morphisms $\id_x \Rightarrow \id_x$ as an $(\infty,1)$-category enriched in $(\infty,1)$-categories.  Then $\underline{\rB^2\bN} = \{x\}$ is the trivial $(\infty,1)$-category with one object and so admits all limits and colimits (all of them being just $x$ itself), whereas $\rB^2\bN$ admits very few limits and colimits in the enriched sense.
  
  It is a standard fact that if $\cC$ is a strict 1-category which admits colimits of some shape $K$, then any functor category $[\Theta,\cC]$, for $\Theta$ an arbitrary strict 1-category, also admits colimits of shape $K$, and moreover the pullback functors $f^* : [\Theta',\cC] \to [\Theta,\cC]$ for functors $f : \Theta \to \Theta'$ preserve colimits of shape $K$.  
  The $(\infty,1)$-version of this fact is described in \cref{complete n-fold segality among sifted-cocomplete functors}.
  However, the example $\cC = \rB^2\bN$ and $\Theta = \Theta^{(2)}$ shows that this is false when extended to $(\infty,n)$-categories if colimits are taken in the unenriched sense.  But if an $(\infty,n)$-category $\cC$ admits colimits of shape $K$ in the sense of categories enriched in $(\infty,n-1)$-categories, then it should be true that all functor categories $[\Theta,\cC]$, for $\Theta$ an arbitrary $(\infty,n)$-category, also admit enriched $K$-shaped colimits, and that these are preserved by pullbacks.

  It would follow that if $\cC$ is $\otimes$-GR-cocomplete (resp.\ \ensuremath\otimes-sifted-cocomplete) in the enriched sense, then $\cC^{\smallbox{strong}}$ is automatically $\otimes$-GR-cocomplete (resp.\ \ensuremath\otimes-sifted-cocomplete).  In particular, $\Alg_d^{\mathrm{strong}}(\cC)$ would exists,  where ``$\Alg_d$'' is interpreted in the sense of \cite{HaugsengEn} (resp.\ \cite{ClaudiaDamien2}).  
\end{remark}

We turn now to proving \cref{end of the proof Haugseng version} and \cref{end of the proof CS version}, thereby completing the proof of \cref{thm.existence of even higher morita category}.  This requires recalling  the main ingredients of the two constructions of the higher Morita category $\Alg_d(\cC)$ from \cite{HaugsengEn, ClaudiaDamien2}.  We discuss the construction from \cite{HaugsengEn} first, and then turn to the construction from \cite{ClaudiaDamien2}.

\Subsection{The unpointed version from \cite{HaugsengEn}}

The paper \cite{HaugsengEn} uses ``generalized $\Delta^{\times d}$-$\infty$-operads'' and algebras over them, which are an extension of Lurie's machinery for algebras over $\infty$-operads developed in \cite{LurieHA}.  By design, (generalized) $\Delta^{\times d}$-$\infty$-operads are the appropriate notion of ``(generalized) operads'' for which one can form algebras over them  in $E_d$-monoidal categories.

By \cite[Definition 4.29]{HaugsengEn}, there is an adjunction between generalized $\Delta^{\times d}$-$\infty$-operads and (symmetric) $\infty$-operads given by the Segal functor. Moreover, by \cite[Section 2.2.4]{LurieHA}, every symmetric monoidal $(\infty,1)$-category has an underlying $\infty$-operad, and this forgetful map admits a left adjoint, the symmetric monoidal envelope. Piecing this together, we have an adjunction
$$\operatorname{Und} : \{\text{symmetric monoidal $(\infty,1)$-categories}\} \rightleftarrows \{\text{generalized $\Delta^{\times d}$-$\infty$-operad}\} : \operatorname{Env}.$$ 

The main step of the construction of $\cat{Alg}_d(-)$ in \cite{HaugsengEn} is a combinatorial definition of generalized $\Delta^{\times d}$-$\infty$-operads $\Delta^{\times d,\op}_{/(i)}$ for each $i\leq d$.  Given a symmetric monoidal $(\infty, 1)$-category $\cC$, the space of $i$-morphisms in $\cat{Alg}_d(\cC)$ is then defined to be
$$ \Alg_d(\cC)_{(i)} := \Alg^d_{\Delta^{\times d,\op}_{/(i)}}(\cC) := \maps^h_{\text{$\Delta^{\times d}$-$\infty$-\cat{Op}}}(\Delta^{\times d,\op}_{/(i)},\operatorname{Und}(\cC)).$$

\begin{remark}
  The space $\Alg_d(\cC)_{\vec k}$ of $\vec k$-tuples of composable morphisms given in \cite{HaugsengEn} is a bit more complicated, requiring the notion of ``composite $\Delta^{\times d,\op}_{/\vec k}$-algebra''; see Section 6.1 of  \cite{HaugsengEn} for details.
\end{remark}

\begin{proposition}\label{end of the proof Haugseng version}
For fixed $\vec k\in \mathbb{N}^d$, the functor $\cat{Alg}_d(-)_{\vec k}$ from $\cat{Cat}^\otimes$ to spaces as defined by \cite{HaugsengEn} preserves limits, and thus in particular fibered products.
\end{proposition}

\begin{proof}
  The main result of \cite{HaugsengEn} is that when $\cC$ is $\otimes$-GR-cocomplete, $\cat{Alg}_d(\cC)_{\vec \bullet}$ is an $n$-fold Segal space.  In particular, $\cat{Alg}_d(\cC)_{\vec k}$ is a limit of $\cat{Alg}_d(\cC)_{(i)}$s for various $i$s.  Since limits commute \cite[Lemma 5.5.2.3]{LurieHTT}, it suffices to check that $\cC \mapsto \Alg_d(\cC)_{(i)}$ preserves limits along $\otimes$-GR-cocontinuous functors.  

  By \cite[Lemma 5.4.5.5]{LurieHTT}, the inclusion of $\otimes$-GR-cocomplete $(\infty,1)$-categories among all symmetric monoidal $(\infty,1)$-categories preserves fibered products.  In fact, it preserves all limits \cite{RV2014}.  Thus it suffices to check that the functor $\Alg_d(-)_{(i)}$ from symmetric monoidal $(\infty,1)$-categories to spaces preserves limits.
  
As we saw above, the functor $\operatorname{Und}$ is a right adjoint to the envelope $\operatorname{Env}$. Being a right adjoint, it preserves $\infty$-categorical limits \cite[Proposition 5.2.3.5]{LurieHTT}.
  Theorem 4.2.4.1 from \cite{LurieHTT} identifies $\infty$-categorical limits with homotopy limits, which are precisely the limits tested by the representable functors $\hom(X,-)$ \cite[Remark A.3.3.13]{LurieHTT}. In particular, the functor $$\maps^h_{\text{$\Delta^{\times d}$-$\infty$-\cat{Op}}}(\Delta^{\times d,\op}_{/(i)},-) : \{\text{$\Delta^{\times d}$-$\infty$-operads}\} \to \cat{Spaces}$$ preserves ($\infty$-categorical limits aka homotopy) limits.  Thus $\Alg_d(-)_{(i)}$, being a composition of limit-preserving functors, preserves limits.
\end{proof}

\begin{remark}
For the purposes of this paper, we care only about the algebras for an operad among symmetric monoidal $(\infty,1)$-categories, but the results  generalize if one works with $E_d$-monoidal categories $\cC$ using the more general machinery of generalized $\Delta^{\times d}$-operads. 
\end{remark}

\Subsection{The pointed version from \cite{ClaudiaDamien2}}

We turn now to the construction of $\Alg_d$ from \cite{ClaudiaDamien2}.
The basic geometric input used 
 are stratifications of $\bR^d$ of very simple types, namely, stratifications whose components are finite unions of complements of finite intersections of hyperplanes given by the equations $x_i = s^i_j$ for some $s^i_j\in\bR$.  Some examples when $d = 2$ and $3$ are:
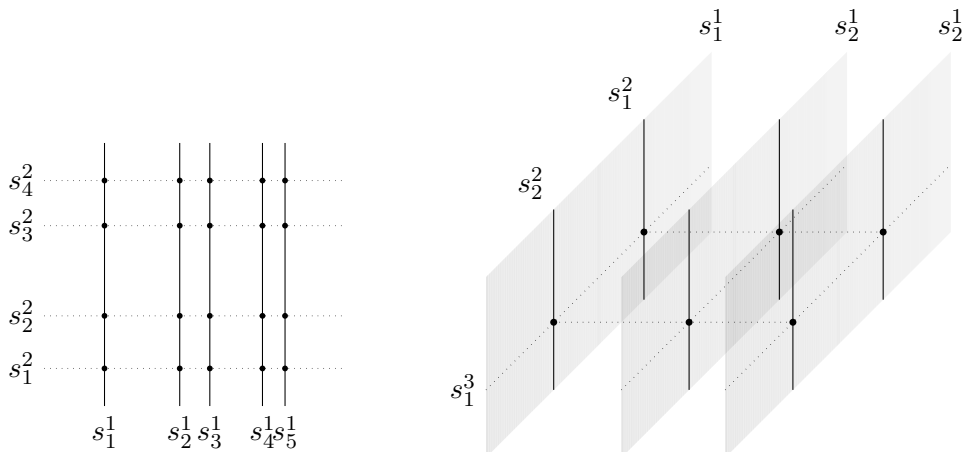
\begin{center}
\begin{tikzpicture}
\draw (0.8, -0.5) node [anchor=north] {$s^1_1$} -- (0.8,3);
\draw (1.8, -0.5) node [anchor=north] {$s^1_2$} -- (1.8, 3);
\draw (2.2, -0.5) node [anchor=north] {$s^1_3$} -- (2.2, 3);
\draw (2.9, -0.5) node [anchor=north] {$s^1_4$} -- (2.9, 3);
\draw (3.2, -0.5) node [anchor=north] {$s^1_5$} -- (3.2, 3);

\draw[dotted, very thin] (0,0) node [anchor=east] {$s^2_1$} -- (4,0);
\draw[dotted, very thin] (0,0.7) node [anchor=east] {$s^2_2$} -- (4,0.7);
\draw[dotted, very thin] (0,1.9) node [anchor=east] {$s^2_3$} -- (4,1.9);
\draw[dotted, very thin] (0,2.5) node [anchor=east] {$s^2_4$} -- (4,2.5);

\fill (0.8,0) circle (0.1em);
\fill (1.8,0) circle (0.1em);
\fill (2.2,0) circle (0.1em);
\fill (2.9,0) circle (0.1em);
\fill (3.2,0) circle (0.1em);

\fill (0.8,0.7) circle (0.1em);
\fill (1.8,0.7) circle (0.1em);
\fill (2.2,0.7) circle (0.1em);
\fill (2.9,0.7) circle (0.1em);
\fill (3.2,0.7) circle (0.1em);

\fill (0.8,1.9) circle (0.1em);
\fill (1.8,1.9) circle (0.1em);
\fill (2.2,1.9) circle (0.1em);
\fill (2.9,1.9) circle (0.1em);
\fill (3.2,1.9) circle (0.1em);

\fill (0.8,2.5) circle (0.1em);
\fill (1.8,2.5) circle (0.1em);
\fill (2.2,2.5) circle (0.1em);
\fill (2.9,2.5) circle (0.1em);
\fill (3.2,2.5) circle (0.1em);
\end{tikzpicture}
\hspace{1cm}
\begin{tikzpicture}[scale=0.6]
\fill[shading=axis,shading angle=90, ultra nearly transparent] (0,0) -- (5,5) -- (5,9)  -- (0,4);
\draw (5,9) node[anchor=south] {$s^1_1$};
\fill[shading=axis,shading angle=90, ultra nearly transparent] (3,0) -- (8,5) -- (8,9) -- (3,4);
\draw (8,9) node[anchor=south] {$s^1_2$};
\fill[shading=axis,shading angle=90, ultra nearly transparent] (5.3,0) -- (10.3,5) -- (10.3,9) -- (5.3,4);
\draw (10.3,9) node[anchor=south] {$s^1_2$};

\draw (3.5, 3.5) -- (3.5, 7.5);
\draw (6.5, 3.5) -- (6.5, 7.5);
\draw (8.8, 3.5) -- (8.8, 7.5);
\draw (3.5, 7.5) node[anchor=south east] {$s^2_1$};

\draw (1.5, 1.5) -- (1.5, 5.5);
\draw (4.5, 1.5) -- (4.5, 5.5);
\draw (6.8, 1.5) -- (6.8, 5.5);
\draw (1.5, 5.5) node[anchor=south east] {$s^2_2$};

\draw (0, 1.5) node[anchor=east] {$s^3_1$};
\draw[dotted, very thin] (0, 1.5) -- (5, 6.5);
\fill (1.5, 3) circle (0.2em);
\fill (3.5, 5) circle (0.2em);

\draw[dotted, very thin] (3, 1.5) -- (8, 6.5);
\fill (4.5, 3) circle (0.2em);
\fill (6.5, 5) circle (0.2em);

\draw[dotted, very thin] (5.3, 1.5) -- (10.3, 6.5);
\fill (6.8, 3) circle (0.2em);
\fill (8.8, 5) circle (0.2em);

\draw[dotted, very thin] (1.5, 3) -- (6.8, 3);
\draw[dotted, very thin] (3.5, 5) -- (8.8, 5);
\end{tikzpicture}
\end{center}

For each $\vec k\in \bN^d$, there is a contractible space of possible stratifications with at most $k_i$ hyperplanes in the $i$th direction of $\bR^d$. To organize this data into an $(\infty,n)$-category, \cite{ClaudiaDamien2} defines a contractible complete $d$-fold Segal space $\Covers^d_{\vec\bullet}$ such that any element in $\Covers^d_{\vec k}$ gives a stratification of $\bR^d$. To simplify notation for the purposes here, we will denote by $(\bR^d)^\stratified_{\vec I}$ the manifold $\bR^d$ together with the stratification coming from an element $\vec I \in \Covers^d_{\vec k}$.

Given any stratified manifold $M$ and any symmetric monoidal $(\infty,1)$-category $\cC$, there is a space $\Fact^{lc}(M,\cC)$ of $\cC$-valued factorization algebras on $M$ which are locally constant with respect to the stratification.  We refer the reader to \cite{ClaudiaDamien2} or \cite{Ginot} for the precise definition of ``factorization algebra locally constant with respect to a stratification,'' which also might be called ``constructible.''  

  We will use only the following property of the definition: there is a (strict) colored operad $\cat{Opens}(M)$ whose algebras in $\cC$ are \define{$\cC$-valued prefactorization algebras on $M$}; a prefactorization algebra $F : \cat{Opens}(M) \to \cC$ is \define{factorization algebra locally constant with respect to the stratification} if and only if certain diagrams in $\cC$ determined by $F$ are colimits (some coming from a gluing condition, others from the local constancy condition). Since these diagrams are all sifted, the assignment $\cC \mapsto \Fact^{lc}(M,\cC)$ is functorial for $\otimes$-sifted-cocontinuous functors between $\otimes$-sifted-cocomplete $(\infty,1)$-categories.
  
\begin{definition}
For $\vec k\in \bN^d$, the space $ \Alg_d(\cC)_{\vec k}$ is the total space of the bundle over $\Covers^d_{\vec k}$ whose fiber over some element $\vec I \in \Covers^d_{\vec k}$ is $\Fact^{lc}\bigl((\bR^d)^\stratified_{\vec I},\cC\bigr)$.  (For details on the structure of $ \Alg_d(\cC)_{\vec k}$ as a topological space, see \cite{ClaudiaDamien2}.)
\end{definition}

Among the points $\vec I \in \Covers^d_{\vec k}$ are some for which the corresponding stratification is ``maximal'' --- in which there are exactly $k_i$ hypersurfaces in the $i$th direction, and not fewer.  For any such $\vec I$, the inclusion of the fiber over $\vec I$ 
  $$ \Fact^{lc}\bigl((\bR^d)^\stratified_{\vec I},\cC\bigr) \hookrightarrow \Alg_d(\cC)_{\vec k} $$
  is an equivalence. This is because there is a deformation retract rescaling the stratifications appearing on the right hand side to the fixed stratification $(\bR^d)^\stratified_{\vec I}$ in a compatible way.

\begin{proposition} \label{end of the proof CS version}
For fixed $\vec k\in \mathbb{N}^d$, the functor $\Alg_d(-)_{\vec k}$ from $\otimes$-sifted-cocomplete symmetric monoidal $(\infty,1)$-categories and $\otimes$-sifted-cocontinuous functors to spaces as defined by \cite{ClaudiaDamien2} preserves limits, and thus in particular fibered products.
\end{proposition}

\begin{proof}
First, consider the $(\infty,1)$-category $\cat{Cat}^\otimes$ of (not necessarily cocomplete) symmetric monoidal $(\infty,1)$-categories and (not necessarily cocontinuous) symmetric monoidal functors. By Lemma 5.4.5.5 in \cite{LurieHTT} and the fact that fibered products of symmetric monoidal $(\infty,1)$-categories are computed levelwise, the fibered product computed in $\cat{Cat}^\otimes$ of $\otimes$-sifted-cocomplete $(\infty,1)$-categories along $\otimes$-sifted-cocontinuous functors again is an $\otimes$-sifted-cocomplete $(\infty,1)$-category. In fact, by \cite{RV2014}, the same is true for any limits. Moreover, the projection maps are $\otimes$-sifted-cocontinuous. 

Now let $\vec I \in \Covers^d_{\vec k}$ determine a ``maximal'' stratification. The equivalence $\Fact^{lc}\bigl((\bR^d)^\stratified_{\vec I},\cC\bigr) \simeq \Alg_d(\cC)_{\vec k}$ is natural in $\cC$, and so it suffices to prove that $\Fact^{lc}\bigl((\bR^d)^\stratified_{\vec I},-\bigr)$ preserves limits.  We will prove more generally that for any stratified manifold $M$, the functor $\Fact^{lc}(M,-)$ from $\otimes$-sifted-cocomplete categories and $\otimes$-sifted-cocontinuous functors to spaces preserves limits.

For a symmetric monoidal $(\infty,1)$-category $\cC$, let $\operatorname{PreFact}(M,\cC)$ denote the space of $\cat{Open}(M)$-algebras in $\cC$.  Then $\operatorname{PreFact}(M,-) = \Alg_{\cat{Open}(M)}(-)$ preserves limits by the same argument as in the proof of \cref{end of the proof Haugseng version}. 

It remains to see that a (family of) prefactorization algebra(s) valued in a fibered product is a (family of) locally constant factorization algebra(s) if and only if its images in each component are locally constant factorization algebras.
Again we use Lemma 5.4.5.5 in \cite{LurieHTT} and \cite{RV2014}: a diagram in the limit is a colimit of fixed shape if and only if its projections are. Applying this lemma to the diagrams governing the gluing condition and the local constancy proves the result.
\end{proof}

\appendix

\section{Some model categories}

Model categories model homotopy theories; the objects of the modeled theory are the fibrant-cofibrant objects of the model.  Consistent with this philosophy, we will continue to use the word \emph{space} to mean Kan complex, but switch attention now to the \define{model category of spaces}, which we take to be Quillen's model structure on the category $\cat{sSet}$ of simplicial sets; we will henceforth let $\cS = \cat{sSet}_{\mathrm{Quillen}}$ denote this model category.  

There are many competing model-categorical models for the homotopy theory of $(\infty,n)$-categories, some of which are reviewed in~\cite{BSP2011}.  The goal of this appendix is to recall the necessary theory of localizations needed to state the ``complete $n$-fold Segal space'' model.  Following~\cite{Horel2014}, we will focus on the projective version of that model, although this makes no substantive difference (see \cref{lemma.equivalence of inj and proj}); note that \cite{Rezk} uses instead the injective or Reedy model.
We make no claims as to the novelty of any result in this appendix; our goal is only to provide a crimp sheet for those who do not know everything that is ``known to experts.''
The standard reference on localizations of model categories is~\cite{MR1944041}, which consistently uses the projective model on presheaf categories.

We recall first some standard definitions.

\begin{definition}
  A \define{simplicial model category} is a model category $\cM$ whose underlying strict category is enriched, powered, and tensored over $\cS$ such that for every cofibration $i : A \to B$ and every fibration $p : X \to Y$, the morphism $i^* \times p_* : \maps(B,X) \to \maps(A,X) \times_{\maps(A,Y)} \maps(B,Y)$ in $\cS$ is a fibration, and is moreover a weak equivalence if either $i$ or $p$ is.  Here $\maps(-,-)$ denotes the $\cS$-valued hom object determined by the enrichment.
\end{definition}

One can give a general definition of ``{derived mapping spaces}'' in any model category.  The main advantage of simplicial model categories is that the enriched mapping space already furnishes such a notion.  Indeed, if $\cM$ is a simplicial model category, $A\in \cM$ is cofibrant, and $X\in \cM$ is fibrant, then $\maps(A,X) \in \cS$ is automatically a Kan complex satisfying all necessarily conditions.  For general $A, X\in \cM$, the \define{derived mapping space} $\maps^h(A,X)$ is the space $\maps(\widetilde A ,\widehat X)$, where $\widetilde A \to A$ is a cofibrant replacement of $A$ and $X \to \widehat X$ is a fibrant replacement of $X$.  Different choices for the replacements lead to equivalent derived mapping spaces.

That $\cM$ is enriched in $\cS$ assures that $\maps(A,X)$ is an object of $\cS$ when $A,X\in \cM$.  That $\cM$ is \define{powered} in $\cS$ means for any $A\in \cS$ and $X\in \cM$, there is a \define{mapping object} $\maps(A,X)\in \cM$ depending functorially on $A$ and $X$ and with all the usual properties.  The \define{derived} mapping object $\maps^{h}(A,X)$ is defined in terms of fibrant and cofibrant replacements as in the derived mapping space.

\begin{definition}\label{defn.global model structure}
  Given a small category $\Phi$ and any category $\cM$, we will write $\cM^{\Phi}$ for the category of $\cM$-valued presheaves on $\Phi$.  (We will not use the category of covariant functors $\Phi \to \cM$, which we would write as $^{\Phi}\cM$ if we needed it; the notation is consistent the slogan that covariant functors are like left modules whereas contravariant functors are like right modules.)  Given $X\in \cM^{\Phi}$ and $\phi \in \Phi$, we will write $X_{\phi}$ for the value of the presheaf $X$ at $\phi$.
  
  Let $\cM$ be a model category.  If it exists, the \define{projective} model structure $\cM^{\Phi}_{\proj}$ is determined by declaring that a natural transformation $f : A \to X$ is a fibration (resp.\ weak equivalence) if for each $\phi\in \Phi$, the map $f_{\phi} : A_{\phi} \to X_{\phi}$ is a fibration (resp.\ weak equivalence).  If it exists, the \define{injective} model structure $\cM^{\Phi}_{\inj}$ is determined by declaring that a natural transformation $f : A \to X$ is a cofibration (resp.\ weak equivalence) if for each $\phi\in \Phi$, the map $f_{\phi} : A_{\phi} \to X_{\phi}$ is a cofibration (resp.\ weak equivalence).
\end{definition}

We will care most about the case $\cM = \cS$, in which case both the projective and injective model structures exist for any small category $\Phi$ (see~\cite[Proposition~A.2.8.2]{LurieHTT} for a recent treatment). 
We recall a few more existence results in \cref{lemma.basic facts about localizations} and \cref{remark.combinatorial}.    

Suppose that $\cM$ is simplicial.  Then the powering of $\cM$ over $\cS$ implies further that for any $A \in \cS^{\Phi}$ and $X \in \cM^{\Phi}$, there is a mapping object $\maps(A,X) \in \cM$.  The corresponding \define{derived mapping object} is $\maps^{h}(A,X) = \maps(\widetilde A,\widehat X) \in \cM$, where $\widetilde A \to A$ is a cofibrant replacement in $\cS^{\Phi}$ and $X \to \widehat X$ is a fibrant replacement in $\cM^{\Phi}$, assuming at least one of the projective or injective model structures on $\cM^{\Phi}$ exists.  Note that if both model structures exist, then the derived mapping object is independent (up to canonical equivalence) of the choice.  Indeed, projective cofibrancy implies injective cofibrancy and injective fibrancy implies projective fibrancy, so one may assume $\widetilde A$ to be projective-cofibrant and $\widehat X$ to be injective-fibrant.

\begin{definition}
  A \define{presentation} consists of a small category $\Phi$ and a set of maps $\cE$ in $\cS^{\Phi}$.  Let $\cM$ be a simplicial model category such that at least one of 
  $\cM^\Phi_\proj$ and $\cM^\Phi_\inj$ exists.  An object $X\in \cM^{\Phi}$ is \define{$\cE$-local} if for each element $(A\overset e \to B) \in \cE$, the induced map $e^{*} : \maps^{h}(B,X) \to \maps^{h}(A,X)$ is a weak equivalence in $\cM$.
  We will refer to $\cE$-local objects in $\cM^{\Phi}$ as \define{$(\Phi,\cE)$-objects in $\cM$}.
\end{definition}

Let us call a presentation $(\Phi,\cE)$ \define{discrete} if every morphism in $\cE$ is from the full subcategory $\cat{Set}^{\Phi} \hookrightarrow \cS^{\Phi}$.  Given a complete strict category $\cM$, a \define{strict $(\Phi,\cE)$-object of $\cM$} is an object of $\cM^{\Phi}$ which is \define{$\cE$-local} in the sense that each $(A\overset e \to B) \in \cE$ induces an isomorphism $\hom(B,X) \overset{e^{*}}\to \hom(A,X)$ --- these objects make sense since $\cM$ is complete.  Let $\yo : \Phi \to \cat{Set}^{\Phi}$ denote the Yoneda embedding (\yo\ being the first letter of ``Yoneda'' in Hiragana).  
 Every presheaf is a colimit of representable presheaves, and so asking for a presheaf to be $\cE$-local is the same as asking that certain limits built from its values be isomorphic.
  
\begin{example} \label{eg.Gamma}
  Recall Segal's category $\Gamma = \Fin_{*}^{\mathrm{op}}$ from \cref{defn.Gamma} and \cref{remark.Gamma}.  We make it into a discrete presentation by declaring the distinguished class of maps $\cE$ to consist of the inclusions $\yo([1])^{\sqcup k} \hookrightarrow \yo([k])$, where  the $i$th map $\yo([1]) \to \yo([k])$ is the $i$th Segal morphism~$\gamma_{i}$.  The strict $(\Gamma,\cE)$ objects in a complete category $\cM$ are precisely the commutative monoids in~$\cM$.
\end{example}

As \cref{eg.Gamma} illustrates, presentations provide a framework for universal algebra, generalizing Lawvere algebraic theories.  The $(\Phi,\cE)$-objects in a model category $\cM$ are a homotopic avatar of strict $(\Phi,\cE)$-objects.  Indeed, the homotopy theory of $(\Phi,\cE)$-objects can be readily modeled by a special case of left Bousfield localization:

\begin{definition}
  Let $\cM$ be a simplicial model category.  If it exists, the \define{projective  (resp.\ injective) model category of  $(\Phi,\cE)$-objects in $\cM$}, denoted $\rL_{\cE}\cM^{\Phi}_{\proj}$ (resp.\ $\rL_{\cE}\cM^{\Phi}_{\inj}$), is the model category structure on $\cM^{\Phi}$ whose cofibrations are the same as those in $\cM^{\Phi}_{\proj}$ (resp.\ $\cM^{\Phi}_{\inj}$) and for which a morphism $f:A \to B$ is a weak equivalence if the induced map of spaces $f^{*} : \maps^{h}(B,X) \to \maps^{h}(A,X)$ is a weak equivalence for every $\cE$-local object in $X \in \cM^{\Phi}_{\proj}$ (resp.\ $X \in \cM^{\Phi}_{\inj}$).
\end{definition}

\begin{lemma}\label{lemma.basic facts about localizations}
  Suppose that $\cM$ is a left proper cellular simplicial model category.  
  Then $\rL_\cE\cM^\Phi_\proj$ exists and is also left proper cellular simplicial.
  The fibrant objects of $\rL_{\cE}\cM^{\Phi}_{\proj}$ are precisely the  $\cE$-local objects in $\cM^{\Phi}$ that are fibrant in $\cM^{\Phi}_{\proj}$.  Weak equivalences in $\rL_{\cE}\cM^{\Phi}_{\proj}$ between fibrant objects are levelwise weak equivalences.  
  The simplicial structure on $\rL_\cE\cM^\Phi_\proj$ is the simplicial enrichment of $\cM^\Phi_{\proj}$, so that derived mapping spaces in $\rL_\cE\cM^\Phi_\proj$ can be computed as mapping spaces in $\cM^\Phi$ between cofibrant and fibrant replacements.
\end{lemma}

For the technical notion of ``left proper cellular'' we refer the reader to~\cite{MR1944041}, which uses always the projective model structure on presheaf categories.  Note in particular that $\cS$ itself is left proper cellular simplicial~\cite[Proposition 4.1.4]{MR1944041}.

\begin{proof}
  These claims are special cases of \cite[Theorem 4.1.1 and Proposition 4.1.7]{MR1944041}.
\end{proof}

\begin{remark} \label{remark.combinatorial}
One can replace ``cellular'' by the related notion of ``combinatorial,'' in which case all the same statements  hold for both $\rL_{\cE}\cM^{\Phi}_\proj$ and $\rL_{\cE}\cM^{\Phi}_\inj$~\cite[Propositions~A.2.8.2 and~A.3.7.3]{LurieHTT}. 
\end{remark}

We can talk about ``the homotopy theory of $(\Phi,\cE)$-objects in $\cM$'' without deciding whether to work with the projective or injective models:

\begin{lemma}\label{lemma.equivalence of inj and proj}
  Let $\cM$ be a model category and $(\Phi,\cE)$ a presentation such that the model categories $\rL_{\cE}\cM^{\Phi}_{\proj}$ and $\rL_{\cE}\cM^{\Phi}_{\inj}$ exist.  The identity functor is a left Quillen equivalence $\id : \rL_{\cE}\cM^{\Phi}_{\proj} \to \rL_{\cE}\cM^{\Phi}_{\inj}$.
\end{lemma}
\begin{proof}
This follows from \cite[Theorem~3.3.20]{MR1944041} and the fact $\cM^{\Phi}_{\inj}$ and $\cM^{\Phi}_{\proj}$ have the same weak equivalences.
\end{proof}

When $\cM$ is itself of the form $\rL_{\cF}\cS^{\Psi}_{\proj}$ for some presentation $(\Psi,\cF)$, then $\rL_{\cE}\cM^{\Phi}_{\proj}$ is as well:

\begin{proposition}\label{prop.AinB}
  Let $(\Psi,\cF)$ and $(\Phi,\cE)$ be presentations.  
  Define a new presentation $(\Psi,\cF)\boxtimes (\Phi,\cE)$ whose underlying category is the product of categories $\Psi \times \Phi$ and whose set of distinguished maps in $\cS^{\Psi \times\Phi}$ is the set
  $$\{f\boxtimes\id_{\yo(\phi)}\}_{(f,\phi)\in \cF\times \Phi} \cup \{\id_{\yo(\psi)}\boxtimes e \}_{( \psi ,e) \in \Psi \times\cE},$$
  where $\yo$ denotes the Yoneda embedding and $\boxtimes : \cS^{\Psi}\times \cS^{\Phi} \to \cS^{\Psi \times \Phi}$ denotes the functor sending the pair of presheafs $(A,B)$ to the presheaf $(A\boxtimes B):(\psi,\phi) \mapsto A(\psi)\times B(\phi)$.
  Under the canonical equivalence of categories $(\cS^{\Phi})^{\Psi} \cong \cS^{\Phi\times\Psi}$, the model structures defining $\rL_{\cF}(\rL_{\cE}\cS^{\Phi}_{\proj})^{\Psi}_{\proj}$ and $\rL_{\cE\boxtimes\cF}\cS^{\Phi \times \Psi}_{\proj}$ agree.
\end{proposition}

\begin{proof}
  The model categories $\rL_{\cF}(\rL_{\cE}\cS^{\Phi}_{\proj})^{\Psi}_{\proj}$ and $\rL_{\cE\boxtimes\cF}\cS^{\Phi \times \Psi}_{\proj}$ have  cofibrations inherited from $(\cS^{\Phi}_{\proj})^{\Psi}_{\proj}$ and $\cS^{\Phi \times \Psi}_{\proj}$ respectively; but these are the same as model categories, since in both categories the weak equivalences and fibrations are computed levelwise.  It thus suffices to show~\cite[Proposition E.1.10]{Joyal2008} that $\rL_{\cF}(\rL_{\cE}\cS^{\Phi}_{\proj})^{\Psi}_{\proj}$ and $\rL_{\cE\boxtimes\cF}\cS^{\Phi \times \Psi}_{\proj}$ have the same fibrant objects.
  
  Consider first a fibrant object of $\rL_{\cF}(\rL_{\cE}\cS^{\Phi}_{\proj})^{\Psi}_{\proj}$.  It is just a fibrant object of $(\rL_{\cE}\cS^{\Phi}_{\proj})^{\Psi}_{\proj}$ which is $\cF$-local in the sense that each $f\in \cF$ induces a weak equivalence of $\rL_{\cE}\cS^{\Phi}_{\proj}$-valued derived mapping spaces.  Derived mapping spaces are always fibrant-cofibrant, and so the asserted weak equivalences are also weak equivalences in $\cS^{\Phi}_{\proj}$, which is to say levelwise weak equivalences.  Thus $\cF$-locality in $(\rL_{\cE}\cS^{\Phi}_{\proj})^{\Psi}_{\proj}$ unpacks to locality in $\cS^{\Phi \times \Psi}_{\proj}$ for the set $\{f\boxtimes\id_{\yo(\phi)}\}_{(f,\phi)\in \cF\times \Phi} $.  On the other hand, the fibrant objects in $(\rL_{\cE}\cS^{\Phi}_{\proj})^{\Psi}_{\proj}$ are the objects that are levelwise-in-$\Psi$ fibrant-in-$\rL_{\cE}\cS^{\Phi}_{\proj}$.  But fibrancy in $\rL_{\cE}\cS^{\Phi}_{\proj}$ is the same as fibrancy in $\cS^{\Phi}_{\proj}$ and $\cE$-locality.  Together, we see that the fibrant objects of $(\rL_{\cE}\cS^{\Phi}_{\proj})^{\Psi}_{\proj}$ are the levelwise fibrant objects of $\cS^{\Phi \times \Psi}_{\proj}$ that are local for the set $\{\id_{\yo(\psi)}\boxtimes e \}_{( \psi ,e) \in \Psi \times\cE}$.  
\end{proof}

An immediate corollary of \cref{prop.AinB} is that (via the canonical equivalence $\cS^{\Phi\times\Psi} \cong \cS^{\Psi\times\Phi}$) the model categories $\rL_{\cF}(\rL_{\cE}\cS^{\Phi}_{\proj})^{\Psi}_{\proj}$ and $\rL_{\cE}(\rL_{\cF}\cS^{\Psi}_{\proj})^{\Phi}_{\proj}$ agree.
This can be summarized by the following slogan: $(\Psi,\cF)$-objects among $(\Phi,\cE)$-objects are the same as $(\Phi,\cE)$-objects among $(\Psi,\cF)$-objects, as both are the same as simultaneously-$(\Phi,\cE)$-and-$(\Psi,\cF)$-objects.

We conclude with our main examples, which are the presentations presenting complete $n$-fold Segal spaces and symmetric monoidal complete $n$-fold Segal spaces:

\begin{example}\label{eg.Rezk}
  Let $\Delta$ denote the usual simplex category with objects $[k] = \{0,1,\dots,k\}$ for $k\in \bN$ and nondecreasing maps.
  We will define a discrete presentation $(\Delta,\cF)$ that we will call the \define{Rezk presentation}.  The set $\cF\subseteq \cat{Set}^\Delta$ consists of the following maps:
  \begin{description}
    \item[Segal]   For each $k$ and each $0 < i \leq k$, the \define{Segal morphism} $\gamma_i : [1] \to [k]$ is the inclusion along the edge connecting vertex $i-1$ to vertex $i$.  The set $\cF$ includes the following map for each~$k$:
    $$  \ourunderbrace{\yo([1]) \underset{\yo([0])}\cup \yo([1]) \underset{\yo([0])}\cup \dots \underset{\yo([0])}\cup\yo([1])}{k\text{ times}} \overset{\bigcup_{i=1}^k \yo(\gamma_i) }\longto \yo([k]) $$
    The gluings of edges $\yo([1])$ along vertices $\yo([0])$ always identify the second vertex of the left copy with the first vertex of the right copy.  We will call this map $\operatorname{Segal}(k)$.
    \item[Completeness]  The set $\cF$ includes one final map:
    $$  \yo([0]) \underset{\yo([1])}\cup \yo([3]) \underset{\yo([1])}\cup  \yo([0]) \longto \yo([0]), $$
    where the first inclusion $\yo([1]) \to \yo([3])$ sends $0\mapsto 0$ and $1\mapsto 2$, and the second inclusion  sends $0\mapsto 1$ and $1\mapsto 3$.
     (The maps to $\yo([0])$ are all forced, since $\yo([0])$ is the terminal object in $\cat{Set}^\Delta$.)
     Following~\cite{BSP2011}, we let $K$ denote the simplicial set on the left hand side.
  \end{description}
  
  To compute the  $\cF$-local objects in $\cS$ (in the projective model structure), one must compute projective-cofibrant resolutions of the left hand sides; note that the representable presheaf $\yo(\phi)$ is  cofibrant in $\cS^\Phi_\proj$ for any $\Phi$ and any object $\phi$ therein.  One finds precisely the complete Segal spaces in the sense of \cref{defn.segalspace,defn.1fold complete}.  The main result of~\cite{Rezk} is that $\rL_\cF\cS^\Delta$ models the homotopy theory of $(\infty,1)$-categories.  (Rezk uses the Reedy model structure, which for $\Delta$ is equal to the injective model structure, but that choice is Quillen-equivalent to the projective choice by \cref{lemma.equivalence of inj and proj}.)
  
  By applying \cref{prop.AinB}, we can write down a model for complete $n$-uple Segal spaces as the model category of $(\Delta,\cF)^{\boxtimes n}$-spaces.  Explicitly, it is the localization of $\cS^{\Delta^{\times n}}$ at the following maps:
  \begin{description}
    \item[Segal$^{\boxtimes n}$]  For each $k\in \bN$, each $i \in \{1,\dots,n\}$, and each $m_1,\dots,m_{i-1},m_{i+1},\dots,m_{n} \in \bN$, we localize at the map
    $$ \yo([m_1]) \boxtimes \dots \boxtimes \yo([m_{i-1}]) \boxtimes \operatorname{Segal}(k) \boxtimes \yo([m_{i+1}]) \boxtimes \dots \boxtimes \yo([m_n]), $$
    where $\operatorname{Segal}(k) =  \bigcup_{i=1}^k \yo(\gamma_i) \in \cF$ is as above.
    \item[Completeness$^{\boxtimes n}$]  For each $i \in \{1,\dots,n\}$ and each $m_1,\dots,m_{i-1},m_{i+1},\dots,m_{n} \in \bN$, we localize at the map
    $$ \yo([m_1]) \boxtimes \dots \boxtimes \yo([m_{i-1}]) \boxtimes \bigl\{K \to \yo([0])\bigr\} \boxtimes \yo([m_{i+1}]) \boxtimes \dots \boxtimes \yo([m_n]), $$
    where the simplicial set $K$ is as above.
  \end{description}
  To get a model of complete $n$-fold Segal spaces as defined in \cref{defn.complete}, we further localize at:
  \begin{description}
    \item[Essential constancy] For each $i \in \{1,\dots,n\}$ and each $m_1,\dots,m_{i-1},m_{i+1},\dots,m_{n} \in \bN$, we localize at the map
    \begin{multline*} \yo([m_1]) \boxtimes \dots \boxtimes \yo([m_{i-1}]) \boxtimes \yo([0]) \boxtimes \yo([m_{i+1}]) \boxtimes \dots \boxtimes \yo([m_{n}]) \\ \longto \yo([m_1]) \boxtimes \dots \boxtimes \yo([m_{i-1}]) \boxtimes \yo([0]) \boxtimes \yo([0]) \boxtimes \dots \boxtimes \yo([0]). \end{multline*}
  \end{description}
  We will call this the \define{Lurie presentation} of complete $n$-fold Segal spaces $(\Delta^{\times n},\cF_n)$.  
  
  In \cite{BSP2011}, yet another presentation with underlying category $\Delta^{\times n}$ is used to model $(\infty,n)$-categories (see Remark 12.4 therein).  This \define{Barwick presentation} consists of the sets \textbf{Segal$^{\boxtimes n}$} and \textbf{Essential constancy} as above, but replaces the set \textbf{Completeness$^{\boxtimes n}$} by the subset consisting of maps
  $$ \yo([m_1]) \boxtimes \dots \boxtimes \yo([m_{i-1}]) \boxtimes \bigl\{K \to \yo([0])\bigr\} \boxtimes \yo([0]) \boxtimes \dots \boxtimes \yo([0])$$
  for $i \in \{1,\dots,n\}$ and $\vec m \in \bN^{i-1}$.  But this smaller Barwick presentation is equivalent to the larger Lurie presentation by \cref{lemma.weaker completeness condition}.
  One of the main theorems of~\cite{BSP2011} is that the Barwick presentation presents a model of $(\infty,n)$-categories.
\end{example}

\begin{example}\label{eg.model category of symmetric monoidal $n$-fold complete Segal spaces}
One can always use \cref{prop.AinB} to combine presentations.  The model category of \define{complete $n$-fold Segal spaces internal to complete $m$-fold Segal spaces} is
the category of local objects in $\cS$ for the presentation $(\Delta^{\times n},\cF_n) \boxtimes (\Delta^{\times m},\cF_m)$.  The model category of \define{symmetric monoidal complete $n$-fold Segal spaces} corresponds to the presentation $(\Gamma,\cE) \boxtimes (\Delta^{\times n},\cF_n)$, where $(\Gamma,\cE)$ is as in \cref{eg.Gamma}.
\end{example}


\end{document}